\pdfoutput=1
\documentclass[reqno]{amsart}
\usepackage[utf8]{inputenc}

\usepackage[margin=1in]{geometry}
\usepackage[expansion=false]{microtype}

\usepackage{amsmath, amsfonts, amssymb, amsthm, amsopn, mathtools}
\usepackage{eucal}

\usepackage{tikz-cd}
\usetikzlibrary{patterns}
\usepackage{bm}
\usepackage{rotating}
\usepackage{xparse}
\usepackage{pict2e}

\usepackage[pdfusetitle,unicode,colorlinks]{hyperref}

\numberwithin{equation}{subsection}

\theoremstyle{plain}
\newtheorem{theorem}[equation]{Theorem}
\newtheorem{proposition}[equation]{Proposition}

\newtheorem{lemma}[equation]{Lemma}
\newtheorem{corollary}[equation]{Corollary}

\makeatletter
\newtheorem*{rep@theorem}{\rep@title}
\newcommand{\newreptheorem}[2]{%
\newenvironment{rep#1}[1]{%
 \def\rep@title{#2 \ref{##1}}%
 \begin{rep@theorem}}%
 {\end{rep@theorem}}}
\makeatother
\newreptheorem{theorem}{Theorem}

\theoremstyle{definition}
\newtheorem{definition}[equation]{Definition}
\newtheorem{example}[equation]{Example}
\newtheorem{remark}[equation]{Remark}

\newtheorem{construction}[equation]{Construction}
\newtheorem{warning}[equation]{Warning}

\makeatletter
\def\th@definition{%
	\thm@notefont{\mathversion{normal}\normalfont}
	\normalfont 
}
\def\th@plain{%
	\thm@notefont{\mathversion{normal}\normalfont}
	\itshape 
}
\makeatother

\let\scr=\mathcal
\let\bb=\mathbf
\let\phi=\varphi

\let\into=\hookrightarrow
\let\onto=\twoheadrightarrow

\def\AA{\scr A}
\def\BB{\scr B}
\def\CC{\scr C}
\def\DD{\scr D}
\def\EE{\scr E}

\def\GG{\scr G}
\def\II{\scr I}

\def\KK{\scr K}
\def\LL{\scr L}

\def\PP{\scr P}
\def\QQ{\scr Q}
\def\RR{\scr R}
\def\SS{\scr S}

\def\UU{\scr U}

\def\XX{\scr X}

\def\BBB{\widehat{\BB}}

\def\SSS{\widehat{\SS}}

\def\SSSS{\hathat{\SS}}
\def\BBBB{\hathat{\BB}}

\def\bU{\bb{U}}
\def\bV{\bb{V}}
\def\bW{\bb{W}}

\DeclareMathOperator{\id}{id}
\DeclareMathOperator{\ev}{ev}

\DeclareMathOperator{\PSh}{PSh}
\DeclareMathOperator{\Cat}{Cat}
\DeclareMathOperator{\Shv}{Sh}
\DeclareMathOperator{\Cart}{Cart}

\DeclareMathOperator{\RFib}{RFib}
\DeclareMathOperator{\LFib}{LFib}

\DeclareMathOperator{\ILFib}{\mathsf{LFib}}
\DeclareMathOperator{\IRFib}{\mathsf{RFib}}
\DeclareMathOperator{\Sml}{\mathsf{Small}}
\DeclareMathOperator{\ICat}{\mathsf{Cat}}

\DeclareMathOperator{\ILConst}{\mathsf{LConst}}

\DeclareMathOperator{\RPr}{Pr^R}
\DeclareMathOperator{\LPr}{Pr^L}

\DeclareMathOperator{\Sub}{Sub}
\DeclareMathOperator{\SubObj}{SubObj}
\DeclareMathOperator{\Tw}{Tw}
\DeclareMathOperator{\Fun}{Fun}
\DeclareMathOperator{\Map}{map}

\DeclareMathOperator{\Grpd}{Grpd}

\DeclareMathOperator{\Pyk}{Pyk}
\DeclareMathOperator{\const}{const}
\DeclareMathOperator{\diag}{diag}

\DeclareMathOperator{\MH}{H}

\DeclareMathOperator{\Sch}{Sch}
\DeclareMathOperator{\pr}{pr}

\DeclareMathOperator{\cosk}{cosk}
\DeclareMathOperator{\sk}{sk}

\DeclareMathOperator{\ISm}{\mathsf{Sm}}
\DeclareMathOperator{\IPrp}{\mathsf{Prp}}

\DeclareMathOperator{\ISh}{\underline{Sh}}

\newcommand{\ord}[1]{\langle{#1}\rangle}
\newcommand{\map}[1]{\Map_{#1}}

\newcommand{\Over}[2]{#1_{\hspace{-1pt}/#2}}
\newcommand{\Under}[2]{#1_{\hspace{-1pt}#2/}}
\newcommand{\I}[1]{\mathsf{#1}}

\renewcommand{\smallint}{\textstyle\int}
\newcommand{\iFun}[1][\BB]{\underline{\mathsf{Fun}}_{#1}}
\newcommand{\IPSh}[1][\BB]{\underline{\mathsf{PSh}}_{#1}}
\newcommand{\Comma}[3]{{#1}\downarrow_{#2}{#3}}
\newcommand{\gpdcl}[1]{\overline{#1}}

\newcommand{\Simp}[1]{#1_{\Delta}}
\newcommand{\op}{\mathrm{op}}
\newcommand{\core}{\simeq}
\newcommand{\gp}{\mathrm{gpd}}
\newcommand{\CatS}{\Cat_\infty}
\newcommand{\CatSS}{\widehat{\Cat}_\infty}
\newcommand{\cocont}[1]{{#1\text{\normalfont{-cc}}}}
\newcommand{\cc}{\text{\normalfont{cc}}}
\newcommand{\Hom}{\underline{\operatorname{hom}}}

\NewDocumentCommand{\Gen}{m o}{%
	\IfNoValueTF{#2}{%
		\langle #1\rangle%
	}{%
		\langle #1\rangle^{#2}%
	}%
}

\NewDocumentCommand{\Univ}{o}{%
	\IfNoValueTF{#1}{%
		\I{\Omega}%
	}{%
		\I{\Omega}_{#1}%
	}%
}
\NewDocumentCommand{\UnivHat}{o}{%
	\IfNoValueTF{#1}{%
		\widehat{\I{\Omega}}%
	}{%
		\widehat{\I{\Omega}}_{#1}%
	}%
}

\let\lim=\relax
\DeclareMathOperator*{\lim}{lim}
\DeclareMathOperator*{\colim}{colim}

\makeatletter
\g@addto@macro\bfseries{\boldmath}
\makeatother

\makeatletter
\newcommand{\hathatInternal}[2]{%
	\begingroup%
	\let\macc@kerna\z@%
	\let\macc@kernb\z@%
	\let\macc@nucleus\@empty%
	\widehat{\raisebox{#2}{\vphantom{\ensuremath{#1}}}\smash{\widehat{#1}}}%
	\endgroup%
}

\newcommand{\hathat}[1]{\mathchoice
	{\hathatInternal{#1}{.2ex}}
	{\hathatInternal{#1}{.2ex}}
	{\hathatInternal{#1}{-1.5pt}}
	{\hathatInternal{#1}{1pt}}
}

\title{Colimits and Cocompletions in internal higher category theory}
\author{Louis Martini}
\address{Norwegian University of Science and Technology (NTNU)\\
Alfred Getz' vei 1\\
7034 Trondheim\\
Norway}
\email{\href{mailto:louis.o.martini@ntnu.no}{louis.o.martini@ntnu.no}}
\author{Sebastian Wolf}
\address{University of Regensburg\\
Universit\"atsstra{\ss}e 31 \\
93047 Regensburg\\
Germany}
\email{\href{mailto:sebastian1.wolf@mathematik.uni-regensburg.de}{sebastian1.wolf@mathematik.uni-regensburg.de}}
\date{\today}
\begin{document}
\maketitle

\begin{abstract}
    We develop a number of basic concepts in the theory of categories internal to an $\infty$-topos. We discuss adjunctions, limits and colimits as well as Kan extensions for internal categories, and we use these results to establish the universal property of internal presheaf categories. We furthermore construct the free cocompletion of an internal category by colimits that are indexed by an arbitrary class of diagram shapes.
\end{abstract}

\tableofcontents

\section{Introduction}

\subsection*{Motivation}
A main advantage of formalising and proving mathematical results using the language of category theory is that this helps to separate the formal aspects of a mathematical theory from its non-formal core.
Also in homotopical contexts one would like to have the advantage of such a systematic treatment of the formal parts of a theory, and for this reason higher categories have been invented.
This machinery, developed by Joyal, Lurie and many others, gives rise to a language of homotopy-coherent mathematics that allows one to deal with non-trivial coherence issues in an elegant and efficient way.

This paper is the second in a series in which we intend to add to this language by developing categorical tools in the context of higher categories internal to an $\infty$-topos $\BB$.
As a first step towards this goal, Yoneda's lemma for internal higher categories was proven in \cite{martini2021} by the first-named author.
In this article we continue this work with an extensive discussion of adjunctions, (co)limits, Kan extensions and free (co)completions in the internal setting.

A higher category $\I{C}$ internal to an $\infty$-topos $\BB$, which we usually refer to as a \emph{$\BB$-category}, is a simplicial object in $\BB$ satisfying the Segal conditions and univalence (see \cite[Definition 3.2.4]{martini2021} for a precise definition).
The study of $\BB$-categories is equivalent to the study of sheaves of $\infty$-categories on $\BB$. These sheaves arise in many different contexts.
In fact most higher-categorical invariants in modern algebraic geometry and topology are constructed in a functorial way and satisfy suitable descent conditions. As a concrete example from motivic homotopy theory, the assignment that carries a scheme $S$ to the unstable motivic homotopy $\infty$-category $\MH(S)$ defines a sheaf for the Nisnevich topology on the category $\Sch$ of schemes.
Such invariants may be studied using the methods developed in this article. One benefit of approaching such objects from an internal point of view is that one can apply essentially the same line of reasoning as for their unparametrised counterpart. In this way, the added complexity level that arises from considering invariants that are parametrised by some base $\infty$-topos $\BB$ can be hidden by working internally in $\BB$, so that in effect both the parametrised and unparametrised theory can be treated on an equal footing.

On a related note, we took great care to formulate our results in a model-independent way, in the sense that we do not rely on the choice of a particular model for $\infty$-groupoids and thereby work exclusively within a fully homotopy-coherent setup from the very start. In this regard, our approach is similar in spirit to the synthetic theory of higher categories~\cite{Riehl2017b,Buchholtz2023}. 

Conceptually, there are two major areas in which the theory of internal higher categories finds applications: the first is comprised of those branches of homotopy theory in which one needs to replace bare $\infty$-groupoids by objects that admit more structure. For example, in \emph{condensed} or \emph{pyknotic} mathematics~\cite{barwick2019, scholze2019}, the basic objects of interest are no longer bare $\infty$-groupoids but \emph{pyknotic} $\infty$-groupoids, i.e.\ the $\infty$-topos $\SS$ is replaced by the $\infty$-topos $\Pyk(\SS)$ of pyknotic $\infty$-groupoids. Our results therefore make it possible to handle certain aspects in condensed/pyknotic mathematics in essentially the same way as one would handle analogous problems in ordinary homotopy theory. In that way, internalisation may lead to a significant reduction of complexity. As a concrete example, the second-named author has shown that internal to $\Pyk(\SS)$, the pro\'etale $\infty$-topos of a coherent scheme $S$ is simply given by the category of internal copresheaves on the Galois category of $S$~\cite{Wolf2022}. In this case, the process of internalisation has therefore transformed a seemingly more complicated $\infty$-topos into a structurally very simple internal topos. 

The second area in which ideas from internal higher category theory find applications is comprised of those branches of homotopy theory in which the study of categorified invariants plays a prominent role. We have already pointed out above that many of these already form examples of internal higher categories and can therefore be studied using the results that are developed within our framework. As a prominent example, \emph{six functor formalisms}, which are of great interest in areas such as motivic homotopy theory, fit into this picture~\cite{cisinski2019b,Drew2022}. For example, parts of the structure in a six functor formalism corresponds to the condition that the associated internal category admits internal colimits that are indexed by a certain collection of internal groupoids. Therefore, the toolset that we provide in this article for studying internal colimits and internal cocompletions by arbitrary collections of diagram shapes can already be used for studying such structures. See for Example~\cite{bachmann2021} where certain internal colimits in motivic homotopy theory have already been put to use.

The goal of this article is to provide a systematic framework for this area of ideas in order to open the way for future application.
Therefore we put special emphasis on presenting and formulating our results in a way that makes them as easy to use in practice as possible.

\subsection*{Main results}
Let us from now on fix an $\infty$-topos $\BB$.
In \cite{martini2021} the first-named author proved Yoneda's lemma for internal higher categories in $\BB$.
More precisely, it was shown that for any $\BB$-category $\I{C}$ there is a fully faithful functor of $\BB$-categories $h \colon \I{C} \rightarrow \iFun(\I{C}^\op,\Univ)$ such that the functor
\[
	\map {\iFun(\I{C}^\op,\Univ)}(h(-),-) \colon \I{C} \times \iFun(\I{C}^\op,\Univ) \rightarrow \Univ
\]
is equivalent to the evaluation functor.
Here $\Univ$ denotes the \emph{universe} in $\BB$,  which is the internal analogue of the $\infty$-category $\SS$ of $\infty$-groupoids and which is explicitly defined as the $\BB$-category corresponding to the sheaf $\BB_{/-} \colon \BB^\op \rightarrow \CatSS$. Furthermore, $\iFun(-,-)$ denotes the internal hom of $\BB$-categories, so that $\iFun(\I{C}^\op,\Univ)$ is the $\BB$-category of internal presheaves on $\I{C}$.
We will therefore write $\IPSh(\I{C}) = \iFun(\I{C}^\op,\Univ)$ hereafter.
One of the main goals of the present article is to show that the embedding $h\colon \I{C}\into\IPSh(\I{C})$ satisfies the expected universal property:

\begin{reptheorem}{thm:universalPropertyPSh}
	For any small $\BB$-category $\I{C}$ and any cocomplete large $\BB$-category $\I{E}$, the functor of left Kan extension $h_!$ along the Yoneda embedding $h\colon \I{C}\into\IPSh(\BB)$ induces an equivalence
	\begin{equation*}
	h_!\colon\iFun^\cc(\I{C},\I{E})\simeq \iFun(\IPSh(\I{C}),\I{E}),
	\end{equation*}
	where $\iFun^\cc(\IPSh(\I{C}),\I{E})$ denotes the full subcategory of $\iFun(\IPSh(\I{C}),\I{E})$ that is spanned by the cocontinuous functors.
	In other words, the Yoneda embedding $h\colon \I{C}\into\IPSh(\I{C})$ exhibits the $\BB$-category of presheaves on $\I{C}$ as the free cocompletion of $\I{C}$.
\end{reptheorem}

In order to make the above result precise we first have to introduce a good deal of categorical tools in the internal setting.
After discussing a few necessary preliminaries, we start out by studying adjunctions of $\BB$-categories.
The main interesting observation here is that, while one gets a good theory of internal adjunctions which is completely parallel to the case of adjunctions between $\infty$-categories, one can also directly relate these with section-wise adjunctions of the corresponding sheaves on $\BB$ that are compatible with étale base change in a suitable sense.
See Proposition~\ref{prop:existenceAdjointsBeckChevalley} for a precise formulation of this statement.

We then move on to develop the theory of (co)limits in internal higher category theory.
Most of the story is quite analogous to the corresponding theory for $\infty$-categories, but again one can provide explicit section-wise criteria for the existence or preservation of certain internal (co)limits. 
For example, we show that a large $\BB$-category $\I{C}$ is internally cocomplete if and only if the associated functor $\I{C} \colon \BB^\op \rightarrow \CatSS$
factors through the subcategory of cocomplete $\infty$-categories and cocontinuous functors such that the transition functors admit left adjoints that are compatible in a suitable sense. In this way, our theory of internally cocomplete $\BB$-categories connects to the already established theory of \emph{Beck-Chevalley fibrations} of $\infty$-categories, cf.~\cite{hopkins2013}.

As a next step, we discuss Kan extensions in the world of $\BB$-categories. 
Building on the work of the earlier chapters we obtain the expected existence theorem:

\begin{reptheorem}{thm:existenceKanExtension}
	Let $\I{E}$ be a cocomplete (large) $\BB$-category and let $f \colon \I{C} \rightarrow \I{D}$ be a functor of small $\BB$-categories.
	Then the functor $f^\ast\colon \iFun(\I{D},\I{E})\to\iFun(\I{C},\I{E})$ has a left adjoint $f_!$ which is fully faithful whenever $f$ is fully faithful.
\end{reptheorem}

In fact one only needs that colimits of certain comma $\BB$-categories exist in $\I{E}$ (see \S~\ref{sec:KanExtensions} for a refined formulation).
As explained above, one can explicitly check whether a $\BB$-category is cocomplete, therefore it will be easy to resort to Theorem~\ref{thm:existenceKanExtension} in applications.

We have now built enough machinery to be able to prove our main result.
With future applications in mind we generalise our main results and also construct the free $\I{U}$-cocompletion of a small $\BB$-category for any so called \emph{internal class} $\I{U}$ (see Definition~\ref{def:internalClass}) of indexing $\BB$-categories.

\subsection*{Parametrisation and internalisation}
In a joint project, Barwick, Dotto, Glasman, Nardin and Shah developed what they call \emph{parametrised higher category theory and higher algebra} with the aim of laying new foundations for a systematic study of equivariant homotopy theory (see \cite{barwick2016}). Parametrised higher category theory is the study of $\infty$-categories that are parametrised by a fixed base $\infty$-category $\CC$. Therefore, they exactly correspond to $\PSh(\CC)$-categories in our terminology.
As part of this joint project, Shah already developed a great deal of theory for parametrised $\infty$-categories, such as parametrised limits and colimits, Kan extensions and free cocompletions~\cite{Shah2023,Shah2021}.
As a consequence, whenever the base $\infty$-topos $\BB$ is a presheaf $\infty$-topos, most results in the present article are already contained in~\cite{Shah2023} or~\cite{Shah2021}. But even for a general $\infty$-topos $\BB$, the fact that $\BB$ is always a left exact localisation of a presheaf $\infty$-topos $\PSh(\CC)$ and the observation that many results in internal higher category theory in $\BB$ only depend on the underlying presheaf $\infty$-topos $\PSh(\CC)$ imply that these results can also be deduced from their analogue in parametrised higher category theory.
Nevertheless, we see compelling reasons for developing our own framework to study internal higher categories:

First, our results are a priori more general as they are valid internal to any $\infty$-topos $\BB$ and not only to presheaf $\infty$-topoi. 
In some of the applications we have in mind, this added level of generality will be crucial.
As a concrete example, an arbitrary $\infty$-topos $\BB$ need not be compactly generated and the final object $1\in\BB$ need not be compact.
Working internal to $\BB$ offers a way to fix this behaviour: 
the internal mapping $\BB$-groupoid functor $\map {\Univ[\BB]}(1_{\Univ},-) \colon \Univ_{\BB} \rightarrow \Univ_{\BB}$ is equivalent to the identity and therefore commutes with any kind of colimit.
By building on this, one can for example show that for a ring $R$, every dualisable object in the $\infty$-category $\bb{D}(\BB,R) = \BB \otimes \bb{D}(R)$ is \emph{internally} compact (when viewing $\bb{D}(\BB,R)$ as a $\BB$-category in a suitable way).
It would not be possible to make sense of these statements by working internal to presheaf $\infty$-topoi, since there need not exist a presentation $L\colon\PSh(\mathcal{C}) \rightarrow \BB$ such that the inclusion $\BB \into \PSh(\mathcal{C})$ preserves filtered colimits. 
In future work we intend to classify dualisable objects in $\bb{D}(\BB,R)$ using the above observations and are therfore forced to work internal to $\BB$ itself.

Second, our framework is inherently functorial in the base $\infty$-topos, which would be difficult to realise with the parametrised approach as the choice of a presentation for an $\infty$-topos is not functorial. With regard to future applications of our framework to relative higher topos theory, such functoriality is however indispensible for us.

Third, from a more conceptual point of view, our approach differs from the approach taken in parametrised higher category theory in that we put special emphasis on employing the \emph{internal logic} of the base $\infty$-topos for the development of our theory. That is, instead of proving statements for $\BB$-categories by way of reducing them to the analogous well-known statements for $\infty$-categories, our general strategy is to find proofs of the latter that only make use of the abstract properties of the $\infty$-topos of spaces and that can therefore be \emph{interpreted} internally in $\BB$.
In particular, we consistently avoid to choose a strict point-set model for $\infty$-categories and $\infty$-groupoids. As a consequence, our proofs tend to be very different from the proofs in~\cite{Shah2023} and~\cite{Shah2021}.

\subsection*{Other related work}
The study of categories internal to a $1$-topos (or more generally a $1$-category with finite limits) was initated in the second half of the 20th century and has been studied thereafter by numerous mathematicians. Therefore, most results that are presented in this paper are well-known in the $1$-categorical context. An excellent account of this theory can be found in~\cite{johnstone2002}.

As mentioned above the theory of synthetic higher categories developed by Shulman and Riehl in \cite{Riehl2017b} is closely related to our theory of $\BB$-categories. Indeed as a consequence of Shulman's $\infty$-topos semantics \cite{shulman2019}, synthetic higher category theory can be interpreted in simplicial objects in any $\infty$-topos. Many concepts of internal higher category theory have been developed from this point of view by Buchholtz-Weinberger \cite{Buchholtz2023} and Weinberger in \cite{Weinberger2022}, \cite{Weinberger2022a}. In particular colimits indexed by internal groupoids have been studied in the latter.

In the more general setup of higher categories internal to any $\infty$-category with finite limits, parts of this theory, in particular adjunctions and limits, have been developed by Nima Rasekh in \cite{Rasekh2022}.

\subsection*{Acknowledgements}
We would like to thank our respective advisors Rune Haugseng and Denis-Charles Cisinksi for their advice and support while writing this article.
We would like to thank Bastiaan Cnossen for many helpful discussions and his very valuable feedback on earlier drafts of this paper.
We would also like to thank the anonymous referee for many helpful suggestions to improve the exposition of this paper.
 S.W.\ was supported by the SFB 1085 ‘Higher
Invariants’ in Regensburg, funded by the DFG.

\section{Background on $\BB$-categories}
In this section we recall the basic framework of higher category theory internal to an $\infty$-topos from~\cite{martini2021}. We refer the reader to~\cite{martini2021} for proofs and a more detailed discussion.

\subsection{General conventions and notation}
\label{sec:conventions}
We generally follow the conventions and notation from~\cite{martini2021}. For the convenience of the reader, we will briefly recall the main setup. 

Throughout this paper we freely make use of the language of higher category theory. We will generally follow a model-independent approach to higher categories. This means that as a general rule, all statements and constructions that are considered herein will be invariant under equivalences in the ambient $\infty$-category, and we will always be working within such an ambient $\infty$-category.

We denote by $\Delta$ the simplex category, i.e.\ the category of non-empty totally ordered finite sets with order-preserving maps. Every natural number $n\in\mathbb N$ can be considered as an object in $\Delta$ by identifying $n$ with the totally ordered set $\ord{n}=\{0,\dots n\}$. For $i=0,\dots,n$ we denote by $\delta^i\colon \ord{n-1}\to \ord{n}$ the unique injective map in $\Delta$ whose image does not contain $i$. Dually, for $i=0,\dots n$ we denote by $\sigma^i\colon \ord{n+1}\to \ord{n}$ the unique surjective map in $\Delta$ such that the preimage of $i$ contains two elements. Furthermore, if $S\subset n$ is an arbitrary subset of $k$ elements, we denote by $\delta^S\colon \ord{k}\to \ord{n}$ the unique injective map in $\Delta$ whose image is precisely $S$. In the case that $S$ is an interval, we will denote by $\sigma^S\colon \ord{n}\to \ord{n-k}$ the unique surjective map that sends $S$ to a single object. If $\CC$ is an $\infty$-category, we refer to a functor $C\colon\Delta^{\op}\to\CC$ as a simplicial object in $\CC$. We write $C_n$ for the image of $n\in\Delta$ under this functor, and we write $d_i$, $s_i$, $d_S$ and $s_S$ for the image of the maps $\delta^i$, $\sigma^i$, $\delta^S$ and $\sigma^S$ under this functor. Dually, a functor $C^{\bullet}\colon \Delta\to\CC$ is referred to as a cosimplicial object in $\CC$. In this case we denote the image of $\delta^i$, $\sigma^i$, $\delta^S$ and $\sigma^S$ by $d^i$, $s^i$, $d^S$ and $\sigma^S$.

The $1$-category $\Delta$ embeds fully faithfully into the $\infty$-category of $\infty$-categories by means of identifying posets with $0$-categories and order-preserving maps between posets with functors between such $0$-categories. We denote by $\Delta^n$ the image of $\ord{n} \in\Delta$ under this embedding.

\subsection{Set-theoretical foundations}
\label{sec:setTheory}
Once and for all we will fix three Grothendieck universes $\bU\in\bV\in\bW$ that contain the first infinite ordinal $\omega$. A set is \emph{small} if it is contained in $\bU$, \emph{large} if it is contained in $\bV$ and \emph{very large} if it is contained in $\bW$. An analogous naming convention will be adopted for $\infty$-categories and $\infty$-groupoids. The large $\infty$-category of small $\infty$-groupoids is denoted by $\SS$, and the very large $\infty$-category of large $\infty$-groupoids by $\SSS$. The (even larger) $\infty$-category of very large $\infty$-groupoids will be denoted by $\SSSS$. Similarly, we denote the large $\infty$-category of small $\infty$-categories by $\CatS$ and the very large $\infty$-category of large $\infty$-categories by $\CatSS$. We shall not need the $\infty$-category of very large $\infty$-categories in this article.

\subsection{$\infty$-topoi}
\label{sec:inftyTopoi}
For $\infty$-topoi $\AA$ and $\BB$, a \emph{geometric morphism} is a functor $f_\ast\colon \BB\to \AA$ that admits a left exact left adjoint, and an \emph{algebraic morphism} is a left exact functor $f^\ast\colon \AA\to \BB$ that admits a right adjoint. The \emph{global sections} functor is the unique geometric morphism $\Gamma_{\BB}\colon \BB\to \SS$ into the $\infty$-topos of $\infty$-groupoids $\SS$. Dually, the unique algebraic morphism originating from $\SS$ is denoted by $\const_{\BB}\colon \SS\to \BB$ and referred to as the \emph{constant sheaf} functor. We will often omit the subscripts if they can be inferred from the context.
For an object $A \in \BB$, we denote the induced étale geometric morphism by $(\pi_A)_\ast \colon \BB_{/A} \rightarrow \BB$.

\subsection{Universe enlargement}
\label{sec:universeEnlargement}
If $\BB$ is an $\infty$-topos, we define its \emph{universe enlargement} $\BBB=\Shv_{\SSS}(\BB)$, where the right-hand side denotes the $\infty$-category of presheaves $\BB^{\op}\to\SSS$ which preserve small limits; this is an $\infty$-topos relative to the larger universe $\bV$~\cite[Remark~6.3.5.17]{htt}. Moreover, the Yoneda embedding gives rise to an inclusion $\BB\into\BBB$ that commutes with small limits and colimits and with the internal hom~\cite[Proposition~2.4.4]{martini2021}. The operation of enlarging universes is transitive: when defining the $\infty$-topos $\BBBB$ relative to $\bW$ as the universe enlargement of $\BBB$ with respect to the inclusion $\bV\in\bW$, the $\infty$-category $\BBBB$ is equivalent to the universe enlargement of $\BB$ with respect to $\bU\in\bW$~\cite[Remark~2.4.1]{martini2021}.

\subsection{Factorisation systems} 
\label{sec:factorisationSystems}
If $\CC$ is a presentable $\infty$-category and if $S$ is a small set of maps in $\CC$, there is a unique factorisation system $(\LL,\RR)$ in which a map is contained in $\RR$ if and only if it is \emph{right orthogonal} to the maps in $S$, and where $\LL$ is dually defined as the set of maps that are left orthogonal to the maps in $\RR$. We refer to $\LL$ as the \emph{saturation} of $S$; this is the smallest set of maps containing $S$ that is stable under pushouts, contains all equivalences and is stable under small colimits in $\Fun(\Delta^1,\CC)$. An object $c\in\CC$ is said to be \emph{$S$-local} if the unique morphism $c\to 1$ is contained in $\RR$. 
    
If $\CC$ is cartesian closed, one can analogously construct a factorisation system $(\LL^\prime,\RR^\prime)$ in which $\RR^\prime$ is the set of maps in $\BB$ that are \emph{internally} right orthogonal to the maps in $S$. Explicitly, a map is contained in $\RR^\prime$ if and only if it is right orthogonal to maps of the form $s\times \id_c$ for any $s\in S$ and any $c\in \CC$. The left complement $\LL^\prime$ is comprised of the maps in $\CC$ that are left orthogonal to the maps in $\RR^\prime$ and is referred to as the \emph{internal} saturation of $S$. Equivalently, $\LL^\prime$ is the saturation of the set of maps $s\times\id_c$ for $s\in S$ and $c\in\CC$. An object $c\in\CC$ is said to be \emph{internally $S$-local} if the unique morphism $c\to 1$ is contained in $\RR^\prime$. 
    
Given any factorisation system $(\LL,\RR)$ in $\CC$ in which $\LL$ is the saturation of a small set of maps in $\CC$, the inclusion $\RR\into\Fun(\Delta^1,\CC)$ admits a left adjoint that carries a map $f\in\Fun(\Delta^1,\CC)$ to the map $r\in\RR$ that arises from the unique factorisation $f\simeq rl$ into maps $l\in \LL$ and $r\in \RR$. By taking fibres over an object $c\in\CC$, one furthermore obtains a Bousfield localisation $\Over{\CC}{c}\leftrightarrows\Over{\RR}{c}$ such that if $f\colon d\to c$ is an object in $\Over{\CC}{c}$ and if $f\simeq rl$ is its unique factorisation into maps $l\in \LL$ and $r\in \RR$, the adjunction unit is given by $l$.

\subsection{Simplicial objects, $\BB$-categories and $\BB$-groupoids}
\label{sec:BCategories}
If $\BB$ is an arbitrary $\infty$-topos, we denote by $\Simp{\BB}=\Fun(\Delta^{\op},\BB)$ the $\infty$-topos of simplicial objects in $\BB$. Note that the adjunction $(\const\dashv \Gamma)\colon\SS\leftrightarrows\BB$ yields via postcomposition an induced adjunction $(\const\dashv\Gamma)\colon \Simp{\SS}\leftrightarrows \Simp\BB$ on the level of simplicial objects. We will often implicitly identify a simplicial $\infty$-groupoid $K$ with its image in $\Simp\BB$ along $\const_{\BB}$.

For every $n\geq 1$, we denote by $I^n=\Delta^1\sqcup_{\Delta^0}\cdots\sqcup_{\Delta^0}\Delta^1\into\Delta^n$ the $n$-spine, viewed as a simplicial $\infty$-groupoid. Furthermore, we denote by $E^1=(\Delta^0\sqcup\Delta^0)\sqcup_{(\Delta^1\sqcup\Delta^1)}\Delta^3$ the walking equivalence.
\begin{definition}[{\cite[Definitions~3.1.5 and~3.2.1]{martini2021}}]
	\label{def:BCategories}
	A \emph{$\BB$-category} is a simplicial object $\I{C}\in\Simp\BB$ that is internally local with respect to $I^2\into\Delta^2$ (Segal conditions) and $E^1\to 1$ (univalence). We denote by $\Cat(\BB)\into\Simp\BB$ the full subcategory spanned by the $\BB$-categories. A \emph{$\BB$-groupoid} is a simplicial object $G\in\Simp\BB$ which is internally local with respect to $\Delta^1\to\Delta^0$. We denote by $\Grpd(\BB)\into\Simp\BB$ the full subcategory spanned by the $\BB$-groupoids.
\end{definition}

\begin{remark}[{\cite[Proposition~3.2.7]{martini2021}}]
	\label{rem:BCategoriesExplicitly}
	More explicitly, a simplicial object $C$ is a $\BB$-category if and only if for all $n\geq 2$ the maps $C_n\to C_1\times_{C_0} \cdots\times_{C_0} C_1$ as well as the map $C_0\to (C_0\times C_0)\times_{C_1\times C_1} C_3$ are equivalences.
\end{remark}
\begin{remark}
	\label{rem:walkingEquivalence}
	There are several non-equivalent definitions of the walking equivalence. For example, Charles Rezk~\cite[\S~6]{rezk2001} defines the walking equivalence as the simplicial set $J$ that arises as the nerve of the category with two objects and a unique isomorphism between them. Our model $E^1$ (that we adopted from~\cite[Notation~1.1.12]{lurie2009b}), on the other hand, is comprised of a map together with \emph{separate} left and right inverses. Nevertheless, either choice gives rise to the same notion of $\BB$-categories: there is a natural map $E^1\to J$ which is contained in the internal saturation of $I^2\into\Delta^2$, i.e.\ which becomes an equivalence when imposing the Segal conditions. This can be extracted from the discussion in~\cite[\S~6]{rezk2001}, see also~\cite[\S~2.4]{Rasekh2018b}.
\end{remark}

\begin{proposition}[{\cite[Proposition~3.2.9, Remark~3.2.10 and Proposition~3.2.11]{martini2021}}]
	\label{prop:CatBPresentable}
	The inclusion $\Cat(\BB)\into\Simp\BB$ preserves filtered colimits and admits a left adjoint which preserves finite products. Therefore, $\Cat(\BB)$ is presentable and an exponential ideal in $\Simp\BB$, so in particular cartesian closed.
\end{proposition}
We will denote by $\iFun(-,-)$ the internal hom in $\Cat(\BB)$ and refer to it as the \emph{functor $\BB$-category} bifunctor.

\begin{proposition}[{\cite[after Corollary~3.2.12]{martini2021}}]
	\label{prop:GroupoidificationCore}
	A simplicial object in $\BB$ is a $\BB$-groupoid if and only if it is constant (i.e.\ contained in the essential image of the diagonal embedding $\iota\colon\BB\into\Simp\BB$), and every $\BB$-groupoid is a $\BB$-category. Moreover, the resulting embedding $\BB\simeq \Grpd(\BB)\into\Cat(\BB)$ admits both a left adjoint $(-)^{\gp}$ (the \emph{groupoidification functor}) and a right adjoint $(-)^\core$ (the \emph{core $\BB$-groupoid functor}). Explicitly, if $\I{C}$ is a $\BB$-category, one has $\I{C}^\gp \simeq \colim_{\Delta^{\op}}\I{C}$ and $\I{C}^{\simeq}\simeq \I{C}_0$.
\end{proposition}

\begin{definition}
	\label{def:oppositeBCategory}
	If $\I{C}$ is a $\BB$-category, we denote by $\I{C}^\op$ the simplicial object that is obtained by precomposing $\I{C}\colon\Delta^\op\to\BB$ with the involution $(-)^\op\colon\Delta\simeq\Delta$ that carries $\ord{n}$ (viewed as a $0$-category) to its opposite $\ord{n}^\op$. The simplicial object $\I{C}^\op$ is again a $\BB$-category that we refer to as the \emph{opposite $\BB$-category} of $\I{C}$.
\end{definition}

\begin{remark}
	The equivalence $(-)^\op\colon\Cat(\BB)\simeq\Cat(\BB)$ from Definition~\ref{def:oppositeBCategory} restricts to the identity on $\Grpd(\BB)$. In fact, this follows immediately from the observation that $\BB$-groupoids are constant simplicial objects (see Proposition~\ref{prop:GroupoidificationCore}).
\end{remark}

\begin{remark}[{\cite[\S~3.3]{martini2021}}]
	\label{rem:functorialityBCategories}
	If $f_\ast\colon \BB\to\AA$ is a geometric morphism and if $f^\ast$ is the associated algebraic morphism, postcomposition induces an adjunction $f^\ast\dashv f_\ast\colon \Cat(\AA)\leftrightarrows\Cat(\BB)$. In particular, one obtains an adjunction $\const_{\BB}\dashv\Gamma_{\BB}\colon \CatS\leftrightarrows\Cat(\BB)$. We will often implicitly identify an $\infty$-category $\CC$ with the associated \emph{constant $\BB$-category} $\const_{\BB}(\CC)\in\Cat(\BB)$. Furthermore, if the geometric morphism $f_\ast$ is \emph{\'etale}, the further left adjoint $f_!$ of $f^\ast$ also induces a functor $f_!\colon \Cat(\BB)\to\Cat(\AA)$ that identifies $\Cat(\BB)$ with $\Over{\Cat(\AA)}{f_! 1}$. 
\end{remark}

By making use of the adjunction $\const_{\BB}\dashv\Gamma_{\BB}\colon \CatS\leftrightarrows\Cat(\BB)$ and the internal hom $\iFun(-,-)$ as well as the product $-\times -$ in $\Cat(\BB)$, one can define bifunctors
\begin{align}
	\Fun_{\BB}(-,-)=\Gamma_{\BB}\circ\iFun(-,-)&\colon \Cat(\BB)^\op\times\Cat(\BB)\to\CatS \tag{Functor $\infty$-category} \\
	(-)^{(-)}=\iFun(\const_{\BB}(-),-)&\colon \CatS^\op\times\Cat(\BB)\to\Cat(\BB) \tag{Powering}\\
	-\otimes - = \const_{\BB}(-)\times -&\colon\CatS\times\Cat(\BB)\to\Cat(\BB) \tag{Tensoring}
\end{align}
which fit into equivalences
\begin{equation*}
	\map{\Cat(\BB)}(-\otimes -, -)\simeq \map{\CatS}(-,\Fun_{\BB}(-,-))\simeq\map{\Cat(\BB)}(-, (-)^{(-)})
\end{equation*}
(see~\cite[\S~3.4]{martini2021}). In particular, we have $\Fun_{\BB}(-,-)^\core\simeq\map{\Cat(\BB)}(-,-)$, so that $\Fun_{\BB}(-,-)$ gives rise to a $\CatS$-enrichement of $\Cat(\BB)$ and therefore an $(\infty,2)$-categorical enhancement of $\Cat(\BB)$~\cite[Remark~3.4.3]{martini2021}.

\begin{remark}[{\cite[Proposition~3.1.2]{martini2021}}]
	\label{rem:identificationObjectOfNMorphisms}
	There is an equivalence of functors $\id_{\Cat(\BB)}\simeq ((-)^{\Delta^\bullet})^\simeq$. In other words, for any $\BB$-category $\I{C}$ and any integer $n\geq 0$ one may canonically identify $\I{C}_n\simeq (\I{C}^{\Delta^n})_0$.
\end{remark}

We conclude this section with a remark on \emph{large} $\BB$-categories: observe that postcomposition with the universe enlargement $\BB\into\BBB$ from \S~\ref{sec:universeEnlargement} determines an inclusion $\Cat(\BB)\into\Cat(\BBB)$ that is natural in $\BB$ both with respect to geometric and algebraic morphisms of $\infty$-topoi~\cite[\S~3.3]{martini2021}. Furthermore, the inclusion commutes with small limits and the internal hom~\cite[Proposition~3.4.1]{martini2021} and therefore also the tensoring, powering and functor $\infty$-category bifunctors~\cite[Corollary~3.4.2]{martini2021}. We refer to the objects in $\Cat(\BBB)$ as \emph{large} $\BB$-categories (or as $\BBB$-categories) and to the objects in $\Cat(\BB)$ as \emph{small} $\BB$-categories. If not specified otherwise, every $\BB$-category is small. Note, however, that by replacing the universe $\bU$ with the larger universe $\bV$ (i.e.\ by working internally to $\BBB$), every statement about $\BB$-categories carries over to one about large $\BB$-categories as well. Also, we will often omit specifying the relative size of a $\BB$-category if it is evident from the context, and we will continue writing $\iFun(\I{C},\I{D})$ for the internal hom even if $\I{C}$ and $\I{D}$ are large.

\subsection{$\BB$-categories as sheaves of $\infty$-categories}
\label{sec:parametrisedCategories}
One may equivalently regard a $\BB$-category as a \emph{sheaf} of $\infty$-categories on $\BB$, by which we mean a functor $\BB^{\op}\to\CatS$ that preserves small limits:
\begin{proposition}[{\cite[Proposition~3.5.1 and Remark~3.5.6]{martini2021}}]
	\label{prop:equivalenceBCategoriesSheaves}
	There is a natural equivalence of $\infty$-categories $\Cat(\BB)\simeq\Shv_{\CatS}(\BB)$ that sends $\I{C}\in\Cat(\BB)$ to the sheaf $\Fun_{\BB}(\iota(-),\I{C})$ (where $\iota\colon \BB\into\Cat(\BB)$ is the diagonal embedding) and that restricts along the diagonal embedding $\iota\colon \BB\into\Cat(\BB)$ to the equivalence $\BB\simeq \Shv_{\SS}(\BB)$.
\end{proposition}
Hereafter, we will often implicitly identify a $\BB$-category $\I{C}$ with the associated sheaf $\Fun_{\BB}(\iota(-),\I{C})$. That is, we usually write $\I{C}(A)=\Fun_{\BB}(\iota(A),\I{C})$ for the $\infty$-category of \emph{local sections} over $A\in \BB$, and we write $s^\ast\colon \I{C}(B)\to\I{C}(A)$ for the restriction functor along a map $s\colon B\to A$ in $\BB$.

\begin{remark}[{cf.~\cite[Remark~3.1.1]{martini2021}}]
	\label{rem:sheafFromBCategoryExplicitly}
	More explicitly, the $\infty$-category $\I{C}(A)=\Fun_{\BB}(\iota(A),\I{C})$ is given by the complete Segal space whose space of $n$-morphisms is given by the $\infty$-groupoid $\map{\BB}(A, \I{C}_n)$. In particular, the equivalence $\Cat(\BB)\simeq\Shv_{\CatS}(\BB)$ from Proposition~\ref{prop:equivalenceBCategoriesSheaves} commutes both with taking core $\BB$-groupoids and opposite $\BB$-categories, in the sense that we have equivalences of sheaves $\I{C}^\core(-)\simeq\I{C}(-)^\core$ and $\I{C}^\op(-)\simeq\I{C}(-)^\op$.
\end{remark}

\begin{remark}
	\label{rem:internalVsParametrised}
	One may interpret Proposition~\ref{prop:equivalenceBCategoriesSheaves} as a correspondence between \emph{internal} and \emph{parametrised} higher category theory. Both approaches have their specific advantages: the upshot of the internal approach is that one can often use a statement about $\infty$-categories and simply interpret it internally in $\BB$ in order to obtain the corresponding statement for $\BB$-categories. On the other hand, it is usually easier to construct a particular $\BB$-category via its associated sheaf of $\infty$-categories. In fact, most examples that are of practical interest arise in this way.
\end{remark}

\begin{remark}[{\cite[\S~3.5]{martini2021}}]
	\label{rem:BCategoriesSheavesFunctoriality}
	The equivalence $\Cat(\BB)\simeq\Shv_{\CatS}(\BB)$ is natural in $\BB$: if $f_\ast\colon \BB\to\AA$ is a geometric morphism and $f^\ast$ denotes its left adjoint, one obtains commutative squares
	\begin{equation*}
		\begin{tikzcd}
			\Cat(\BB)\arrow[d, "f_\ast"]\arrow[r, "\simeq"] & \Shv_{\CatS}(\BB)\arrow[d, "f_\ast"] & & \Cat(\BB)\arrow[from=d, "f^\ast"]\arrow[r, "\simeq"] & \Shv_{\CatS}(\BB)\arrow[from=d, "f^\ast"]\\
			\Cat(\AA)\arrow[r, "\simeq"] & \Shv_{\CatS}(\AA) && \Cat(\AA)\arrow[r, "\simeq"] & \Shv_{\CatS}(\AA).
		\end{tikzcd}
	\end{equation*}
	Explicitly, $f_\ast\colon \Shv_{\CatS}(\BB)\to\Shv_{\CatS}(\AA)$ is given by restriction along $f^\ast\colon \AA\to\BB$. In particular, we may identify $\I{C}(1)\simeq\Gamma_{\BB}(\I{C})$ for every $\BB$-category $\I{C}$. Furthermore,  $f^\ast\colon\Shv_{\CatS}(\AA)\to\Shv_{\CatS}(\BB)$ is given by left Kan extension along $f^\ast\colon\AA\to\BB$. Thus, if the latter functor admits an additional left adjoint $f_!$, then $f^\ast\colon\Shv_{\CatS}(\AA)\to\Shv_{\CatS}(\BB)$ is simply given by precomposition with $f_!$.
\end{remark}

\begin{remark}[{\cite[Proposition~3.5.1]{martini2021}}]
	\label{rem:BCategoriesSheavesSize}
	The equivalence between $\BB$-categories and sheaves of $\infty$-categories respects universe enlargement in the following sense: there is a commutative square
	\begin{equation*}
		\begin{tikzcd}
			\Cat(\BB)\arrow[r, "\simeq"]\arrow[d, hookrightarrow] & \Shv_{\CatS}(\BB)\arrow[d, hookrightarrow]\\
			\Cat(\BBB)\arrow[r, "\simeq"] & \Shv_{\CatSS}(\BB)
		\end{tikzcd}
	\end{equation*}
	in which the lower horizontal equivalence is obtained by sending a large $\BB$-category $\I{C}$ to $\Fun_{\BBB}(\iota(-), \I{C})$, where $\iota\colon \BB\into\BBB\into \Cat(\BBB)$ is the inclusion.
\end{remark}

We conclude this section by noting that the sheaf-theoretic perspective on $\BB$-categories also gives rise to a \emph{fibrational} point of view: on account of the inclusion $\Shv_{\CatSS}(\BB)\into\PSh_{\CatSS}(\BB)$ and by making use of the straightening/unstraightening equivalence $\PSh_{\CatSS}(\BB)\simeq\Cart(\BB)$ between $\CatSS$-valued presheaves on $\BB$ and \emph{cartesian fibrations} over $\BB$ (see~\cite[\S~3.2]{htt}), we obtain a full embedding $\Cat(\BBB)\into\Cart(\BB)$ which sends a (large) $\BB$-category $\I{C}$ to its underlying cartesian fibration $\int\I{C}\to\BB$.

\subsection{Objects and morphisms in $\BB$-categories}
\label{sec:objectsMorphisms}
Observe that by combining Proposition~\ref{prop:equivalenceBCategoriesSheaves} with the two-variable adjunctions between the bifunctors $\Fun_{\BB}(-,-)$, $-\otimes -$ and $(-)^{(-)}$, one obtains equivalences
\begin{equation*}
	\I{C}^{\Delta^n}(A)^\core\simeq\map{\Cat(\BB)}(A, \I{C}^{\Delta^n})\simeq \map{\Cat(\BB)}(\Delta^n\otimes A, \I{C})\simeq \map{\CatS}(\Delta^n, \I{C}(A))
\end{equation*}
for every $A\in\BB$, every $\I{C}\in\Cat(\BB)$ and each $n\in\mathbb N$
(where we leave the diagonal embedding $\BB\into\Cat(\BB)$ implicit). Moreover, by combining Proposition~\ref{prop:GroupoidificationCore} with Remark~\ref{rem:identificationObjectOfNMorphisms}, we may furthermore compute
\begin{equation*}
	\map{\Cat(\BB)}(A,\I{C}^{\Delta^n})\simeq\map{\BB}(A, \I{C}_n).
\end{equation*}
In other words, the datum of a map $A\to \I{C}^{\Delta^n}$ in $\Cat(\BB)$ is equivalent to that of a map $\Delta^n\otimes A\to \I{C}$ in $\Cat(\BB)$, a map $A\to\I{C}_n$ in $\BB$ as well as a functor $\Delta^n\to\I{C}(A)$ of $\infty$-categories.

\begin{definition}
	\label{def:nMorphisms}
	Let $\I{C}$ be a $\BB$-category and let $A\in\BB$ be an object. For a given integer $n\geq 0$, an \emph{$n$-morphism in $\I{C}$ in context $A$} is a map $A\to \I{C}^{\Delta^n}$ in $\Cat(\BB)$. If $n=0$, we simply speak of an \emph{object} in $\I{C}$ in context $A$, and for $n=1$ we refer to such a map as a \emph{morphism} in $\I{C}$ in context $A$. Given objects $c,d\colon A\rightrightarrows \I{C}$, one defines the \emph{mapping $\Over{\BB}{A}$-groupoid} $\map{\I{C}}(c,d)$ as the pullback
	\begin{equation*}
		\begin{tikzcd}
			\map{\I{C}}(c,d)\arrow[r]\arrow[d] & \I{C}_1\arrow[d, "{(d_1,d_0)}"]\\
			A\arrow[r, "{(c,d)}"]& \I{C}_0\times\I{C}_0.
		\end{tikzcd}
	\end{equation*}
	We denote a section $f\colon A\to \map{\I{C}}(c,d)$ by $f\colon c\to d$. 
\end{definition}

\begin{remark}[{\cite[\S~3.6]{martini2021}}]
	Equivalently, the mapping $\Over{\BB}{A}$-groupoid $\map{\I{C}}(c,d)$ can be defined as the pullback of $(d_1,d_0)\colon\I{C}^{\Delta^1}\to\I{C}\times\I{C}$ along $(c,d)\colon A\to\I{C}\times\I{C}$.
\end{remark}

\begin{remark}
	 Viewed as an $\SS$-valued sheaf on $\Over{\BB}{A}$, the object $\map{\I{C}}(c,d)$ from Definition~\ref{def:nMorphisms} is given by the assignment
	\begin{equation*}
		\Over{\BB}{A}\ni(s\colon B\to A)\mapsto \map{\I{C}(B)}(s^\ast c,s^\ast d)
	\end{equation*}
	where $s^\ast c= c s$ and likewise for $d$.
\end{remark}

More generally, if $c_0,\dots,c_n$ are objects in context $A$ in $\I{C}$, one writes $\map{\I{C}}(c_0,\dots,c_n)$ for the pullback of $(d_n,\dots,d_0)\colon \I{C}_n\to \I{C}_0^{n+1}$ along the map $(c_0,\dots,c_n)\colon A\to \I{C}_0^{n+1}$. Using the Segal conditions, one obtains an equivalence
\begin{equation*}
	\map{\I{C}}(c_0,\dots,c_n)\simeq\map{\I{C}}(c_0,c_1)\times_A\cdots\times_A\map{\I{C}}(c_{n-1},c_n).
\end{equation*}
By combining this identification with the map $\map{\I{C}}(c_0,\dots,c_n)\to\map{\I{C}}(c_0,c_n)$ that is induced by the map $d_{\{0,n\}}\colon \I{C}_n\to\I{C}_1$, one obtains a composition map
\begin{equation*}
	\map{\I{C}}(c_0,c_1)\times_A\cdots\times_A\map{\I{C}}(c_{n-1},c_n)\to\map{\I{C}}(c_0,c_n).
\end{equation*}
Given maps $f_i\colon c_{i-1}\to c_i$ in $\I{C}$ for $i=1,\dots, n$, we write $f_1\cdots f_n$ for their composition. By making use of the simplicial identities, it is straightforward to verify that composition is associative and unital, i.e.\ that the relations $f(gh)\simeq (fg)h$ and $f\id\simeq f\simeq \id f$ as well as their higher analogues hold whenever they make sense, see~\cite[Proposition~5.4]{rezk2001} for a proof.

\begin{remark}
	\label{rem:contexts}
	As a $\BB$-category $\I{C}$ is determined by the associated sheaf of $\infty$-categories on $\BB$ but not just by the underlying $\infty$-category $\Gamma_{\BB}(\I{C})$ of global sections, it is crucial that we allow objects and morphisms in $\I{C}$ to have arbitrary context $A\in\BB$. In other words, we need to allow objects and morphisms to be only \emph{locally} defined, where by the term \emph{local} we allude to the point of view that the base $\infty$-topos $\BB$ can be thought of as a spatial object. Alternatively, this phenomenon can be viewed as a shadow of the notion of contexts in type theory (hence the name), where they are needed to keep track of the types of the variables that occur in a formula. More precisely, when regarding the theory of $\BB$-categories as a model of simplicial homotopy type theory~\cite{Riehl2017b}, the type-theoretic notion of contexts exactly translates into our notion of contexts.
\end{remark}

\begin{remark}
	\label{rem:localityPrincipleObjectsMorphisms}
	At first, the fact that objects and morphisms of a $\BB$-category $\I{C}$ have non-global context $A$ might appear to complicate things, but in practice this is usually not the case: in fact, by making use of the adjunction $(\pi_A)_!\dashv\pi_A^\ast\colon \Over{\BB}{A}\leftrightarrows\BB$ and by the observations made in Remark~\ref{rem:functorialityBCategories}, the datum of an object $c\colon A\to \I{C}$ precisely corresponds to that of an object $\bar c\colon1_{\Over{\BB}{A}}\to \pi_A^\ast\I{C}$, where $\pi_A^\ast\I{C}\in\Cat(\Over{\BB}{A})$ is the image of $\I{C}$ along the base change functor $\pi_A^\ast\colon \Cat(\BB)\to\Cat(\Over{\BB}{A})\simeq\Over{\Cat(\BB)}{A}$. In other words, upon replacing $\BB$ with $\Over{\BB}{A}$ and $\I{C}$ with $\pi_A^\ast\I{C}$, object in context $A$ are turned into objects in global context. Very often, we will make use of this correspondence in order to be able to restrict our attention to objects and morphisms in global context (see \S~\ref{sec:localityPrinciple} below for more details on this strategy).
\end{remark}

\begin{remark}
	\label{rem:tautologicalObject}
	Observe that for every $\BB$-category $\I{C}$ there is a distinguished object $\tau\colon \I{C}_0\to \I{C}$ that is determined by the counit of the adjunction $\iota\dashv (-)_0\colon \Cat(\BB)\leftrightarrows\BB$ from Proposition~\ref{prop:GroupoidificationCore}. We refer to $\tau$ as the \emph{tautological} object of $\I{C}$. By definition, \emph{every} object $c\colon A\to\I{C}$ arises as a pullback of $\tau$, in the sense that we have $c\simeq c^\ast \tau$ (where $c^\ast\colon \I{C}(\I{C}_0)\to\I{C}(A)$ is the restriction functor). In that way, many questions about an arbitrary object in a $\BB$-category can be reduced to questions about the tautological object.
\end{remark}

We conclude this section with a discussion of \emph{equivalences} in $\BB$-categories. To that end, given any object $c\colon A\to\I{C}$ in a $\BB$-category $\I{C}$, let us denote by $\id_c\colon c\to c$ the morphism that is determined by the lift $s_0 c\colon A\to \I{C}_0\to\I{C}_1$ of $(c,c)\colon A\to \I{C}_0\times\I{C}_0$.
\begin{definition}
	\label{def:equivalencesBCategories}
	A morphism $f\colon c\to d$ in $\I{C}$ is an \emph{equivalence} if there are maps $g\colon c\to d$ and $h\colon c\to d$ (all in context $A$) such that $gf\simeq \id_c$ and $fh\simeq \id_d$.
\end{definition}

As a consequence of univalence, one finds:
\begin{proposition}[{\cite[Corollary~3.6.3]{martini2021}}]
	\label{prop:characterisationEquivalences}
	A map $f\colon A\to \I{C}^{\Delta^1}$ in a $\BB$-category $\I{C}$ is an \emph{equivalence} if it factors through $s_0\colon \I{C}\into\I{C}^{\Delta^1}$.
\end{proposition}
In other words, every equivalence $f\colon A\to\I{C}_1$ is equivalent (in the $\infty$-groupoid $\I{C}_1(A)$) to an identity.

\subsection{Fully faithful functors and full subcategories}
\label{sec:fullyFaithfulFunctors}
A functor $f\colon \I{C}\to\I{D}$ between $\BB$-categories is said to be \emph{fully faithful} if it is internally right orthogonal to the map $\Delta^0\sqcup\Delta^0\to \Delta^1$. Dually, a functor is \emph{essentially surjective} if is (internally) left orthogonal to the class of fully faithful functors. Therefore, it formally follows that fully faithful functors are stable under small limits in $\Fun(\Delta^1,\Cat(\BB))$ and are preserved by the endofunctor $\iFun(\I{C},-)$ for every $\BB$-category $\I{C}$~\cite[Proposition~3.8.4]{martini2021}. Moreover, a functor of $\BB$-categories is an equivalence if and only if it is fully faithful and essentially surjective~\cite[Proposition~3.8.3]{martini2021}, and every functor can be uniquely factored into an essentially surjective and a fully faithful functor. In other words, the \emph{essential image} of a functor between $\BB$-categories is well-defined.

Fully faithful and essentially surjective functors can be characterised as follows:
\begin{proposition}[{\cite[Proposition~3.8.6 and~3.8.7]{martini2021}}]
	\label{prop:classificationFullSubcategories}
	For a functor $f\colon \I{C}\to\I{D}$ of $\BB$-categories, the following are equivalent:
	\begin{enumerate}
	\item The functor $ f $ is fully faithful;
	\item the square
	\begin{equation*}
		\begin{tikzcd}
			\I{C}_1\arrow[r, "f_1"]\arrow[d] & \I{D}_1\arrow[d]\\
			\I{C}_0\times \I{C}_0\arrow[r, "f_0\times f_0"] & \I{D}_0\times \I{D}_0
		\end{tikzcd}
	\end{equation*}
	is a pullback;
	\item for every $A\in\BB$ and any two objects $c_0,c_1\colon A\to \I{C}$ in context $A$, the morphism
	\begin{equation*}
		\map{\I{C}}(c_0,c_1)\to\map{\I{D}}(f(c_0), f(c_1))
	\end{equation*}
	that is induced by $ f $ is an equivalence in $\Over{\BB}{A}$;
	\item for every $A\in \BB$ the functor $f(A)\colon \I{C}(A)\to\I{D}(A)$ of $\infty$-categories is fully faithful.
\end{enumerate}
\end{proposition}
\begin{proposition}[{\cite[Corollary~3.8.12]{martini2021}}]
	\label{prop:characterisationEssentiallySurjective}
	A functor $f\colon \I{C}\to\I{D}$ is essentially surjective if and only if $f_0\colon\I{C}_0\to\I{D}_0$ is a cover (i.e.\ an effective epimorphism) in $\BB$.
\end{proposition}

Fully faithful functors are in particular monomorphisms, hence the full subcategory $\Sub^{\mathrm{full}}(\I{D})\into\Over{\Cat(\BB)}{\I{D}}$ that is spanned by the fully faithful functors into $\I{D}$ is a poset whose objects we call \emph{full subcategories} of $\I{D}$.
\begin{proposition}[{\cite[Proposition~3.9.3]{martini2021}}]
	\label{prop:fullSubcategoriesParametrisation}
	Taking core $\BB$-groupoids determines an equivalence of posets $\Sub^{\mathrm{full}}(\I{D})\simeq \Sub(\I{D}_0)$ between the poset of full subcategories of $\I{D}$ and the poset of subobjects of $\I{D}_0\in\BB$. 
\end{proposition}
In particular, Proposition~\ref{prop:fullSubcategoriesParametrisation} implies that specifying a full subcategory of $\I{D}$ is equivalent to specifying a subobject of $\I{D}_0$. Therefore, if $(d_i\colon A_i\to \I{D})_{i\in I}$ is a family of objects in $\I{D}$, we may define the full subcategory of $\I{D}$ that is \emph{spanned} by these objects as the unique full subcategory of $\I{D}$ whose core $\BB$-groupoid is given by the image of the induced morphism $(d_i)\colon\bigsqcup_i A_i\to \I{D}$ in $\BB$~\cite[Definition~3.9.7]{martini2021}. Note that this is possible even if the family is large~\cite[Remark~3.9.8]{martini2021}.

\subsection{The universe for $\BB$-groupoids}
\label{sec:universe}
By straightening the codomain fibration $\Fun(\Delta^1,\BB)\to\BB$, one obtains a functor $\Over{\BB}{-}\colon \BB^{\op}\to\CatSS$ that preserves small limits since $\BB$ is an $\infty$-topos~\cite[Proposition~6.1.3.9]{htt}. In other words, $\Over{\BB}{-}$ is a sheaf of (large) $\infty$-categories and therefore (by Remark~\ref{rem:BCategoriesSheavesSize}) determined by a large $\BB$-category $\Univ[\BB]$ that we refer to as the \emph{universe for $\BB$-groupoids}~\cite[\S~3.7]{martini2021}. We will often omit the subscript if it is clear from the context. By definition, we have equivalences $\Univ(A)\simeq\Over{\BB}{A}\simeq\Grpd(\Over{\BB}{A})$. In other words, the objects in $\Univ$ in context $A$ are precisely given by the $\Over{\BB}{A}$-groupoids, an observation which justifies its name. Moreover, we have:
\begin{proposition}[{\cite[Proposition~3.7.3]{martini2021}}]
	\label{prop:mappingObjectsInternalUniverse}
	For any two objects $g,h$ in $\Univ$ in context $A\in\BB$ that correspond to $\Over{\BB}{A}$-groupoids $\I{G},\I{H}\in\Grpd(\Over{\BB}{A})\simeq\Over{\BB}{A}$, there is an equivalence
	\begin{equation*}
		\map{\Univ}(g,h)\simeq \Hom_{\Over{\BB}{A}}(\I{G},\I{H})
	\end{equation*}
	in $\Over{\BB}{A}$, where $\Hom_{\Over{\BB}{A}}(-,-)$ denotes the internal  hom in $\Over{\BB}{A}$.
\end{proposition}

\begin{remark}
	\label{rem:universeUMP}
	The universe $\Univ$ is to be regarded as the $\BB$-categorical analogue of the $\infty$-category $\SS$ of $\infty$-groupoids. In fact, the main result of this paper (Theorem~\ref{thm:universalPropertyPSh}) implies in particular that $\Univ$ is characterised among $\BB$-categories by the same universal property that characterises $\SS$ among $\infty$-categories (namely as the free cocompletion of the point).
\end{remark}

We refer to a full subcategory of $\Univ$ as a \emph{subuniverse}. It follows from item~(4) of Proposition~\ref{prop:classificationFullSubcategories} and the definition of $\Univ$ that every such subuniverse corresponds precisely to \emph{local class} of morphisms in $\BB$, i.e.\ a class $S$ that satisfies the condition that a morphism $p\colon P\to A$ in $\BB$ is contained in $S$ if and only if it is \emph{locally} contained in $S$, i.e.\ if and only if for every cover $(s_i)\colon\sqcup_i A_i\onto A$ in $\BB$, the maps $s_i^\ast(p)\colon A_i\times_A P\to A_i$ are contained in $S$ (see~\cite[\S~6.1.3 and Proposition~6.2.3.14]{htt}). In other words. we have:
\begin{proposition}[{\cite[Proposition~3.9.12]{martini2021}}]
	\label{prop:classificationSubuniverses}
	There is an equivalence between the partially ordered set of local classes in $\BB$ and $\Sub^{\mathrm{full}}(\Univ)$.
\end{proposition}
For a given local class $S$, we denote the associated subuniverse by $\Univ[S]$.

\begin{example}[{see the discussion towards the end of~\cite[\S~4.5]{martini2021}}]
	\label{ex:smallLargeUniverse}
	Let us say that a map $p\colon P\to A$ in $\BBB$ is \emph{small} if for every map $A^\prime\to A$ in which $A^\prime \in \BB$, the pullback $A^\prime\times_A P$ is contained in $\BB$ as well. This determines a local class of morphisms in $\BBB$ and therefore by Proposition~\ref{prop:classificationSubuniverses} a subuniverse of $\Univ[\BBB]\in\Cat(\BBBB)$ which can be identified with $\Univ[\BB]\in\Cat(\BBB)\into\Cat(\BBBB)$. This exhibits $\Univ[\BB]$ as a full subcategory of $\Univ[\BBB]$.
\end{example}

\subsection{Left fibrations and the Grothendieck construction}
\label{sec:leftFibrations}
A functor $p\colon \I{P}\to\I{C}$ between $\BB$-categories is called a \emph{left fibration} if it is internally right orthogonal to the map $d^1\colon \Delta^0\into\Delta^1$. A functor that is contained in the internal saturation of this map is said to be \emph{initial}. In that way, one obtains a factorisation system between initial maps and left fibrations.
\begin{proposition}[{\cite[Proposition~4.1.3]{martini2021}}]
	\label{prop:leftFibrationsExplicitly}
	A functor $p\colon\I{P}\to\I{C}$ of $\BB$-categories is a left fibration if and only if for every $n\geq 1$ the commutative square
	\begin{equation*}
		\begin{tikzcd}
			\I{P}_n\arrow[r, "p_n"]\arrow[d, "d_{\{0\}}"] & \I{C}_n\arrow[d, "d_{\{0\}}"]\\
			\I{P}_0\arrow[r, "p_0"] & \I{C}_0
		\end{tikzcd}
	\end{equation*}
	is a pullback.
\end{proposition}
The restriction of the codomain fibration $d_0\colon\Fun(\Delta^1,\Cat(\BB))\to\Cat(\BB)$ to the full subcategory of $\Fun(\Delta^1,\Cat(\BB))$ that is spanned by the left fibrations is a cartesian fibration (as left fibrations are stable under pullback) and therefore determines via straightening a functor $\LFib\colon \Cat(\BB)^\op\to\CatSS$. By precomposing this functor with the product bifunctor $-\times -\colon \BB\times\Cat(\BB)\to\Cat(\BB)$ (where we leave the diagonal embedding $\BB\into\Cat(\BB)$ implicit), we therefore end up with a functor
\begin{equation*}
	\LFib(-\times -)\colon \Cat(\BB)^\op\to\PSh_{\CatSS}(\BB),\quad \I{C}\mapsto \ILFib_{\I{C}}=\LFib(-\times\I{C}).
\end{equation*}

\begin{theorem}[{\cite[Theorem~4.5.1]{martini2021}}]
	\label{thm:straightening}
	For every $\BB$-category $\I{C}$, the presheaf $\ILFib_{\I{C}}$ is a sheaf and therefore defines a large $\BB$-category. Furthermore, there is an equivalence
	\begin{equation*}
		\ILFib_{\I{C}}\simeq\iFun(\I{C},\Univ)
	\end{equation*}
	of large $\BB$-categories that is natural in $\I{C}\in\Cat(\BB)$.
\end{theorem}
\begin{remark}
	Theorem~\ref{thm:straightening} is the $\BB$-categorical analogue of straightening/unstraightening for left fibrations~\cite[Theorem~2.2.1.2]{htt}.
\end{remark}

\begin{remark}
	\label{rem:LFibExplicitly}
	By means of the projection $\pr_0\colon A\times\I{C}\to A$, every functor $p\colon\I{P}\to A\times\I{C}$ can be regarded as a map in $\Over{\Cat(\BB)}{A}\simeq\Cat(\Over{\BB}{A})$ (cf.~Remark~\ref{rem:functorialityBCategories}). Now since the forgetful functor $(\pi_A)_!\colon \Over{\BB}{A}\to\BB$ creates pullbacks, it follows (using Proposition~\ref{prop:leftFibrationsExplicitly}) that $p$ is a left fibration of $\Over{\BB}{A}$-categories if and only if it is a left fibration of $\BB$-categories. Consequently, the functor $(\pi_A)_!$ induces an equivalence
	\begin{equation*}
		\LFib_{\Over{\BB}{A}}(\pi_A^\ast\I{C})\simeq\LFib_{\BB}(A\times\I{C})
	\end{equation*}
	(where the subscript indicates internal to which $\infty$-topos we are taking left fibrations). In other words, the objects of $\ILFib_{\I{C}}$ in context $A$ are precisely given by the left fibrations (internal to $\Over{\BB}{A}$) over $\pi_A^\ast\I{C}$.
\end{remark}

\begin{remark}
	\label{rem:rightFibrations}
	Dually, a functor $p\colon \I{P}\to\I{C}$ of $\BB$-categories is a \emph{right fibration} if it is internally right orthogonal to $d^0\colon\Delta^0\into\Delta^1$, and a functor that is contained in the internal saturation of the latter map is said to be \emph{final}. Equivalently, $p$ is a right fibration precisely if $p^\op$ (see Definition~\ref{def:oppositeBCategory}) is a left fibration, and a functor $j$ is final if and only if $j^\op$ is initial. Again, one obtains a factorisation system between final maps and right fibrations, and by the same construction as for left fibrations (or by simply dualising this construction in the appropriate way) one ends up with a functor
	\begin{equation*}
		\RFib(-\times -)\colon \Cat(\BB)^\op\to\PSh_{\CatSS}(\BB),\quad \I{C}\mapsto \IRFib_{\I{C}}=\RFib(\I{C}\times -).
	\end{equation*}
	For every $\BB$-category $\I{C}$, we have $\IRFib_{\I{C}}\simeq \ILFib_{\I{C}^\op}$, hence $\IRFib_{\I{C}}$ defines a large $\BB$-category as well, and one furthermore obtains a natural straightening/unstraightening equivalence
	\begin{equation*}
		\IRFib_{\I{C}}\simeq\IPSh(\I{C}),
	\end{equation*}
	where $\IPSh(\I{C})=\iFun(\I{C}^\op,\Univ)$ is the large $\BB$-category of \emph{presheaves} on $\I{C}$.
\end{remark}

\subsection{Slice $\BB$-categories and initial objects}
\label{sec:sliceBCategories}
We now turn to the most important example of a left fibration:
\begin{definition}
	\label{def:sliceBCategories}
	For any $\BB$-category $\I{C}$ and any object $c\colon A\to \I{C}$, one defines the \emph{slice $\BB$-category} $\Under{\I{C}}{c}$ via the pullback
	\begin{equation*}
		\begin{tikzcd}
			\Under{\I{C}}{c}\arrow[d, "(\pi_c)_!"]\arrow[r] & \I{C}^{\Delta^1}\arrow[d, "{(d^1,d^0)}"]\\
			A\times\I{C}\arrow[r, "{c\times\id}"] & \I{C}\times\I{C}.
		\end{tikzcd}
	\end{equation*}
\end{definition}

\begin{remark}[{\cite[Remark~4.2.2]{martini2021}}]
	\label{rem:baseChangeSlice}
	In the situation of Definition~\ref{def:sliceBCategories}, Remark~\ref{rem:localityPrincipleObjectsMorphisms} allows us to transpose $c\colon A\to\I{C}$ to an object $\bar{c}\colon 1_{\Over{\BB}{A}}\to \pi_A^\ast\I{C}$. Thus, we can also define the slice $\Over{\BB}{A}$-category $\Under{(\pi_A^\ast\I{C})}{\bar c}$, which also comes with a projection $(\pi_{\bar c})_!\colon \Under{(\pi_A^\ast\I{C})}{\bar c}\to\pi_A^\ast\I{C}$. This turns out to produce the same result, in the sense that when applying the forgetful functor $(\pi_A)_!\colon\Cat(\Over{\BB}{A})\to\Cat(\BB)$ to the map $(\pi_{\bar c})_!\colon \Under{(\pi_A^\ast\I{C})}{\bar c}\to\pi_A^\ast\I{C}$, we recover the map $(\pi_c)_!\colon \Under{\I{C}}{c}\to A\times\I{C}$ from Definition~\ref{def:sliceBCategories}. Thus, when regarded as  a $\Over{\BB}{A}$-category, we may identify $\Under{\I{C}}{c}$ with $\Under{(\pi_A^\ast\I{C})}{\bar c}$.
\end{remark}

\begin{remark}
	\label{rem:dualSliceBCategory}
	Dually, by performing the pullback of $(d^1,d^0)$ along $\id\times c\colon \I{C}\times A\to\I{C}\times\I{C}$, one defines the slice $\BB$-category $\Over{\I{C}}{c}$ together with its projection $(\pi_c)_!\colon \Over{\I{C}}{c}\to\I{C}\times A$. Alternatively, this $\BB$-category can be defined via the identity $\Over{\I{C}}{c}\simeq(\Under{\I{C}^\op}{c})^\op$.
\end{remark}

\begin{proposition}[{\cite[Proposition~4.2.7]{martini2021}}]
	\label{prop:sliceProjectionLeftFibration}
	For every object $c\colon A\to\I{C}$ in a $\BB$-category $\I{C}$, the functor $(\pi_c)_!\colon\Under{\I{C}}{c}\to A\times\I{C}$ is a left fibration of $\BB$-categories.
\end{proposition}

\begin{remark}
	By Remark~\ref{rem:baseChangeSlice}, the functor $(\pi_c)_!$ in Proposition~\ref{prop:sliceProjectionLeftFibration} can be regarded as a map in $\Cat(\Over{\BB}{A})$ and is as such a left fibration as well (by either applying Proposition~\ref{prop:sliceProjectionLeftFibration} to the transposed object $\bar c\colon 1_{\Over{\BB}{A}}\to\pi_A^\ast\I{C}$ or by using Remark~\ref{rem:LFibExplicitly}).
\end{remark}

\begin{definition}
	\label{def:initialObject}
	Let $\I{C}$ be a $\BB$-category. An object $c\colon A\to \I{C}$ is said to be \emph{initial} if the transpose map $1\to \pi_A^\ast\I{C}$ defines an initial functor in $\Cat(\Over{\BB}{A})$. 
\end{definition}

\begin{remark}
	In the situation of Definition~\ref{def:initialObject}, one dually says that $c$ is \emph{final} if the transpose map $1\to \pi_A^\ast \I{C}$ defines a final functor in $\Cat(\Over{\BB}{A})$.
\end{remark}

\begin{remark}[{\cite[Remark~4.3.7]{martini2021}}]
	\label{rem:clarificationInitialObject}
	For every object $A\in\BB$, the forgetful functor $(\pi_A)_!\colon\Cat(\Over{\BB}{A})\simeq\Over{\Cat(\BB)}{A}\to\Cat(\BB)$ creates initial maps. Therefore, if $\I{C}$ is a $\BB$-category, an object $c\colon A\to\I{C}$ is initial if and only if the map $(c,\id)\colon A\to\I{C}\times A$ is an initial functor in $\Cat(\BB)$.
\end{remark}

Observe that if $c\colon A\to\I{C}$ is an object in a $\BB$-category $\I{C}$, the identity $\id_c\colon A\to \I{C}^{\Delta^1}$ takes values in $\Under{\I{C}}{c}$. We therefore obtain a section $\id_c\colon A\to \Under{\I{C}}{c}$ of the structure map $\Under{\I{C}}{c}\to A$ (which coincides with the image of $\id_{\bar c}\colon 1_{\Over{\BB}{A}}\to \Under{(\pi_A^\ast\I{C})}{\bar c}$ along the forgetful functor $(\pi_A)_!$, see Remark~\ref{rem:baseChangeSlice}).

\begin{proposition}[{\cite[Proposition~4.3.9 and Remark~4.3.10]{martini2021}}]
	\label{prop:initialityCanonicalSection}
	For any $\BB$-category and any object $c\colon A\to\I{C}$, the section $\id_c\colon A\to \Under{\I{C}}{c}$ is initial as a map in $\Cat(\Over{\BB}{A})$ and therefore defines an initial object of $\Under{\I{C}}{c}$.
\end{proposition}

\begin{corollary}[{\cite[Corollary~4.3.19]{martini2021}}]
	\label{cor:factorisationInternalObject}
	Let $\I{C}$ be a $\BB$-category and let $c\colon A\to\I{C}$ be an object in $\I{C}$. The factorisation of $c$ into an initial map and a left fibration is given by the composition $\pr_1(\pi_c)_!\id_c\colon A\to \Under{(\I{C})}{c}\to \I{C}$ where $\pr_1\colon A\times\I{C}\to\I{C}$ is the projection.
\end{corollary}

\begin{proposition}[{\cite[Proposition~4.3.20]{martini2021}}]
	\label{prop:characterisationinitialObject}
	Let $\I{C}$ be a $\BB$-category. For any object $c\colon A\to \I{C}$, the following are equivalent:
	\begin{enumerate}
		\item  $c$ is an initial object;
		\item the projection $(\pi_c)_!\colon\Under{\I{C}}{c}\to A\times \I{C}$ is an equivalence;
		\item for any object $d\colon B\to \I{C}$ the map $\map{\I{C}}(\pr_0^\ast c,\pr_1^\ast d)\to A\times B$ is an equivalence in $\BB$.
	\end{enumerate}
\end{proposition}

\begin{corollary}[{\cite[Corollary~4.3.21]{martini2021}}]
	\label{cor:universalPropertyInitialObject}
	Let $\I{C}$ be a $\BB$-category and let $c$ and $d$ be objects in $\I{C}$ in context $A\in\BB$ such that $c$ is initial. Then there is a unique map $c\to d$ in $\I{C}$ in context $A$ that is an equivalence if and only if $d$ is initial as well.
\end{corollary}

\subsection{Yoneda's lemma}
\label{sec:YonedaLemma}
The theory of left fibrations can be used to derive a version of Yoneda's lemma for $\BB$-categories. First, we need a functorial version of the mapping $\BB$-groupoid construction. To that end, let us denote by $-\star -\colon\Delta\times\Delta\to\Delta$ the ordinal sum bifunctor. We may now define:

\begin{definition}[{\cite[Definition~4.2.4]{martini2021}}]
	\label{def:twistedArrow}
	Let $\epsilon\colon \Delta\to\Delta$ denote the functor $\ord{n}\mapsto \ord{n}^{\op}\star \ord{n}$. For any $\BB$-category $\I{C}$, we define the \emph{twisted arrow $\BB$-category} $\Tw(\I{C})$ to be the simplicial object given by the composition
	\begin{equation*}
		\Delta^{\op}\xrightarrow{\epsilon^{\op}} \Delta^{\op}\xrightarrow{\I{C}} \BB.
	\end{equation*}
	This defines a functor $\Tw\colon \Cat(\BB)\to\Simp\BB$.
\end{definition}

Note that the functor $\epsilon$ in Definition~\ref{def:twistedArrow} comes along with two canonical natural transformations
\begin{equation*}
	(-)^{\op}\to \epsilon \leftarrow \id_{\Delta}
\end{equation*}
which induces a map of simplicial objects
\begin{equation*}
	\Tw(\I{C})\to \I{C}^{\op}\times\I{C}
\end{equation*}
that is natural in $\I{C}$.
\begin{proposition}[{\cite[Proposition~4.2.5]{martini2021}}]
	For every $\BB$-category $\I{C}$, the simplicial object $\Tw(\I{C})$ is a $\BB$-category, and the map $\Tw(\I{C})\to\I{C}^\op\times\I{C}$ is a left fibration.
\end{proposition}

By applying the straightening/unstraightening equivalence from Theorem~\ref{thm:straightening} to the left fibration $\Tw(\I{C})\to\I{C}^\op\times\I{C}$, one now ends up with a bifunctor
\begin{equation*}
	\map{\I{C}}\colon\I{C}^\op\times\I{C}\to\Univ
\end{equation*}
that sends a pair of objects $(c,d)\colon A\to\I{C}^{\op}\times\I{C}$ to the object $\map{\I{C}}(c,d)\in\Over{\BB}{A}$ from Definition~\ref{def:nMorphisms}.
Upon transposing this bifunctor across the adjunction $\I{C}^\op\times -\dashv \iFun(\I{C}^\op,-)$, one obtains the \emph{Yoneda embedding}
\begin{equation*}
	h_{\I{C}}\colon\I{C}\to\IPSh(\I{C}).
\end{equation*}

\begin{theorem}[{\cite[Theorem~4.7.8]{martini2021}}]
	\label{thm:YonedaLemma}
	For any $\BB$-category $\I{C}$, there is a commutative diagram
	\begin{equation*}
	\begin{tikzcd}
		\I{C}^{\op}\times \IPSh(\I{C})\arrow[dr, "\ev"'] 
		\arrow[r, "h\times \id"] & {\IPSh(\I{C})^{\op}}\times\IPSh(\I{C})\arrow[d, "{\map{\IPSh(\I{C})}(-,-)}"]  \\
		& \Univ
	\end{tikzcd}
	\end{equation*}
	in $\Cat(\BBB)$ (where $\ev$ is the evaluation map).
\end{theorem}

\begin{corollary}[{\cite[Corollary~4.7.16]{martini2021}}]
	\label{cor:YonedaEmbedding}
	For every $\BB$-category $\I{C}$, the Yoneda embedding $h_{\I{C}}$ is fully faithful.
\end{corollary}

\begin{remark}[{\cite[Proposition~4.7.20]{martini2021}}]
	\label{rem:representableFunctorsFibrationalCriterion}
	Explicitly, an object $A\to \IPSh(\I{C})$ is contained in $\I{C}$ if and only if the associated right fibration $p\colon \I{P}\to \I{C}\times A$ admits a final section $A\to \I{P}$ over $A$ (i.e.\ if $\I{P}$ has a final object in global context when viewed as a $\Over{\BB}{A}$-category). If this is the case, one obtains an equivalence $\Over{\I{C}}{c}\simeq \I{P}$ over $\I{C}\times A$ where $c$ is the image of the final section $A\to \I{P}$ along the functor $\I{P}\to \I{C}$.
\end{remark}

\subsection{Context reduction techniques}
\label{sec:localityPrinciple}
As a general rule, every construction and every statement that we make in $\BB$-category theory has to be \emph{local} in $\BB$ and has to be \emph{invariant under \'etale transposition}, in the following sense:
\begin{description}
	\item[locality] For every $A\in\BB$, the base change functor $\pi_{A}^\ast\colon \Cat(\BB)\to\Cat(\Over{\BB}{A})$ preserves all of the structure that we use when reasoning about $\BB$- (resp.\ $\Over{\BB}{A}$-)categories. Furthermore, for every cover (i.e.\ effective epimorphism) $(s_i)\colon\bigsqcup_i A_i\onto A$ in $\BB$ and every object $c\colon A\to\I{C}$ in a $\BB$-category $\I{C}$, a proposition is true for $c$ if and only if it is true for each of the pullbacks $s_i^\ast(p)\colon A_i\to \I{C}$.
	\item[\'etale transposition invariance] For every object $c\colon A\to\I{C}$ in a $\BB$-category $\I{C}$, a proposition holds for $c$ if and only if the same proposition, interpreted internally in $\Over{\BB}{A}$, is true for the transposed object $\bar c\colon 1_{\Over{\BB}{A}}\to\pi_A^\ast\I{C}$ (see Remark~\ref{rem:localityPrincipleObjectsMorphisms}).
\end{description}

\begin{remark}
	\label{rem:localityPrinciplePreservationStructure}
	More concretely, the locality rule asserts that
	\begin{enumerate}
		\item $\pi_A^\ast$ preserves limits and colimits;
		\item there is an equivalence $\pi_A^\ast\const_{\BB}\simeq\const_{\Over{\BB}{A}}$;
		\item $\pi_A^\ast$ commutes with the internal hom $\iFun(-,-)$~\cite[Lemma~4.2.3]{martini2021};
		\item $\pi_A^\ast$ carries the universe $\Univ[\BB]$ to the universe $\Univ[\Over{\BB}{A}]$~\cite[Remark~3.7.2]{martini2021}.
	\end{enumerate}
	From these preservation properties, one can now infer that virtually all constructions that we carry out in $\Cat(\BB)$ are preserved by $\pi_A^\ast$, see Example~\ref{ex:representabilityLocalCondition} below for a few specific  instances.
\end{remark}

\begin{remark}
	\label{rem:everyCoverISSmall}
	In the locality rule, we need not assume that a cover $(s_i)\colon\bigsqcup_i A_i\onto A$ is \emph{small}. In fact, since $\BB$ is presentable and therefore admits a small full subcategory $\GG\subset \BB$ that is \emph{dense} in $\BB$ (i.e.\ which has the property that every $A\in\BB$ is the colimit of the diagram $\Over{\GG}{A}\to \BB$), every large cover can be refined by a small one.
\end{remark}

\begin{remark}
	\label{rem:etaleInvarianceImposed}
	Very often, we simply impose invariance under \'etale transposition by \emph{defining} a property of $c\colon A\to\I{C}$ as a property of its transpose $\bar c\colon 1_{\Over{\BB}{A}}\to\pi_A^\ast\I{C}$ (see for example Definition~\ref{def:initialObject}).
\end{remark}

\begin{remark}
	\label{rem:etaleInvarianceReductionGlobalContext}
	Locality and invariance under \'etale transposition imply that the context of an object is largely irrelevant: if we wish to study the properties of an object $c\colon A\to\I{C}$ in an $\BB$-category $\I{C}$, we may simply pass to the slice $\infty$-topos $\Over{\BB}{A}$, replace $\I{C}$ by $\pi_A^\ast\I{C}$ and $c$ by its transpose $\bar c\colon 1_{\Over{\BB}{A}}\to\pi_A^\ast\I{C}$ and can thus assume that $c$ has had global context to begin with. Note that by locality, $\pi_A^\ast\I{C}$ arises from the very same constructions (internally in $\Over{\BB}{A}$) that are used to define $\I{C}$ (internally in $\BB$), hence every statement about the objects of $\I{C}$  also makes sense as a statement about the objects of $\pi_A^\ast\I{C}$.
\end{remark}

\begin{remark}
	\label{rem:localityPrinciplePropositions}
	If $\I{C}$ is a $\BB$-category and if $P(c)$ is a proposition about an object $c\colon A\to\I{C}$ in context $A\in\BB$, then locality implies that there is a full subcategory $\I{P}\into\I{C}$ that classifies $P$, in the sense that an object $c\colon A\to\I{C}$ is contained in $\I{P}$ if and only if $P(c)$ is true. In fact, we may define $\I{P}$ as the full subcategory that is spanned by the objects $c\colon A\to\I{C}$ in arbitrary context $A$ for which $P(c)$ holds. Explicitly, $\I{P}$ is the unique full subcategory of $\I{C}$ for which $\I{P}_0\into\I{C}_0$ is the image of the map
	\begin{equation*}
		\bigsqcup_{\substack{c\colon A\to\I{C}\\ P(c)~\textrm{holds}}} A\to \I{C}_0
	\end{equation*}
	(cf.~Proposition~\ref{prop:fullSubcategoriesParametrisation}). This means that for the tautological object $\tau\colon \I{P}_0\to\I{P}$ (see Remark~\ref{rem:tautologicalObject}) there is a cover $(s_i)\colon\bigsqcup_i A_i\onto \I{P}_0$ such that $P(s_i^\ast\tau)$ holds for each $i$. Since every object of $\I{P}$ is a pullback of $\tau$ and since covers are stable under base change in $\BB$, this implies that for every object $c\colon A\to\I{P}$ there is a cover $(s_i)\colon \bigsqcup_i A_i\onto A$ such that $P(s_i^\ast c)$ holds. Using the locality rule, we thus deduce that $P(c)$ must be true. Consequently, an object $c\colon A\to\I{C}$ is contained in $\I{P}$ if and only if $P(c)$ holds, as claimed.
\end{remark}

\begin{remark}
	\label{rem:localityPrincipleBaseChangeProposition}
	By combining Remarks~\ref{rem:etaleInvarianceReductionGlobalContext} and~\ref{rem:localityPrinciplePropositions}, if $P(c)$ is a proposition about an object $c\colon A\to\I{C}$ in a $\BB$-category $\I{C}$ and if $\I{P}\into\I{C}$ is the associated classifying full subcategory, then $\pi_A^\ast\I{P}\into\pi_A^\ast\I{C}$ classifies the proposition $P$ interpreted internally in $\Over{\BB}{A}$. In fact, $\pi_A^\ast\I{P}$ is the full subcategory of $\pi_A^\ast\I{C}$ that is spanned by those objects $\bar c\colon B\to\pi_A^\ast\I{C}$ in context $B\in\Over{\BB}{A}$ for which the transpose $c\colon B\to\I{C}$ satisfies $P(c)$ (interpreted internally in $\BB$), which by invariance under \'etale transposition is equivalent to $\bar c$ satisfying $P(\bar c)$ (interpreted internally in $\Over{\BB}{A}$).
\end{remark}

\begin{example}
	\label{ex:representabilityLocalCondition}
	Suppose that $\I{C}$ is a $\BB$-category. Then locality asserts that for every $A\in\BB$, one obtains an equivalence $\pi_A^\ast\IPSh(\I{C})\simeq \IPSh[\Over{\BB}{A}](\pi_A^\ast\I{C})$ (cf.\ the list in Remark~\ref{rem:localityPrinciplePreservationStructure}). In light of this equivalence, one can furthermore identify $\pi_A^\ast(h_{\I{C}})$ with $h_{\pi_A^\ast\I{C}}$~\cite[Lemma~4.7.14]{martini2021} (where $h_{\I{C}}$ is the Yoneda embedding). Hence, an object $F\colon A\to\IPSh(\I{C})$ is representable if and only if its transpose $\bar F\colon1_{\Over{\BB}{A}}\to\IPSh(\pi_A^\ast\I{C})$ is representable, so that this property is indeed invariant under \'etale transposition. It also satisfies the second part of the locality principle, which can be seen as follows: given a cover $(s_i)\colon\bigsqcup_i A_i\onto A$ in $\BB$, the presheaf $F$ being representable precisely means that the map $F\colon A\to \IPSh(\I{C})$ factors through the Yoneda embedding $h\colon \I{C}\into\IPSh(\I{C})$, so clearly $F$ being representable implies that $s_i^\ast(F)=Fs_i$ is representable. Conversely, if each $s_i^\ast(F)$ is representable, one can form the lifting problem
	\begin{equation*}
		\begin{tikzcd}
			\bigsqcup_i A_i\arrow[d, twoheadrightarrow, "(s_i)"]\arrow[r] & \I{C}\arrow[d, "h", hookrightarrow]\\
			A\arrow[r, "F"] \arrow[ur, dashed]& \IPSh(\I{C})
		\end{tikzcd}
	\end{equation*}
	which admits a unique solution (since covers and monomorphisms form a factorisation system in $\BBB$), hence the result follows.
\end{example}

\section{Adjunctions}
\label{chap:Adjunctions}
In this section we will study \emph{adjunctions} between $\BB$-categories. We begin in \S~\ref{sec:adjunctionsDefinitions} by defining such adjunctions as ordinary adjunctions in the underlying bicategory of $\Cat(\BB)$. In \S~\ref{sec:relativeAdjunctions} we compare our definition with \emph{relative} adjunctions and prove a convenient section-wise criterion for when a functor admits a left or right adjoint. In \S~\ref{sec:adjunctionsMappingGroupoids} we discuss an alternative approach to adjunctions based on an equivalence of mapping $\BB$-groupoids. Finally, we discuss the special case of \emph{reflective subcategories} in \S~\ref{sec:reflectiveSubcategories}. 

\subsection{Definitions and basic properties}
\label{sec:adjunctionsDefinitions}
Let $\I{C}$ and $\I{D}$ be $\BB$-categories, let $f,g\colon \I{C}\rightrightarrows\I{D}$ be two functors and let $\alpha\colon f\to g$ be a morphism of functors, i.e.\ a map in $\Fun_{\BB}(\I{C},\I{D})$.  If $h\colon \I{E}\to\I{C}$ is any other functor, we denote by $\alpha h\colon fh\to gh$ the map $h^\ast(\alpha)$ in $\Fun_{\BB}(\I{E},\I{D})$. Dually, if $k\colon \I{D}\to\I{E}$ is an arbitrary functor, we denote by $k\alpha\colon kf\to kg$ the map $k_\ast(\alpha)$ in $\Fun_{\BB}(\I{C},\I{E})$. Having established the necessary notational conventions, we may now define:
\begin{definition}
	\label{def:internalAdjunction}
	Let $\I{C}$ and $ \I{D}$ be $\BB$-categories. An \emph{ adjunction} between $\I{C}$ and $\I{D}$ is a tuple $(l,r,\eta,\epsilon)$, where $l\colon \I{C}\to\I{D}$ and $r\colon\I{D}\to\I{C}$ are functors and where $\eta\colon \id_{ \I{D}}\to rl$ and $\epsilon \colon lr\to \id_{ \I{C}}$ are maps such that there are commutative triangles
	\begin{equation*}
		\begin{tikzcd}
			l\arrow[r, "l\eta"]\arrow[dr, "\id"'] & lrl \arrow[d, "\epsilon l"] & & rlr \arrow[from=r, "\eta r"']\arrow[d, "r\epsilon"'] & r \arrow[dl, "\id"]\\
			& l & & r
		\end{tikzcd}
	\end{equation*}
	in $\Fun_{\BB}(\I{C}, \I{D})$ and in $\Fun_{\BB}(\I{D}, \I{C})$, respectively. We denote such an adjunction by $l\dashv r$, and we refer to $\eta$ as the \emph{unit} and to $\epsilon$ as the \emph{counit} of the adjunction. We say that a pair $(l,r)\colon \I{C}\leftrightarrows \I{D}$ \emph{defines an adjunction} if there exist transformations $\eta$ and $\epsilon$ as above such that the tuple $(l,r,\eta,\epsilon)$ is an adjunction.
\end{definition}

Analogous to how adjunctions between $\infty$-categories can be defined (see~\cite[\S 17]{joyal2008notes}), Definition~\ref{def:internalAdjunction} is equivalent to an adjunction in the underlying homotopy bicategory of the $(\infty,2)$-category $\Cat(\BB)$ (see \S~\ref{sec:BCategories}). We may therefore make use of the usual bicategorical arguments to derive results for adjunctions in $\Cat(\BB)$. 
Hereafter, we list a few of these results, we refer the reader to~\cite[\S~I.6]{gray1974} and \cite[\S~2.1]{Riehl2022a} for proofs.
\begin{proposition}
	\label{prop:adjunctionComposition}
	If $(l \dashv r)\colon \I{C}\leftrightarrows \I{D}$ and $(l^\prime\dashv r^\prime)\colon \I{D}\leftrightarrows \I{E}$ are adjunctions between $\BB$-categories, then the composite functors define an adjunction $(ll^\prime\dashv r^\prime r)\colon \I{C}\leftrightarrows \I{E}$.\qed
\end{proposition}
\begin{proposition}
	\label{prop:uniquenessAdjoints}
	Adjoints are unique if they exist, i.e\ if $(l\dashv r)$ and $(l\dashv r^\prime)$ are adjunctions between $\BB$-categories, then $r\simeq r^\prime$. Dually, if $(l\dashv r)$ and $(l^\prime\dashv r)$ are adjunctions, then $l\simeq l^\prime$.\qed
\end{proposition}

\begin{proposition}
	\label{prop:minimalAdjunctionData}
	In order for a pair $(l,r)\colon \I{C}\leftrightarrows \I{D}$ of functors between $\BB$-categories to define an adjunction, it suffices to provide maps $\eta\colon \id_{ \I{D}}\to rl$ and $\epsilon \colon lr\to \id_{ \I{C}}$ such that the compositions $\epsilon l\circ l\eta$ and $r\epsilon\circ \eta r$ are equivalences.\qed
\end{proposition}
\begin{corollary}
	If $f\colon \I{C}\to \I{D}$ is an equivalence between $\BB$-categories, then the pair $(f, f^{-1})$ defines an adjunction.\qed
\end{corollary}
\begin{corollary}
	For any adjunction $(l\dashv r)\colon \I{C}\leftrightarrows \I{D}$ between $\BB$-categories and any equivalence $f\colon D\simeq D^\prime$, the induced pair $(lf^{-1}, fr)\colon \I{C}\leftrightarrows \I{D}^{\prime}$ defines an adjunction as well.\qed
\end{corollary}

If $\AA$ and $\BB$ are $\infty$-topoi and $f\colon \Cat(\BB)\to\Cat(\AA)$ is a functor, we will often need to know whether $f$ carries an adjunction $l\dashv r$ in $\Cat(\BB)$ to an adjunction $f(l)\dashv f(r)$ in $\Cat(\AA)$. This is obviously the case whenever $f$ is a functor of $(\infty,2)$-categories. Since we do not wish to dive too deep into $(\infty,2)$-categorical arguments, we will instead make use of the straightforward observation that $f$ preserves adjunctions whenever there is a bifunctorial map
\begin{equation*}
\Fun_{\BB}(-,-)\to \Fun_{\AA}(f(-),f(-))
\end{equation*}
that recovers the action of $f$ on mapping $\infty$-groupoids upon postcomposition with the core $\infty$-groupoid functor.
\begin{lemma}
	\label{lem:2functors}
	Let $\AA$ and $\BB$ be $\infty$-topoi and let $f\colon\Cat(\BB)\to\Cat(\AA)$ be a functor that preserves finite products. Suppose furthermore that there is a morphism of functors $\const_{{\AA}}\to f\circ\const_{\BB}$, where $\const_{\BB}\colon\CatS\to\Cat(\BB)$ and $\const_{\AA}\colon \CatS\to\Cat(\AA)$ are the constant sheaf functors. Then $f$ induces a bifunctorial map $\Fun_{\BB}(-,-)\to \Fun_{\AA}(f(-),f(-))$ that recovers the action of $f$ on mapping $\infty$-groupoids upon postcomposition with the core $\infty$-groupoid functor. Moreover, if $f$ is fully faithful and if the map $\const_{{\AA}}\to f\circ\const_{\BB}$ restricts to an equivalence on the essential image of $f$, then this map is an equivalence.
\end{lemma}
\begin{proof}
	Since $f$ preserves finite products, the map $\const_{\AA}\to f\circ\const_{\BB}$ induces a map
	\begin{equation*}
	-\otimes f(-)\to f(-\otimes -)
	\end{equation*}
	of bifunctors $\CatS\times\Cat(\BB)\to \Cat(\AA)$. This map gives rise to the first arrow in the composition 
	\begin{equation*}
	\map{\Cat(\AA)}(f(-\otimes -), f(-))\to \map{\Cat(\AA)}(-\otimes f(-), f(-))\simeq \map{\CatS}(-, \Fun_{\AA}(f(-), f(-))),
	\end{equation*}
	and by precomposition with the morphism $\map{\Cat(\BB)}(-\otimes-,-)\to\map{\Cat(\AA)}(f(-\otimes -), f(-))$ that is induced by $f$ and Yoneda's lemma, we end up with the desired morphism of functors
	\begin{equation*}
	\Fun_{\BB}(-,-)\to\Fun_{\AA}(f(-), f(-))
	\end{equation*}
	that recovers the morphism $\map{\Cat(\BB)}(-,-)\to\map{\Cat(\AA)}(f(-), f(-))$ upon restriction to core $\infty$-groupoids. By construction, this map is an equivalence whenever $f$ is fully faithful and the map $\const_{{\AA}}\to f\circ\const_{\BB}$ is an equivalence.
\end{proof}
\begin{remark}
	\label{rem:2FunctorialityExplicit}
	In the situation of Lemma~\ref{lem:2functors}, the construction in the proof shows that if $\I{C}$ and $\I{D}$ are $\BB$-categories, the functor
	\begin{equation*}
		\Fun_{\BB}(\I{C},\I{D})\to\Fun_{\AA}(f(\I{C}),f(\I{D}))
	\end{equation*}
	that is induced by $f$ and the morphism of functors $\phi\colon -\otimes f(-)\to f(-\otimes -)$ is given as the transpose of the composition
	\begin{equation*}
			\Fun_{\BB}(\I{C},\I{D})\otimes f(\I{C})\xrightarrow{\phi} f(\Fun_{\BB}(\I{C},\I{D})\otimes \I{C})\xrightarrow{f(\ev)} f(\I{D})
	\end{equation*}
	in which $\ev\colon \Fun_{\BB}(\I{C},\I{D})\otimes\I{C}\to \I{D}$ denotes the counit of the adjunction $-\otimes \I{C}\dashv \Fun_{\BB}(\I{C},-)$.
\end{remark}
Using Lemma~\ref{lem:2functors}, one now finds:
\begin{corollary}
	\label{cor:geometricMorphismAdjunction}
	Let $f_\ast\colon \BB\to \AA$ be a geometric morphism of $\infty$-topoi. If a pair $(l,r)$ of functors in $\Cat(\BB)$ defines an adjunction, then the pair $(f_\ast(l), f_\ast(r))$ defines an adjunction in $\Cat(\AA)$. Moreover, the converse is true whenever $f_\ast$ is fully faithful.
	
	Dually, for any algebraic morphism $f^\ast\colon \AA\to\BB$ of $\infty$-topoi, if a pair $(l,r)$ of functors in $\Cat(A)$ defines an adjunction, then the pair $(f^\ast(l), f^\ast(r))$ defines an adjunction in $\Cat(\BB)$, and the converse is true whenever $f^\ast$ is fully faithful.
\end{corollary}
\begin{proof}
	This follows immediately from Lemma~\ref{lem:2functors} on account of the equivalence $\const_{\BB}\simeq f^\ast\circ\const_{{\AA}}$ and the map $\const_{\AA}\to f_\ast \const_{\BB}$ that is induced by the adjunction unit $\id_{\AA}\to f_\ast f^\ast$.
\end{proof}

Recall from Proposition~\ref{prop:GroupoidificationCore} that the inclusion $\BB\simeq\Grpd(\BB)\into\Cat(\BB)$ admits a left adjoint $(-)^\gp$. We now obtain:
\begin{corollary}
    \label{cor:groupoidificationAdjunction}
    The groupoidification functor $(-)^{\gp}\colon \Cat(\BB)\to\Grpd(\BB)$ preserves adjunctions and therefore carries any left or right adjoint functor to an equivalence in $\Grpd(\BB)$.
\end{corollary}
\begin{proof}
    The first part follows by applying Lemma~\ref{lem:2functors} to the map $\eta\const_{\BB}\colon\const_{\BB}\to (-)^{\gp}\circ\const_{\BB}$ in which $\eta\colon \id_{\Cat(\BB)}\to (-)^{\gp}$ denotes the adjunction unit. As for the second part, it suffices to note that if $(l\dashv r)\colon \I{G}\leftrightarrows\I{H}$ is an adjunction between $\BB$-groupoids, then since both $\Fun_{\BB}(\I{G},\I{G})$ and $\Fun_{\BB}(\I{H},\I{H})$ are $\infty$-groupoids both unit and counit must be an equivalence.
\end{proof}

\begin{corollary}
	\label{cor:internalFunctorCategoryAdjunction}
	For any simplicial object $K\in\Simp\BB$, the endofunctor $\iFun(K,-)$ on $\Cat(\BB)$ preserves adjunctions in $\Cat(\BB)$.
\end{corollary}
\begin{proof}
	By bifunctoriality of $\iFun(-,-)$, precomposition with the terminal map $K\to 1$ in $\Simp\BB$ gives rise to the diagonal functor $\id_{\Cat(\BB)}\to \iFun(K,-)$, and combining this map with the functor $\const_{\BB}$ then defines a map $\const_{\BB}(-)\to \iFun(K, \const_{\BB}(-))$, hence Lemma~\ref{lem:2functors} applies.
\end{proof}
\begin{remark}
	\label{rem:2functorInternalHomExplicit}
	In the situation of Corollary~\ref{cor:internalFunctorCategoryAdjunction}, Remark~\ref{rem:2FunctorialityExplicit} shows that for any two $\BB$-categories $\I{C}$ and $\I{D}$, the induced map
	\begin{equation*}
		\Fun_{\BB}(\I{C},\I{D})\to\Fun_{\BB}(\iFun(K,\I{C}), \iFun(K,\I{D}))
	\end{equation*}
	is the one that is determined by the composition
	\begin{equation*}
		\Fun_{\BB}(\I{C},\I{D})\otimes(\iFun(K,\I{C})\times K)\xrightarrow{\id\otimes \ev_{K}} \Fun_{\BB}(\I{C},\I{D})\otimes \I{C}\xrightarrow{\ev_{\I{C}}} \I{D}
	\end{equation*}
	in light of the two adjunctions $-\times K\dashv \iFun(K,-)$ and $-\otimes \I{C}\dashv \Fun_{\BB}(\I{C},-)$. Here $\ev_{K}$ and $\ev_{\I{C}}$, respectively, denote the counits of these adjunctions.
\end{remark}
Combining Corollary~\ref{cor:geometricMorphismAdjunction} with Corollary~\ref{cor:internalFunctorCategoryAdjunction}, one furthermore obtains:
\begin{corollary}
	\label{cor:2functorFun}
	For any simplicial object $K\in\Simp\BB$, the functor $\Fun_{\BB}(K,-)\colon \Cat(\BB)\to\CatS$ carries adjunctions in $\Cat(\BB)$ to adjunctions in $\CatS$.\qed
\end{corollary}
Similarly as above, if $\AA$ and $\BB$ are $\infty$-topoi and if $f\colon \Cat(\BB)\to\Cat(\AA)$ is a functor such that there is a bifunctorial map
\begin{equation*}
\Fun_{\BB}(-,-)\to \Fun_{\AA}(f(-),f(-))^{\op}
\end{equation*}
that recovers the action of $f$ on mapping $\infty$-groupoids upon postcomposition with the core $\infty$-groupoid functor, the functor $f$ sends an adjunction $l\dashv r$ in $\Cat(\BB)$ to an adjunction $f(r)\dashv f(l)$ in $\Cat(\AA)$. One therefore finds:
\begin{proposition}
	\label{prop:adjunctionOpposite}
	The equivalence $(-)^{\op}\colon\Cat(\BB)\to\Cat(\BB)$ sends an adjunction $l\dashv r$ to an adjunction $r^{\op}\dashv l^{\op}$.
\end{proposition}
\begin{proof}
	This follows from the evident equivalence
	\begin{equation*}
		(-)^{\op}\colon\Fun_{\BB}(-,-)\simeq \Fun_{\BB}((-)^{\op}, (-)^{\op})^{\op}
	\end{equation*}
	of bifunctors $\Cat(\BB)^{\op}\times\Cat(\BB)\to\CatS$, which shows that if $l\dashv r$ is an adjunction with unit $\eta$ and counit $\epsilon$, then the pair $(r^{\op},l^{\op})$ defines an adjunction on account of the maps $\epsilon^{\op}\colon\id\to l^{\op} r^{\op}$ and $\eta^{\op}\colon r^{\op}l^{\op}\to \id$ that correspond to $\epsilon$ and $\eta$ via the above equivalence.
\end{proof}

The contravariant versions of the functors considered in Corollary~\ref{cor:internalFunctorCategoryAdjunction} and Corollary~\ref{cor:2functorFun} preserve adjunctions as well: If $\I{C}$ is an arbitrary $\BB$-category, functoriality of $\iFun(-,\I{C})$ defines a map
\begin{equation*}
	\map{\Cat(\BB)}(\I{E}, \I{D})\to \map{\Cat(\BB)}(\iFun(\I{D},\I{C}), \iFun(\I{E},\I{C}))
\end{equation*}
that is natural in $\I{E}$ and $\I{D}$. The composition
\begin{align*}
	\map{\Cat(\BB)}(-\otimes \I{E},\I{D})&\to \map{\Cat(\BB)}(\iFun(\I{D},\I{C}), \iFun(-\otimes\I{E},\I{C}))\\
	&\simeq \map{\Cat(\BB)}(\iFun(\I{D},\I{C})\times(-\otimes\I{E}), \I{C})\\
	&\simeq \map{\Cat(\BB)}((-\otimes\iFun(\I{D},\I{C}))\times\I{E}, \I{C})\\
	&\simeq\map{\Cat(\BB)}(-\otimes\iFun(\I{D},\I{C}), \iFun(\I{E},\I{C}))
\end{align*}
(in which each step is natural in $\I{D}$ and $\I{E}$) and Yoneda's lemma now give rise to a map
\begin{equation*}
	\Fun_{\BB}(\I{E},\I{D})\to \Fun_{\BB}(\iFun(\I{D},\I{C}), \iFun(\I{E},\I{C}))
\end{equation*}
that defines a morphism of functors $\Cat(\BB)^{\op}\times\Cat(\BB)\to \CatS$ and that recovers the action of $\iFun(-,\I{C})$ on mapping $\infty$-groupoids upon postcomposition with the core $\infty$-groupoid functor. One therefore finds:
\begin{proposition}
	\label{prop:2functorFunContravariant}
	For any $\BB$-category $\I{C}$, the two functors $\iFun(-,\I{C})$ and $\Fun_{\BB}(-, \I{C})$ carry an adjunction $l \dashv r$ in $\Cat(\BB)$ to an adjunction $r^\ast\dashv l^\ast$ in $\Cat(\BB)$ and in $\CatS$, respectively. \qed
\end{proposition}

\subsection{Adjunctions via relative adjunctions of cartesian fibrations}
\label{sec:relativeAdjunctions}
Recall from the discussion in~\ref{sec:parametrisedCategories} that every pair $(l,r)\colon\I{C}\leftrightarrows\I{D}$ of functors between (large) $\BB$-categories give rise to a pair of functors $(\int l,\int r)\colon \int \I{C}\leftrightarrows\int \I{D}$ between the associated cartesian fibrations over $\BB$. In this section, our goal is to characterise those pairs $(\int l,\int r)$ that come from an adjunction $l\dashv r$.

Given any small $\infty$-category $\CC$, there is a bifunctor
\begin{equation*}
	-\otimes -\colon \CatS\times\Cart(\CC)\to\Cart(\CC)
\end{equation*}
that sends a pair $(\XX,\PP\to\CC)$ to the cartesian fibration $\XX\times\PP\to\PP\to\CC$ in which the first arrow is the natural projection. Explicitly, a morphism in $\XX\times \PP$ is cartesian precisely if its projection to $\PP$ is cartesian in $\PP$ and its projection to $\XX$ is an equivalence. For an arbitrary fixed cartesian fibration $\PP\to\CC$, the functor $-\otimes\PP\colon \CatS\to\Cart(\CC)\into\Over{(\CatSS)}{\CC}$ admits a right adjoint $\Fun_{/\CC}(\PP,-)$ that sends a map $\QQ\to \CC$ to the $\infty$-category that is defined by the pullback square
\begin{equation*}
	\begin{tikzcd}
		\Fun_{/\CC}(\PP,\QQ)\arrow[d]\arrow[r] & \Fun(\PP,\QQ)\arrow[d]\\
		1\arrow[r] & \Fun(\PP,\CC)
	\end{tikzcd}
\end{equation*}
in which the vertical map on the right is given by postcomposition with $\QQ\to\CC$ and in which the lower horizontal arrow picks out the cartesian fibration $\PP\to \CC$~\cite[Proposition~5.2.5.1]{htt}.
If $\QQ\to \CC$ is a cartesian fibration, let $\Fun_{/\CC}^{\Cart}(\PP,\QQ)\into\Fun_{/\CC}(\PP,\QQ)$ denote the full subcategory that is spanned by those functors that preserve cartesian edges, and observe that this defines a functor 
\begin{equation*}
	\Fun_{/\CC}^{\Cart}(\PP,-)\colon \Cart(\CC)\to \CatS.
\end{equation*}
As the equivalence $\map{/\CC}(\XX\otimes\PP,\QQ)\simeq \map{\CatSS}(\XX, \Fun_{/\CC}(\PP,\QQ))$ identifies functors $\XX\otimes \PP\to \QQ$ that preserve cartesian arrows with functors $\XX\to\Fun_{/\CC}(\PP,\QQ)$ that take values in $\Fun_{/\CC}^{\Cart}(\PP,\QQ)$, one obtains an adjunction $(-\otimes \PP\dashv \Fun_{/\CC}^{\Cart}(\PP,-))\colon \CatSS\leftrightarrows\Cart(\CC)$. By making use of the bifunctoriality of $-\otimes -$, the assignment $\PP\mapsto \Fun_{/\CC}^{\Cart}(\PP,-)$ gives rise to a bifunctor $\Fun_{/\CC}^{\Cart}(-,-)$ in a unique way such that there is an equivalence
\begin{equation*}
	\map{\Cart(\CC)}(-\otimes -, -)\simeq\map{\CatS}(-, \Fun_{/\CC}^{\Cart}(-,-)).
\end{equation*}
Note that there is an equivalence $\int (-\otimes -)\simeq -\otimes \int (-)$ of bifunctors $\CatS\times \Cat(\PSh_{\SS}(\CC))\to \Cart(\CC)$ in which the tensoring on the left-hand side is given by the canonical tensoring in $\Cat(\PSh_{\SS}(\CC))$ over $\CatS$, i.e.\ by the bifunctor $\const(-)\times -$. By the uniqueness of adjoints, one therefore finds:
\begin{proposition}
	\label{prop:comparisonMappingFunctorsGrothendieckConstruction}
	For any small $\infty$-category $\CC$, there is an equivalence
	\begin{equation*}
		\Fun_{\PSh_{\SS}(\CC)}(-,-)\simeq \Fun_{/\CC}^{\Cart}(\smallint(-),\smallint(-))
	\end{equation*}
	of bifunctors $\Cat(\PSh_{\SS}(\CC))^{\op}\times\Cat(\PSh_{\SS}(\CC))\to\CatS$ that recovers the action of $\int\colon \PSh_{\SS}(\CC)\to\Cart(\CC)$ on mapping $\infty$-groupoids upon postcomposition with the core $\infty$-groupoid functor. \qed
\end{proposition}
Recall the notion of a \emph{relative adjunction} between cartesian fibrations as defined by Lurie in~\cite[\S~7.3]{lurie2017}:
\begin{definition}
	\label{def:relativeAdjunction}
	Let $\CC$ be an $\infty$-category and let $\PP$ and $\QQ$ be cartesian fibrations over $\CC$. A \emph{relative adjunction} between $\PP$ and $\QQ$ is defined to be an adjunction $(l \dashv r)\colon\QQ\leftrightarrows\PP$ between the underlying $\infty$-categories such that both $l$ and $r$ define maps in $\Cart(\CC)$ and such that the structure map $p\colon \PP\to \CC$ sends the adjunction counit $\epsilon$ to the identity transformation on $p$ and the structure map $q\colon\QQ\to \CC$ sends the adjunction unit $\eta$ to the identity transformation on $q$.
\end{definition}
By construction of the bifunctor $\Fun_{\CC}^{\Cart}(-,-)$, it is immediate that a pair $(l,r)\colon \QQ\leftrightarrows \PP$ of maps in $\Cart(\CC)$ defines a relative adjunction if and only if there are maps $\eta \colon \id_{\QQ}\to rl$ and $\epsilon\colon lr\to\id_{\PP}$ in $\Fun_{\CC}^{\Cart}(\QQ,\QQ)$ and $\Fun_{\CC}^{\Cart}(\PP,\PP)$, respectively, that satisfy the triangle identities from Definition~\ref{def:internalAdjunction}. Proposition~\ref{prop:comparisonMappingFunctorsGrothendieckConstruction} therefore implies:
\begin{corollary}
    \label{cor:internalRelativeAdjunctionPSh}
    For any small $\infty$-category $\CC$, a pair $(l,r)\colon\I{C}\leftrightarrows\I{D}$ of functors between $\PSh_{\SS}(\CC)$-categories defines an adjunction if and only if the associated pair $(\int l, \int r)$ defines a relative adjunction in $\Cart(\CC)$.\qed
\end{corollary}
Observe that as by~\cite[Lemma~6.3.5.28]{htt} the inclusion $\BBB\into\PSh_{\SSS}(\BB)$ defines a geometric morphism of $\infty$-topoi (relative to the universe $\bV$), Corollary~\ref{cor:geometricMorphismAdjunction} implies that the pair $(l,r)$ defines an adjunction between large $\BB$-categories if and only if it defines an adjunction in $\PSh_{\SSS}(\BB)$. We may therefore conclude:
\begin{corollary}
	\label{cor:adjunctionCartesianFibrations}
	A pair $(l,r)\colon\I{C}\leftrightarrows\I{D}$ of functors between large $\BB$-categories defines an adjunction if and only if the associated pair $(\int l, \int r)$ defines a relative adjunction in $\Cart(\BB)$.\qed
\end{corollary}
The upshot of Corollary~\ref{cor:adjunctionCartesianFibrations} is that we may make use of Lurie's results on relative adjunctions in order to formulate a useful criterion for when a functor between $\BB$-categories admits a right and a left adjoint, respectively. For this we need to recall the \emph{mate} construction:
\begin{definition}
	\label{def:mates}
	For any right lax square in $\Cat(\BB)$ of the form
	\begin{equation*}
		\begin{tikzcd}
			\I{C}_1\arrow[r, "r_1"] \arrow[d, "f"] & \I{D}_1 \arrow[d, "g"] \arrow[dl, Rightarrow, "\phi"', shorten=2mm]\arrow[d]\\
			\I{C}_2\arrow[r, "r_2"]  & \I{D}_2
		\end{tikzcd}    
	\end{equation*}
	such that both $r_1$ and $r_2$ admit left adjoints $l_1$ and $l_2$ exhibited by units $\eta_i\colon \id\to r_i l_i$ and counits $\epsilon_i\colon l_i r_i\to \id$, there is a left lax square
	\begin{equation*}
		\begin{tikzcd}
			\I{C}_1\arrow[from = r, "l_1"'] \arrow[d, "f"]\arrow[from =dr, Rightarrow, "\psi"', shorten=2mm] & \I{D}_1 \arrow[d, "g"] \arrow[d]\\
			\I{C}_2\arrow[from = r, "l_2"']  & \I{D}_2
		\end{tikzcd}    
	\end{equation*}
	in which $\psi$ is defined as the composite map
	\begin{equation*}
			l_2 g\xrightarrow{l_2 g \eta_1} l_2 g r_1 l_1 \xrightarrow{l_2 \phi l_1} l_2 r_2 f l_1\xrightarrow{\epsilon_2 f l_1}  f l_1.
	\end{equation*}
	Conversely, when starting with the latter left lax square, the original right lax square is recovered by means of the composition
	\begin{equation*}
			g r_1\xrightarrow{\eta_2 g r_1} r_2 l_2 g r_1 \xrightarrow{r_2 \psi r_1} r_2 f l_1 r_1\xrightarrow{r_2 f \epsilon_1} r_2 f .  
	\end{equation*}
	The left lax square determined by $\psi$ is referred to as the \emph{mate} of the right lax square determined by $\psi$, and vice versa.
\end{definition}

\begin{remark}
In the 2-categorical context mates have been studied under the name \emph{adjoint squares} by Gray in \cite[\S I.6]{gray1974}, and under the name mate in \cite[\S 2]{kelly2006}.
In the $(\infty,2)$-categorical setting they have been studied by Haugseng, see the discussion following \cite[Remark~4.5]{Haugseng2021a}.
In the case where the starting 2-cell is invertible, which we will mostly use, they are also already considered in \cite[Definition 4.7.4.13]{lurie2017}.
\end{remark}

\begin{remark}
	\label{rem:functorialityMates}
	The mate construction is functorial in the following sense: Consider the composition of right lax squares
	\begin{equation*}
		\begin{tikzcd}
			\I{C}_1\arrow[r, "r_1"] \arrow[d, "f_1"] & \I{D}_1 \arrow[d, "g_1"] \arrow[dl, Rightarrow, "\phi_1"', shorten=2mm]\arrow[d]\\
			\I{C}_2\arrow[r, "r_2"] \arrow[d, "f_2"] & \I{D}_2 \arrow[d, "g_2"] \arrow[dl, Rightarrow, "\phi_2"', shorten=2mm]\arrow[d]\\
			\I{C}_3\arrow[r, "r_3"]  & \I{D}_3, 
		\end{tikzcd}    
	\end{equation*}
	by which we simply mean the composition $(\phi_2 f_1)\circ(g_2\phi_1)$. Then the mate of the composite square is given by the composition of left lax squares
	\begin{equation*}
		\begin{tikzcd}
			\I{C}_1\arrow[from = r, "l_1"'] \arrow[d, "f_1"]\arrow[from =dr, Rightarrow, "\psi_1"', shorten=2mm] & \I{D}_1 \arrow[d, "g_1"] \arrow[d]\\
			\I{C}_2\arrow[from = r, "l_2"'] \arrow[d, "f_2"]\arrow[from =dr, Rightarrow, "\psi_2 "', shorten=2mm] & \I{D}_2 \arrow[d, "g_2"] \arrow[d]\\
			\I{C}_3\arrow[from = r, "l_3"']  & \I{D}_3,
		\end{tikzcd}    
	\end{equation*}
	in which $\psi_1$ denotes the mate of $\phi_1$ and $\psi_2$ denotes the mate of $\phi_2$. This is easily checked using the triangle identities for adjunctions and the interchange law in bicategories.
	
	Similarly, one can show that the mate of the \emph{horizontal} composition of right lax squares
	\begin{equation*}
	\begin{tikzcd}
		\I{C}_1\arrow[r, "r_1"] \arrow[d, "f"] & \I{D}_1 \arrow[d, "g"] \arrow[dl, Rightarrow, "\phi_1"', shorten=2mm]\arrow[d] \arrow[r, "r_1^\prime"] & \I{E}_1\arrow[d, "h"]\arrow[dl, Rightarrow, "\phi_2"', shorten=2mm]\\
		\I{C}_2\arrow[r, "r_2"] & \I{D}_2\arrow[r, "r_2^\prime"] & \I{E}_2\\
	\end{tikzcd}    
	\end{equation*}
	(i.e.\ the composite $r_2^\prime \phi_1\circ\phi_2r_1$) is given by the horizontal composition of the associated mates.
\end{remark}

\begin{lemma}
	\label{lem:relativeAdjointBeckChevalley}
	Let $\CC$ be an $\infty$-category and let $p\colon\PP\to\CC$ and $q\colon\QQ\to\CC$ be cartesian fibrations. A map $r\colon \PP\to\QQ$ in $\Cart(\CC)$ is a relative right adjoint if and only if
	\begin{enumerate}
		\item for all $c\in\CC$ the functor $r\vert_c\colon \PP\vert_c\to\QQ\vert_c$ that is induced by $r$ on the fibres over $c$ admits a left adjoint $l_c\colon \QQ\vert_c\to\PP\vert_c$;
		\item For every morphism $g\colon d\to c$ in $\CC$, the mate of the commutative square
		\begin{equation*}
			\begin{tikzcd}
			\PP\vert_c\arrow[r, "r\vert_c"] \arrow[d, "g^\ast"'] & \QQ\vert_c \arrow[d, 	"g^\ast"] \arrow[dl, Rightarrow, "\simeq"', shorten=2mm]\arrow[d]\\
			\PP\vert_d\arrow[r, "r\vert_d"]  &\QQ\vert_d
			\end{tikzcd}
		\end{equation*}
		commutes.
	\end{enumerate}
	If this is the case, the relative left adjoint $l$ of $r$ recovers the map $l_c$ on the fibres over $c\in\CC$.
		
	Dually, a map $l\colon \QQ\to\PP$ in $\Cart(\CC)$ is a relative left adjoint if and only if
	\begin{enumerate}
		\item for all $c\in\CC$ the functor $l\vert_c\colon \QQ\vert_c\to\PP\vert_c$ that is induced by $l$ on the fibres over $c$ admits a right adjoint $r_c\colon \PP\vert_c\to\QQ\vert_c$;
		\item For every morphism $g\colon d\to c$ in $\CC$, the mate of the commutative square
		\begin{equation*}
			\begin{tikzcd}
				\PP\vert_c\arrow[from=r, "l\vert_c"'] \arrow[d, "g^\ast"'] & \QQ\vert_c \arrow[d, 	"g^\ast"] \arrow[d]\\
				\PP\vert_d\arrow[from=r, "l\vert_d"']  &\QQ\vert_d\arrow[ul, Rightarrow, "\simeq"', shorten=2mm]
			\end{tikzcd}
		\end{equation*}
		commutes.
	\end{enumerate}
	If this is the case, the relative right adjoint $r$ of $l$ recovers the map $r_c$ on the fibres over $c\in\CC$.	
\end{lemma}
\begin{proof}
	The second statement is the content of (the dual of) ~\cite[Proposition~7.3.2.11]{lurie2017}. The first statement, on the other hand, is a formal consequence of the second: in fact, in light of the straightening equivalence, there is an equivalence $(-)^{\vee,\op}\colon\Cart(\CC)\simeq\Cart(\CC)$ that is determined by the equivalence $(-)^\op_\ast\colon\PSh_{\CatSS}(\CC)\simeq\PSh_{\CatSS}(\CC)$ given by postcomposition with the involution $(-)^\op\colon\CatSS\simeq\CatSS$. By combining Proposition~\ref{prop:adjunctionOpposite} with Corollary~\ref{cor:internalRelativeAdjunctionPSh}, the equivalence $(-)^{\vee,\op}$ carries a relative left adjoint to a relative right adjoint, and it is evidently true that it translates the two conditions in the second statement to the two conditions in the first one. Since we already know that the second statement is verified, the first one therefore follows as well.
\end{proof}
By combining Corollary~\ref{cor:adjunctionCartesianFibrations} with Lemma~\ref{lem:relativeAdjointBeckChevalley}, we conclude:
\begin{proposition}
	\label{prop:existenceAdjointsBeckChevalley}
	A functor $r\colon\I{C}\to\I{D}$ in $\Cat(\BBB)$ is a right adjoint if and only if the following two conditions hold:
	\begin{enumerate}
		\item For any object $A\in \BB$, the induced functor $r(A)\colon \I{C}(A)\to \I{D}(A)$ is the right adjoint in an adjunction $(l_A,r(A),\eta_A,\epsilon_A)$.
		\item For any morphism $s\colon B\to A$ in $\BB$, the mate of the commutative square
		\begin{equation*}
			\begin{tikzcd}
				\I{C}(A)\arrow[r, "r(A)"] \arrow[d, "s^\ast"'] & \I{D}(A) \arrow[d, "s^\ast"] \arrow[dl, Rightarrow, "\simeq"', shorten=2mm]\arrow[d]\\
				\I{C}(B)\arrow[r, "r(B)"]  & \I{D}(B)
			\end{tikzcd}
		\end{equation*}
		commutes.
	\end{enumerate}
	If this is the case, then the left adjoint $l$ of $r$ is given on objects $A\in \BB$ by $l_A$ and on morphisms $s\colon B\to A$ by the mate of the commutative square defined by $r(s)$.
	
	Dually, a functor $l\colon \I{D}\to\I{C}$ in $\Cat(\BBB)$ is a left adjoint if and only if the following two conditions hold:
	\begin{enumerate}
		\item For any object $A\in \BB$, the induced map $l(A)\colon \I{D}(A)\to \I{C}(A)$ is the left adjoint in an adjunction $(l(A),r_A,\eta_A,\epsilon_A)$.
		\item For any morphism $s\colon B\to A$ in $\BB$, the mate of the commutative square
		\begin{equation*}
			\begin{tikzcd}
				\I{C}(A)\arrow[from=r, "l(A)"'] \arrow[d, "s^\ast"']\arrow[from=dr, Rightarrow, "\simeq"', shorten=2mm] & \I{D}(A) \arrow[d, "s^\ast"] \arrow[d]\\
				\I{C}(B)\arrow[from=r, "l(B)"']  & \I{D}(B)
			\end{tikzcd}
		\end{equation*}
		commutes.
	\end{enumerate}
	If this is the case, then a right adjoint $r$ of $l$ is given on objects $A\in \BB$ by $r_A$ and on morphisms $s\colon B\to A$ by the mate of the commutative square defined by $l(s)$.\qed
\end{proposition}

\begin{remark}
	\label{rem:existenceAdjointsBeckChevalleyLocalisation}
	In the situation of Proposition~\ref{prop:existenceAdjointsBeckChevalley}, suppose that the functor $r\colon \I{C}\to\I{D}$ is fully faithful and suppose that condition~(1) is satisfied. Since the mate of the commutative square in condition~(2) is given by the composition
	\begin{equation*}
	l_B g^\ast\xrightarrow{l_B g^\ast \eta_A} l_B g^\ast r(A) l_A\xrightarrow{\simeq} l_B r(B) g^\ast l_A\xrightarrow{\epsilon_B g^\ast l_A}g^\ast l_A
	\end{equation*}
	in which the map $\epsilon_B$ is an equivalence, the composition is an equivalence whenever the map $l_B g^\ast\eta_A$ is an equivalence. Since furthermore the map $l_A \eta_A$ is an equivalence as well, we may in this case replace condition~(2) by the a priori weaker condition that there exists an \emph{arbitrary} equivalence $l_B g^\ast\simeq g^\ast l_A$.
\end{remark}

Combining Lemma~\ref{lem:relativeAdjointBeckChevalley} with Corollary~\ref{cor:internalRelativeAdjunctionPSh} furthermore implies:
\begin{corollary}
	\label{cor:restricedBeckChevalleyAdj}
	Let $r \colon \I{C} \rightarrow \I{D}$ be a functor of $\BB$-categories and let $L \colon \PSh(\CC) \rightarrow \BB$ be a left exact localisation where $\CC$ is some small $\infty$-category.
	Then $r$ is a right adjoint if and only if the following two conditions hold:
		\begin{enumerate}
		\item For any object $c \in \CC$, the induced functor $r(Lc)\colon \I{C}(Lc)\to \I{D}(Lc)$ is a right adjoint.
		\item For any morphism $s\colon d \to c$ in $\CC$, the mate of the commutative square
		\begin{equation*}
			\begin{tikzcd}
				\I{C}(Lc)\arrow[r, "r(Lc)"] \arrow[d, "Ls^\ast"'] & \I{D}(Lc) \arrow[d, "Ls^\ast"] \arrow[dl, Rightarrow, "\simeq"', shorten=2mm]\arrow[d]\\
				\I{C}(Ld)\arrow[r, "r(Ld)"]  & \I{D}(Ld)
			\end{tikzcd}
		\end{equation*}
		commutes.\qed
	\end{enumerate}
\end{corollary}

Using the criterion from Proposition~\ref{prop:existenceAdjointsBeckChevalley}, we are now able to provide a large class of examples for adjunctions between $\BB$-categories:
\begin{example}
	\label{ex:RightadjOfPresCatsGivesInternalAdj}
	In Construction~\ref{constr:PresCatsAsInternalCats}, we defined a functor $-\otimes\Univ\colon \RPr\to\Cat(\BB)$ that carries a presentable $\infty$category $\CC$ to the sheaf of $\infty$-categories $\CC\otimes \Over{\BB}{-}$ (where $-\otimes -$ is Lurie's tensor product of presentable $\infty$-categories). Therefore, if $g \colon \CC \rightarrow \DD$ is a right adjoint functor between presentable $\infty$-categories, we get an induced functor
	\[
		g \otimes {\Univ} \colon \CC \otimes {\Univ} \rightarrow \DD \otimes {\Univ}
	\]
	of large $\BB$-categories.
	We note that for any morphism $s \colon B \rightarrow A$ in $\BB$ the mate of the commutative square
	\[
		\begin{tikzcd}
\CC \otimes \BB_{/A} \arrow[d, "\CC \otimes s^*"] \arrow[r, "g \otimes \BB_{/A}"] & \DD \otimes \BB_{/A} \arrow[d, "\DD \otimes s^*"] \\
\CC \otimes \BB_{/B} \arrow[r, "g \otimes \BB_{/B}"]                                     & \DD \otimes \BB_{/B}                                    
\end{tikzcd}
	\]
	may be identified with the square induced by passing to left adjoints in the commutative diagram
	\[
	\begin{tikzcd}
\CC \otimes \BB_{/A} \arrow[r, "g \otimes \BB_{/A}"]                                      & \DD \otimes \BB_{/A}                                      \\
\CC \otimes \BB_{/B} \arrow[r, "g \otimes \BB_{/B}"] \arrow[u, "\CC \otimes s_*"] & \DD \otimes \BB_{/B} \arrow[u, "\DD \otimes s_*"]
\end{tikzcd}
	\]
	Thus it follows from Proposition~\ref{prop:existenceAdjointsBeckChevalley} that $g \otimes {\Univ}$ is a right adjoint.
\end{example}

We conclude this section by applying the above example in two concrete cases.
At first we note that the large $\BB$-category $\Simp\Univ=\IPSh(\Delta)$ (where $\Delta$ is viewed as a constant $\BB$-category) may naturally be identified with the large $\BB$-category $\Simp\SS \otimes {\Univ}$. Therefore, by applying the functor $-\otimes\Univ$ from Construction~\ref{constr:PresCatsAsInternalCats} to the inclusion $\CatS \hookrightarrow \PSh_{\SS}(\Delta)$, one obtains a canonical inclusion of large $\BB$-categories
\[
	\iota \colon \ICat_\BB \into \IPSh(\Delta).
\] 
Now Example~\ref{ex:RightadjOfPresCatsGivesInternalAdj} shows:

\begin{proposition}
	\label{prop:presentationCategoryOfCategories}
	The inclusion $\iota\colon \ICat_{\BB}\into\IPSh(\Delta)$ admits a left adjoint $L\colon \IPSh(\Delta)\to\ICat_{\BB}$.\qed
\end{proposition}

Similarly, the inclusion $\SS \hookrightarrow \CatS$ induces an inclusion $\Univ\into\ICat_{\BB}$, so that Example~\ref{ex:RightadjOfPresCatsGivesInternalAdj} together with Proposition~\ref{prop:GroupoidificationCore} yields:
\begin{proposition}
	\label{prop:internalCoreGroupoidification}
	The inclusion $\Univ\into\ICat_{\BB}$ admits both a right adjoint $(-)^{\core}$ and a left adjoint $(-)^{\gp}$ that recover the core $\BB$-groupoid and the groupoidification functor on local sections.\qed
\end{proposition}

\subsection{Adjunctions in terms of mapping $\BB$-groupoids}
\label{sec:adjunctionsMappingGroupoids}
The notion of an adjunction between $\infty$-categories can be formalised in several ways. One way is the bicategorical approach that we have chosen in Definition~\ref{def:internalAdjunction}, but an equivalent way to define an adjunction is by means of a triple $(l,r, \alpha)$ in which $(l,r)\colon \CC\leftrightarrows\DD$ is a pair of functors and 
\begin{equation*}
\alpha\colon \map{\DD}(-,r(-))\simeq\map{\CC}(l(-),-)
\end{equation*}
is an equivalence (see for Example~\cite[Theorem~6.1.23]{cisinski2019a}). The aim of this section is to obtain an analogous characterisation for adjunctions between $\BB$-categories. To that end, recall from \S~\ref{sec:leftFibrations} that there is a factorisation system in $\Cat(\BB)$ between initial functors and left fibrations. Recall, furthermore, that there is a functor $\Cat(\BB)^\op\to\Cat(\BBB)$ that carries a $\BB$-category $\I{C}$ to the large $\BB$-category $\ILFib_{\I{C}}$ of left fibrations over $\I{C}$ and that carries a functor $f\colon\I{C}\to\I{D}$ to the pullback functor $f^\ast\colon\ILFib_{\I{C}}\to\ILFib_{\I{D}}$ that carries a left fibration $q\colon \I{Q}\to A\times\I{D}$ in context $A\in\BB$ to its pullback along $\id\times f\colon A\times\I{C}\to A\times\I{D}$. Now the key result from which we will derive our desired characterisation of adjunctions is the following statement:

\begin{proposition}
	\label{prop:YonedaExtensionFibrations}
	Let $f\colon \I{C}\to\I{D}$ be a functor between $\BB$-categories. Then the pullback functor
	\begin{equation*}
		f^\ast\colon \ILFib_{\I{D}}\to\ILFib_{\I{C}}
	\end{equation*}
	admits a left adjoint $f_!$ that is fully faithful whenever $f$ is.
	If $p\colon \I{P}\to A\times\I{C}$ is an object in $ \ILFib_{\I{C}}$, the left fibration $f_!(p)$ over $A\times\I{D}$ is the unique functor that fits into a commutative diagram
	\begin{equation*}
		\begin{tikzcd}
		\I{P}\arrow[d, "p"]\arrow[r, "i"] & f_!\I{P}\arrow[d, "f_!(p)"]\\
		A\times\I{C}\arrow[r, "\id\times f"] & A\times\I{D}
		\end{tikzcd}
	\end{equation*}
	such that $i$ is initial.
\end{proposition}
In order to prove Proposition~\ref{prop:YonedaExtensionFibrations}, we need the following lemma:
\begin{lemma}
	\label{lem:cancellabilityInitialFF}
	If $f\colon \I{C}\to\I{D}$ and $g\colon \I{D}\to\I{E}$ are functors in $\Cat(\BB)$ such that $g$ is fully faithful, then $gf$ is initial if and only if both $f$ and $g$ are initial.
\end{lemma}
\begin{proof}
	As initial functors are closed under composition, $gf$ is initial whenever both $f$ and $g$ are, so it suffices to show the converse direction. Since initial functors are the left complement in a factorisation system, they satisfy the left cancellability property, so that it suffices to show that $f$ is initial given that $gf$ is. We will make use of the $\BB$-categorical version of Quillen's theorem A~\cite[Corollary~4.4.8]{martini2021}. Let therefore $d\colon A\to\I{D}$ be an object in context $A\in\BB$. On account of the commutative diagram
	\begin{equation*}
		\begin{tikzcd}
			\Over{\I{C}}{d}\arrow[d]\arrow[r] & \Over{\I{D}}{d}\arrow[d]\arrow[r] & \Over{\I{E}}{g(d)}\arrow[d]\\
			\I{C}\times A\arrow[r] & \I{D}\times A\arrow[r] & \I{E}\times A
		\end{tikzcd}
	\end{equation*}
	in which the left square is a pullback, it suffices to show that the right square is a pullback as well, which follows immediately from $g$ being fully faithful.
\end{proof}

\begin{proof}[{Proof of Proposition~\ref{prop:YonedaExtensionFibrations}}]
	We wish to apply Proposition~\ref{prop:existenceAdjointsBeckChevalley}. Fixing an object $A\in\BB$, first note that the functor
	\begin{equation*}
	f^\ast\colon \LFib(A\times\I{D})\to \LFib(A\times\I{C})
	\end{equation*}
	that is given by pullback along $(\id\times f)\colon A\times \I{C}\to A\times\I{D}$
	has a left adjoint $f_!$. In fact, on account of the commutative square
	\begin{equation*}
	\begin{tikzcd}
	\LFib(A\times\I{D})\arrow[r, "f^\ast"]\arrow[d, hookrightarrow, "i"] & \LFib(A\times\I{C})\arrow[d, hookrightarrow, "i"] \\
	\Over{\Cat(\BB)}{A\times\I{D}}\arrow[r, "f^\ast"] & \Over{\Cat(\BB)}{A\times\I{C}},
	\end{tikzcd}
	\end{equation*}
	one may define the desired left adjoint $f_!$ on the level of left fibrations as the composition $\Over{L}{A\times\I{D}}\circ (\id\times f)_!\circ i$, where $\Over{L}{A\times\I{C}}\colon \Over{\Cat(\BB)}{A\times\I{D}}\to\LFib(A\times\I{D})$ denotes the localisation functor and where $(\id\times f)_!$ denotes the forgetful functor. By construction, this functor sends $p\colon \I{P}\to A\times\I{C}$ to the left fibration $f_!(p)\colon \I{Q}\to A\times\I{D}$ that arises from the factorisation of $(\id\times f)p\colon \I{P}\to A\times\I{D}$ into an initial map and a left fibration. Note that the counit of this adjunction is given by the canonical map $\I{P}\to \I{Q}\times_{\I{C}}\I{D}$. If $f$ is fully faithful, Lemma~\ref{lem:cancellabilityInitialFF} implies that this map is initial and therefore an equivalence since it is already a left fibration. As a consequence $f$ being fully faithful implies that $f_!$ is fully faithful as well.
	Therefore, by using Proposition~\ref{prop:existenceAdjointsBeckChevalley} the proof is complete once we show that for any map $s\colon B\to A$ in $\BB$, the lax square
	\begin{equation*}
		\begin{tikzcd}
				\LFib(A\times\I{D})\arrow[from=r, "f_!"'] \arrow[d, "s^\ast"']\arrow[from=dr, Rightarrow, shorten=2mm, "\phi"'] & \LFib(A\times\I{C}) \arrow[d, "s^\ast"] \arrow[d]\\
				\LFib(B\times\I{D})\arrow[from=r, "f_!"']  & \LFib(B\times\I{C})
		\end{tikzcd}
	\end{equation*}
	commutes. To see this, let $p\colon \I{P}\to A\times\I{C}$ be a left fibration, and consider the commutative diagram
	\begin{equation*}
	\begin{tikzcd}[column sep=small, row sep=small]
	& s^\ast f_!\I{P}\arrow[rr]\arrow[dd, "s^\ast f_!(p)", near start] && f_!\I{P}\arrow[dd, "f_!(p)", near start] \\
	s^\ast\I{P}\arrow[rr, crossing over]\arrow[dd, "s^\ast p", near start]\arrow[ur, "s^\ast i"] && \I{P}\arrow[ur, "i"] &\\
	& B\times\I{D}\arrow[rr, "s\times\id"', near start] && A\times\I{D}\\
	B\times\I{C}\arrow[rr, "s\times\id"']\arrow[ur, "\id\times f"'] && A\times\I{C}\arrow[ur, "\id\times f"'] \arrow[from=uu, "p", near start, crossing over]&
	\end{tikzcd}
	\end{equation*}
	in which $f_!(p) i\colon\I{P}\to f_!\I{P}\to A\times\I{D}$ is the factorisation of $(\id\times f)p$ into an initial map and a left fibration. The map $\phi\colon f_!s^\ast(p)\to s^\ast f_!(p)$ is given by the unique lift in the commutative square
	\begin{equation*}
	\begin{tikzcd}[column sep=large]
	s^\ast \I{P}\arrow[d, "j"]\arrow[r, "s^\ast i"] & s^\ast f_!\I{P}\arrow[d, "s^\ast f_!(p)"]\\
	f_!s^\ast \I{P}\arrow[ur, dotted, "\phi"]\arrow[r, "f_!s^\ast p"] & B\times\I{D}
	\end{tikzcd}
	\end{equation*}
	in which $j$ is initial. To complete the proof, it therefore suffices to show that $s^\ast i$ is initial, which follows from the fact that the map $s\colon B\to A$ is a right fibration and therefore proper, cf.~\cite[\S~4.4]{martini2021}.
\end{proof}

\begin{corollary}
	\label{cor:existenceYonedaExtension}
	For any functor $f\colon \I{C}\to\I{D}$ between $\BB$-categories, the functor
	\begin{equation*}
		f^\ast\colon \iFun(\I{D},\Univ)\to\iFun(\I{C},\Univ)
	\end{equation*}
	admits a left adjoint $f_!$ that fits into a commutative diagram
	\begin{equation*}
	\begin{tikzcd}
	\I{C}^{\op}\arrow[r, "f^\op"]\arrow[d, hookrightarrow, "h_{\I{C}^{\op}}"] & \I{D}^{\op}\arrow[d, hookrightarrow, "h_{\I{D}^{\op}}"]\\
	\iFun(\I{C},\Univ)\arrow[r, "f_!"] & \iFun(\I{D},\Univ)
	\end{tikzcd}
	\end{equation*}
	in which the two vertical arrows are given by the Yoneda embedding. Moreover, $f$ is fully faithful if and only if $f_!$ is fully faithful.
\end{corollary}
\begin{proof}
	The existence of the left adjoint $f_!$ follows immediately from Proposition~\ref{prop:YonedaExtensionFibrations} on account of the straightening/unstraightening equivalence for left fibrations (Theorem~\ref{thm:straightening}). To show that the composition $\I{C}^{\op}\into \iFun(\I{C},\Univ)\to\iFun(\I{D},\Univ)$ factors through the Yoneda embedding $\I{D}^{\op}\into\iFun(\I{D},\Univ)$, it suffices to show that for every representable left fibration $p\colon\I{P}\to A\times\I{C}$ the associated left fibration $f_!(p)\colon\I{Q}\to A\times\I{D}$ is representable as well. This follows immediately from the fact that there is an initial map $i\colon \I{P}\to\I{Q}$, which implies that $\I{Q}$ admits an initial section $A\to\I{Q}$ whenever $\I{P}$ admits such a section (cf.~Remark~\ref{rem:representableFunctorsFibrationalCriterion}).
\end{proof}

\begin{proposition}
	\label{prop:characterizationAdjunctionMappingGroupoids}
	A pair of functors $(l, r)\colon\I{C}\leftrightarrows\I{D}$ between $\BB$-categories defines an adjunction if and only if there is an equivalence of functors
	\begin{equation*}
		\alpha\colon\map{\I{D}}(l(-),-)\simeq\map{\I{C}}(-,r(-)).
	\end{equation*}
\end{proposition}
\begin{proof}
	Suppose that $l\dashv r$ is an adjunction in $\Cat(\BB)$. Then Proposition~\ref{prop:2functorFunContravariant} gives rise to an adjunction $l^\ast\dashv r^\ast\colon \IPSh(\I{D})\leftrightarrows\IPSh(\I{C})$. On the other hand, Corollary~\ref{cor:existenceYonedaExtension} provides a left adjoint $r_!$ to $r^\ast$, hence the uniqueness of adjoints implies that there is an equivalence $\beta\colon r_!\simeq l^\ast$. We therefore conclude that there is an equivalence $\alpha\colon h_{\I{C}}r\simeq l^\ast h_{\I{D}}$, where $h_{\I{C}}$ and $h_{\I{D}}$ denotes the Yoneda embedding of $\I{C}$ and $\I{D}$, respectively. On account of the adjunction $-\times\I{D}^{\op}\dashv \iFun(\I{D}^\op,-)$, the datum of such an equivalence corresponds precisely to an equivalence
	\begin{equation*}
			\alpha\colon\map{\I{D}}(l(-),-)\simeq\map{\I{C}}(-,r(-)),
	\end{equation*}
	as desired.
	
	Conversely, suppose that the pair $(l,r)$ comes along with an equivalence $\alpha$ as above.
	As functoriality of the twisted arrow construction (Definition~\ref{def:twistedArrow}) gives rise to a morphism of functors $\map{\I{C}}(-,-)\to\map{\I{D}}(l(-),l(-))$, one obtains a map
		\begin{equation*}
			\map{\I{C}}(-,-)\to\map{\I{D}}(l(-),l(-))\simeq\map{\I{C}}(-,rl(-)).
		\end{equation*}
	As the Yoneda embedding is fully faithful (Corollary~\ref{cor:YonedaEmbedding}), this map arises uniquely from a map $\eta\colon \id_{\I{C}}\to rl$. In fact, we may view the above map as a functor
		\begin{equation*}
			\I{C}\to \IPSh(\I{C})^{\Delta^1}
		\end{equation*}
		that sends an object $d\colon A\to \I{C}$ to the map
		\begin{equation*}
			\map{\I{C}}(-,d)\to\map{\I{D}}(l(-),l(d))\simeq\map{\I{C}}(-,rl(d))
		\end{equation*}
		in $\IPSh(\I{C})$. As the Yoneda embedding $\I{C}\into\IPSh(\I{C})$ is fully faithful, this map must arise from a map in $\I{C}$, hence the above functor factors through the fully faithful functor $\I{C}^{\Delta^1}\into \IPSh(\I{C})^{\Delta^1}$ that is induced by the Yoneda embedding. By a similar argument, one obtains a map $\epsilon\colon lr\to \id_{\I{D}}$. We complete the proof by showing that $\eta$ and $\epsilon$ satisfy the conditions of Proposition~\ref{prop:minimalAdjunctionData}, i.e.\ that the maps $r\epsilon\circ \eta r$ and $\epsilon l\circ l\eta$ are equivalences. We show this for the first case, the second case follows from an analogous argument. Since equivalences of functors can be detected objectwise by~\cite[Corollary~4.7.17]{martini2021}, it suffices to show that for any object $d\colon A\to \I{D}$ the map
		\begin{equation*}
				r(d)\xrightarrow{\eta r d} rlr(d)\xrightarrow{r\epsilon d} r(d)
		\end{equation*}
		is an equivalence. Now bifunctoriality of the equivalence $\map{\I{D}}(l(-),-)\simeq\map{\I{C}}(-,r(-))$ implies that there is a commutative diagram
		\begin{equation*}
			\begin{tikzcd}
				r(d)\arrow[r, "\eta rd"]\arrow[d, "\id_{r(d)}"] & rlr(d)\arrow[d, "r\epsilon d"] \\
				r(d)\arrow[r, "\id_{r(d)}"] & r(d)
			\end{tikzcd}
		\end{equation*}
		that arises from the transposed commutative diagram
		\begin{equation*}
			\begin{tikzcd}
				lr(d)\arrow[r, "\id_{lr(d)}"]\arrow[d, "\id_{lr(d)}"] & lr(d)\arrow[d, "\epsilon d"] \\
				lr(d)\arrow[r, "\epsilon d"] & d,
			\end{tikzcd}
		\end{equation*}
		which proves the claim.
\end{proof}
Recall that if $r\colon \I{D}\to\I{C}$ is a functor between $\BB$-categories and if $c\colon A\to \I{D}$ is an arbitrary object, the functor $\map{\I{C}}(c,r(-))\colon A\times\I{D}\to\Univ$ precisely classifies the left fibration $\Under{\I{D}}{c}\to A\times\I{D}$ that arises as the pullback of the slice projection $(\pi_c)_!\colon\Under{\I{C}}{c}\to A\times\I{C}$ along $\id\times r\colon A\times\I{D}\to A\times\I{C}$ (see~\cite[Definition~4.2.1]{martini2021}). We now obtain:
\begin{corollary}
	\label{cor:representabilityCriterionAdjoint}
	Let $r\colon\I{D}\to\I{C}$ be a functor between large $\BB$-categories. Then $r$ admits a left adjoint $l$ if and only if for any object $c\colon A\to \I{C}$ in context $A\in\BB$ the copresheaf $\map{\I{C}}(c, r(-))$ (viewed as an object in $\iFun(\I{D},\Univ)$ in context $A$) is representable by an object in $\I{D}$, in which case the representing object is given by $l(c)$ and the associated initial object in $\Under{\I{D}}{c}$ is given by the unit map $\eta c\colon c\to rl(c)$.
\end{corollary}
\begin{proof}
	By Proposition~\ref{prop:characterizationAdjunctionMappingGroupoids}, the functor $r$ admits a left adjoint if and only if there is a functor $l\colon\I{C}\to\I{D}$ and an equivalence
	\begin{equation*}
		\alpha\colon\map{\I{D}}(l(-),-)\simeq \map{\I{C}}(-,r(-)).
	\end{equation*}
	Therefore, if $r$ admits a left adjoint then $\map{\I{C}}(c,r(-))$ is representable by $l(c)\colon A\to\I{D}$, and the explicit construction of the equivalence $\alpha$ in Proposition~\ref{prop:characterizationAdjunctionMappingGroupoids} shows that the equivalence
	\begin{equation*}
		\Under{\I{D}}{l(c)}\simeq \Under{\I{D}}{c}
	\end{equation*}
	over $A\times\I{D}$ that arises from $\alpha$ sends the initial section $\id_{l(c)}\colon A\to \Under{\I{D}}{l(c)}$ to the unit map $\eta c\colon c\to rl(c)$.
	
	Conversely, if $\map{\I{C}}(c,r(-))$ is representable for every object $c$ in $\I{C}$ in context $A\in\BB$, then the functor $hr\colon \I{D}\to\I{C}\into\IPSh[\BBB](\I{C})=\iFun[\BB](\I{C}^\op,\Univ[\BBB])$ transposes to a functor
	\begin{equation*}
		\I{C}^{\op}\to\iFun(\I{D},\Univ[\BBB])
	\end{equation*}
	that factors through the Yoneda embedding $\I{D}^{\op}\into\iFun(\I{D},\Univ[\BBB])$ by~\cite[Proposition~3.9.4]{martini2021} and therefore defines a functor $l\colon\I{C}\to\I{D}$. By construction, this functor comes with an equivalence $\map{\I{D}}(l(-),-)\simeq \map{\I{C}}(-,r(-))$, hence the claim follows.
\end{proof}

Let $\I{C}$ and $\I{D}$ be $\BB$-categories and let $\iFun^R(\I{D},\I{C})\into\iFun(\I{D},\I{C})$ be the full subcategory that is spanned by those functors $\pi_A^\ast\I{D}\to\pi_A^\ast\I{C}$ in $\Cat(\Over{\BB}{A})$ (for every $A\in\BB$) that admit a left adjoint. Dually, let $\iFun^L(\I{C},\I{D})\into\iFun(\I{C},\I{D})$ denote the full subcategory spanned by those functors that admit a right adjoint.

\begin{remark}[locality of adjunctions]
	\label{rem:adjunctionsAreLocal}
	If $\I{C}$ and $\I{D}$ are $\BB$-categories and $A\in \BB$ is an arbitrary object, the property of a functor $f\colon\pi_A^\ast\I{C}\to\pi_A^\ast\I{D}$ to be a right adjoint is local in $\BB$ (see \S~\ref{sec:localityPrinciple}). In fact, by Corollary~\ref{cor:representabilityCriterionAdjoint} this property is equivalent to the condition that for every object $c$ in $\pi_A^\ast\I{C}$ (in arbitrary context), the functor $\map{\pi_A^\ast\I{C}}(c, f(-))$ is representable. Hence the claim follows from the fact that the representability of such functors is a local condition (see Example~\ref{ex:representabilityLocalCondition}).
	In particular, this implies that every object in $\iFun^R(\I{C},\I{D})$ in context $A\in\BB$ encodes a right adjoint functor $\pi_A^\ast\I{C}\to\pi_A^\ast\I{D}$, and one furthermore has a canonical equivalence $\pi_A^\ast\iFun^R(\I{D},\I{C})\simeq\iFun[\Over{\BB}{A}]^R(\pi_A^\ast\I{D},\pi_A^\ast\I{C})$ for every $A\in\BB$ (see Remarks~\ref{rem:localityPrinciplePropositions} and~\ref{rem:localityPrincipleBaseChangeProposition}).
\end{remark}

\begin{remark}[\'etale transposition invariance]
	\label{rem:etaleInvarianceAdjunctions}
	By its very definition, the property of an object $f\colon A\to \iFun(\I{D},\I{C})$ to be a right adjoint (i.e.\ to be contained in $\iFun^R(\I{D},\I{C})$) is invariant under \'etale transposition (see \S~\ref{sec:localityPrinciple}).
\end{remark}

\begin{corollary}
	\label{cor:functorialityAdjunction}
	For any two $\BB$-categories $\I{C}$ and $\I{D}$, there is an equivalence
	\begin{equation*}
		\iFun^R(\I{D},\I{C})\simeq\iFun^L(\I{C},\I{D})^{\op}
	\end{equation*}
	that sends a functor between $\I{D}$ and $\I{C}$ to its left adjoint, and vice versa.
\end{corollary}
\begin{proof}
	By postcomposition with the Yoneda embedding $\I{C}\into\IPSh(\I{C})$, the $\BB$-category $\iFun^R(\I{D},\I{C})$ embeds into $\iFun(\I{D}\times\I{C}^{\op},\Univ)$. Likewise, the $\BB$-category $\iFun^L(\I{C},\I{D})^{\op}\simeq\iFun^R(\I{C}^{\op},\I{D}^{\op})$ embeds into the $\BB$-category $\iFun(\I{D}\times\I{C}^{\op},\Univ)$. To finish the proof, we only need to show that an object $f\colon A\times\I{D}\times\I{C}^{\op}\to\Univ$ in $\iFun(\I{D}\times\I{C}^{\op},\Univ)$ in context $A\in\BB$ is contained in the essential image of $\iFun^R(\I{D},\I{C})$ if and only if it is contained in the essential image of $\iFun^L(\I{C},\I{D})^\op$.
	By Remarks~\ref{rem:adjunctionsAreLocal} and~\ref{rem:etaleInvarianceAdjunctions} (and the fact that the base change functor $\pi_A^\ast$ preserves the internal hom, cf.~Remark~\ref{rem:localityPrinciplePreservationStructure}), we may replace $\BB$ with $\Over{\BB}{A}$ and can thus assume that $A\simeq 1$ (see Remark~\ref{rem:etaleInvarianceReductionGlobalContext}). By Corollary~\ref{cor:representabilityCriterionAdjoint}, the functor $f$ is contained in $\iFun^R(\I{D},\I{C})$ if and only if	 $f(d,-)$ is representable for any object $d$ in $\I{D}$ and $f(-,c)$ is representable for any object $c$ in $\I{C}$, which is in turn equivalent to $f$ being contained in the essential image of $\iFun^L(\I{C},\I{D})^{\op}$. Thus the claim follows.
\end{proof}

\subsection{Reflective subcategories}
\label{sec:reflectiveSubcategories}
In this brief section we discuss the special case of an adjunction where the right adjoint is fully faithful.
Again this material is quite standard for ordinary $\infty$-categories, see for example \cite[\S 5.2.7]{htt}.
\begin{definition}
	\label{def:reflectiveCoreflectiveSubcategory}
	Let $i\colon\I{C}\into \I{D}$ be a fully faithful functor between $\BB$-categories. Then $\I{C}$ is said to be \emph{reflective} in $\I{D}$ if $i$ admits a left adjoint. Dually, $\I{C}$ is \emph{coreflective} if $i$ admits a right adjoint.
\end{definition}
\begin{proposition}
	\label{prop:adjunctionFullyFaithful}
	If $(l\dashv r)\colon\I{C}\leftrightarrows \I{D}$ is an adjunction between $\BB$-categories, then $l$ is fully faithful if and only if the adjunction unit $\eta$ is an equivalence, and $r$ is fully faithful if and only if the adjunction counit $\epsilon$ is an equivalence.
\end{proposition}
\begin{proof}
	The functor $l$ is fully faithful if and only if the map
	\begin{equation*}
		\map{\I{C}}(-,-)\to\map{\I{D}}(l(-),l(-))
	\end{equation*}
	is an equivalence~\cite[Proposition~3.8.7]{martini2021}. By postcomposition with the equivalence
	\begin{equation*}
		\map{\I{D}}(l(-),l(-))\simeq\map{\I{C}}(-,rl(-))
	\end{equation*}
	that is provided by Proposition~\ref{prop:characterizationAdjunctionMappingGroupoids}, this is in turn equivalent to the map
	\begin{equation*}
		\map{\I{C}}(-,-)\to\map{\I{C}}(-,rl(-))
	\end{equation*}
	being an equivalence. But this map is obtained as the image of the adjunction unit $\eta\colon\Delta^1\to \iFun(\I{C},\I{C})$ along the fully faithful functor $\iFun(\I{C},\I{C})\into\iFun(\I{C}^{\op}\times\I{C},\Univ)$ that is induced by postcomposition with the Yoneda embedding $\I{C}\into\IPSh(\I{C})$.  The claim thus follows from the observation that fully faithful functors are conservative (since the map $\Delta^1\to\Delta^0$ is essentially surjective, see~\cite[Lemma~3.8.8]{martini2021}). The dual statement about $r$ and $\epsilon$ is proved by an analogous argument.
\end{proof}
By combining Proposition~\ref{prop:adjunctionFullyFaithful} with Proposition~\ref{prop:minimalAdjunctionData}, one immediately deduces:
\begin{corollary}
	\label{cor:criterionSubcategoryReflective}
	Let $i\colon\I{D}\into\I{C}$ be a fully faithful functor between $\BB$-categories. Then $\I{D}$ is reflective in $\I{C}$ if and only if $i$ admits a retraction $L\colon \I{C}\to\I{D}$ together with a map $\eta\colon\id_{\I{C}}\to iL$ such that both $\eta i$ and $L\eta$ are equivalences.\qed
\end{corollary}
If $\I{D}\into\I{C}$ is a reflective subcategory, then the reflection functor $L\colon \I{C}\to\I{D}$ is a retraction and therefore in particular essentially surjective (cf.~Proposition~\ref{prop:characterisationEssentiallySurjective}). Consequently, we may recover the subcategory $\I{D}$ from the endofunctor $iL\colon \I{C}\to\I{C}$ be means of its factorisation into an essentially surjective and a fully faithful functor. Conversely, given an arbitrary endofunctor $f\colon \I{C}\to\I{C}$, Corollary~\ref{cor:criterionSubcategoryReflective} shows that the essential image of $f$ defines a reflective subcategory precisely if there is a map $\eta\colon \id_\I{C}\to f$ such that both $\eta f$ and $f\eta$ are equivalences. Let us record this observation for future use in the following proposition.
\begin{proposition}
	\label{prop:reflectiveSubcategoryFromEndofunctor}
	Let $\I{C}$ be a $\BB$-category, let $f\colon \I{C}\to\I{C}$ be a functor and let $iL\colon \I{C}\onto\I{D}\into\I{C}$ be its factorisation into an essentially surjective and a fully faithful functor. Then $L\dashv i$ precisely if there is a map $\eta\colon\id_{\I{C}}\to f$ such that both $\eta f$ and $f\eta$ are equivalences.\qed
\end{proposition}
\begin{example}
	\label{ex:StableFactorisationSystemReflectiveSubuniverse}
	If $(\LL,\RR)$ is a factorisation system in $\BB$, then for any $A\in\BB$ the full subcategory $\Over{\RR}{A}\into \Over{\BB}{A}$ is reflective: the associated reflection functor $\Over{L}{A}\colon \Over{\BB}{A}\to\Over{\RR}{A}$ is induced by the unique factorisation of maps. Such a factorisation system $(\LL,\RR)$ is called a \emph{modality} if $\LL$ is closed under base change in $\BB$, which precisely means that for every map $s\colon B\to A$ in $\BB$ the natural map $\Over{L}{B}s^\ast\to s^\ast\Over{L}{A}$ is an equivalence.
	Using Proposition~\ref{prop:existenceAdjointsBeckChevalley}, we thus conclude that the right orthogonality class $\RR$ of any modality $(\LL,\RR)$ defines a reflective subcategory of $\Univ$. In Example~\ref{ex:StableFactorisationSystemReflectiveSubuniverseCocomplete} below, we will characterise those reflective subcategories of $\Univ$ that arise in such a way.
\end{example}

Reflective subcategories are examples of \emph{localisations} in the sense of Appendix~\ref{sec:localisation}:
\begin{proposition}
	\label{prop:reflectiveSubcategoryBousfield}
	Let $(l\dashv r)\colon \I{C}\leftrightarrows\I{D}$ be a reflective subcategory. Then $l$ is the localisation of $\I{C}$ at the subcategory $\I{S}= l^{-1}\I{D}^{\core}\into \I{C}$.
\end{proposition}
\begin{proof}
	By construction of $\I{S}$, we obtain a commutative diagram
	\begin{equation*}
	\begin{tikzcd}
	\I{S}\arrow[r] \arrow[d, hookrightarrow] & \I{S}^{\gp}\arrow[r]\arrow[d] & \I{D}^{\core}\arrow[d, hookrightarrow]\\
	\I{C}\arrow[r, "L"] \arrow[rr, bend right, "l"]& \I{S}^{-1}\I{C}\arrow[r, "g"] & \I{D}, 
	\end{tikzcd}
	\end{equation*}
	hence we only need to show that $g$ is an equivalence. Let us define $h=Lr$. Then $gh\simeq lr\simeq \id$, hence $h$ is a right inverse of $g$. We finish the proof by showing that $h$ is a left inverse of $g$ as well. Since $L^\ast\colon \iFun(\I{S}^{-1}\I{C},\I{S}^{-1}\I{C})\to \iFun(\I{C},\I{S}^{-1}\I{C})$ is fully faithful by Proposition~\ref{prop:universalPropertyLocalisation}, it suffices to produce an equivalence $hgL\simeq L$. Let $\eta\colon \id\to rl$ be the adjunction unit. Since $l\eta$ is an equivalence, the map $l\eta c$ factors through the core $\I{D}^\core\into\I{D}$ for every object $c\colon A\to \I{C}$ in context $A\in\BB$. By construction of $\I{S}$, this means that $\eta c$ is contained in $\I{S}$, hence $L\eta c$ is an equivalence. Since equivalences of functors can be detected objectwise~\cite[Corollary~4.7.17]{martini2021}, we conclude that $L\eta\colon L\to Lrl\simeq hgL$ is the desired equivalence.
\end{proof}

It will be useful to have a name for the class of localisations that arise from reflective subcategories:
\begin{definition}
	\label{def:bousfieldLocalisation}
	Let $\I{S}\to\I{C}$ be a functor between $\BB$-categories. The localisation $L\colon \I{C}\to S^{-1}\I{C}$ is said to be a \emph{Bousfield localisation} if $L$ admits a fully faithful right adjoint $i\colon \I{S}^{-1}\I{C}\into\I{C}$.
\end{definition}

\begin{remark}
    The extra condition on the right adjoint in Definition~\ref{def:bousfieldLocalisation} to be fully faithful is superfluous: in fact, by Proposition~\ref{prop:universalPropertyLocalisation} the functor $L^\ast\colon \IPSh(S^{-1}\I{C})\to\IPSh(\I{C})$ is fully faithful and by Proposition~\ref{prop:2functorFunContravariant} $L^\ast$ is left adjoint to $i^\ast$. We therefore obtain an equivalence $L^\ast\simeq i_!$, hence Corollary~\ref{cor:existenceYonedaExtension} implies that $i$ must be fully faithful as well.
\end{remark}

\section{Limits and colimits}
	\label{chap:limitsColimits}
In this chapter we discuss limits and colimits in a $\BB$-category. We set up the general theory in \S~\ref{sec:definitionsLimitColimit}--\ref{sec:limitsColimitsFunctorCategories}. 
All in all our treatment is quite parallel to the one in ordinary higher category theory, see for example \cite[\S 19]{joyal2008notes} or~\cite[\S 6.2]{cisinski2019a}.
In \S~\ref{sec:limitsColimitsUniverse} and \S~\ref{sec:limitsColimitsCat} we discuss limits and colimits in the universe $\Univ$ and in the $\BB$-category of $\BB$-categories $\ICat_{\BB}$. In \S~\ref{sec:initialFunctorsLimits} we show that initial and final functors can be characterised by their property of preserving limits and colimits. Finally, in \S~\ref{sec:decompositionColimits} we explain how general internal limits and colimits can be decomposed into groupoidal and constant limits and colimits.

\subsection{Definitions and first examples}
\label{sec:definitionsLimitColimit}
Let $\I{C}$ be a $\BB$-category. Recall from Proposition~\ref{prop:CatBPresentable} that for any simplicial object $I$ in $\BB$ the internal hom $\iFun(I,\I{C})$ in $\Simp\BB$ is a $\BB$-category. We refer to the objects of this $\BB$-category as \emph{$I$-indexed diagrams in $\I{C}$}. Note that this $\BB$-category is equivalent to $\iFun(\I{I},\I{C})$, where $\I{I}$ is the image of the simplicial object $I$ along the localisation functor $\Simp\BB\to\Cat(\BB)$. Thus, in what follows we can always safely assume that $I$ is a $\BB$-category.

Now recall from~\cite[Definition~4.2.1]{martini2021} that to any pair of maps $f\colon \I{D}\to\I{C}$ and $g\colon\I{E}\to\I{C}$ in $\Cat(\BB)$ we can associate the \emph{comma $\BB$-category} $\Comma{\I{D}}{\I{C}}{\I{E}}=(\I{D}\times\I{E})\times_{\I{C}\times\I{C}}\I{C}^{\Delta^1}$. We may now define:

\begin{definition}
	\label{def:cones}
	Let $\I{C}$ be a $\BB$-category and let $d\colon A\to \iFun(I,\I{C})$ be an $I$-indexed diagram  in $\I{C}$ in context $A\in \BB$, for some $I\in \Simp\BB$. The \emph{$\BB$-category of cones over $d$} is defined as the comma $\BB$-category $\Over{\I{C}}{d}=\Comma{\I{C}}{\iFun(I,\I{C})}{A}$ formed from $d\colon A\to\iFun(I,\I{C})$ and the diagonal map $\diag\colon\I{C}\to\iFun(I,\I{C})$. Dually, the \emph{$\BB$-category of cocones under $d$} is defined as the comma $\BB$-category $\Under{\I{C}}{d}=\Comma{A}{\iFun(I,\I{C})}{\I{C}}$.
\end{definition}

In the situation of Definition~\ref{def:cones}, the $\BB$-category of cones $\Over{\I{C}}{d}$ admits a structure map into $\I{C}\times A$ that fits into the pullback square
\begin{equation*}
	\begin{tikzcd}
	\Over{\I{C}}{d}\arrow[d]\arrow[r] & \Over{\iFun(I,\I{C})}{d}\arrow[d, "(\pi_d)_!"]\\
		\I{C}\times A\arrow[r, "\diag\times\id"] & \iFun(I,\I{C})\times A
	\end{tikzcd}
\end{equation*}
in which the vertical map on the right is the forgetful functor from the slice $\BB$-category, cf.~Definition~\ref{def:sliceBCategories}. Since this is a right fibration (Proposition~\ref{prop:sliceProjectionLeftFibration}), so is the map $\Over{\I{C}}{d}\to\I{C}\times A$. In other words, we may regard this map as an object in $\IRFib_{\I{C}}$ in context $A$. Dually, the map $\Under{\I{C}}{d}\to A\times\I{C}$ is a left fibration and therefore defines an object in $\ILFib_{\I{C}}$ in context $A$. With respect to the straightening/unstraightening equivalence $\IRFib_{\I{C}}\simeq\IPSh(\I{C})$ from Theorem~\ref{thm:straightening}, the right fibration $\Over{\I{C}}{d}\to\I{C}\times A$ corresponds to the presheaf $\map{\iFun(I,\I{C})}(\diag(-), d)$ on $\I{C}$, and the left fibration $\Under{\I{C}}{d}\to A\times\I{C}$ corresponds to the copresheaf $\map{\iFun(I,\I{C})}(d,\diag(-))$ on $\I{C}$.

\begin{remark}[locality of cones]
	\label{rem:localityCones}
	In the situation of Definition~\ref{def:cones}, if $B\in\BB$ is an arbitrary object, it follows immediately from Remark~\ref{rem:localityPrinciplePreservationStructure} that one obtains a canonical equivalence of $\Over{\BB}{B}$-categories $\pi_B^\ast(\Over{\I{C}}{d})\simeq\Over{(\pi_B^\ast \I{C})}{\pi_B^\ast(d)}$.
\end{remark}

\begin{remark}[\'etale transposition invariance for cones]
	\label{rem:etaleInvarianceCones}
	In the situation of Definition~\ref{def:cones}, let us denote by $\bar{d}\colon 1_{\Over{\BB}{A}}\to \pi_A^\ast\iFun(I,\I{C})\simeq\iFun[\Over{\BB}{A}](\pi_A^\ast I,\pi_A^\ast\I{C})$ the transpose of $d$. Since the forgetful functor $(\pi_A)_!\colon \Cat(\Over{\BB}{A})\to\Cat(\BB)$ preserves pullbacks, we deduce from Remark~\ref{rem:baseChangeSlice} that the map $\Over{\I{C}}{d}\to\I{C}\times A$ arises as the image of $\Over{(\pi_A^\ast\I{C})}{\bar{d}}\to\pi_A^\ast\I{C}$ along $(\pi_A)_!$. In other words, when regarded as a $\Over{\BB}{A}$-category, we can identify $\Over{\I{C}}{d}$ with $\Over{(\pi_A^\ast\I{C})}{\bar d}$.
\end{remark}

\begin{remark}
	\label{rem:cones}
	Let $I$ be a simplicial object in $\BB$ and let $\I{C}$ be a $\BB$-category. Recall from~\cite[Definition~4.3.11]{martini2021} the definition of the \emph{right cone} $I^{\triangleright}$ as the pushout
	\begin{equation*}
	\begin{tikzcd}
	I\sqcup I\arrow[d, "\id\times \pi_I"]\arrow[r, "{(d^1,d^0)}"] & \Delta^1\otimes I\arrow[d]\\
	I\sqcup 1\arrow[r, "{(\iota, \infty)}"] & I^{\triangleright}.
	\end{tikzcd}
	\end{equation*}
	By applying the functor $\iFun(-,\I{C})$ to this diagram, one obtains an equivalence
	\begin{equation*}
	\iFun(I^{\triangleright},\I{C})\simeq \Comma{\iFun(I,\I{C})}{\iFun(I,\I{C})}{\I{C}}
	\end{equation*}
	over $\iFun(I,\I{C})\times\I{C}$, in which the right-hand side denotes the comma $\BB$-category that is formed from the cospan
	\begin{equation*}
	\iFun(I,\I{C})\xrightarrow{\id}\iFun(I,\I{C})\xleftarrow{\diag}\I{C}.
	\end{equation*}
	By construction, if $d\colon A\to \iFun(I,\I{C})$ is an $I$-indexed diagram in $\I{C}$ in context $A\in\BB$, one obtains a pullback square
	\begin{equation*}
	\begin{tikzcd}
	\Under{\I{C}}{d}\arrow[r]\arrow[d] & \iFun(I^{\triangleright},\I{C})\arrow[d, "({\iota^\ast, \infty^\ast})"]\\
	A\times\I{C}\arrow[r, "d\times\id"] & \iFun(I,\I{C})\times\I{C}. 
	\end{tikzcd}
	\end{equation*}
	In other words, the pullback of $\iFun(I^{\triangleright},\I{C})$ along $d\times \id$ recovers the $\BB$-category of cocones under $d$. We may therefore regard any object $\bar d\colon A\to \iFun(I^{\triangleright},\I{C})$  as a cocone $d\to \diag c$ under the diagram $d=\iota^\ast \bar d$ with $c=\infty^\ast\bar d$.
	
	Dually, one defines the \emph{left cone} $I^{\triangleleft}$ as the pushout
	\begin{equation*}
	\begin{tikzcd}
	I\sqcup I\arrow[d, "\pi_I\times\id"]\arrow[r, "{(d^1,d^0)}"] & \Delta^1\otimes I\arrow[d]\\
	1\times I\arrow[r] & I^{\triangleleft}
	\end{tikzcd}
	\end{equation*}
	and therefore obtains an equivalence 
	\begin{equation*}
	\iFun(I^{\triangleleft},\I{C})\simeq \Comma{\I{C}}{\iFun(I,\I{C})}{\iFun(I,\I{C})}
	\end{equation*}
	over $\I{C}\times\iFun(I,\I{C})$. Consequently, the pullback of $\iFun(I^{\triangleleft},\I{C})$ along $\id\times d$ recovers the $\BB$-category of cones $\Over{\I{C}}{d}$ over $d$.
\end{remark}
\begin{definition}
	\label{def:limitColimit}
	Let $\I{C}$ be a $\BB$-category and let $d\colon A\to \iFun(I,\I{C})$ be an $I$-indexed diagram in context $A$ in $\I{C}$, for some $A\in \BB$ and some $I\in \Simp\BB$. A \emph{limit} cone of $d$ is a map $\diag(\lim d)\to d$ in $\iFun(I,\I{C})$ in context $A$ that defines a final section $A\to \Over{\I{C}}{d}$ over $A$. Dually, a \emph{colimit} cocone of $d$ is a map $d\to\diag(\colim d)$ in $\iFun(I,\I{C})$ in context $A$ that defines an initial section $A\to \Under{\I{C}}{d}$ over $A$.
\end{definition}

\begin{remark}
    The above definition is a direct analogue of Joyal's original definition of limits and colimits in an $\infty$-category \cite{Joyal2002}. 
\end{remark}

\begin{remark}
	\label{rem:conesRepresentable}
	In the situation of Definition~\ref{def:limitColimit}, Remark~\ref{rem:representableFunctorsFibrationalCriterion} implies that an $I$-indexed diagram $d\colon A\to \iFun(I,\I{C})$ admits a colimit cocone if and only if the presheaf $\map{\iFun(I,\I{C})}(d,\diag(-))$ is representable, in which case the representing object is given by $\colim d$. In other words, if $d$ admits a colimit cocone, one obtains an equivalence $\Under{\I{C}}{\colim d}\simeq \Under{\I{C}}{d}$ over $A\times\I{C}$, and conversely if there is an object $c\colon A\to \I{C}$ and an equivalence $\Under{\I{C}}{c}\simeq \Under{\I{C}}{d}$ over $A\times\I{C}$ then the image of the object $\id_{c}$ in $\Under{\I{C}}{c}$ along this equivalence defines a colimit cocone of $d$. A similar observation can be made for limits. In particular, the colimit and limit of a diagram are unique up to equivalence if they exist.
\end{remark}
\begin{remark}[locality of limits and colimits]
	\label{rem:limitsAreLocal}
	The existence of limits and colimits is a \emph{local} condition: in fact, by the same reasoning as in Remark~\ref{rem:adjunctionsAreLocal}, a diagram $d\colon A\to \iFun(I,\I{C})$ admits a limit in $\I{C}$ if and only if for every cover $(s_i)\colon  \bigsqcup_i A_i\onto A$ the diagram $s_i^\ast(d)\colon A_i\to \iFun(I,\I{C})$ admits a limit in $\I{C}$. Analogous observations can be made for colimits.
\end{remark}

\begin{remark}[\'etale transposition invariance for limits and colimits]
	\label{rem:etaleInvarianceLimits}
	In light of Remark~\ref{rem:etaleInvarianceCones}, a cone $\diag c \to d $ in $\iFun(I,\I{C}) $ in context $ A $ transposes to a cone $\diag \bar{c} \to  \bar{d}$ in $\iFun[\BB_{/A}](\pi_A^\ast I,\pi_A^\ast\I{C}) $ in context $1_{\BB_{/A}}$ (where $\bar{d}\colon 1_{\Over{\BB}{A}}\to\iFun(\Over{\BB}{A})(\pi_A^\ast I,\pi_A^\ast \I{C})$ and $\bar{c}\colon 1_{\Over{\BB}{A}}\to\pi_A^\ast\I{C}$ are the transpose of $d$ and $c$, respectively), and the former defines an intial section $A \to \I{C}_{/d}$ over $A$ if and only if the latter defines an initial object $1_{\BB_{/A}} \to \pi_A^* \I{C}_{/\bar{d}}$. In other words, we may compute the limit of $d\colon A\to\iFun(I,\I{C})$ as the transpose of the limit of $\bar d\colon 1_{\Over{\BB}{A}}\to\iFun[\Over{\BB}{A}](\pi_A^\ast I,\pi_A^\ast \I{C})$. Analogous observations can be made for colimits.
\end{remark}

\begin{example}
	\label{ex:objectAsDiagramLimit}
	Let $\I{C}$ be $\BB$-category and let $c\colon A\to\I{C}$ be an object, viewed as a $1$-indexed diagram $c\colon A\to \iFun(1,\I{C})\simeq\I{C}$. Then there are equivalences $\lim c\simeq c\simeq \colim c$, and the associated limit and colimit cocones are given by $\id_c\colon A\to \Over{\I{C}}{c}$ and $\id_c\colon A\to \Under{\I{C}}{c}$.
\end{example}

\begin{example}
	\label{ex:initialObjectColimit}
	For any $\BB$-category $\I{C}$ and any object $c\colon A\to\I{C}$, the object $c$ is initial if and only if it defines a colimit of the initial diagram $d\colon \varnothing\to \I{C}$, and dually $c$ is final if and only if it defines a limit of $d$. In fact, since $\varnothing$ is initial in $\Cat(\BB)$, there is an equivalence $\iFun(\varnothing,\I{C})\simeq 1$, which implies that the left fibration $\Under{\I{C}}{d}\to A\times\I{C}$ is an equivalence. Consequently, a section $A\to \Under{\I{C}}{d}$ is initial if and only if the map $A\to A\times\I{C}$ is, which is in turn the case if and only if the associated map $1\to \pi_A^\ast\I{C}$ is initial in $\Cat(\Over{\BB}{A})$. As this is precisely the condition that $c$ is an initial object in $\I{C}$, the result follows. The statement about final objects and limits follows by dualisation.
\end{example}

\begin{proposition}
	\label{prop:limitColimitFunctor}
	Let $\I{C}$ be a $\BB$-category and let $I$ be a simplicial object in $\BB$. The following conditions are equivalent:
	\begin{enumerate}
		\item every diagram $d\colon A\to \iFun(I,\I{C})$ admits a colimit $\colim d$;
		\item the diagonal functor $\diag\colon \I{C}\to \iFun(I,\I{C})$ admits a left adjoint $\colim\colon \iFun(I,\I{C})\to \I{C}$.
	\end{enumerate}
	If either of these conditions are satisfied, the functor $\colim$ carries $d$ to $\colim d$, and the adjunction unit $d\to \diag \colim d$ defines a colimit cocone of $d$. 
	The dual statement for limits holds as well.
\end{proposition}
\begin{proof}
	By the dual of Corollary~\ref{cor:representabilityCriterionAdjoint}, the functor $\diag$ admits a left adjoint if and only if for every diagram $d\colon A\to \iFun(I,\I{C})$ the functor $\map{\iFun(I,\I{D})}(d, \diag(-))$ is representable by an object in $\I{C}$, in which case the left adjoint sends $d$ to the representing object in $\I{C}$. By definition, this functor classifies the left fibration $\Under{\I{C}}{d}\to A\times\I{C}$. Therefore, Remark~\ref{rem:conesRepresentable} shows that $\diag$ admits a left adjoint if and only if every diagram $d$ admits a colimit $\colim d\colon A\to \I{C}$, in which case this is the representing object of the functor $\map{\iFun(I,\I{C})}(d,\diag(-))$. Corollary~\ref{cor:representabilityCriterionAdjoint} moreover shows that in this case the adjunction unit $d\to \diag\colim d$ defines an initial section $A\to\Under{\I{C}}{d}$.
 \end{proof}

\begin{example}
	\label{ex:groupoidalLimitsColimits}
	Let $\I{C}$ be a large $\BB$-category and $\I{G}$ be a $\BB$-groupoid. By using Proposition~\ref{prop:existenceAdjointsBeckChevalley}, the following two conditions are equivalent:
	\begin{enumerate}
	\item $\I{C}$ admits $\I{G}$-indexed colimits;
	\item for every $A\in\BB$ the functor $\pi_{\I{G}}^\ast\colon \I{C}(A)\to\I{C}(\I{G}\times A)$ admits a left adjoint $(\pi_{\I{G}})_!$ such that for every map $s\colon B\to A$ in $\BB$ the natural morphism $(\pi_{\I{G}})_!s^\ast\to s^\ast (\pi_{\I{G}})_!$ is an equivalence.
	\end{enumerate}
	In particular, if $\I{C}$ has $\I{G}$-indexed colimits, then the colimit of a diagram $d\colon A\to\iFun(\I{G},\I{C})$ can be identified with the image of $d\in \I{C}(\I{G}\times A)$ along the functor $(\pi_{\I{G}})_!$.
	
	Dually, the following two conditions are equivalent:
		\begin{enumerate}
		\item $\I{C}$ admits $\I{G}$-indexed limits;
		\item for every $A\in\BB$ the functor $\pi_{\I{G}}^\ast\colon \I{C}(A)\to\I{C}(\I{G}\times A)$ admits a right adjoint $(\pi_{\I{G}})_\ast$ such that for every map $s\colon B\to A$ in $\BB$ the natural morphism $s^\ast(\pi_{\I{G}})_\ast\to (\pi_{\I{G}})_\ast s^\ast$ is an equivalence.
		\end{enumerate}
	In particular, if $\I{C}$ has $\I{G}$-indexed limits, then the limit of a diagram $d\colon A\to\iFun(\I{G},\I{C})$ can be identified with the image of $d\in \I{C}(\I{G}\times A)$ along the functor $(\pi_{\I{G}})_\ast$.
\end{example}

\begin{example}
	\label{ex:externalLimitsColimits}
	Let $\I{C}$ be a large $\BB$-category and let $\II$ be an $\infty$-category. By using Proposition~\ref{prop:existenceAdjointsBeckChevalley}, the following two conditions are equivalent:
	\begin{enumerate}
			\item $\I{C}$ admits $\II$-indexed colimits;
			\item for every $A\in\BB$ the $\infty$-category $\I{C}(A)$ admits $\II$-indexed colimits, and for every map $s\colon B\to A$ in $\BB$ the functor $s^\ast\colon \I{C}(A)\to\I{C}(B)$ preserves such colimits. 
	\end{enumerate}
	Dually, the following two conditions are equivalent:
	\begin{enumerate}
			\item $\I{C}$ admits $\II$-indexed limits;
			\item for every $A\in\BB$ the $\infty$-category $\I{C}(A)$ admits $\II$-indexed limits, and for every map $s\colon B\to A$ in $\BB$ the functor $s^\ast\colon \I{C}(A)\to\I{C}(B)$ preserves such limits. 
	\end{enumerate}
\end{example}

\begin{remark}
    \label{rem:existenceColimitsGenerators}
    Let $\CC$ be a small $\infty$-category such that $\BB$ is a left exact and accessible localisation of $\PSh(\CC)$. Let $L\colon \PSh(\CC)\to\BB$ be the localisation functor. Then Corollary~\ref{cor:restricedBeckChevalleyAdj} implies that in the situation of Example~\ref{ex:groupoidalLimitsColimits} and Example~\ref{ex:externalLimitsColimits}, it suffices to check the condition in~(2) for the special case where $A=L(c)$, $B=L(d)$ and $s=L(t)$ for some objects $c,d\in \CC$ and some map $t\colon d\to c$ in $\CC$.
\end{remark}

\subsection{Preservation of limits and colimits}

Let $f\colon \I{C}\to\I{D}$ be a functor between $\BB$-categories and let $I$ be a simplicial object in $\BB$. Let $f_\ast\colon \iFun(I,\I{C})\to\iFun(I,\I{D})$ be the functor that is given by postcomposition with $f$. For any diagram $d\colon A\to \iFun(I,\I{C})$, the functor $f_\ast$ gives rise to an evident commutative square
\begin{equation*}
	\begin{tikzcd}
		\Over{\I{C}}{d}\arrow[r, "f_\ast"]\arrow[d] & \Over{\I{D}}{f_\ast d}\arrow[d]\\
		\I{C}\times A\arrow[r, "f\times\id"] & \I{D}\times A.
	\end{tikzcd}
\end{equation*}
Suppose that $d$ has a limit in $\I{C}$, i.e.\ there is a limit cone given by a final section $ A\to \Over{\I{C}}{d}$ over $A$. We say that the functor $f$ \emph{preserves} this limit if the image of this limit cone along $f_\ast$ defines a final section of $\Over{\I{D}}{f_\ast d}$. Dually, if $d$ has a colimit in $\I{C}$ then $f$ is said to preserve this colimit if the image of the colimit cocone along $f_\ast$ is an initial section of $\Under{\I{D}}{f_\ast d}$ over $A$.

\begin{remark}[locality of preservation of limits and colimits]
	\label{rem:preservationLimitsIsLocal}
	The property that a functor $f\colon\I{C}\to\I{D}$ preserves the limit (colimit) of a diagram $d\colon A\to \iFun(I,\I{C})$ is a \emph{local} condition: in fact, the same reasoning as in Remark~\ref{rem:adjunctionsAreLocal} implies that $f$ preserves the limit of $d$ if and only if for every cover $(s_i)\colon \bigsqcup_i A_i\onto A$ in $\BB$ the limit of the induced diagram $s_i^\ast(d)$ is preserved by $f$. Analogous observations can be made for colimits.
\end{remark}

\begin{remark}[\'etale transposition invariance for the preservation of limits and colimits]
	\label{rem:etaleInvariancePreservationLimits}
	Note that by means of the projections to $A$, the functor $f_\ast\colon\Over{\I{C}}{d}\to\Over{\I{D}}{f_\ast d}$ can be regarded as a map in $\Cat(\Over{\BB}{A})$. When viewed as such,  Remark~\ref{rem:etaleInvarianceCones} implies that this map can be identified with the functor $(\pi_A^\ast f)_\ast\colon \Over{(\pi_A^\ast\I{C})}{\bar d}\to\Over{(\pi_A^\ast\I{D})}{(\pi_A^\ast f)_\ast \bar d}$ (where $\bar d\colon 1_{\Over{\BB}{A}}\to \iFun[\Over{\BB}{A}](\pi_A^\ast I,\pi_A^\ast\I{C})$ denotes the transpose of $d$). Together with Remark~\ref{rem:etaleInvarianceLimits}, this implies that $f$ preserves the limit of $d$ if and only if $\pi_A^\ast f$ preserves the limit of $\bar d$. Analogous observations hold for colimits.
\end{remark}

\begin{lemma}
	\label{lem:equivalenceTransposeFinalObject}
	Let $(l\dashv r)\colon\I{C}\leftrightarrows\I{D}$ be an adjunction between $\BB$-categories, and let $f\colon c\to r(d)$ be a map in $\I{C}$ in context $A\in \BB$. Then $f$ is an equivalence if and only if the transpose map $g\colon l(c)\to d$ defines a final section of $\Over{\I{C}}{d}$ over $A$.
\end{lemma}
\begin{proof}
	By Corollary~\ref{cor:representabilityCriterionAdjoint}, the counit $\epsilon d\colon lr(d)\to d$ defines a final section of $\Over{\I{C}}{d}$ over $A$, hence the dual of Corollary~\ref{cor:universalPropertyInitialObject} implies that there is a map $g\to \epsilon d$ in $\Over{\I{C}}{d}$  that is an equivalence if and only if $g$ is final. On account of the equivalence $\Over{\I{C}}{d}\simeq\Over{\I{C}}{r(d)}$, this map corresponds to a map $f\to \id_{r(d)}$ in $\Over{\I{C}}{r(d)}$. The result now follows from the straightforward observation that the latter is an equivalence if and only if $f$ is an equivalence in $\I{C}$.
\end{proof}
\begin{proposition}
	\label{prop:globalDefinitionPreservationLimits}
	Let $f\colon\I{C}\to\I{D}$ be a functor between $\BB$-categories and let $I$ be a simplicial object in $\BB$ such that $\I{C}$ and $\I{D}$ admit all $I$-indexed limits, i.e\ the diagonal maps $\I{C}\to\iFun(I,\I{C})$ and $\I{D}\to\iFun(I,\I{D})$ admit right adjoints (cf. Proposition~\ref{prop:limitColimitFunctor}). Then $f$ preserves all $I$-indexed limits precisely if the mate of the commutative square
	\begin{equation*}
		\begin{tikzcd}
			\iFun(I,\I{C})\arrow[d, "f_\ast"]\arrow[from=r, "\diag"] & \I{C}\arrow[d, "f"]\\
			\iFun(I,\I{D})\arrow[from=r, "\diag"] & \I{D}
		\end{tikzcd}
	\end{equation*}
	commutes. The dual statement about colimits holds as well.
\end{proposition}
\begin{proof}
	Suppose that $f$ preserves all $I$-indexed limits. The mate of the commutative square in the statement of the proposition is encoded by a map $\phi\colon f\lim\to \lim f_\ast$ that is given by the composite
	\begin{equation*}
			f\lim\xrightarrow{\eta f\lim} \lim \diag f\lim\xrightarrow{\simeq} \lim f_\ast \diag \lim\xrightarrow{\lim f_\ast \epsilon}\lim f_\ast
	\end{equation*}
	in which $\eta$ and $\epsilon$ are the units and counits of the two adjunctions $\diag\dashv \lim$. By~\cite[Corollary~4.7.17]{martini2021}, this map is an equivalence if and only if for any $d\colon A\to \iFun(I,\I{D})$ the associated map $\phi(d)\colon f(\lim d)\to\lim f_\ast d$ is an equivalence in $\I{D}$. Now since the transpose map $\diag f(\lim d)\to f_\ast d$ is given by postcomposing the equivalence $\diag f(\lim d)\simeq f_\ast\diag(\lim d)$ with the map $f_\ast\epsilon d$ and since Proposition~\ref{prop:limitColimitFunctor} implies that $\epsilon d$ is precisely the limit cone over $d$ in $\I{D}$, the claim follows from lemma~\ref{lem:equivalenceTransposeFinalObject}.
\end{proof}

\begin{remark}
	\label{rem:preservationColimitsMateExplicitly}
	Let $f\colon\I{C}\to\I{D}$ be a functor between $\BB$-categories, let $I$ be an arbitrary simplicial object in $\BB$ and let $d\colon A\to \iFun(\I{I},\I{C})$ be a diagram that has a limit in $\I{C}$. Suppose furthermore that $f_\ast d$ has a limit in $\I{D}$. Then the universal property of final objects (see~Corollary~\ref{cor:universalPropertyInitialObject}) gives rise to a unique map 
	\begin{equation*}
	\begin{tikzcd}[column sep=tiny]
	\diag f(\lim d)\arrow[rr] && \diag \lim f_\ast d \\
		& f_\ast d\arrow[from=ul]\arrow[from=ur] & 
	\end{tikzcd}
	\end{equation*}
	in $\Over{\I{D}}{f_\ast d}$ that is an equivalence if and only if $f$ preserves the limit of $d$. Since $\Over{\I{D}}{f_\ast d}\to \I{D}$ is a right fibration and therefore in particular conservative (cf.~\cite[Definition~4.1.10]{martini2021}), this is in turn equivalent to the map $f(\lim d)\to \lim f_\ast d$ being an equivalence in $\I{D}$. If both $\I{C}$ and $\I{D}$ admit $I$-indexed limits, this map is nothing but the mate transformation $f\lim \to \lim f_\ast$ from Proposition~\ref{prop:globalDefinitionPreservationLimits} evaluated at the object $d$.
\end{remark}

\begin{example}
	\label{ex:groupoidalLimitsColimitsPreservation}
	Let $f\colon\I{C}\to\I{D}$ be a functor between large $\BB$-categories and let $\I{G}$ be a $\BB$-groupoid. Suppose that both $\I{C}$ and $\I{D}$ admit $\I{G}$-indexed colimits. By using Proposition~\ref{prop:existenceAdjointsBeckChevalley} and Proposition~\ref{prop:globalDefinitionPreservationLimits}, the following two conditions are equivalent:
	\begin{enumerate}
	\item $f$ preserves $\I{G}$-indexed colimits;
	\item for every $A\in\BB$ the natural morphism $(\pi_{\I{G}})_! f(\I{G}\times A)\to f(A) (\pi_{\I{G}})_!$ is an equivalence.
	\end{enumerate}
	
	Dually, if $\I{C}$ and $\I{D}$ admit $\I{G}$-indexed limits, the following two conditions are equivalent:
	\begin{enumerate}
	\item $f$ preserves $\I{G}$-indexed limits;
	\item for every $A\in\BB$ the natural morphism $  f(A)(\pi_{\I{G}})_\ast\to(\pi_{\I{G}})_\ast f(A)$ is an equivalence.
	\end{enumerate}
\end{example}

\begin{example}
	\label{ex:externalLimitsColimitsPreservation}
	Let $f\colon\I{C}\to\I{D}$ be a functor between large $\BB$-categories, let $\II$ be an $\infty$-category and suppose that both $\I{C}$ and $\I{D}$ admit $\II$-indexed colimits. By using Proposition~\ref{prop:existenceAdjointsBeckChevalley} and Proposition~\ref{prop:globalDefinitionPreservationLimits}, the following two conditions are equivalent:
	\begin{enumerate}
			\item $f$ preserves $\II$-indexed colimits;
			\item for every $A\in\BB$ the functor $f(A)\colon\I{C}(A)\to\I{D}(A)$ preserves $\II$-indexed colimits.
	\end{enumerate}
	Dually, if $\I{C}$ and $\I{D}$ admit $\II$-indexed limits, the following two conditions are equivalent:
	\begin{enumerate}
			\item $f$ preserves $\II$-indexed limits;
			\item for every $A\in\BB$ the functor $f(A)\colon\I{C}(A)\to\I{D}(A)$ preserves $\II$-indexed limits.
	\end{enumerate}
\end{example}

Checking whether a functor between $\BB$-categories preserves certain limits or colimits becomes simpler when the functor is fully faithful:
\begin{proposition}
	\label{prop:limitsFullyFaithfulFunctor}
	Let $f\colon \I{C}\into\I{D}$ be a fully faithful functor between $\BB$-categories, let $I$ be a simplicial object in $\BB$ and let $d\colon A\to \iFun(I,\I{C})$ be a diagram in $\I{C}$. Suppose that $f_\ast(d)$ admits a colimit in $\I{D}$ such that $\colim f_\ast d$ is contained in $\I{C}$. Then $\colim f_\ast d$ already defines a colimit of $d$ in $\I{C}$. The analogous statement for limits holds as well.
\end{proposition}
\begin{proof}
	Since $f$ is fully faithful, the canonical square
	\begin{equation*}
		\begin{tikzcd}
		\Under{\I{C}}{d}\arrow[r, hookrightarrow, "f_\ast"] \arrow[d]& \Under{\I{D}}{f_\ast d} \arrow[d]\\
		A\times\I{C}\arrow[r, hookrightarrow, "\id\times f"] & A\times\I{D}
		\end{tikzcd}
	\end{equation*}
	is a pullback and $f_\ast$ is fully faithful. Therefore, if $\colim f_\ast d\colon A\to \Under{\I{D}}{f_\ast d}$ is an initial section such that the underlying object $\colim f_\ast d$ in $\I{D}$ is contained in $\I{C}$, then the entire colimit cocone is contained in the essential image of $f_\ast$, i.e.\ defines a section $A\to \Under{\I{C}}{d}$ over $A$. By Lemma~\ref{lem:cancellabilityInitialFF}, this section must be initial as well, hence the result follows.
\end{proof}
\begin{corollary}
	\label{cor:preservationColimitsFF}
	Let $f\colon \I{C}\into\I{D}$ be a fully faithful functor between $\BB$-categories, and suppose that both $\I{C}$ and $\I{D}$ admit $I$-indexed colimits for some simplicial object $I$ in $\BB$. Then $f$ preserves $I$-indexed colimits if and only if the restriction of $\colim\colon \iFun(I,\I{D})\to\I{D}$ along the inclusion $f_\ast\colon\iFun(I,\I{C})\into\iFun(I,\I{D})$ factors through the inclusion $f\colon \I{C}\into\I{D}$. The analogous statement for limits holds as well.\qed
\end{corollary}
We conclude this section with a discussion of the preservation of (co)limits by adjoint functors. We will need the following lemma:
\begin{lemma}
	\label{lem:mateBifunctor}
	Let $(l\dashv r)\colon \I{C}\leftrightarrows\I{D}$ be an adjunction between $\BB$-categories and let $i\colon L\to K$ be a map between simplicial objects in $\BB$. Then the two commutative squares
	\begin{equation*}
		\begin{tikzcd}
			\iFun(K,\I{C})\arrow[r, "l_\ast"]\arrow[d, "i^\ast"] & \iFun(K,\I{D})\arrow[d, "i^\ast"] && \iFun(K,\I{C})\arrow[from=r, "r_\ast"']\arrow[d, "i^\ast"] & \iFun(K,\I{D})\arrow[d, "i^\ast"]\\
			\iFun(L,\I{C})\arrow[r, "l_\ast"]& \iFun(L,\I{D}) && \iFun(L,\I{C})\arrow[from=r, "r_\ast"']& \iFun(L,\I{D})
		\end{tikzcd}
	\end{equation*}
	that are obtained from the bifunctoriality of $\iFun(-,-)$ are related by the mate correspondence.
\end{lemma}
\begin{proof}
    To prove the lemma, we may argue in the homotopy bicategory of the $(\infty,2)$-category $\Cat(\BB)$.
    Then the claim follows from the fact that the natural transformation $\iFun(K,-) \to \iFun(L,-)$ determines a pseudonatural transformation between 2-functors. 
    See \cite[Proposition 2.5]{kelly2006} for an argument in the strict case.
\end{proof}

\begin{proposition}
	\label{prop:adjointsPreserveLimitsColimits}
	Let $(l\dashv r)\colon \I{C}\leftrightarrows\I{D}$ be an adjunction between $\BB$-categories. Then $l$ preserves all colimits that exist in $\I{C}$, and $r$ preserves all limits that exist in $\I{D}$.
\end{proposition}
\begin{proof}
	We will show that the right adjoint $r\colon\I{D}\to\I{C}$ preserves all limits that exist in $\I{D}$, the dual statement about $l$ and colimits follows by taking opposite $\BB$-categories.
	Let therefore $I$ be a simplicial object in $\BB$ and let $d\colon A\to\iFun(I,\I{D})$ be a diagram that has a limit in $\I{D}$. We need to show that the image of the final section $\diag\lim d\to d$ along $r_\ast\colon \Over{\I{D}}{d}\to\Over{\I{C}}{r_\ast d}$ is final. By Corollary~\ref{cor:internalFunctorCategoryAdjunction}, the functor $\iFun(I,-)$ sends the adjunction $l\dashv r$ to an adjunction $l_\ast\dashv r_\ast\colon \iFun(I,\I{C})\leftrightarrows\iFun(I,\I{D})$, hence by using Proposition~\ref{prop:characterizationAdjunctionMappingGroupoids} one obtains a chain of equivalences
	\begin{align*}
		\map{\I{C}}(-,r(\lim d)) &\simeq \map{\I{D}}(l(-),\lim d)\\
		&\simeq\map{\iFun(I,\I{D})}(\diag l(-), d)\\
		&\simeq\map{\iFun(I,\I{D})}(l_\ast\diag (-), d)\\
		&\simeq\map{\iFun(I,\I{C})}(\diag(-),r_\ast d)
	\end{align*}
	of presheaves on $\I{C}$. We complete the proof by showing that this equivalence sends the identity $\id_{r(\lim d)}$ to the map $\diag r(\lim d)\simeq r_\ast\diag\lim d\to r_\ast d$ that arises as the image of the limit cone $\diag\lim d\to d$ under the functor $r_\ast$. By construction, the image of the identity $\id_{r(\lim d)}$ under this chain of equivalences is given by the composition
	\begin{equation*}
			\diag r(\lim d)\xrightarrow{\eta\diag r} r_\ast l_\ast \diag r(\lim d)\xrightarrow{\simeq} r_\ast\diag lr(\lim d)\xrightarrow{r_\ast\diag \epsilon} r_\ast\diag \lim d\to r_\ast d
	\end{equation*}
	in which the right-most map is the image of the limit cone $\diag\lim d\to d$ under the functor $r_\ast$, the map $\eta$ denotes the unit of the adjunction $l_\ast\dashv r_\ast$ and $\epsilon$ denotes the counit of the adjunction $l\dashv r$. As the composition of the first three maps is precisely the mate of the equivalence $l_\ast\diag\simeq \diag l$ and therefore recovers the equivalence $\diag r(\lim d)\simeq r_\ast\diag(\lim d)$ by Lemma~\ref{lem:mateBifunctor}, the result follows.
\end{proof}

\begin{proposition}
	\label{prop:reflectiveSubcategoryLimitsColimits}
	Let $(l\dashv r)\colon \I{C}\leftrightarrows\I{D}$ be an adjunction in $\Cat(\BB)$ that exhibits $\I{D}$ as a reflective subcategory of $\I{C}$, let $I$ be a simplicial object in $\BB$ and let $d\colon A\to \iFun(I,\I{D})$ be a diagram in context $A\in\BB$ such that $r_\ast d$ admits a colimit in $\I{C}$. Then $l(\colim r_\ast d)$ defines a colimit of $d$ in $\I{D}$.
	Dually, if $r_\ast d$ admits a limit in $\I{C}$, then $l(\lim r_\ast d)$ defines a limit of $d$ in $\I{D}$.
\end{proposition}
\begin{proof}
	Suppose first that $r_\ast d$ admits a colimit in $\I{C}$. Since $r$ is fully faithful, we obtain a chain of equivalences
		\begin{align*}
				\map{\iFun(I,\I{D})}(d,\diag(-)) &\simeq \map{\iFun(I,\I{C})}(r_\ast d, \diag r(-))\\
				&\simeq \map{\I{C}}(\colim r_\ast d, r(-))\\
				&\simeq \map{\I{D}}(l(\colim r_\ast d), -),
		\end{align*}
	which shows that the colimit of $d$ in $\I{D}$ exists and is explicitly given by $l(\colim r_\ast d)$.
	
	Next, let us suppose that $r_\ast d$ admits a limit in $\I{C}$. By the triangle identities, the functor $l$ sends the adjunction unit $\eta\colon \id\to rl$ to an equivalence. In particular, the map $\lim r_\ast d\to rl(\lim r_\ast d)$ is sent to an equivalence in $\I{D}$. Note that on account of the equivalence
	\begin{equation*}
			\map{\I{C}}(-, \lim r_\ast d)\simeq \map{\iFun(I,\I{D})}(\diag l(-), d),
	\end{equation*}
	the presheaf $\map{\I{C}}(-, \lim r_\ast d)$ sends any map in $\I{C}$ that is inverted by $l$ to an equivalence in $\Univ$.
	Applying this observation to $\eta\colon \lim r_\ast d\to rl(\lim r_\ast d)$, we obtain a retraction $\phi\colon rl(\lim r_\ast d)\to \lim r_\ast d$ of $\eta$ that gives rise to a retract diagram
	\begin{equation*}
		\begin{tikzcd}
		\lim r_\ast d\arrow[d, "\eta"]\arrow[r, "\eta"] & rl(\lim r_\ast d)\arrow[r, "\phi"] \arrow[d, "\eta"]& \lim r_\ast d\arrow[d, "\eta"]\\
		rl(\lim r_\ast d)\arrow[r, "rl\eta"] & rlrl(\lim r_\ast d)\arrow[r, "rl\phi"] & rl(\lim r_\ast d)
		\end{tikzcd}
	\end{equation*}
	in which the two maps in the lower row are equivalences. By the triangle identities and the fact that since $r$ is fully faithful the adjunction counit $\epsilon\colon lr\to\id$ is an equivalence (see Proposition~\ref{prop:adjunctionFullyFaithful}), the vertical map in the middle must be an equivalence as well, hence we conclude that $\eta\colon \lim r_\ast d\to rl(\lim r_\ast d)$ too is an equivalence. Therefore, the computation
	\begin{align*}
		\map{\iFun(I,\I{D})}(\diag(-),d) &\simeq \map{\iFun(I,\I{C})}(\diag r(-), r_\ast d)\\
		&\simeq \map{\I{C}}(r(-), \lim r_\ast d)\\
		&\simeq \map{\I{C}}(r(-), rl(\lim r_\ast d))\\
		&\simeq\map{\I{D}}(lr(-), l(\lim r_\ast d))\\
		&\simeq \map{\I{D}}(-, l(\lim r_\ast d))
	\end{align*}
	proves the claim.
\end{proof}

\begin{remark}
	We adopted the strategy for the proof of the second claim in proposition~\ref{prop:reflectiveSubcategoryLimitsColimits} from Denis-Charles Cisinski's proof of the analogous statement for $\infty$-categories, see~\cite[Proposition~6.2.17]{cisinski2019a}.
\end{remark}

\subsection{Limits and colimits in functor categories}
\label{sec:limitsColimitsFunctorCategories}
In this section, we discuss the familiar fact that limits and colimits in functor $\infty$-categories can be computed objectwise in the context of $\BB$-categories. 

\begin{proposition}
	\label{prop:limitsFunctorCategories}
	Let $I$ be a simplicial object in $\BB$ and let $\I{C}$ be a $\BB$-category that admits all $I$-indexed limits. Then $\iFun(K,\I{C})$ admits all $I$-indexed limits for any simplicial object $K$ in $\BB$, and the precomposition functor $i^\ast\colon \iFun(K,\I{C})\to\iFun(L,\I{C})$ preserves $I$-indexed limits for any map $i\colon L\to K$ in $\Simp\BB$. The dual statement for colimits is true as well.
\end{proposition}
\begin{proof}
    Proposition~\ref{prop:limitColimitFunctor} implies that the diagonal functor $\diag\colon \I{C}\to \iFun(I,\I{C})$ admits a right adjoint $\lim\colon \iFun(I,\I{C})\to\I{C}$. By Corollary~\ref{cor:internalFunctorCategoryAdjunction}, the functor $\lim_\ast\colon \iFun(K,\iFun(I,\I{C}))\to\iFun(K,\I{C})$ therefore defines a right adjoint to the diagonal functor $\diag_\ast\colon \iFun(K,\I{C})\to\iFun(K,\iFun(I,\I{C}))$. As postcomposing the latter with the equivalence $\iFun(K,\iFun(I,\I{C}))\simeq\iFun(I,\iFun(K,\I{C}))$ recovers the diagonal functor $\diag\colon \iFun(K,\I{C})\to\iFun(I,\iFun(K,\I{C}))$, Corollary~\ref{cor:internalFunctorCategoryAdjunction} implies that $\iFun(K,\I{C})$ admits all $I$-indexed limits. If $i\colon L\to K$ is an arbitrary map in $\Simp\BB$, the commutative diagram
	\begin{equation*}
		\begin{tikzcd}
			\iFun(K,\I{C})\arrow[d, "i^\ast"]\arrow[r, "\diag_\ast"]\arrow[rr, bend left, "\diag"] & \iFun(K,\iFun(I,\I{C}))\arrow[d, "i^\ast"]\arrow[r, "\simeq"] & \iFun(I,\iFun(K,\I{C}))\arrow[d, "(i^\ast)_\ast"]\\
			\iFun(L,\I{C})\arrow[r, "\diag_\ast"] \arrow[rr, bend right, "\diag"]& \iFun(L,\iFun(I,\I{C}))\arrow[r, "\simeq"] &\iFun(I,\iFun(L,\I{C}))
		\end{tikzcd}
	\end{equation*}
	and the functoriality of the mate construction (cf.\ Remark~\ref{rem:functorialityMates}) imply that in order to show that the functor $i^\ast\colon \iFun(K,\I{C})\to\iFun(L,\I{C})$ preserves $I$-indexed limits, we only need to show that the mate of the left square in the above diagram commutes, which is an immediate consequence of Lemma~\ref{lem:mateBifunctor}.
\end{proof}

\begin{proposition}
    \label{prop:limitsFunctorCategoryObjectwise}
    Let $I$ be a simplicial object in $\BB$ and let $\I{C}$ and $\I{D}$ be $\BB$-categories such that $\I{D}$ admits $I$-indexed limits. Let $d\colon A\to\iFun(I,\iFun(\I{C},\I{D}))$ be a diagram in context $A\in \BB$, and let $\diag F\to d$ be a cone over $d$, where $F\colon A\to \iFun(\I{C},\I{D})$ is an arbitrary object. Then $\diag F\to d$ is a limit cone if and only if for every map $s\colon B\to A$ in $\BB$ and every $c\colon B\to \I{C}$ the induced map $\diag (\overline F)(c)\to \bar d(c)$ is a limit cone in $\pi_A^\ast\I{D}$ in context $B$ (where $\diag \overline F\to \bar d$ denotes the transpose of $\diag F\to d$ across the adjunction $(\pi_A)_!\dashv \pi_A^\ast$).
    The dual statements for colimits holds as well.
\end{proposition}
\begin{proof}
	Using that $\pi_A^\ast$ preserves the internal hom (Remark~\ref{rem:localityPrinciplePreservationStructure}) together with the \'etale transposition invariance of limits (Remark~\ref{rem:etaleInvarianceLimits}), we may replace $\BB$ with $\Over{\BB}{A}$ and can therefore assume $A\simeq 1$ (see Remark~\ref{rem:etaleInvarianceReductionGlobalContext}).
    By means of the adjunction $\diag\dashv \lim$ and Lemma~\ref{lem:equivalenceTransposeFinalObject}, the map $\diag F\to d$ defines a limit cone if and only if the transpose map $F\to\lim d$ is an equivalence in $\iFun(\I{C},\I{D})$.
    Using that equivalences in functor $\BB$-categories are detected object-wise (see~\cite[Corollary~4.7.17]{martini2021}), this is in turn the case precisely if for every $c\colon B\to \I{C}$ the map $F(c)\to(\lim d)(c)$ is an equivalence in context $ B$. Note that by Remark~\ref{rem:limitsAreLocal}, this map transposes to the map $\pi_B^\ast(F)(\bar c)\to \lim \pi_B^\ast(d)(\bar c)$ (where $\bar c\colon 1_{\Over{\BB}{B}}\to\pi_B^\ast\I{C}$ is the transpose of $c$). Using Proposition~\ref{prop:limitsFunctorCategories}, we can identify the latter with the map $\pi_B^\ast(F)(\bar c)\to \lim(\pi_B^\ast(d)(c))$, i.e.\ with the transpose of the morphism of diagrams $\diag \pi_B^\ast(F)(\bar c)\to \pi_B^\ast(d)(\bar c)$. Hence, we conclude that $\diag F\to d$ is a limit cone if and only if $\diag \pi_B^\ast(F)(\bar c)\to \pi_B^\ast(d)(\bar c)$ is one for each $c\colon B\to \I{C}$. Now by Remark~\ref{rem:etaleInvarianceCones}, the latter transposes to $\diag{F}(c)\to d(c)$, hence the claim follows from the invariance of limit cones under \'etale transposition (Remark~\ref{rem:etaleInvarianceLimits}).
\end{proof}

\begin{proposition}
	\label{prop:postcompositionPreservationLimits}
	Let $f\colon \I{C}\to\I{D}$ be a functor between $\BB$-categories, let $I$ be a simplicial object in $\BB$ and suppose that both $\I{C}$ and $\I{D}$ admits $I$-indexed limits and that $f$ preserves such limits. Then for every simplicial object $K$ in $\BB$, the induced functor $f_\ast\colon \iFun(K,\I{C})\to\iFun(K,\I{D})$ preserves $I$-indexed limits as well. The dual statement for colimits holds too.
\end{proposition}
\begin{proof}
    Similarly as in the proof in Proposition~\ref{prop:limitsFunctorCategories}, we need to show that the mate of the left square in the commutative diagram
	\begin{equation*}
			\begin{tikzcd}
				\iFun(K,\I{C})\arrow[d, "f_\ast"]\arrow[r, "\diag_\ast"]\arrow[rr, bend left, "\diag"] & \iFun(K,\iFun(I,\I{C}))\arrow[d, "(f_\ast)_\ast"]\arrow[r, "\simeq"] & \iFun(I,\iFun(K,\I{C}))\arrow[d, "(f_\ast)_\ast"]\\
				\iFun(K,\I{D})\arrow[r, "\diag_\ast"] \arrow[rr, bend right, "\diag"]& \iFun(K,\iFun(I,\I{D}))\arrow[r, "\simeq"] & \iFun(I,\iFun(K,\I{D}))
			\end{tikzcd}
	\end{equation*}
	commutes, which follows from the observation that this mate is obtained by applying the functor $\iFun(K,-)$ to the mate of the commutative square
	\begin{equation*}
	\begin{tikzcd}
	\I{C}\arrow[d, "f"] \arrow[r, "\diag"] & \iFun(I,\I{C})\arrow[d, "f_\ast"] \\
	\I{D}\arrow[r, "\diag"] & \iFun(I,\I{D}),
	\end{tikzcd}
	\end{equation*}
	which by assumption is an equivalence. Hence the claim follows.
\end{proof}

\subsection{Limits and colimits in the universe $\Univ$}
\label{sec:limitsColimitsUniverse}
Our goal of this section is to prove that the universe $\Univ$ for $\BB$-groupoids admits small limits and colimits, and to give explicit constructions of those. We start with the case of colimits:
\begin{proposition}
	\label{prop:colimitsUniverse}
	The universe $\Univ$ for small $\BB$-groupoids admits small colimits. Moreover, if $\I{I}$ is a $\BB$-category and if $d\colon A\to \iFun(\I{I},\Univ)$ is an $\I{I}$-indexed diagram in context $A\in\BB$, then the colimit $\colim d\colon A\to\Univ$ is given by the $\Over{\BB}{A}$-groupoid $(\int d)^{\gp}$ , where $\int d\to A\times\I{I}$ denotes the left fibration that is classified by $d$.
\end{proposition}
\begin{proof}
	In light of Proposition~\ref{prop:limitColimitFunctor}, we need to show that the diagonal functor $\diag\colon\Univ\to\iFun(\I{I},\Univ)$ has a left adjoint, which is a consequence of Corollary~\ref{cor:existenceYonedaExtension}. The explicit description of this colimit furthermore follows from Proposition~\ref{prop:YonedaExtensionFibrations}.
\end{proof}
\begin{remark}
    For the special case $\BB\simeq\SS$, the explicit construction of colimits in Proposition~\ref{prop:colimitsUniverse} is given in~\cite[Corollary~3.3.4.6]{htt}.
\end{remark}

\begin{remark}
	\label{rem:howtocomputecolimitsUniverse}
	Let $i \colon \BB \hookrightarrow \PSh(\mathcal{C})$ be a left exact accessible localisation with left adjoint $L$, where $\mathcal{C}$ is a small $\infty$-category.
	Let $\I{I}$ be a $\BB$-category and let $d \colon 1 \rightarrow \iFun(\I{I},\Univ)$ be a diagram classified by a left fibration $\I{P} \rightarrow \I{I}$.
	By Proposition~\ref{prop:colimitsUniverse} we have that $\colim d \simeq \I{P}^{\gp} \simeq \colim_{\Delta^{\op}} \I{P}$.
    Therefore $\colim d$ is given by applying $L$ to the presheaf
	\[
	     c \mapsto (\colim_{\Delta^{\op}} \I{P})(c) \simeq \colim_{\Delta^{\op}} (\I{P}(c)) \simeq \I{P}(c)^{\gp}.
	\]
    Since~\cite[Corollary~4.6.8]{martini2021} implies that for every $c \in \CC$ the left fibration $\I{P}(c) \rightarrow \I{I}(c)$ classifies the functor $\Gamma_{\Over{\BB}{L(c)}} \circ d(c) \colon \I{I}(c) \rightarrow \mathcal{S}$, we conclude that $\colim d \in \BB$ is given by applying $L$ to the presheaf $ c \mapsto \colim (\Gamma \circ d(c))$.
\end{remark}

We will now proceed by showing that $\Univ$ also admits small limits. By Proposition~\ref{prop:limitColimitFunctor}, we need to show that for any $\BB$-category $\I{I}$ the diagonal functor $\diag\colon \Univ\to\iFun(\I{I},\Univ)$ admits a right adjoint. To that end, recall that since $\Cat(\BB)$ is cartesian closed, the pullback functor $\pi_{\I{I}}^\ast\colon\Cat(\BB)\to\Over{\Cat(\BB)}{\I{I}}$ admits a right adjoint $(\pi_{\I{I}})_\ast$ that is given by sending a functor $p\colon\I{P}\to\I{I}$ to the $\BB$-category $\Over{\iFun(\I{I},\I{P})}{\I{I}}$ that is defined by the pullback square
\begin{equation*}
\begin{tikzcd}
\Over{\iFun(\I{I},\I{P})}{\I{I}}\arrow[r]\arrow[d] & \iFun(\I{I},\I{P})\arrow[d, "p_\ast"]\\
1\arrow[r, "\id_{\I{I}}"] &\iFun(\I{I},\I{I}). 
\end{tikzcd}
\end{equation*}
If $p$ is a left fibration, then so is $p_\ast$, hence $(\pi_\I{I})_\ast$ sends $p$ to a $\BB$-groupoid in this case. Upon replacing $\BB$ with $\Over{\BB}{A}$ (where $A\in\BB$ is an arbitrary object) and using the locality of $\ILFib$ (see Remark~\ref{rem:LFibExplicitly}), this argument also shows that the pullback functor $\pi_{\I{I}}^\ast\colon \Over{\BB}{A}\to\LFib(A\times\I{I})$ admits a right adjoint $(\pi_{\I{I}})_\ast$ for any $A\in\BB$. Moreover, if $s\colon B\to A$ is a map in $\BB$, the natural map $s^\ast(\pi_{\I{I}})_\ast\to (\pi_{\I{I}})_\ast s^\ast$ is an equivalence whenever the transpose map $s_!(\pi_{\I{I}})^\ast\to (\pi_{\I{I}})^\ast s_!$ is one, and as this latter condition is evidently satisfied, Proposition~\ref{prop:existenceAdjointsBeckChevalley} and Theorem~\ref{thm:straightening} now show:

\begin{proposition}
	\label{prop:limitsUniverse}
	The universe $\Univ$ for small $\BB$-groupoids admits small limits. More precisely, if $\I{I}$ is a $\BB$-category and if $d\colon A\to \iFun(I,\Univ)$ is an $\I{I}$-indexed diagram in context $A\in\BB$, then the limit $\lim d\colon A\to \Univ$ is given by the $\Over{\BB}{A}$-groupoid $\Over{\iFun[\Over{\BB}{A}](\pi_{A}^\ast\I{I},\int \bar d)}{\pi_A^\ast\I{I}}$ in $\Over{\BB}{A}$, where $\int \bar d\to \pi_A^\ast\I{I}$ is the left fibration that is classified by the transpose $\bar d\colon \pi_A^\ast\I{I}\to\Univ[\Over{\BB}{A}]$ of $d$. \qed
\end{proposition}
\begin{proof}
    The discussion before the proposition shows the existence of limits.
    The explicit description of the limit follows from the description of the right adjoint $(\pi_{\I{I}})_*$ in the case $A = 1$ and the invariance of limits under étale transposition, Remark~\ref{rem:etaleInvarianceLimits}.
\end{proof}
\begin{remark}
    For the special case $\BB\simeq\SS$, the explicit construction of limits in Proposition~\ref{prop:limitsUniverse} is given in~\cite[Corollary~3.3.3.3]{htt}.
\end{remark}

If $\I{I}$ is an arbitrary $\BB$-category, the fact that right adjoint functors preserve limits (Proposition~\ref{prop:adjointsPreserveLimitsColimits}) combined with the fact that the final object $1_{\Univ}$ is the limit of the unique diagram $\varnothing\to \Univ$ (Example~\ref{ex:initialObjectColimit}) show that $\diag(1_{\Univ})\colon 1\to \iFun(\I{I},\Univ)$ defines a final object in $\iFun(\I{I},\Univ)$. We will denote this object by $1_{\iFun(\I{I},\Univ)}$. Proposition~\ref{prop:limitsUniverse} now implies:
\begin{corollary}
	\label{cor:limitFunctorUniverse}
	For any $\BB$-category $\I{I}$, the limit functor $\lim_{\I{I}}\colon \iFun(\I{I},\Univ)\to\Univ$ is explicitly given by the representable functor $\map{\iFun(\I{I},\Univ)}(1_{\iFun(\I{I},\Univ)},-)$, where $1_{\iFun(\I{I},\Univ)}\colon 1\to \iFun(\I{I},\Univ)$ denotes the final object in $\iFun(\I{I},\Univ)$.
\end{corollary}
\begin{proof}
	Since Proposition~\ref{prop:limitsUniverse} already implies the existence of $\lim_{\I{I}}$, the claim follows from the equivalence $\map{\iFun(\I{I},\Univ)}(1_{\iFun(\I{I},\Univ)},-)\simeq\map{\Univ}(1_{\Univ},\lim_{\I{I}}(-))$ and the fact that $\map{\Univ}(1_{\Univ},-)$ is equivalent to the identity functor on $\Univ$, see~\cite[Proposition~4.6.3]{martini2021}.
\end{proof}

Recall from \S~\ref{sec:universe} that there is a canonical embedding $i\colon \Univ[\BB]\into\Univ[\BBB]$. For later use, we note:
\begin{proposition}
	\label{prop:universeEnlargementColimits}
	The inclusion $i\colon \Univ[\BB]\into\Univ[\BBB]$ preserves small limits and colimits.
\end{proposition}
\begin{proof}
	We begin with the case of colimits.
	Using Corollary~\ref{cor:preservationColimitsFF}, it suffices to show that the restriction of the colimit functor $\colim\colon \iFun(\I{I},\Univ[\BBB])\to\Univ[\BBB]$ along the inclusion $i_\ast\colon \iFun(\I{I},\Univ[\BB])\into\iFun(\I{I},\Univ[\BBB])$ takes values in $\Univ[\BB]$ for any $\BB$-category $\I{I}$. Since Proposition~\ref{prop:colimitsUniverse} implies that the colimit of any diagram $d\colon A\to \iFun(\I{I},\Univ[\BBB])$ is given by the (large) $\Over{\BB}{A}$-groupoid $(\int d)^\gp$, the claim follows from~\cite[Proposition~3.3.3]{martini2021}, together with the fact that $d$ taking values in $\iFun(\I{I},\Univ[\BB])$ is tantamount to $\int d$ being a small $\Over{\BB}{A}$-category, cf.~\cite[Corollary~4.5.9]{martini2021}.
	
	As for the case of limits, by Corollary~\ref{cor:limitFunctorUniverse} we need to verify that $\map{\iFun(\I{I},\Univ[\BBB])}(1_{\iFun(\I{I},\Univ[\BBB])},i_\ast(-))$ takes values in $\Univ[\BB]$. Since we have $1_{\Univ[\BBB]}\simeq i(1_{\Univ})$, we find that $1_{\iFun(\I{I},\Univ[\BBB])}\simeq i_\ast(1_{\iFun(\I{I},\Univ)})$, so that the functor $\map{\iFun(\I{I},\Univ[\BBB])}(1_{\iFun(\I{I},\Univ[\BBB])},i_\ast(-))$ can be identified with $\map{\iFun(\I{I},\Univ)}(1_{\iFun(\I{I},\Univ)},-)$ (since $i_\ast$ is fully faithful). Hence the claim follows.
\end{proof}

We have now assembled the necessary results in order to prove the following:
\begin{proposition}
	\label{prop:limitsRepresentably}
	For any $\BB$-category $\I{C}$, the $\BB$-category $\IPSh(\I{C})$ of presheaves on $\I{C}$ admits small limits and colimits. Moreover, for any $\BB$-category $\I{I}$ and any diagram $d\colon A\to\iFun(\I{I},\I{C})$, a cone $\diag c\to d$ defines a limit of $d$ if and only if the induced cone $\diag h(c)\to h_\ast d$ defines a limit in $\IPSh(\I{C})$. In particular, the Yoneda embedding $h$ preserves small limits. 
\end{proposition}
\begin{proof}
	The fact that $\IPSh(\I{C})$ admits small limits and colimits follows immediately from combining Proposition~\ref{prop:limitsFunctorCategories} with Propositions~\ref{prop:limitsUniverse} and~\ref{prop:colimitsUniverse}. Now if $d\colon A\to \iFun(\I{I},\I{C})$ is an $\I{I}$-indexed diagram in $\I{C}$ and if $\diag c\to d$ is an arbitrary cone that is represented by a section $A\to \Over{\I{C}}{d}$ over $A$, we obtain a commutative diagram
	\begin{equation*}
		\begin{tikzcd}[column sep={6em,between origins}, row sep={3em,between origins}]
			& \Over{\I{C}}{c}\arrow[rr, hookrightarrow]\arrow[dd]\arrow[dl] && \Over{\IPSh(\I{C})}{h(c)}\arrow[dd]\arrow[dl]\\
			\Over{\I{C}}{d}\arrow[rr, crossing over, hookrightarrow]\arrow[dd] && \Over{\IPSh(\I{C})}{h_\ast d} &\\
			&\I{C}\times A \arrow[rr, hookrightarrow, "h\times\id", near start]\arrow[dl, "\id"] && \IPSh(\I{C})\times A\arrow[dl, "\id"]\\
			\I{C}\times A\arrow[rr, hookrightarrow, "h\times\id"] && \IPSh(\I{C})\times A \arrow[from=uu, crossing over]&
		\end{tikzcd}
	\end{equation*}
	in which the square in the front and the one in the back are cartesian as $h$ is fully faithful. Therefore, the upper horizontal square must be cartesian as well. The cone $\diag c\to d$ defines a limit of $d$ if and only if the map $\Over{\I{C}}{c}\to \Over{\I{C}}{d}$ is an equivalence. Likewise, the induced cone $\diag h(c)\to h_\ast d$ defines a limit of $h_\ast d$ precisely if the map $\Over{\IPSh(\I{C})}{h(c)}\to \Over{\IPSh(\I{C})}{h_\ast d}$ is an equivalence. To complete the proof, we therefore need to show that the first map is an equivalence if and only if the second map is one. As the upper square in the previous diagram is cartesian, the second condition implies the first. Conversely, the map $\Over{\IPSh(\I{C})}{h(c)}\to \Over{\IPSh(\I{C})}{h_\ast d}$ corresponds via Theorem~\ref{thm:straightening} to a map between presheaves on $\IPSh(\I{C})$ which are both representable by objects in $\IPSh(\I{C})$. Therefore, there is a unique map $h(c)\to \lim h_\ast d$ in $\IPSh(\I{C})$ such that the induced map
	\begin{equation*}
		\map{\IPSh(\I{C})}(-, h(c))\to \map{\IPSh(\I{C})}(-, \lim h_\ast d)
	\end{equation*}
	recovers the morphism $\Over{\IPSh(\I{C})}{h(c)}\to \Over{\IPSh(\I{C})}{h_\ast d}$ on the level of presheaves on $\IPSh(\I{C})$. As Yoneda's lemma (Theorem~\ref{thm:YonedaLemma}) implies that restricting this map along $h\colon \I{C}\into \IPSh(\I{C})$ recovers the map $h(c)\to \lim h_\ast d$, the latter being an equivalence implies that the morphism $\Over{\IPSh(\I{C})}{h(c)}\to \Over{\IPSh(\I{C})}{h_\ast d}$ is an equivalence as well, as desired.
\end{proof}

\begin{corollary}
	\label{cor:corepresentablePreserveLimits}
	For any $\BB$-category $\I{C}$ and any object $c\colon A\to\I{C}$ in context $A\in\BB$, the corepresentable functor $\map{\I{C}}(c,-)\colon A\times\I{C}\to \Univ$ transposes to a functor $\pi_A^\ast\I{C}\to\Univ[\Over{\BB}{A}]$ that preserves all limits that exist in $\pi_A^\ast\I{C}$.
\end{corollary}
\begin{proof}
    By Example~\ref{ex:representabilityLocalCondition}, the transpose of $\map{\I{C}}(c,-)$ can be identified with $\map{\pi_A^\ast\I{C}}(\bar c,-)$, where $\bar c\colon 1_{\Over{\BB}{A}}\to\pi_A^\ast\I{C}$ is the transpose of $c$. Therefore, by replacing $\BB$ with $\Over{\BB}{A}$, we may assume that $A\simeq 1$.
	On account of Yoneda's lemma, the functor $\map{\I{C}}(c,-)$ is equivalent to the composition $c^\ast h$, where $h$ denotes the Yoneda embedding and $c^\ast\colon \IPSh(\I{C})\to\Univ$ is the evaluation functor at $c$. By Proposition~\ref{prop:limitsRepresentably} and Proposition~\ref{prop:limitsFunctorCategories}, both of these functors preserve limits, hence the claim follows.
\end{proof}

Our next goal is to show that $\Univ$ is \emph{cartesian closed}. To that end, denote by $-\times -\colon \Univ\times\Univ\to\Univ$ the product functor.  One now finds:
\begin{proposition}
	\label{prop:CartesianClosureUniverse}
	The universe $\Univ$ for small $\BB$-groupoids is cartesian closed, in that there is an equivalence
	\begin{equation*}
		\map{\Univ}(-\times-,-)\simeq\map{\Univ}(-,\map{\Univ}(-,-))
	\end{equation*}
	of functors $\Univ^{\op}\times\Univ^{\op}\times\Univ\to\Univ$.
\end{proposition}
\begin{proof}
	First, we claim that the transpose $\phi\colon \Univ\to\iFun(\Univ,\Univ)$ of the product bifunctor $-\times -\colon \Univ\times\Univ\to\Univ$ takes values in $\iFun^L(\Univ,\Univ)$. To see this, we need to show that the image of every $\Over{\BB}{A}$-groupoid $\I{G}$ along $\phi$ defines a left adjoint functor of $\Over{\BB}{A}$-categories. Note that since $\pi_A^\ast$ preserves adjunctions (Corollary~\ref{cor:geometricMorphismAdjunction}) and the internal hom (Remark~\ref{rem:localityPrinciplePreservationStructure}), we may identify $\pi_A^\ast(-\times -)$ with the product bifunctor of $\pi_A^\ast\Univ$ and $\pi_A^\ast(\phi)$ with its transpose. Together with the equivalence $\pi_A^\ast\Univ\simeq\Univ[\Over{\BB}{A}]$ from Remark~\ref{rem:localityPrinciplePreservationStructure}, this implies that the image $\phi(\I{G})\colon A\to\iFun(\Univ,\Univ)$ transposes to the product functor $\I{G}\times-\colon \Univ[\Over{\BB}{A}]\to\Univ[\Over{\BB}{A}]$. Thus, by replacing $\BB$ with $\Over{\BB}{A}$, we may assume without loss of generality that $A\simeq 1$. In this case, Example~\ref{ex:externalLimitsColimits} implies that the functor $\I{G}\times -\colon \Univ\to\Univ$ is given on local sections over $A\in\BB$ by the $\infty$-categorical product functor
	\begin{equation*}
	\begin{tikzcd}
	\Over{\BB}{A}\arrow[r, "\pi_A^\ast\I{G}\times -"] & \Over{\BB}{A}
	\end{tikzcd}
	\end{equation*}
	which admits a right adjoint $\Hom_{\Over{\BB}{A}}(\pi_A^\ast\I{G},-)$. If $s\colon B\to A$ is a map in $\BB$, we deduce from~\cite[Lemma~4.2.3]{martini2021} that the natural map $s^\ast\Hom_{\Over{\BB}{A}}(\pi_A^\ast\I{G},-)\to \Hom_{\Over{\BB}{B}}(\pi_B^\ast\I{G},s^\ast(-))$ is an equivalence, hence Proposition~\ref{prop:existenceAdjointsBeckChevalley} shows that the functor $\I{G}\times -\colon \Univ\to\Univ$ admits a right adjoint, as desired. 
	
	As a consequence of what we've just shown and Corollary~\ref{cor:functorialityAdjunction}, we now obtain a bifunctor $f\colon \Univ^{\op}\times\Univ\to\Univ$ that fits into an equivalence
	\begin{equation*}
		\map{\Univ}(-\times-, -)\simeq\map{\Univ}(-, f(-,-)).
	\end{equation*}
	We complete the proof by showing that $f$ is equivalent to the mapping bifunctor $\map{\Univ}(-,-)$. Note that by~\cite[Proposition~4.6.3]{martini2021} the functor $\map{\Univ}(1_{\Univ},-)$ is equivalent to the identity on $\Univ$. Hence the chain of equivalences
	\begin{equation*}
	f(-,-)\simeq\map{\Univ}(1_{\Univ}, f(-,-))\simeq \map{\Univ}(1_{\Univ}\times -, -)\simeq \map{\Univ}(-,-)
	\end{equation*}
	in which the second step follows from the evident equivalence $1_{\Univ}\times - \simeq \id_{\Univ}$ gives rise to the desired identification.
\end{proof}
In~\cite[Proposition~3.7.3]{martini2021}, it was shown that for any two objects $g,h\colon A\rightrightarrows\Univ$ in context $A\in\BB$ that correspond to $\Over{\BB}{A}$-groupoids $\I{G},\I{H}$, there is an equivalence $\Hom_{\Over{\BB}{A}}(\I{G},\I{H})\simeq\map{\Univ}(g,h)$ of $\Over{\BB}{A}$-groupoids (where $\Hom_{\Over{\BB}{A}}(\I{G},\I{H})$ denotes the internal hom in $\Over{\BB}{A}$). We are now able to upgrade this result to a \emph{functorial} equivalence.
\begin{proposition}
	\label{prop:MappingGroupoidBifunctorUniverse}
	The mapping $\BB$-groupoid bifunctor $\map{\Univ}(-,-)$ recovers the internal hom bifunctor $\Hom_{\Over{\BB}{A}}(-,-)\colon \Over{\BB}{A}^{\op}\times\Over{\BB}{A}\to\Over{\BB}{A}$ when taking local sections over $A\in\BB$.
\end{proposition}
\begin{proof}
	By~\cite[Lemma~4.7.13]{martini2021} and Remark~\ref{rem:localityPrinciplePreservationStructure}, we can identify $\pi_A^\ast(\map{\Univ}(-,-))$ with $\map{\Univ[\Over{\BB}{A}]}(-,-)$. Therefore, by replacing $\BB$ with $\Over{\BB}{A}$ we may assume without loss of generality that $A\simeq 1$. Also,~\cite[Corollary~4.6.8]{martini2021} implies that one may identify the bifunctor $\map{\BB}(-,-)\colon\BB^\op\times\BB\to\SS$ with the composition
	\begin{equation*}
		\BB^\op\times\BB\xrightarrow{\Gamma_{\BB}(\map{\Univ}(-,-))}\BB\xrightarrow{\Gamma_{\BB}}\SS.
	\end{equation*}
	Since applying $\Gamma_{\BB}$ to the bifunctor $-\times -\colon \Univ\times\Univ\to\Univ$ recovers the ordinary product bifunctor on $\BB$, Proposition~\ref{prop:CartesianClosureUniverse} yields an equivalence
	\begin{equation*}
	\map{\BB}(-\times -,-)\simeq \map{\BB}(-, \Gamma_{\BB}(\map{\Univ}(-,-))),
	\end{equation*}
	which finishes the proof.
\end{proof}

\subsection{Limits and colimits in $\ICat_{\BB}$}
\label{sec:limitsColimitsCat}
Recall that by the discussion in Appendix~\ref{sec:CatB}, the assignment $A \mapsto \Cat(\BB_{/A})$ defines a sheaf of $\infty$-categories on $\BB$ that we denote by $\ICat_{\BB}$ and that we refer to as the $\BB$-category of (small) $\BB$-categories.
By combining Proposition~\ref{prop:reflectiveSubcategoryLimitsColimits} with Proposition~\ref{prop:presentationCategoryOfCategories} and the fact that presheaf $\BB$-categories admits small limits and colimits (Proposition~\ref{prop:limitsRepresentably}), we find:
\begin{proposition}
	\label{prop:CategoryOfCategoriesLimitsColimits}
	The $\BB$-category $\ICat_{\BB}$ admits small limits and colimits.\qed
\end{proposition}
\begin{remark}
    Similar to the case of diagrams in $\Univ$, one can give explicit formulas for limits and colimits of diagrams in $\ICat_{\BB}$. However, these formulas rely on the theory of cartesian and cocartesian fibrations for $\BB$-categories, which we plan to feature in upcoming work.
\end{remark}
Next, our goal is to show that $\ICat_{\BB}$ is \emph{cartesian closed}. To that end, let $-\times -\colon \ICat_{\BB}\times\ICat_{\BB}\to\ICat_{\BB}$ be the product functor.
\begin{proposition}
	\label{prop:CatCartesianClosed}
	There is a functor $\iFun(-,-)\colon \ICat_{\BB}^{\op}\times\ICat_{\BB}\to\ICat_{\BB}$ together with an equivalence
	\begin{equation*}
		\map{\ICat_{\BB}}(-\times -,-)\simeq \map{\ICat_{\BB}}(-, \iFun(-,-)).
	\end{equation*}
	In other words, the $\BB$-category $\ICat_{\BB}$ is cartesian closed.
\end{proposition}
\begin{proof}
	This is proved in exactly the same way as Proposition~\ref{prop:CartesianClosureUniverse}. Namely, by using Corollary~\ref{cor:functorialityAdjunction}, it is enough to show that the product bifunctor transposes to a functor $\ICat_{\BB}\to \iFun^L(\ICat_{\BB},\ICat_{\BB})$. Using the equivalence $\pi_A^\ast\ICat_{\BB}\simeq \ICat_{\Over{\BB}{A}}$ from Remark~\ref{rem:BCCatB}, we may carry out the same reduction steps as in the proof of Proposition~\ref{prop:CartesianClosureUniverse}, so that it will be sufficient to prove that for every $\BB$-category $\I{C}$ the functor $\I{C}\times -\colon \ICat_{\BB}\to\ICat_{\BB}$ has a right adjoint. To see this, note that this functor is given on local sections over $A\in\BB$ by the $\infty$-categorical product functor
	\begin{equation*}
	\begin{tikzcd}
	\Cat(\Over{\BB}{A})\arrow[r, "\pi_A^\ast\I{C}\times -"] & \Cat(\Over{\BB}{A}).
	\end{tikzcd}
	\end{equation*}
	which admits a right adjoint $\iFun[\Over{\BB}{A}](\pi_A^\ast\I{C},-)$. Furthermore, if $s\colon B\to A$ is a map in $\BB$, we deduce from~\cite[Lemma~4.2.3]{martini2021} that the natural map $s^\ast\iFun[\Over{\BB}{A}](\pi_A^\ast\I{C},-)\to \iFun[\Over{\BB}{B}](\pi_B^\ast\I{C},s^\ast(-))$ is an equivalence. Hence, Proposition~\ref{prop:existenceAdjointsBeckChevalley} shows that the functor $\I{C}\times -\colon \ICat_{\BB}\to\ICat_{\BB}$ admits a right adjoint, as desired.
\end{proof}
\begin{remark}
	\label{rem:internalExternalCartesianClosureCat}
	By making use of~\cite[Corollary~4.6.8]{martini2021} and the fact that the product bifunctor $-\times -$ on $\ICat_{\BB}$ recovers the $\infty$-categorical product bifunctor on $\Cat(\Over{\BB}{A})$ upon taking local sections over $A\in\BB$, the equivalence
	\begin{equation*}
			\map{\ICat_{\BB}}(-\times -,-)\simeq \map{\ICat_{\BB}}(-, \iFun(-,-))
	\end{equation*}
	from Proposition~\ref{prop:CatCartesianClosed} implies that the bifunctor $\iFun(-,-)\colon\ICat_{\BB}^{\op}\times\ICat_{\BB}\to\ICat_{\BB}$ recovers the internal hom of $\Cat(\Over{\BB}{A})$ when being evaluated at $A\in \BB$, which justifies our choice of notation.
\end{remark}
\begin{corollary}
    \label{cor:mappingGroupoidBifunctorCatB}
	The mapping $\BB$-groupoid bifunctor $\map{\ICat_{\BB}}(-,-)\colon\ICat_{\BB}^\op\times\ICat_{\BB}\to\Univ$ is equivalent to the composition of the bifunctor $\iFun(-,-)\colon \ICat_{\BB}^{\op}\times\ICat_{\BB}\to\ICat_{\BB}$ with the core $\BB$-groupoid functor $(-)^{\core}\colon \ICat_{\BB}\to\Univ$.
\end{corollary}
\begin{proof}
	On account of Proposition~\ref{prop:internalCoreGroupoidification} and the fact that the functor $\map{\Univ}(1_{\Univ},-)$ is equivalent to the identity on $\Univ$ (see~\cite[Proposition~4.6.3]{martini2021}), we obtain equivalences
	\begin{align*}
		\iFun(-,-)^{\core} &\simeq \map{\Univ}(1_{\Univ}, \iFun(-,-)^\core) \\
		&\simeq \map{\ICat_{\BB}}(1_{\Univ}, \iFun(-,-))\\
		&\simeq \map{\ICat_{\BB}}(1_{\Univ}\times -, -)\\
		&\simeq \map{\ICat_{\BB}}(-,-)
	\end{align*}
	in which the last equivalence follows from the evident equivalence $1_{\Univ}\times -\simeq \id_{\ICat_{\BB}}$.
\end{proof}

\subsection{A characterisation of initial and final functors}
\label{sec:initialFunctorsLimits}
In this section, we show that initial and final functors (see \S~\ref{sec:leftFibrations}) can be characterised as those functors along which restriction of diagrams does not change their limits and colimits, respectively. For the case $\BB\simeq \SS$, this characterisation is proved in~\cite[Proposition~4.1.1.8]{htt} or \cite[Theorem 6.4.5]{cisinski2019a}. For the general case, note that precomposition with a functor $i\colon \I{J}\to \I{I}$ of $\BB$-categories defines a functor $i^\ast\colon\iFun(\I{I},\I{C})\to\iFun(\I{J},\I{C})$ that induces a functor $i^\ast\colon \Under{\I{C}}{d}\to \Under{\I{C}}{i^\ast d}$ over $A\times\I{C}$ for every $\I{I}$-indexed diagram $d\colon A\to \iFun(\I{I},\I{C})$ in $\I{C}$.
\begin{proposition}
	\label{prop:characterisationFinalLimit}
	For any functor $i\colon\I{J}\to\I{I}$ between $\BB$-categories, the following are equivalent:
	\begin{enumerate}
		\item $i$ is final;
		\item for every large $\BB$-category $\I{C}$ and every diagram $d\colon A\to \iFun(\I{I},\I{C})$ in context $A\in\BB$, the functor $i^\ast\colon \Under{\I{C}}{d}\to \Under{\I{C}}{i^\ast d}$ is an equivalence;
		\item For every large $\BB$-category $\I{C}$ and every diagram $d\colon A\to \iFun(\I{I},\I{C})$ in context $A\in\BB$ that admits a colimit $\colim d$, the image of the colimit cocone $d\to \diag\colim d$ along the functor $i^\ast\colon \Under{\I{C}}{d}\to \Under{\I{C}}{i^\ast d}$ defines a colimit cocone of $i^\ast d$.
		\item The mate of the commutative square
		\begin{equation*}
		\begin{tikzcd}
		\Univ\arrow[r, "\diag"]\arrow[d, "\id"] & \iFun(\I{I},\Univ)\arrow[d, "i^\ast"]\\
		\Univ\arrow[r, "\diag"] & \iFun(\I{J},\Univ)
		\end{tikzcd}
		\end{equation*}
		commutes.
	\end{enumerate}
	The dual characterisation of initial functors holds as well.
\end{proposition}
\begin{proof}
	Suppose that $i$ is final, and let $d\colon A\to\iFun(\I{I},\I{C})$ be an arbitrary diagram. By making use of Remark~\ref{rem:etaleInvarianceLimits} and the fact that the base change functor $\pi_A^\ast$ preserves final functors~\cite[Remark~4.4.9]{martini2021}, we may replace $\BB$ with $\Over{\BB}{A}$ and can therefore assume that $A\simeq 1$ (see Remark~\ref{rem:etaleInvarianceReductionGlobalContext}). On account of~\cite[Proposition~4.1.18]{martini2021}, it suffices to show that the induced map $i^\ast\vert_c$ on the fibres over every $c\colon A\to\I{C}$ is an equivalence. By the same argument as above, we may again assume $A\simeq 1$. Now the commutative diagram
	\begin{equation*}
		\begin{tikzcd}[column sep=large]
			1\arrow[r, "c"] \arrow[d, "d"]& \I{C}\arrow[d, "d\times\diag"]\\
			\iFun(\I{I},\I{C})\arrow[r, "\id\times \diag(c)"] & \iFun(\I{I},\I{C})\times\iFun(\I{I},\I{C}) 
		\end{tikzcd}
	\end{equation*}
	shows that the fibre of the left fibration $\Under{\I{C}}{d}\to \I{C}$ over $c$ is equivalent to the fibre of the right fibration $\iFun(\I{I},\Over{\I{C}}{c})\to\iFun(\I{I},\I{C})$ (that is given by postcomposition with $(\pi_c)_!\colon\Over{\I{C}}{c}\to \I{C}$) over $d\colon 1\to \iFun(\I{I},\I{C})$. Similarly, the fibre of $\Under{\I{C}}{i^\ast d}\to \I{C}$ over $c$ is equivalent to the fibre of the right fibration $\iFun(\I{J},\Over{\I{C}}{c})\to\iFun(\I{J},\I{C})$ over $i^\ast d$ such that the map $i^\ast\vert_c$ fits into the commutative diagram
	\begin{equation*}
		\begin{tikzcd}[column sep={5em,between origins}, row sep={3em,between origins}]
			& \Under{\I{C}}{i^\ast d}\vert_c\arrow[rr] \arrow[dd]&& \iFun(\I{J},\Over{\I{C}}{c}) \arrow[dd]\\
			\Under{\I{C}}{d}\vert_c\arrow[ur, "i^\ast\vert_c"]\arrow[rr, crossing over]\arrow[dd] && \iFun(\I{I},\Over{\I{C}}{c})\arrow[ur, "i^\ast"]&\\
			& 1\arrow[rr, "i^\ast d", near start] && \iFun(\I{J},\I{C}) \\
			1\arrow[rr, "d"]\arrow[ur, "\id"] && \iFun(\I{I},\I{C}) \arrow[ur, "i^\ast"]\arrow[from=uu, crossing over]&
		\end{tikzcd}
	\end{equation*}
	in which the two squares in the front and in the back are cartesian. Since $i$ is final, the right square must be cartesian as well, hence $i^\ast\vert_c$ is an equivalence, so that~(2) holds. Condition~(3) follows immediately from~(2).
	For the special case $\I{C}=\Univ$, the same argument as in the proof of Proposition~\ref{prop:globalDefinitionPreservationLimits} shows that condition~(3) is equivalent to the condition that the map $\colim_{\I{J}} i^\ast\to \colim_{\I{I}}$ must be an equivalence, hence condition~(3) implies condition~(4).
	Lastly, suppose that the map $\colim_{\I{J}}i^\ast\to \colim_{\I{I}}$ is an equivalence, and let us show that $i$ is final. It will be enough to show that $i$ is internally left orthogonal to the universal right fibration $\UnivHat^{\op}\to\Univ^{\op}$ (see~\cite[\S~4.6]{martini2021}) as every right fibration between (small) $\BB$-categories arises as a pullback of this functor. By Proposition~\ref{prop:limitsUniverse}, the universe $\Univ$ admits small limits, hence if $d\colon A\to \iFun(\I{I},\Univ^{\op})$ is an arbitrary diagram both $\Under{\Univ^\op}{d}$ and $\Under{\Univ^{\op}}{i^\ast d}$ admits an initial section. By assumption, the functor $i^\ast\colon\Under{\Univ^\op}{d}\to\Under{\Univ^\op}{i^\ast d}$ sends the colimit cocone $d
	\to \diag\colim d$ to an initial section of $\Under{\Univ^\op}{i^\ast d}$, which implies that the functor $i^\ast\colon\Under{\Univ^\op}{d}\to\Under{\Univ^\op}{i^\ast d}$ must be initial as well. But this map is already a left fibration since it can be regarded as a map betwee left fibrations over $\Univ^{\op}$, hence we conclude that this functor must be an equivalence. Similarly as above and by making use of the equivalence $\UnivHat\simeq\Under{\Univ}{1_{\Univ}}$ over $\Univ$ from~\cite[Proposition~4.6.3]{martini2021}, one obtains a commutative diagram
	\begin{equation*}
		\begin{tikzcd}[column sep={5em,between origins}, row sep={3em,between origins}]
			& \Under{\Univ^\op}{i^\ast d}\vert_{\pi_A^\ast(1_{\Univ})}\arrow[rr] \arrow[dd]&& \iFun(\I{J},\UnivHat^{\op}) \arrow[dd]\\
			\Under{\Univ^\op}{d}\vert_{\pi_A^\ast(1_{\Univ})}\arrow[ur, "i^\ast\vert_{\pi_A^\ast(1_{\Univ})}"]\arrow[rr, crossing over]\arrow[dd] && \iFun(\I{I},\UnivHat^{\op}) \arrow[ur, "i^\ast"]&\\
			& A\arrow[rr, "i^\ast d", near start] && \iFun(\I{J},\Univ^{\op}) \\
			A\arrow[rr, "d"]\arrow[ur, "\id"] && \iFun(\I{I},\Univ^{\op}) \arrow[ur, "i^\ast"]\arrow[from=uu, crossing over]&
		\end{tikzcd}
	\end{equation*}
	in which the squares in the front, in the back and on the left are cartesian. As the maps $\iFun(\I{I},\UnivHat^{\op})\to\iFun(\I{I},\Univ^{\op})$ and $\iFun(\I{J},\UnivHat^{\op})\to\iFun(\I{J},\Univ^{\op})$ are right fibrations, the vertical square on the right is cartesian already when its underlying square of core $\BB$-groupoids is. We therefore deduce that this square must be a pullback as well, which means that $i$ is final.
\end{proof}

\begin{remark}
	\label{rem:functorialityColimitIndexingCategory}
	Let $\I{C}$ be a large $\BB$-category, let $i\colon \I{J}\to\I{I}$ be a functor between $\BB$-categories and let us fix an $\I{I}$-indexed diagram $d\colon A\to \iFun(\I{I},\I{C})$. Suppose that both $d$ and $i^\ast d$ admit a colimit in $\I{C}$. Then the universal property of initial objects (see~\cite[Corollary~4.3.21]{martini2021}) gives rise to a unique map 
	\begin{equation*}
	\begin{tikzcd}[column sep=tiny]
	& i^\ast d\arrow[dl]\arrow[dr] & \\
	\diag \colim i^\ast d\arrow[rr] && \diag\colim d
	\end{tikzcd}
	\end{equation*}
	in $\Under{\I{C}}{i^\ast d}$ that is an equivalence if and only if the cocone $i^\ast d\to \diag\colim d$ (which is the image of the colimit cocone $d\to \diag\colim d$ along $i^\ast$) is a colimit cocone. Proposition~\ref{prop:characterisationFinalLimit} now implies that this map is always an equivalence when $i$ is final, and conversely $i$ must be final whenever this map is an equivalence for every $\BB$-category $\I{C}$ and every diagram $d$ that has a colimit in $\I{C}$ (in fact, Proposition~\ref{prop:characterisationFinalLimit} shows that it suffices to consider $\I{C}=\Univ$).
\end{remark}

\subsection{Decomposition of colimits}
\label{sec:decompositionColimits}
In~\cite[\S~4.2]{htt},~Lurie provides techniques for computing colimits in an $\infty$-category by means of decomposing diagrams into more manageable pieces. For example, he proves that an $\infty$-category has small colimits if and only if it has small coproducts and pushouts. In this section, we aim for similar results in the context of internal higher category theory.
We are mainly interested in the decompoisiton of arbitrary colimits into colimits indexed by constant $\BB$-categories (i.e.\ $\BB$-categories that are in the image of the functor $\const_{\BB}\colon \CatS\to\Cat(\BB)$, see Remark~\ref{rem:functorialityBCategories}) and $\BB$-groupoids.
In these two cases, colimits admit rather explicit descriptions that are often simpler to understand in practice (see Examples~\ref{ex:externalLimitsColimits} and \ref{ex:groupoidalLimitsColimits}). 
Note that in $\infty$-category theory such a decomposition is not really visible since internal to the $\infty$-topos of spaces $\mathcal{S}$, any $\SS$-groupoid is automatically constant.
However, the technique of proof that we use is still mostly the same as in~\cite[\S~4.2]{htt}.
Our main result will be the following proposition:

\begin{proposition}
	\label{prop:ColimitsGroupoidsExternal}
	A large $\BB$-category $\I{C}$ admits small colimits if and only if it admits colimits indexed by \emph{constant} $\BB$-categories and by $\BB$-groupoids, and a functor $f\colon \I{C}\to\I{D}$ between large $\BB$-categories that admit small colimits preserves such colimits if and only if it preserves colimits indexed by constant $\BB$-categories and by $\BB$-groupoids.
\end{proposition}
The proof of Proposition~\ref{prop:ColimitsGroupoidsExternal} requires a few preparations.

\begin{lemma}
	\label{lem:initialObjectProduct}
	Let $(\I{C}_i)_{i\in I}$ be a small family of $\BB$-categories, and let $c_i\colon 1\to \I{C}_i$ be an object in context $1\in\BB$ for every $i\in I$. If each $c_i$ is initial then the induced object $c=(c_i)_{i\in I}\colon 1\to \I{C}=\prod_i \I{C}_i$ is initial as well.
\end{lemma}
\begin{proof}
	By Proposition~\ref{prop:characterisationinitialObject}, the object $(c_i)_{i\in I}$ is initial precisely if the projection
	\begin{equation*}
		(\pi_c)_!\colon \Under{\I{C}}{c}\to\I{C}
	\end{equation*}
	is an equivalence. The result thus follows from the observation that $(\pi_c)_!$ is equivalent to the product
	\begin{equation*}
		\prod_i (\pi_{c_i})_!\colon \prod_i \Under{(\I{C}_i)}{{c_i}}\to \prod_i \I{C}_i
	\end{equation*}
	and is therefore an equivalence since each of the maps $(\pi_{c_i})_!$ is one.
\end{proof}

The key input in the proof of Proposition~\ref{prop:ColimitsGroupoidsExternal} is the following Proposition. The strategy of proof is the same as in~\cite[Proposition~4.4.2.6]{htt}.
\begin{proposition}
	\label{prop:decompositionColimits}
	Let $\kappa$ be a regular cardinal, let $\KK$ be a $\kappa$-small $\infty$-category and let 
	\begin{equation*}
	\alpha\colon \KK\to \Cat(\BB),\quad k \mapsto \I{J}_k
	\end{equation*}
	be a diagram with colimit $\I{J}=\colim_k \I{J}_k$ in $\Cat(\BB)$. Suppose that $\I{C}$ is a $\BB$-category and that $d\colon \I{J}\to\I{C}$ is a diagram such that
	\begin{enumerate}
		\item for every $k\in \KK$ the restricted diagram $d_k\colon \I{J}_k\to \I{C}$ admits a colimit in $\I{C}$;
		\item $\I{C}$ admits colimits indexed by $\kappa$-small constant $\BB$-categories.
	\end{enumerate}
	Then $d$ admits a colimit in $\I{C}$.
\end{proposition}
\begin{proof}
        We consider the full subcategory $ \CC$ of $(\CatS)_{/\KK}$ spanned by all functors $\varphi \colon \LL \to \KK$  such that the conclusion of the proposition holds for $\alpha \circ \varphi$.
        We wish to show that the $\CC$ contains $\id_{\KK}$.
        For this it suffices to see that $\CC$ contains all maps $\Delta^n \to \KK$ and is closed under $\kappa$-small coproducts and pushouts (as every $\kappa$-small simplicial set can be build as an interated pushout of $\kappa$-small coproducts of simplices).
        Since $\Delta^n$  has a final object, the first part is clear.
        Thus it remains to prove the proposition in the cases where $\KK$ is a $ \kappa$-small set and $\KK = \Lambda^2_0$.
	Suppose first that $\KK$ is a $\kappa$-small set. Then the inclusions $i_k\colon \I{J}_k\into \I{J}$ for each $k\in \KK$ determine a pullback square
	\begin{equation*}
	\begin{tikzcd}
		\Under{\I{C}}{d}\arrow[d]\arrow[r, "(i_k^\ast)_{k\in \KK}"] & \prod_k \Under{\I{C}}{d_k}\arrow[d]\\
		\I{C}\arrow[r, "\diag"] & \iFun(\KK,\I{C}).
	\end{tikzcd}
	\end{equation*}
	By assumption, each of the categories $\Under{\I{C}}{d_k}$ admits an initial global section, hence Lemma~\ref{lem:initialObjectProduct} implies that the induced global section $1\to \prod_k \Under{\I{C}}{d_k}$ is initial as well. Phrased differently, the functor $\iFun(\KK,\I{C})\to\Univ$ that classifies the left fibration $\prod_k \Under{\I{C}}{d_k}\to \iFun(\KK,\Univ)$ is corepresented by the diagram $(\colim d_k)_{k\in \KK}\colon \KK\to\I{C}$. Since $\diag$ by assumption admits a left adjoint, we thus conclude that the left fibration $\Under{\I{C}}{d}\to\I{C}$ is classified by the functor corepresented by $\bigsqcup_k \colim d_k\colon 1\to \I{C}$, which implies that $d$ has a colimit in $\I{C}$.
	
	Let us now assume $\KK=\Lambda^2_0$, i.e.\ that $\I{J}$ is given by a pushout. Then there is an equivalence
	\begin{equation*}
	\begin{tikzcd}[column sep={5em,between origins}]
	\Under{\I{C}}{d}\arrow[rr, "\simeq"]\arrow[dr] && \lim_k \Under{\I{C}}{d_k} \arrow[dl]\\
	& \I{C}&
	\end{tikzcd}
	\end{equation*}
	of left fibrations over $\I{C}$, which together with Example~\ref{ex:externalLimitsColimits} implies that the functor $\map{\iFun(\I{J},\I{C})}(d,\diag(-))$ is given by the $\KK^{\op}$-indexed limit of functors $\map{\iFun(\I{J}_k,\I{C})}(d_k,\diag(-))$ in $\iFun(\I{C},\Univ)$. Since $\I{C}$ by assumption admits $\KK$-indexed colimits, its opposite $\I{C}^{\op}$ admits $\KK^{\op}$-indexed limits. Moreover, since each of the functors $\map{\iFun(\I{J}_k,\I{C})}(d_k,\diag(-))$ is contained in the essential image of the Yoneda embedding $\I{C}^{\op}\into\iFun(\I{C},\Univ)$, we conclude that $\map{\iFun(\I{J},\I{C})}(d,\diag(-))$ is corepresentable since the Yoneda embedding commutes with limits (Proposition~\ref{prop:limitsRepresentably}). Hence the diagram $d$ admits a colimit in $\I{C}$.
\end{proof}
By a similar argument as in the proof of Proposition~\ref{prop:decompositionColimits} one shows:
\begin{proposition}
	\label{prop:decompositionPreservationColimits}
	Let $\kappa$ be a regular cardinal, let $\KK$ be a $\kappa$-small $\infty$-category and let 
	\begin{equation*}
	\alpha\colon \KK\to \Cat(\BB),\quad k\mapsto \I{J}_k
	\end{equation*}
	be a diagram with colimit $\I{J}=\colim_k \I{J}_k$ in $\Cat(\BB)$. Let $\I{C}$ be a $\BB$-category that satisfies the conditions of Proposition~\ref{prop:decompositionColimits}, let $d\colon \I{J}\to\I{C}$ be a diagram and suppose that $f\colon\I{C}\to\I{D}$ is a functor in $\Cat(\BB)$ such that
	\begin{enumerate}
		\item for every $k \in \KK$ the functor $f$ preserves the colimit of the restricted diagram $d_k\colon \I{J}_k\to \I{C}$;
		\item $f$ preserves colimits indexed by $\kappa$-small constant $\BB$-categories.
	\end{enumerate}
	Then $f$ preserves the colimit of $d$.\qed
\end{proposition}

\begin{proof}[{Proof of Proposition~\ref{prop:ColimitsGroupoidsExternal}}]
Let $\I{J}$ be a $\BB$-category and let $d\colon A\to \iFun(\I{J},\I{C})$ be a diagram in context $A\in\BB$. We want to show that $d$ admits a colimit in $\I{C}$.
By making use of Remark~\ref{rem:etaleInvarianceLimits}, we may replace $\BB$ by $\Over{\BB}{A}$ and can thus assume that $A\simeq 1$ (see Remark~\ref{rem:etaleInvarianceReductionGlobalContext}).
Recall from \cite[Lemma 4.5.2 and the discussion following it]{martini2021} that we have a canoncial equivalence
    \[
        \I{J} \simeq \colim_{(\Delta^n \times \I{G})_{/ \I{J}}} \Delta^n \otimes \I{G}.
    \]
Furthermore it follows from Proposition~\ref{prop:characterisationFinalLimit} that a $\BB$-category $\I{C}$ has $\Delta^n\otimes \I{G}$-indexed colimits if and only if it has $\I{G}$-indexed colimits since $\Delta^n$ admits a final object. 
So if $\I{C}$ admits colimits indexed by constant $\BB$-categories and $\BB$-groupoids $\I{G}$, we may apply Proposition~\ref{prop:decompositionColimits} to conclude that $d$ has a colimit in $\I{C}$. The argument for the preservation of small colimits is analogous, by making use of Proposition~\ref{prop:decompositionPreservationColimits} instead.
\end{proof}

\section{Cocompleteness}
\label{chap:Cocompleteness}
This chapter is dedicated to a more global study of (co)limits in a $\BB$-category. More precisely, if $\I{U}$ is an internal class of $\BB$-categories (i.e.\ a full subcategory of $\ICat_{\BB}$, see Definition~\ref{def:internalClass}), we define and study what it means for a $\BB$-category $\I{C}$ to be \emph{$\I{U}$-(co)complete} and for a functor $f\colon\I{C}\to\I{D}$ between $\BB$-categories to be \emph{$\I{U}$-(co)continuous}. For the special case where $\I{U}=\ICat_{\BB}$, this will yield the correct internal analogue of the usual notion of cocompleteness and cocontinuity in (higher) category theory. One should note that this will be a strictly stronger notion than to simply admit all internal colimits that are indexed by small $\BB$-categories, cf.\ Example~\ref{ex:MotivicHomotopyCategoryColimits} below. We begin in \S~\ref{sec:internalClasses} by defining the notion of an internal class $\I{U}$ of $\BB$-categories, which is the internal analogue of a collection of $\infty$-categories. In \S~\ref{sec:UCocomplete}, we give the definition of $\I{U}$-cocompleteness and $\I{U}$-cocontinuity with respect to such an internal class and we recast some of the results from~\S~\ref{chap:limitsColimits} in this language. In~\S~\ref{sec:CatU}, we define the large $\BB$-category of $\I{U}$-cocomplete $\BB$-categories, and in \S~\ref{sec:cocompleteness} we study the special case where $\I{U}$ is the internal class of \emph{all} (small) $\BB$-categories. Finally, we briefly review the concept of \emph{proper} and \emph{smooth} maps between simplicial objects in $\BB$ in the context of this newly established framework.

\subsection{Internal classes}
\label{sec:internalClasses}
In this section we introduce the correct $\BB$-categorical analogue of \emph{classes} of $\infty$-categories:
\begin{definition}
	\label{def:internalClass}
	An \emph{internal class} of $\BB$-categories is a full subcategory $\I{U}\into\ICat_{\BB}$.
\end{definition}

\begin{remark}
	\label{def:whyInternalClasses}
	The reason why we define an internal class to be a full subcategory $\I{U} \into \ICat_{\BB}$ rather than just a subcategory $\UU\into \Cat(\BB)$ in the usual $\infty$-categorical sense is that when using internal classes as indexing classes for colimits, only the former notion leads to a theory of cocompleteness that is local in $\BB$ (cf.~\S~\ref{sec:localityPrinciple}), whereas the latter does not.
	For example, it is not reasonable to call a $\BB$-category \emph{cocomplete} even when it admits $\I{I}$-indexed colimits for every $\BB$-category $\I{I}$ (see Definition~\ref{def:limitColimit}), because it could still happen that there is a $\Over{\BB}{A}$-category $\I{J}$ (for some $A\in\BB$) such that $\pi_A^\ast\I{C}$ does not have all $\I{J}$-indexed colimits (see Example~\ref{ex:MotivicHomotopyCategoryColimits} below).
	Instead, on should ask that $\I{C}$ admits all colimits indexed by the maximal internal class $\ICat_{\BB}$ (Example~\ref{ex:SmallCategoriesInternalClass}), which precisely amounts to asking that every small diagram $\I{I} \rightarrow \pi_A^* \I{C}$ of $\BB_{/A}$-categories admits a colimit for every $A \in \BB$. In this way, the notion of cocompleteness is forced to be local.
\end{remark}

\begin{example}
	\label{ex:SmallCategoriesInternalClass}
	By Remark~\ref{rem:CatBB}, the (large) $\BB$-category $\ICat_{\BB}$ may be regarded as an internal class of large $\BB$-categories, so as a subcategory of the (very large) $\BB$-category $\ICat_{\BBB}$.
\end{example}

\begin{example}
	\label{ex:locallyConstantCategories}
	On account of the adjunction $\const\dashv \Gamma\colon\CatSS\leftrightarrows \Cat(\BBB)$, the transpose of the functor $\const\colon \CatS\to \Cat(\BB)\simeq\Gamma(\ICat_{\BB})$ defines a map $\const(\CatS)\to \ICat_{\BB}$ in $\Cat(\BBB)$. The essential image of this functor thus defines an internal class of $\BB$-categories that we denote by $\ILConst\into\ICat_{\BB}$ and that we refer to as the internal class of \emph{locally constant $\BB$-categories}. By construction, this is the full subcategory of $\ICat_{\BB}$ that is spanned by the constant $\BB$-categories, i.e.\ by those objects $1\to \ICat_{\BB}$ that correspond to categories of the form $\const(\CC)$ for some $\CC\in\CatS$. Thus, a $\Over{\BB}{A}$-category $\I{C}$ defines an object in $\ILConst$ in context $A\in\BB$ precisely if there is a cover $(s_i)_{i\in I}\colon\bigsqcup_{i\in I} A_i\onto A$ in $\BB$ such that $s_i^\ast\I{C}$ is a constant $\Over{\BB}{A_i}$-category for each $i\in I$.
\end{example}

\begin{example}
	\label{ex:internalClassGroupoids}
	On account of the inclusion $\Univ\into\ICat_{\BB}$ from Proposition~\ref{prop:internalCoreGroupoidification}, the universe $\Univ$ can be viewed as an internal class of $\BB$-categories.
\end{example}

\subsection{$\I{U}$-cocomplete $\BB$-categories}
\label{sec:UCocomplete}

In this section we define and study the condition on a $\BB$-category to admit colimits indexed by objects in an internal class $\I{U}$ of $\BB$-categories (see Definition~\ref{def:internalClass}).

\begin{definition}
    \label{def:cocomplete}
	Let $\I{U}$ be an internal class of $\BB$-categories. A $\BB$-category $\I{C}$ is said to be \emph{$\I{U}$-cocomplete} if $\pi_A^\ast\I{C}$ admits $\I{I}$-indexed colimits for every object $\I{I}\in\I{U}(A)$ and every $A\in\BB$. Similarly, if $f\colon \I{C}\to\I{D}$ is a functor between $\BB$-categories that are both $\I{U}$-cocomplete, we say that $f$ is \emph{$\I{U}$-cocontinuous} if $\pi_A^\ast f$ preserves $\I{I}$-indexed colimits for any $A\in\BB$ and any $\I{I}\in\I{U}(A)$. We simply say that a (large) $\BB$-category $\I{C}$ is \emph{cocomplete} if it is $\ICat_{\BB}$-cocomplete (when viewing $\ICat_{\BB}$ as an internal class of $\BBB$-categories), and we call a functor between cocomplete (large) $\BB$-categories \emph{cocontinuous} if it is $\ICat_{\BB}$-cocontinuous.
	
	Dually, we say that a $\BB$-category $\I{C}$ is \emph{$\I{U}$-complete} if $\pi_A^\ast\I{C}$ admits $\I{I}$-indexed limits for every object $\I{I}\in\I{U}(A)$ and every $A\in\BB$. If $f\colon \I{C}\to\I{D}$ is a functor between $\BB$-categories that are both $\I{U}$-complete, we say that $f$ is \emph{$\I{U}$-continuous} if $\pi_A^\ast f$ preserves $\I{I}$-indexed limits for any $A\in\BB$ and any $\I{I}\in\I{U}(A)$. We simply say that a (large) $\BB$-category $\I{C}$ is \emph{complete} if it is $\ICat_{\BB}$-complete, and we call a functor between complete (large) $\BB$-categories \emph{continuous} if it is $\ICat_{\BB}$-continuous.
\end{definition}

\begin{remark}
    \label{rem:UcocompletenessDuality}
    If $\I{U}$ is an internal class of $\BB$-categories, let $\op(\I{U})$ be the internal class that arises as the image of $\I{U}$ along the equivalence $(-)^{\op}\colon \ICat_{\BB}\simeq
    \ICat_{\BB}$ from Remark~\ref{rem:oppositeCategoriesInternalFunctor}. Then a $\BB$-category $\I{C}$ is $\I{U}$-complete if and only if $\I{C}^{\op}$ is $\op(\I{U})$-cocomplete, and a functor $f$ is $\I{U}$-continuous if and only if $f^{\op}$ is $\op(\I{U})$-cocontinuous. For this reason, we may dualise statements about $\op(\I{U})$-cocompleteness and $\op(\I{U})$-cocontinuity to obtain the corresponding statements about $\I{U}$-completeness and $\I{U}$-continuity.
\end{remark}

\begin{remark}[locality of $\I{U}$-cocompleteness and $\I{U}$-cocontinuity]
	\label{rem:CocompletenessLocalCondition}
	Since both the existence of (co)limits and the preservation of such (co)limits are local conditions (Remark~\ref{rem:limitsAreLocal} and Remark~\ref{rem:preservationLimitsIsLocal}), one finds that if $\bigsqcup_i A_i\onto 1$ is a cover in $\BB$, a $\BB$-category $\I{C}$ is $\I{U}$-(co)complete if and only if $\pi_{A_i}^\ast\I{C}$ is $\pi_{A_i}^\ast\I{U}$-(co)complete, and a functor $f\colon\I{C}\to\I{D}$ between $\I{U}$-(co)complete $\BB$-categories is $\I{U}$-(co)continuous if and only if $\pi_{A_i}^\ast(f)$ is $\pi_{A_i}^\ast\I{U}$-(co)continuous.
\end{remark}

\begin{remark}
	\label{rem:UcolimitsGenerators}
	Let $\I{U}$ be an internal class of $\BB$-categories that is spanned by a collection of objects $(\I{I}_i\in \ICat_{\BB}(A_i))_{i\in I}$ in $\ICat_{\BB}$ (in the sense of \S~\ref{sec:fullyFaithfulFunctors}). Then Remark~\ref{rem:limitsAreLocal} implies that a $\BB$-category $\I{C}$ is $\I{U}$-cocomplete whenever $\pi_{A_i}^\ast\I{C}$ has $\I{I}_i$-indexed colimits for all $i\in I$. Moreover, Remark~\ref{rem:preservationLimitsIsLocal} implies that a functor $f\colon \I{C}\to\I{D}$ between $\I{U}$-cocomplete $\BB$-categories is $\I{U}$-cocontinuous whenever $\pi_{A_i}^\ast f$ preserves  $\I{I}_i$-indexed colimits for all $i\in I$.
\end{remark}

Since by Corollary~\ref{cor:geometricMorphismAdjunction} the functor $\pi_A^\ast$ carries adjunctions in $\BB$ to adjunctions in $\Over{\BB}{A}$ for every $A\in\BB$, Proposition~\ref{prop:adjointsPreserveLimitsColimits} implies:
\begin{proposition}
	\label{prop:leftAdjointCocontinuous}
	 A left adjoint functor between $\I{U}$-cocomplete categories is $\I{U}$-cocontinuous, while a right adjoint between $\I{U}$-complete categories is $\I{U}$-continuous.\qed
\end{proposition}

Similarly, Proposition~\ref{prop:reflectiveSubcategoryLimitsColimits} shows:
\begin{proposition}
	\label{prop:reflectiveSubcategoryCocomplete}
	Suppose that $\I{U}$ is an internal class of $\BB$-categories and let $\I{D}$ be a $\I{U}$-cocomplete $\BB$-category. Then every reflective and every coreflective subcategory of $\I{D}$ is $\I{U}$-cocomplete as well.\qed
\end{proposition}

As we have a natural equivalence $\pi_A^\ast\iFun(-,-)\simeq \iFun[\Over{\BB}{A}](\pi_A^\ast(-),\pi_A^\ast(-))$ for every $A\in\BB$ (see Remark~\ref{rem:localityPrinciplePreservationStructure}), Propositions~\ref{prop:limitsFunctorCategories} and~\ref{prop:postcompositionPreservationLimits} show:
\begin{proposition}
	\label{prop:FunctorCategoryCocomplete}
	Let $f\colon \I{C}\to\I{D}$ be a $\I{U}$-cocontinuous functor between $\I{U}$-cocomplete $\BB$-categories. Then $f_\ast\colon\iFun(K,\I{C})\to\iFun(K,\I{D})$ is a $\I{U}$-cocontinuous functor between $\I{U}$-cocomplete $\BB$-categories for all $K\in\Simp\BB$. Moreover, for all $i\colon L\to K$ in $\Simp\BB$, the map $i^\ast\colon \iFun(K,\I{C})\to\iFun(L,\I{C})$ is $\I{U}$-cocontinuous as well.\qed
\end{proposition}

\begin{example}
	\label{ex:UniverseCompleteCocomplete}
	The universe $\Univ$ for small $\BB$-groupoids is complete and cocomplete since $\Univ$ admits small limits and colimits (Proposition~\ref{prop:colimitsUniverse} and Proposition~\ref{prop:limitsUniverse}) and since for any $A\in\BB$ there is a natural equivalence $\pi_A^\ast\Univ\simeq\Univ[\Over{\BB}{A}]$ (Remark~\ref{rem:localityPrinciplePreservationStructure}). By the same argument and Proposition~\ref{prop:universeEnlargementColimits}, the inclusion $i\colon\Univ[\BB]\into\Univ[\BBB]$ is continuous and cocontinuous.
\end{example}

Furthermore we conclude:
\begin{proposition}
	\label{prop:YonedaEmbeddingComplete}
	For any $\BB$-category $\I{C}$, the presheaf $\BB$-category $\IPSh(\I{C})$ is complete and cocomplete. If $\I{C}$ is $\I{U}$-complete for some internal class $\I{U}$, the Yoneda embedding $h_{\I{C}}\colon\I{C}\into\IPSh(\I{C})$ is $\I{U}$-continuous, and for every $c\colon A\to\I{C}$ the corepresentable copresheaf $\map{\I{C}}(c,-)\colon A\times\I{C}\to\Univ$ transposes to a $\pi_A^\ast\I{U}$-continuous functor $\pi_A^\ast\I{C}\to\Univ[\Over{\BB}{A}]$.
\end{proposition}
\begin{proof}
    The first claim is an immediate consequence of Example~\ref{ex:UniverseCompleteCocomplete} and Proposition~\ref{prop:FunctorCategoryCocomplete}.
    For the second claim, we have to see that $\pi_A^* h \colon \pi_A^* \I{C} \to \pi_A^* \IPSh(\I{C})$ preserves all limits indexed by the objects in $\I{U}(A)$. By Example~\ref{ex:representabilityLocalCondition}, we may identify $\pi_A^\ast h_{\I{C}}$ with $h_{\pi_A^\ast\I{C}}$, so that we may replace $\BB$ with $\Over{\BB}{A}$ and can therefore assume that $A\simeq 1$. Now the claim follows from Proposition~\ref{prop:limitsRepresentably}.
   Lastly, the third claim is a direct consequence of Corollary~\ref{cor:corepresentablePreserveLimits}.
\end{proof}

\begin{example}
	By combining Proposition~\ref{prop:YonedaEmbeddingComplete} with Proposition~\ref{prop:reflectiveSubcategoryCocomplete}, one finds that the $\BB$-category $\ICat_{\BB}$ is complete and cocomplete.
\end{example}

\subsection{The large $\BB$-category of $\I{U}$-cocomplete $\BB$-categories}
\label{sec:CatU}
In Proposition~\ref{prop:parametrisationSubcategories}, we show that in order to define a (non-full) subcategory of a $\BB$-category $\I{C}$, it suffices to specify a subobject of its object of morphisms $\I{C}_1$, i.e.\ an arbitrary family of maps in $\I{C}$. With this in mind, we define:
\begin{definition}
    \label{def:CatU}
    For any internal class $\I{U}$ of $\BB$-categories, the large $\BB$-category of $\I{U}$-cocomplete $\BB$-categories $\ICat_{\BB}^\cocont{\I{U}}$ is defined as the subcategory of $\ICat_{\BB}$ that is spanned by the $\pi_A^\ast\I{U}$-cocontinuous functors between $\pi_A^\ast\I{U}$-cocomplete $\Over{\BB}{A}$-categories for every $A\in\BB$. In the case where $\I{U}=\ICat_{\BB}$ (viewed as an internal class of large $\BB$-categories), we denote the resulting very large $\BB$-category by $\ICat_{\BBB}^\cc$.
\end{definition}

\begin{remark}[locality of $\ICat_{\BB}^{\cocont{\I{U}}}$]
    \label{rem:objectsCatU}
    The subobject of $(\ICat_{\BB})_1$ that is spanned by the $\pi_A^\ast\I{U}$-cocontinuous functors between $\pi_A^\ast\I{U}$-cocomplete $\BB$-categories is stable under equivalences and composition in the sense of Proposition~\ref{prop:classificationSubcategories}. As moreover $\I{U}$-cocompleteness and $\I{U}$-cocontinuity are local conditions (Remark~\ref{rem:CocompletenessLocalCondition}), we conclude (by the same argument as in Remark~\ref{rem:localityPrinciplePropositions}) that an object $A\to \ICat_{\BB}$ is contained in $\ICat_{\BB}^\cocont{\I{U}}$ if and only if the associated $\Over{\BB}{A}$-category is $\pi_A^\ast\I{U}$-complete, and a functor $f\colon\I{C}\to\I{D}$ between $\Over{\BB}{A}$-categories defines a morphism in $\ICat_{\BB}^\cocont{\I{U}}$ in context $A\in\BB$ precisely if it is a $\pi_A^\ast\I{U}$-cocontinuous functor between $\pi_A^\ast\I{U}$-cocomplete $\Over{\BB}{A}$-categories. In particular, if $\I{C}$ and $\I{D}$ are $\pi_A^\ast\I{U}$-cocomplete $\Over{\BB}{A}$-categories, a functor $\pi_A^\ast\I{C}\to\pi_A^\ast\I{D}$ is contained in the image of the monomorphism
    \begin{equation*}
        \map{\ICat_{\BB}^\cocont{\I{U}}}(\I{C},\I{D})\into \map{\ICat_{\BB}}(\I{C},\I{D})
    \end{equation*}
    if and only if it is $\pi_A^\ast\I{U}$-cocontinuous. Moreover, there is a canonical equivalence $\pi_A^\ast\ICat_{\BB}^\cocont{\I{U}}\simeq\ICat_{\Over{\BB}{A}}^\cocont{\pi_A^\ast\I{U}}$ for every $A\in\BB$ (by the same argument as in Remark~\ref{rem:localityPrincipleBaseChangeProposition}).
\end{remark}

\begin{definition}
	\label{def:UColimitPreservingFunctorCategory}
	Let $\I{U}$ be an internal class of $\BB$-categories. If $\I{C}$ and $\I{D}$ are $\I{U}$-cocomplete $\BB$-categories, we will denote by $\iFun^\cocont{\I{U}}(\I{C},\I{D})$ the full subcategory of $\iFun(\I{C},\I{D})$ that is spanned by those objects $A\to \iFun(\I{C},\I{D})$ in context $A\in\BB$ such that the corresponding functor $\pi_A^\ast\I{C}\to\pi_A^\ast\I{D}$ is $\pi_A^\ast\I{U}$-cocontinuous. In the case where $\I{U}=\ICat_{\BB}$, we will denote the associated large $\BB$-category by $\iFun^\cc(\I{C},\I{D})$.
\end{definition} 
\begin{remark}[locality of $\iFun^{\cocont{\I{U}}}(\I{C},\I{D})$]
	\label{rem:UColimitPreservingFunctorCategoryLocal}
	In the situation of Definition~\ref{def:UColimitPreservingFunctorCategory}, note that by combining Remark~\ref{rem:internalExternalCartesianClosureCat} and Corollary~\ref{cor:mappingGroupoidBifunctorCatB} with Remark~\ref{rem:objectsCatU}, we obtain an equivalence
    \begin{equation*}
        \map{\ICat_{\BB}^\cocont{\I{U}}}(\I{C},\I{D})\simeq \iFun^\cocont{\I{U}}(\I{C},\I{D})^\simeq.
    \end{equation*}
	As a consequence, Remark~\ref{rem:objectsCatU} implies that an object $A\to \iFun(\I{C},\I{D})$ is contained in $\iFun^\cocont{\I{U}}(\I{C},\I{D})$ if and only if the associated functor $\pi_A^\ast\I{C}\to\pi_A^\ast\I{D}$ is $\pi_A^\ast\I{U}$-cocontinuous, and we obtain a canonical equivalence $\pi_A^\ast\iFun^\cocont{\I{U}}(\I{C},\I{D})\simeq\iFun[\Over{\BB}{A}]^\cocont{\pi_A^\ast\I{U}}(\pi_A^\ast\I{C},\pi_A^\ast\I{D})$ for every $A\in\BB$ (see Remark~\ref{rem:localityPrincipleBaseChangeProposition}).
\end{remark}

The notion of $\I{U}$-cocompleteness and $\I{U}$-cocontinuity allows for some flexibility in the choice of internal class $\I{U}$. For example, Proposition~\ref{prop:characterisationFinalLimit} implies that whenever $\I{I}$ is a $\BB$-category that is contained in $\I{U}$ and $f\colon \I{I}\to\I{J}$ is a final functor, adjoining the $\BB$-category $\I{J}$ to $\I{U}$ does not affect whether a $\BB$-category is $\I{U}$-cocomplete or not. As it will be convenient later to impose certain stability conditions on an internal class, we define:
\begin{definition}
    \label{def:colimitClass}
    A \emph{colimit class} in $\BB$ is an internal class $\I{U}$ of $\BB$-categories that contains the final $\BB$-category $1$ and that is stable under final functors, i.e.\ satisfies the property that whenever $\I{I}\to\I{J}$ is a final functor in $\Over{\BB}{A}$ for some $A\in\BB$, then $\I{I}\in\I{U}(A)$ implies that $\I{J}\in \I{U}(A)$.
\end{definition}
For every internal class $\I{U}$ of $\BB$-categories one can construct a colimit class $\I{U}^{\colim}$ that is uniquely specified by the condition that $\I{U}^{\colim}$ is the minimal colimit class that contains $\I{U}$. Explicitly, this class is spanned by those $\Over{\BB}{A}$-categories $\I{J}$ that admit a final functor from either an object in $\I{U}(A)$ or the final $\Over{\BB}{A}$-category $1\in \Cat(\Over{\BB}{A})$. Thus, a $\Over{\BB}{A}$-category $\I{I}$ is contained in $\I{U}^{\colim}(A)$ if and only if there is a cover $(s_i)\colon \bigsqcup_i A_i\onto A$ in $\BB$ such that for each $i$ the $\Over{\BB}{{A_i}}$-category $s_i^\ast\I{I}$ admits a final functor from either an object in $\I{U}(A_i)$ or the final object $1\in\Cat(\Over{\BB}{A_i})$. By combining Proposition~\ref{prop:characterisationFinalLimit} with Remark~\ref{rem:UcolimitsGenerators}, we deduce that a $\BB$-category $\I{C}$ is $\I{U}$-cocomplete if and only if it is $\I{U}^{\colim}$-cocomplete, and similarly a functor $f\colon \I{C}\to\I{D}$ is $\I{U}$-cocontinuous if and only if it is $\I{U}^{\colim}$-cocontinuous. Together with the evident observation that  the above description of the objects in $\I{U}^{\colim}$ is local in $\BB$ (so that one obtains an equivalence $\pi_A^\ast(\I{U}^{\colim})\simeq(\pi_A^\ast\I{U})^{\colim}$ for all $A\in\BB$, cf.~\S~\ref{sec:localityPrinciple}), this implies that one has $\ICat_{\BB}^\cocont{\I{U}}\simeq \ICat_{\BB}^\cocont{\I{U}^{\colim}}$. Thus, for the sake of discussing colimits, we may therefore always assume that an internal class is a colimit class.

\subsection{Cocompleteness and cocontinuity}
\label{sec:cocompleteness}
In \S~\ref{sec:decompositionColimits}, we saw that every small internal colimit can be decomposed into colimits indexed by $\BB$-groupoids and by constant $\BB$-categories. In the terminology introduced in~\S~\ref{sec:UCocomplete}, this result can be formulated as follows:

\begin{proposition}
	\label{prop:CocompleteGroupoidsExternal}
	A large $\BB$-category $\I{C}$ is cocomplete if and only if it is both $\Univ$- and $\ILConst$-cocomplete, and a functor between cocomplete large $\BB$-categories is cocontinuous if and only if it is both $\Univ$- and $\ILConst$-cocontinuous.
\end{proposition}
\begin{proof}
	We show the case of cocompleteness, the case of cocontinuity is completely analogous.
	We need to show that for every $A\in\BB$ the $\Over{\BB}{A}$-category $\pi_A^\ast\I{C}$ admits colimits indexed by all small $\Over{\BB}{A}$-categories if it admits colimits indexed by all small $\Over{\BB}{A}$-groupoids and by the objects of $\ILConst(A)$.
	Note that by construction of $\ILConst$ (Example~\ref{ex:locallyConstantCategories}) and by the equivalence $\const_{\Over{\BB}{A}}\simeq\pi_A^\ast\const_{\BB}$ for every $A\in\BB$ (Remark~\ref{rem:localityPrinciplePreservationStructure}), we may identify $\pi_A^\ast\ILConst$  with the internal class of locally constant $\Over{\BB}{A}$-categories. Therefore, we may replace $\BB$ by $\Over{\BB}{A}$ and can thus assume that $A\simeq 1$. In this case, the result follows immediately from Proposition~\ref{prop:decompositionColimits} (since every constant $\BB$-category defines an object in $\ILConst(1)$).
\end{proof}
In light of Proposition~\ref{prop:CocompleteGroupoidsExternal}, it seems reasonable to investigate $\Univ$- and $\ILConst$-cocompleteness separately.
We begin with the case of $\BB$-groupoidal colimits. By combining Example~\ref{ex:groupoidalLimitsColimits} with Example~\ref{ex:groupoidalLimitsColimitsPreservation}, we find:
\begin{proposition}
	\label{prop:UcolimitsGroupoids}
	Let $S$ be a local class of maps in $\BB$ and let $\Univ[S]$ be the associated subuniverse (see \S~\ref{sec:universe}), where we view $\Univ[S]$ as an internal class of large $\BB$-categories. Then a large $\BB$-category $\I{C}$ is $\Univ[S]$-cocomplete if and only if the following two conditions are satisfied:
	\begin{enumerate}
	\item for every map $p\colon P\to A$ in $S$, the functor $p^\ast\colon \I{C}(A)\to\I{C}(P)$ admits a left adjoint $p_!$;
	\item for every pullback square
	\begin{equation*}
	\begin{tikzcd}
	Q\arrow[r, "t"]\arrow[d, "q"] & P\arrow[d, "p"]\\
	B\arrow[r, "s"] & A
	\end{tikzcd}
	\end{equation*}
	in $\BB$ in which $p$ and $q$ are contained in $S$, the natural map $q_!t^\ast\to s^\ast p_!$ is an equivalence.
	\end{enumerate}
	Furthermore, a functor $f\colon\I{C}\to\I{D}$ between (large) $\Univ[S]$-cocomplete $\BB$-categories is $\Univ[S]$-cocontinuous precisely if for every map $p\colon P\to A$ in $S$ the natural map $p_! f(P)\to f(A)p_!$ is an equivalence.\qed
\end{proposition}

\begin{example}
    \label{ex:localClassCocomplete}
    If $S$ is a local class in $\BB$, the associated subuniverse $\Univ[S]\into \Univ$ is closed under $\Univ[S]$-colimits (i.e.\ $\Univ[S]$ is $\Univ[S]$-cocomplete and the inclusion $\Univ[S]\into\Univ$ is $\Univ[S]$-cocontinuous) if and only if $S$ is stable under composition. For example, this is always the case when $S$ is the right class of a factorisation system in $\BB$.
\end{example}

\begin{example}
	\label{ex:StableFactorisationSystemReflectiveSubuniverseCocomplete}
	Recall from Example~\ref{ex:StableFactorisationSystemReflectiveSubuniverse} that every modality $(\LL,\RR)$ in $\BB$ (i.e.\ a factorisation system in wich $\LL$ is stable under base change in $\BB$) determines a reflective subcategory $\Univ[\RR]$ of $\Univ$. Conversely, if $\Univ[\RR]\into\Univ$ is an arbitrary reflective subcategory, then~\cite[Theorem~4.8]{Vergura2019} shows that the associated local class $\RR$ in $\BB$ arises from a modality as in Example~\ref{ex:StableFactorisationSystemReflectiveSubuniverse} precisely if $\RR$ is stable under composition, i.e.\ if $\Univ[\RR]\into\Univ$ is closed under $\Univ[\RR]$-colimits. Hence modalities in $\BB$ correspond precisely to those reflective subuniverses that are closed under self-indexed colimits in $\Univ$.
\end{example}

Let $\KK$ be a class of $\infty$-categories, i.e.\ a full subcategory of $\CatS$. As in example~\ref{ex:locallyConstantCategories} we obtain a functor $\KK \to \ICat_\BB$ by transposing the map $\const_{\BB}\colon \KK\into \CatS \to\Cat(\BB)$ across the adjunction $\const_{\BB}\dashv\Gamma_{\BB}$.
We denote the essential image of this functor by $\ILConst_{\KK}$.
By construction, for every $A\in\BB$ the internal class $\pi_A^\ast\ILConst_{\KK}$ is the full subcategory of $\ICat_{\Over{\BB}{A}}$ that is spanned by $\const_{\Over{\BB}{A}}(\II)$ for each $\II\in\KK$. Hence a $\BB_{/A}$-category $\I{C}$ defines an object in $\ILConst_{\KK}(A)$ if and only if there is a cover $(s_i)_i \colon \bigsqcup A_i \onto A$ such that $s_i^*\I{C}\simeq\const_{\Over{\BB}{A_i}}(\II_i)$ for some $\II_i\in\KK$.
Using Remark~\ref{rem:UcolimitsGenerators}, Examples~\ref{ex:externalLimitsColimits} and~\ref{ex:externalLimitsColimitsPreservation} now imply:
\begin{proposition}
	\label{prop:LConstCocomplete}
	If $\KK$ is a class of $\infty$-categories, a $\BB$-category $\I{C}$ is $\ILConst_{\KK}$-cocomplete if and only if
	for every $A\in\BB$ the $\infty$-category $\I{C}(A)$ admits colimits indexed by every object in $\KK$ and for every map $s\colon B\to A$ in $\BB$ the functor $s^\ast\colon \I{C}(A)\to\I{C}(B)$ preserves such colimits. Furthermore, a functor $f\colon\I{C}\to\I{D}$ between $\ILConst_\KK$-cocomplete $\BB$-categories is $\ILConst_\KK$-cocontinuous if and only if for all $A\in\BB$ the functor $f(A)$ preserves all colimits that are indexed by objects in $\KK$.\qed
\end{proposition}

In Construction~\ref{constr:PresCatsAsInternalCats} we define a functor $-\otimes\Univ\colon\RPr\to\Cat(\BBB)$. Its explicit formula and Proposition~\ref{prop:LConstCocomplete} now yield:
\begin{corollary}
	\label{cor:CatOfLConstKCocompleteCats}
	For every class of $\infty$-categories $\KK$ there is an equivalence $\ICat_{\BB}^{\cocont{\ILConst_{\KK}}}\simeq \CatS^{\cocont{\KK}}\otimes\Univ$ with respect to which the inclusion $\ICat_{\BB}^{\cocont{\ILConst_{\KK}}}\into\ICat_{\BB}$ is obtained by applying $-\otimes\Univ$ to the inclusion $\CatS^{\cocont{\KK}}\into \CatS$.\qed
\end{corollary}

By combining Propositions~\ref{prop:CocompleteGroupoidsExternal}, \ref{prop:UcolimitsGroupoids} and \ref{prop:LConstCocomplete} we now arrive at the following:

\begin{corollary}
    \label{cor:CocompleteExternalBeckChevalley}
    A $\BB$-category $\I{C}$ is cocomplete if and only if the following conditions are satisfied:
    \begin{enumerate}
        \item For every $A \in \BB$ the $\infty$-category $\I{C}(A)$ is cocomplete and for any $s \colon B \to A$ the functor $s^* \colon \I{C}(A) \rightarrow \I{C}(B)$ preserves colimits.
        \item For every map $p \colon P \rightarrow A$ in $\BB$ the functor $p^*$ has a left adjoint $p_!$ such that for every pullback square
	\begin{equation*}
	\begin{tikzcd}
	Q\arrow[r, "t"]\arrow[d, "q"] & P\arrow[d, "p"]\\
	B\arrow[r, "s"] & A
	\end{tikzcd}
	\end{equation*}
	the natural map $q_!t^\ast\to s^\ast p_!$ is an equivalence.
    \end{enumerate}
    Furthermore a functor $f \colon  \I{C} \rightarrow \I{D}$ of cocomplete $\BB$-categories is cocontinuous if and only if for every $A\in\BB$ the functor $f(A)$ preserves colimits, and for every map $p \colon P \rightarrow A$ in $\BB$ the natural map $p_! f(P) \rightarrow f(A)p_!$ is an equivalence.\qed
\end{corollary}

\begin{example}
	Let $\CC$ be a presentable $\infty$-category. 
	Then Corollary~\ref{cor:CocompleteExternalBeckChevalley} and its dual show that the $\BB$-category of Construction~\ref{constr:PresCatsAsInternalCats} is both complete and cocomplete.
	In fact $\CC \otimes {\Univ}$ will give rise to a \emph{presentable $\BB$-category}, which are defined to be suitable localisations of presheaf $\BB$-categories.
	We will pursue a detailed study of presentable $\BB$-categories in upcoming work.
\end{example}

\begin{remark}
    \label{rem:existenceSmallColimitsGenerators}
    Let $\CC$ be a small $\infty$-category such that $\BB$ is a left exact and accessible localisation of $\PSh(\CC)$, and let $L\colon \PSh(\CC)\to\BB$ be the localisation functor. Then in order to see that a $\BB$-category $\I{C}$ is cocomplete, it suffices to check the conditions of Corollary~\ref{cor:CocompleteExternalBeckChevalley} for objects in $\mathcal{C}$:
    Indeed, as the existence of colimits is a local condition (Remark~\ref{rem:limitsAreLocal}), one may assume without loss of generality that the object $A$ appearing in condition (1) and (2) of Corollary~\ref{prop:UcolimitsGroupoids} is of the form $L(a)$ for some $a\in \CC$. By furthermore using Remark~\ref{rem:existenceColimitsGenerators}, one can also assume that $B=L(b)$ and $s=L(s^\prime)$ for some $d\in\CC$ and some map $s^\prime\colon b \to a$ in $\CC$. Finally, provided that $\I{C}$ is $\ILConst$-cocomplete, Proposition~\ref{prop:decompositionColimits} allows us to further assume that $P=L(p)$ and $u=L(u^\prime)$ for some $p\in\CC$ and some map $u^\prime\colon p\to a$ in $\CC$. Together with Proposition~\ref{prop:CocompleteGroupoidsExternal}, these observations imply that $\I{C}$ is cocomplete if and only if
    \begin{enumerate}
        \item for every $a\in\CC$ the $\infty$-category $\I{C}(L(a))$ has small colimits, and for every $t\colon b\to a$ in $\CC$ the functor $L(t)^\ast\colon \I{C}(L(a))\to\I{C}(L(b))$ preserves small colimits;
        \item for every pullback square
        \begin{equation*}
            \begin{tikzcd}
                Q\arrow[r, "t"]\arrow[d, "v"] & p\arrow[d, "u"]\\
                b\arrow[r, "s"] & a
            \end{tikzcd}
        \end{equation*}
        in $\PSh(\CC)$ where $s\colon b\to a$ and $u\colon p\to a$ are maps in $\CC$, the functors $L(u)^\ast\colon \I{C}(L(a))\to\I{C}(L(p))$ and $L(v)^\ast\colon \I{C}(L(d))\to \I{C}(L(Q))$ admits left adjoints $L(u)_!$ and $L(v)_!$ such that the natural map $L(v)_!L(t)^\ast\to L(s)^\ast L(u)_!$ is an equivalence.
    \end{enumerate}
\end{remark}

\begin{example}
	Let $\XX$ be an $\infty$-topos and let $f_* \colon \XX \rightarrow \BB$ be a geometric morphism.
	We may consider the limit-preserving functor
	\[
		\XX_{/f^\ast(-)} \colon \BB^\op \xrightarrow{(f^*)^\op} \XX^\op \xrightarrow{ \XX_{/-}} \CatSS
	\]
	which defines a large $\BB$-category $\I{X}$.
	Clearly $\I{X}$ is $\ILConst$-cocomplete. 
	Furthermore, for every pullback square 
	\[
	\begin{tikzcd}
	Q\arrow[r, "t"]\arrow[d, "q"] & P\arrow[d, "p"]\\
	B\arrow[r, "S"] & A
	\end{tikzcd}
	\]
	in $\BB$, the lax square
	\[
		\begin{tikzcd}
		\XX_{/f^*(Q)} \arrow[d, "f^*(q)_!"] \arrow[dr, Rightarrow, shorten=4mm]& \XX_{/f^*(P)} \arrow[d, "f^*(p)_!"] \arrow[l, "f^*(t)^*"'] \\
		\XX_{/f^*(B)}                       & \XX_{/f^*(A)} \arrow[l, "f^*(s)^*"]                       
		\end{tikzcd}
	\]
	commutes since $f^*$ preserves pullbacks.
	Thus it follows from Corollary~\ref{cor:CocompleteExternalBeckChevalley} that $\I{X}$ is cocomplete.
	Dually one shows that $\I{X}$ is also complete.
	In fact $\I{X}$ will be an example of a \emph{$\BB$-topos}, i.e.\ a left exact localisation (in a suitable sense) of a presheaf $\BB$-category.
	We intend to make these ideas precise in future work.
\end{example}

\begin{example}
    \label{ex:U+LConstcocomplete}
    One may also combine Proposition~\ref{prop:UcolimitsGroupoids} and \ref{prop:LConstCocomplete} in a more general way. 
    Namely let $S$ be a local class of maps in $\BB$ and $\KK$ a class of $\infty$-categories, and consider the internal class $\left<S,\KK \right>$ generated by $\Univ_S$ and $\ILConst_\KK$ (i.e.\ the essential image of the functor $\Univ[S]\sqcup \ILConst_{\KK}\to\ICat_{\BB}$).
    Then Remark~\ref{rem:UcolimitsGenerators} shows that a $\BB$-category $\I{C}$ is $\left<S,\KK \right>$-cocomplete if and only if 
    \begin{enumerate}
    \item for every $A\in\BB$ the $\infty$-category $\I{C}(A)$ admits colimits indexed by objects in $\KK$, and for every map $s\colon B\to A$ in $\BB$ the transition functor $s^* \colon \I{C}(A) \rightarrow \I{C}(B)$ preserve these colimits;
    \item for every map $p \colon P \rightarrow A$ in $S$ the functor $p^*$ admits a left adjoint $p_!$ that is compatible with base change in the sense of Proposition~\ref{prop:UcolimitsGroupoids}.
    \end{enumerate}
\end{example}

\begin{example}
	\label{ex:MotivicHomotopyCategoryColimits}
	The notion of being cocomplete is strictly stronger than simply admitting small colimits.
	For a concrete counterexample, consider be the category of (topological) manifolds $\mathrm{Man}$.
	There is a functor
	\[
		\ISh = \Shv(-) \colon \mathrm{Man} \to \LPr
	\]
	that takes a manifold $M$ to the $\infty$-category of sheaves of spaces on $M $.
	This defines a limit-preserving functor
	\[
		\ISh \colon \PSh_{\SS}(\mathrm{Man})^\op \rightarrow \LPr
	\]
	via Kan extension and thus a $\PSh_{\SS}(\mathrm{Man})$-category that in particular admits all colimits indexed by constant $\PSh_{\SS}(\mathrm{Man})$-categories.
	Furthermore $\ISh$ has all colimits indexed by $\PSh_{\SS}(\mathrm{Man})$-groupoids:
	by Proposition~\ref{prop:decompositionColimits} it suffices to see this for representable $\PSh_{\SS}(\mathrm{Man})$-groupoids. 
	By Corollary~\ref{cor:restricedBeckChevalleyAdj}, we have to check that for any two manifolds $M$ and $N$
	the functor
	\[
		\pi_{M}^* \colon \Shv(N) \rightarrow \Shv(M \times N)	
	\]
	admits a left adjoint and for any map $\alpha \colon N' \rightarrow N$ the mate of the commutative square
	\[
	\begin{tikzcd}
\Shv(N) \arrow[r, "\pi_M^*"] \arrow[d, "\alpha^*"] & \Shv(M \times N) \arrow[d, "\alpha_X^*"] \\
\Shv(N') \arrow[r, "\pi_M^*"]                      & \Shv(M \times  N')                      
\end{tikzcd}
	\]
	is an equivalence.
	    Since the projections $M \times N \to N$ and $M \times N' \to N'$ are topological submersions, the left adjoint exists and the mate is an equivalence by the smooth base change isomorphism, see \cite[Lemma 3.25]{Volpe2021}. Therefore $\ISh$ admits small colimits.
	However, if $\ISh$ was cocomplete, it would follow that for \textit{any} continuous map $f \colon M \rightarrow N$ of manifolds, the pullback functor
	\[
		f^* \colon \Shv(N) \rightarrow \Shv(M)
	\]
	would have a left adjoint.
        This is certainly not the case. 
        For example if $Y$ is a point, the pullback $f^*$ is simply the stalk functor at the point determined by $f$, and in general stalk functors don't preserve infinite products.
	However if we let $\mathrm{Sub}$ denote the local class in $\PSh_{\SS}(\mathrm{Man})$ that is generated by the topolgocial submersions in $\mathrm{Man}$, the above arguments show that the $\PSh_{\SS}(\mathrm{Man})$-category $\ISh$ is in fact $\langle \mathrm{Sub}, \CatS \rangle$-cocomplete (see Example~\ref{ex:U+LConstcocomplete}).
\end{example}

\subsection{Smooth and proper maps}
While our strategy thus far was to fix an internal class $\I{U}$ of $\BB$-categories and study those $\BB$-categories that are $\I{U}$-(co)complete, we will now reverse this discussion. That is, we fix a $\BB$-category $\I{C}$ and consider the largest internal class $\I{U}$ with respect to which $\I{C}$ is $\I{U}$-(co)complete. For simplicity, we restrict our attention to internal classes of $\BB$-groupoids, i.e.\ to subuniverses in $\BB$.

\begin{definition}
	\label{def:smoothProper}
	For any $\BB$-category $\I{C}$, we let $\ISm_{\I{C}}$ be the largest subuniverse for which $\I{C}$ is $\ISm_{\I{C}}$-cocomplete. We say that a map $p\colon P\to A$ in $\BB$ is \emph{$\I{C}$-smooth} if $p$ is contained in $\ISm_{\I{C}}(A)$. Dually, we define $\IPrp_{\I{C}}$ to be the largest subuniverse for which $\I{C}$ is $\IPrp_{\I{C}}$-complete, and we say that a map $p\colon P\to A$ in $\BB$ is \emph{$\I{C}$-proper} if $p$ is contained in $\IPrp_{\I{C}}(A)$. 
\end{definition}
Explicitly,  $\ISm_{\I{C}}$ is the full subcategory of $\Univ$ that is spanned by those $\Over{\BB}{A}$-groupoids $\I{G}$ for which $\pi_A^\ast\I{C}$ admits $\I{G}$-indexed colimits. The dual description holds for $\IPrp_{\I{C}}$. 

\begin{remark}[locality of $\ISm_{\I{C}}$ and $\IPrp_{\I{C}}$]
	\label{rem:localityPrincipleSmoothProper}
	Since the existence of colimits is local (Remark~\ref{rem:limitsAreLocal}), it follows that for every $A\in\BB$ a $\Over{\BB}{A}$-groupoid $\I{G}$ is contained in $\ISm_{\I{C}}$ precisely if $\pi_A^\ast\I{C}$ admits $\I{G}$-indexed colimits (cf.~Remark~\ref{rem:localityPrinciplePropositions}). As a consequence (see Remark~\ref{rem:localityPrincipleBaseChangeProposition}), we obtain an equivalence $\pi_A^\ast\ISm_{\I{C}}\simeq\ISm_{\pi_A^\ast\I{C}}$ of subuniverses in $\Univ[\Over{\BB}{A}]$. A similar observation holds for $\IPrp_{\I{C}}$.
\end{remark}

By combining Proposition~\ref{prop:UcolimitsGroupoids} with Remark~\ref{rem:localityPrincipleSmoothProper}, we now find:
\begin{proposition}
	\label{prop:smoothMapsExplicitly}
	A map $p\colon P\to A$ in $\BB$ is $\I{C}$-smooth precisely if for every cartesian square
	\begin{equation*}
	\begin{tikzcd}
	Q\arrow[d, "q"]\arrow[r, "t"] & P\arrow[d, "p"]\\
	B\arrow[r, "s"] & A
	\end{tikzcd}
	\end{equation*}
	in $\BB$, both $p^\ast\colon \I{C}(A)\to\I{C}(P)$ and $q^\ast\colon\I{C}(B)\to\I{C}(Q)$ admit a left adjoint $p_!$ and $q_!$ such that the natural map $q_! t^\ast\to s^\ast p_!$ is an equivalence.
	
	Dually, $p$ is $\I{C}$-proper precisely if for every cartesian square as above, both $p^\ast$ and $q^\ast$ admit right adjoints $p_\ast$ and $q_\ast$ such that the canonical map $s^\ast p_\ast \to q_\ast t^\ast$ is an equivalence.\qed
\end{proposition}

The notion of a $\I{C}$-proper map in $\BB$ can be made more explicit in the case where $\I{C}$ arises as the subuniverse that is attached to the right class of a factorisation system in $\BB$:
\begin{proposition}
	\label{prop:FactorisationSystemProper}
	Suppose that $(\LL,\RR)$ is a factorisation system in $\BB$, and let $\Univ[\RR]$ be the subuniverse that corresponds to the local class $\RR$. Then a map $p\colon P\to A$ is $\Univ[\RR]$-proper precisely if for every pullback $q\colon Q\to B$ of $p$ along some map $s\colon B\to A$, base change along $q$ preserves the maps in $\LL$.
\end{proposition}
\begin{proof}
	Suppose first that $p$ satisfies this condition. Then $q_\ast\colon \Over{\BB}{Q}\to\Over{\BB}{B}$ restricts to a map $\Over{\RR}{Q}\to\Over{\RR}{B}$. In fact, an object in $\Over{\BB}{B}$ is contained in $\Over{\RR}{B}$ precisely if it is right orthogonal to the collection $\LL_B$ of maps in $\Over{\BB}{B}$ whose underlying map in $\BB$ is contained in $\LL$. Using the adjunction $q^\ast\dashv q_\ast$, the functor $q_\ast$ restricting to a map $\Over{\RR}{Q}\to\Over{\RR}{B}$ is equivalent to $q^\ast$ carrying maps in $\LL_B$ to maps in $\LL_Q$, which precisely means that base change along $q$ preserves the maps in $\LL$. On account of Proposition~\ref{prop:limitsFullyFaithfulFunctor}, this already implies that $p$ is $\Univ[\RR]$-proper. Conversely, suppose that $p$ is $\Univ[\RR]$-proper. Then $q$ is $\Univ[\RR]$-proper as well, hence by replacing $q$ with $p$ it suffices to show that base change along $p$ preserves maps in $\LL$. Since $\RR$ is the right class of a factorisation system, the transition map $s^\ast\colon \Over{\RR}{A}\to\Over{\RR}{B}$ admits a left adjoint $s_!$ for \emph{every} map $s\colon B\to A$ in $\BB$. Explicitly, this left adjoint is computed by the composition
	\begin{equation*}
		\Over{\RR}{B}\into \Over{\BB}{B}\xrightarrow{s_!} \Over{\BB}{A}\xrightarrow{\Over{L}{A}} \Over{\RR}{A}
	\end{equation*}
	in which $\Over{L}{A}$ is the reflection functor, i.e.\ the left adjoint of the inclusion. Therefore, given a pullback as in Proposition~\ref{prop:smoothMapsExplicitly}, the mate transformation $s^\ast p_\ast \to q_\ast t^\ast$ being an equivalence implies that the map $t_! q^\ast\to p^\ast s_!$ is an equivalence as well. Unwinding the definitions, this is tantamount to base change along $p$ preserving maps in $\LL$ (see e.g.\ the argument in the proof of Proposition~\ref{prop:YonedaExtensionFibrations}).
\end{proof}

\begin{example}
	In~\cite[\S~4.4]{martini2021}, we defined the notion of a proper map between simplicial objects in $\BB$. Let $\ILFib$ be the large $\Simp\BB$-category that is associated with the sheaf $\LFib$ of left fibrations on $\Simp\BB$ (see~\cite[\S~4.1]{martini2021}). That is, $\ILFib$ is the subuniverse of $\Univ[\Simp\BB]$ that is determined by the right class of the factorisation system between initial maps and left fibrations in $\Simp\BB$. From this point of view, Proposition~\ref{prop:FactorisationSystemProper} shows that a map between simplicial objects is proper precisely if it is $\ILFib$-proper in the sense as described above. 
	
	Dually, a map in $\Simp\BB$ is \emph{smooth} (in the sense of~\cite[\S~4.4]{martini2021}) if and only if it is $\ILFib$-smooth in the above sense. To see this, let us first suppose that $p\colon P\to C$ is a smooth map in $\Simp\BB$. As the notion of smoothness is stable under base change in $\Simp\BB$, it suffices to show that for every pullback square
	\begin{equation*}
	\begin{tikzcd}
	Q\arrow[d, "q"]\arrow[r, "g"] & P\arrow[d, "p"]\\
	D\arrow[r, "f"] & C
	\end{tikzcd}
	\end{equation*}
	in $\Simp\BB$, the canonical map $q_! g^\ast\to f^\ast p_!$ is an equivalence. Since equivalences between left fibrations are detected fibrewise~\cite[Proposition~4.1.18]{martini2021} and by functoriality of the mate construction, we can assume $D\simeq A$ for some $A\in\BB$. Moreover, by factoring $f$ into a final map followed by a right fibration, we may assume that $f$ is either a right fibration or a final map whose domain is contained in $\BB$. In the first case, right fibrations being proper and proposition~\ref{prop:FactorisationSystemProper} immediately imply that the map $q_! g^\ast\to f^\ast p_!$ is an equivalence (as this is equivalent to the condition that $p^\ast f_\ast \to g_\ast q^\ast$ is an equivalence). The second case, on the other hand, follows from the argument in the proof of~\cite[Proposition~4.4.10]{martini2021}. Conversely, suppose that $p$ is $\ILFib$-smooth. Again, since $\ILFib$-smoothness is stable under base change in $\Simp\BB$, we only need to show that whenever we have a pullback square as the one above, then $g$ is final whenever $f$ is final. Using that the class of $\ILFib$-smooth maps in $\Simp\BB$ is local (by the correspondence explained in Proposition~\ref{prop:classificationSubuniverses}) and~\cite[Lemma~4.1.2]{martini2021}, we may assume that $f$ is given by the inclusion $d^{\{n\}}\colon A\into\Delta^n\otimes A$ for some $A\in\BB$ and some $n\geq 0$ and in particular that $D$ is contained in $\BB$. By~\cite[Proposition~4.4.3]{martini2021} (noting that its proof does not require the base to be a $\BB$-category) it will be enough to show that for every left fibration $Z\to P$ the map $\colim_{\Delta^{\op}}f^\ast Z\to \colim_{\Delta^{\op}}Z$ is an equivalence. Now by assumption on $p$ to be $\ILFib$-smooth, both vertical maps in the pullback square
	\begin{equation*}
	\begin{tikzcd}
	f^\ast Z\arrow[d, "f^\ast j"]\arrow[r] & Z\arrow[d, "j"]\\
	f^\ast p_! Z\arrow[r] & p_!Z
	\end{tikzcd}
	\end{equation*}
	are initial. Moreover, since $f$ is final and $D$ is contained in $\BB$, the lower horizontal map is final and $f^\ast p_! Z$ is contained in $\BB$ as well. By applying $\colim_{\Delta^{\op}}$ to this diagram and using the two out of three property of equivalences in $\BB$, the claim follows. Hence $p$ is smooth.
\end{example}

\section{Kan extensions}
\label{Chap:KanExtensions}
The goal of this chapter is to develop the theory of Kan extensions of functors between $\BB$-categories. The main theorem about the existence of Kan extensions will be discussed in \S~\ref{sec:KanExtensions}, but its proof requires a few preliminary steps. We begin in \S~\ref{sec:coYoneda} by discussing the \emph{co-Yoneda lemma}, which states that every presheaf can be obtained as the colimit of its Grothendieck construction. Secondly, \S~\ref{sec:UAccessible} contains a discussion of what we call \emph{$\I{U}$-small} presheaves, those that can be obtained as $\I{U}$-colimits of representables.

\subsection{The co-Yoneda lemma}
\label{sec:coYoneda}
If $\I{C}$ is a $\BB$-category and if $F\colon \I{C}^{\op}\to\Univ$ is a presheaf on $\I{C}$, Yoneda's lemma (Theorem~\ref{thm:YonedaLemma}) and the straightening/unstraightening equivalence (Theorem~\ref{thm:straightening}) allow us to identify the pullback $p\colon\Over{\I{C}}{F}\to\I{C}$ of the right fibration $(\pi_F)_!\colon \Over{\IPSh(\I{C})}{F}\to\IPSh(\I{C})$ along the Yoneda embedding $h\colon \I{C}\into\IPSh(\I{C})$ with the right fibration $\int F\to\I{C}$ that is classified by $F$. Let us denote by $\Over{h}{F}\colon\Over{\I{C}}{F}\into\Over{\IPSh(\I{C})}{F}$ the induced embedding. Since $\Over{\IPSh(\I{C})}{F}$ admits a final object $\id_{F}\colon 1\to\Over{\IPSh(\I{C})}{F}$, Proposition~\ref{prop:characterisationFinalLimit} implies that the functor $(\pi_F)_!$ admits a colimit that is given by $F$ itself (cf.\ Example~\ref{ex:objectAsDiagramLimit}). Using Remark~\ref{rem:functorialityColimitIndexingCategory}, the functor $\Over{h}{F}$ therefore induces a canonical map
\begin{equation*}
\colim hp\simeq\colim (\pi_F)_!\Over{h}{F}\to \colim (\pi_F)_!\simeq F
\end{equation*}
of presheaves on $\I{C}$.
Our goal in this section is to prove that this map is an equivalence:
\begin{proposition}
	\label{prop:coYoneda}
	Let $\I{C}$ be a $\BB$-category, let $F\colon 1\to\IPSh(\I{C})$ be a presheaf on $\I{C}$ and let $p\colon \Over{C}{F}\to\I{C}$ be the associated right fibration. Then the map $\colim hp\to F$	is an equivalence.
\end{proposition}
\begin{remark}
    The analogue of Proposition~\ref{prop:coYoneda} for $\infty$-categories can be found in~\cite[Lemma~5.1.5.3]{htt}.
\end{remark}

The proof of Proposition~\ref{prop:coYoneda} requires a few preparations. We begin with the following special case:
\begin{proposition}
	\label{prop:colimitYonedaEmbedding}
	For any $\BB$-category $\I{C}$, the colimit of the Yoneda embedding $h\colon \I{C}\into\IPSh(\I{C})$ is given by the final object $1_{\IPSh(\I{C})}$.
\end{proposition}
\begin{proof}
	Using Proposition~\ref{prop:universeEnlargementColimits} in conjunction with Proposition~\ref{prop:limitsFullyFaithfulFunctor}, it suffices to show that the  colimit of $\hat h\colon \I{C}\into \IPSh[\BBB](\I{C})=\iFun(\I{C}^\op,\Univ[\BBB])$ is given by $1_{\IPSh[\BBB](\I{C})}$. 
	On account of the commutative diagram
	\begin{equation*}
	\begin{tikzcd}
		\IPSh[\BBB](\I{C})\arrow[r, "\diag"]\arrow[d, "\pr_0^\ast"] & \iFun(\I{C},\IPSh[\BBB](\I{C}))\arrow[d, "\simeq"]\\
		\iFun(\I{C}^{\op}\times \IPSh(\I{C}),\Univ[\BBB])\arrow[r, "(\id\times h)^\ast"] & \iFun(\I{C}^{\op}\times\I{C},\Univ[\BBB])
	\end{tikzcd}
	\end{equation*}
	and Corollary~\ref{cor:existenceYonedaExtension}, the colimit of $\hat h$ in $\IPSh[\BBB](\I{C})$ is equivalent to $(\pr_0)_!(\id\times h)_!(i\map{\I{C}})$, where $i\colon \Univ\into\Univ[\BBB]$ denotes the inclusion. On the other hand, Yoneda's lemma provides a commutative square
	\begin{equation*}
			\begin{tikzcd}
			\Tw(\I{C})\arrow[d]\arrow[r, "j"] & \int \ev \arrow[d]\\
			\I{C}^{\op}\times\I{C}\arrow[r, "\id\times h"] & \I{C}^{\op}\times\IPSh(\I{C})
			\end{tikzcd}
	\end{equation*}
	in which $j$ is initial (see the proof of~\cite[Theorem~4.7.8]{martini2021}), which together with Proposition~\ref{prop:YonedaExtensionFibrations} implies that $(\id\times h)_!(i\map{\I{C}})$ is given by the functor $i\circ\ev$. Note that by postcomposing $\pr_0^\ast$ with the equivalence $\iFun(\I{C}^{\op}\times\IPSh(\I{C}),\Univ[\BBB])\simeq\iFun(\IPSh(\I{C}),\IPSh[\BBB](\I{C}))$, we recover the diagonal functor $\diag\colon \IPSh[\BBB](\I{C})\to \iFun(\IPSh(\I{C}),\IPSh[\BBB](\I{C}))$. As this equivalence furthermore transforms the composition $i\circ\ev$ into the inclusion $i_\ast\colon\IPSh(\I{C})\into\IPSh[\BBB](\I{C})$, we conclude that the colimit of $\hat h$ is equivalent to the colimit of $i_\ast$. Since $1_{\IPSh(\I{C})}$ is a final object, the result thus follows from Proposition~\ref{prop:characterisationFinalLimit}, together with Example~\ref{ex:objectAsDiagramLimit}.
\end{proof}
\begin{remark}
	\label{rem:twistedArrowProjectionInitial}
	In the situation of Proposition~\ref{prop:colimitYonedaEmbedding}, Proposition~\ref{prop:YonedaExtensionFibrations} implies that the colimit of the Yoneda embedding $h\colon \I{C}\into\IPSh(\I{C})$ classifies the left fibration $q\colon \I{Q}\to\I{C}^{\op}$ that is defined by the unique commutative square
	\begin{equation*}
	\begin{tikzcd}
	\Tw(\I{C})\arrow[d, "p"]\arrow[r, "i"]& \I{Q}\arrow[d, "q"]\\
	\I{C}^{\op}\times\I{C}\arrow[r, "\pr_0"] & \I{C}^{\op}
	\end{tikzcd}
	\end{equation*}
	in which $i$ is initial. By Proposition~\ref{prop:colimitYonedaEmbedding}, the map $q$ is an equivalence, hence we conclude that the projection $\pr_0 p\colon\Tw(\I{C})\to\I{C}^{\op}$ must be initial.
\end{remark}

\begin{lemma}
	\label{lem:PshSlice}
	Let $\I{C}$ be a $\BB$-category and let $F\colon\I{C}^{\op}\to\Univ$ be a presheaf on $\I{C}$. Then there is a canonical equivalence $\IPSh(\Over{\I{C}}{F})\simeq \Over{\IPSh(\I{C})}{F}$ that fits into a commutative diagram
	\begin{equation*}
	\begin{tikzcd}
	\Over{\I{C}}{F}\arrow[dr, "\Over{(h_{\I{C}})}{F}"] \arrow[d, hookrightarrow, "h_{(\Over{\I{C}}{F})}"]& \\
	\IPSh(\Over{\I{C}}{F})\arrow[r, "\simeq"] & \Over{\IPSh(\I{C})}{F}  
	\end{tikzcd}
	\end{equation*}
\end{lemma}
\begin{proof}
	Let $p\colon \Over{\I{C}}{F}\to\I{C}$ be the projection, and let $p_!\colon \IPSh(\Over{\I{C}}{F})\to\IPSh(\I{C})$ be the left adjoint of the precomposition functor $p^\ast$. By Corollary~\ref{cor:existenceYonedaExtension}, there is an equivalence $p_!h_{(\Over{\I{C}}{F})}\simeq h_{\I{C}} p$, hence it suffices to show that $p_!$ factors through $(\pi_F)_!\colon \Over{\IPSh(\I{C})}{F}\to\IPSh(\I{C})$ via an equivalence. By construction of $p_!$, this functor sends the final object $1_{\IPSh(\I{C})}$ to $F$, hence we obtain a lifting problem
	\begin{equation*}
	\begin{tikzcd}
		1\arrow[r, "F"]\arrow[d, "1_{\IPSh(\I{C})}"] & \Over{\IPSh(\I{C})}{F}\arrow[d, "(\pi_F)_!"]\\
		\IPSh(\Over{\I{C}}{F})\arrow[r, "p_!"] \arrow[ur, dashed]& \IPSh(\I{C})
	\end{tikzcd}
	\end{equation*}
	in which $F$ and $1_{\IPSh(\I{C})}$ define final maps and $(\pi_F)_!$ is a right fibration. On account of the factorisation system between final maps and right fibrations, the dashed arrow exists and has to be final as well. To complete the proof, it therefore suffices to show that it is also a right fibration, which follows once we verify that $p_!$ is a right fibration. By Proposition~\ref{prop:YonedaExtensionFibrations}, this map evaluates at any $A\in\BB$ to the the functor $\RFib(A\times\Over{\I{C}}{F})\to \RFib(A\times\I{C})$ that is given by restricting the right fibration $\Over{\Cat(\BB)}{A\times \Over{\I{C}}{F}}\to\Over{\Cat(\BB)}{A\times \I{C}}$ of $\infty$-categories. Since the canonical square
	\begin{equation*}
	\begin{tikzcd}
	\RFib(A\times\Over{\I{C}}{F})\arrow[d]\arrow[r, hookrightarrow] & \Over{\Cat(\BB)}{A\times\Over{\I{C}}{F}}\arrow[d]\\
	\RFib(A\times\I{C})\arrow[r, hookrightarrow] & \Over{\Cat(\BB)}{A\times\I{C}}
	\end{tikzcd}
	\end{equation*}
	is a pullback, it thus follows that $p_!$ is sectionwise a right fibration and must therefore be a right fibration itself.
\end{proof}

\begin{proof}[{Proof of Proposition~\ref{prop:coYoneda}}]
	The map $\colim hp \to F$
	is determined by the cocone under $hp\simeq (\pi_F)_!\Over{h}{F}$ that arises as the image of the colimit cocone $(\pi_F)_!\to \diag(F)$ along the functor
	\begin{equation*}
		(\Over{h}{F})^\ast\colon 	\Under{\IPSh(\I{C})}{(\pi_F)_!}\to \Under{\IPSh(\I{C})}{hp}.
	\end{equation*}
	By making use of the equivalence $\phi\colon \Over{\IPSh(\I{C})}{F}\simeq \IPSh(\Over{\I{C}}{F})$ from Lemma~\ref{lem:PshSlice}, we now obtain a commutative square
	\begin{equation*}
	\begin{tikzcd}			
		 	\Under{\IPSh(\Over{\I{C}}{F})}{\phi}\arrow[r, "(\Over{h}{F})^\ast"]\arrow[d, "(p_!)_\ast"] & \Under{\IPSh(\Over{\I{C}}{F})}{h_{\Over{\I{C}}{F}}}\arrow[d, "(p_!)_\ast"]\\
			\Under{\IPSh(\I{C})}{(\pi_F)_!}\arrow[r, "(\Over{h}{F})^\ast"] & \Under{\IPSh(\I{C})}{hp}.
	\end{tikzcd}
	\end{equation*}
	As $p_!$ is a left adjoint and therefore preserves colimits, we may thus replace $\I{C}$ by $\Over{\I{C}}{F}$ and can therefore assume without loss of generality $F=1_{\IPSh(\I{C})}$, in which case the desired result follows from Proposition~\ref{prop:colimitYonedaEmbedding}.
\end{proof}

\subsection{$\I{U}$-small presheaves}
\label{sec:UAccessible}

In this section we study those subcategories of the $\BB$-category $\IPSh(\I{C})$ of presheaves on a $\BB$-category $\I{C}$ that are spanned by $\I{U}$-colimits of representable presheaves for an arbitrary internal class $\I{U}$ of $\BB$-categories. 

\begin{definition}
	\label{def:smallPresheaf}
	Let $\I{C}$ be a $\BB$-category and let $\I{U}$ be an internal class of $\BB$-categories. We say that a presheaf $F\colon A\to \IPSh[\BB](\I{C})$ in context $A\in\BB$ is \emph{$\I{U}$-small} if $\Over{\I{C}}{F}$ is contained in $\I{U}^{\colim}(A)$ (see the discussion after Definition~\ref{def:colimitClass}). We denote by $\Sml_{\BB}^{\I{U}}(\I{C})$ the full subcategory of $\IPSh(\I{C})$ that is spanned by the $\I{U}$-small presheaves.
\end{definition}

\begin{remark}[locality of $\I{U}$-small presheaves]
	\label{rem:localitySmall}
	The property of a presheaf $F\colon A\to\IPSh(\I{C})$ to be $\I{U}$-small is local in $\BB$. That is, for every cover $(s_i)\colon\bigsqcup_i A_i\onto A$ in $\BB$, the presheaf $F$ is $\I{U}$-small if and only if $s_i^\ast(F)$ is $\I{U}$-small. This follows immediately from the fact that since $\I{U}^{\colim}(-)$ is a subsheaf of $\Cat(\Over{\BB}{-})$, the property to be contained in $\I{U}^{\colim}(A)$ can be checked locally. As a consequence (see Remark~\ref{rem:localityPrinciplePropositions}), a presheaf $F$ is contained in $\Sml_{\BB}(\I{C})$ if and only if $F$ is $\I{U}$-small. From this observation and Remark~\ref{rem:etaleInvarianceSmallPresheaves} below, it furthermore follows (by the argument in Remark~\ref{rem:localityPrincipleBaseChangeProposition}) that there is a natural equivalence
	\begin{equation*}
	\Sml_{\Over{\BB}{A}}^{\pi_A^\ast\I{U}}(\pi_A^\ast\I{C})\simeq\pi_A^\ast\Sml_{\BB}^{\I{U}}(\I{C}).
	\end{equation*}
	for every $A\in\BB$.
\end{remark}

\begin{remark}[\'etale transposition invariance of $\I{U}$-small presheaves]
	\label{rem:etaleInvarianceSmallPresheaves}
	Since for every $A\in\BB$ we have an equivalence $\pi_A^\ast(\I{U}^{\colim})\simeq(\pi_A^\ast\I{U})^{\colim}$ (see the discussion following Definition~\ref{def:colimitClass}) and on account of Remark~\ref{rem:baseChangeSlice}, it follows that a presheaf $F\colon A\to\IPSh(\I{C})$ is $\I{U}$-small if and only if its transpose $\hat F\colon 1_{\Over{\BB}{A}}\to\IPSh[\Over{\BB}{A}](\pi_A^\ast\I{C})$ is $\pi_A^\ast\I{U}$-small.
\end{remark}

\begin{remark}
	For the special case where $\BB\simeq\SS$ and where $\I{U}$ is the class of $\kappa$-filtered $\infty$-categories for some regular cardinal $\kappa$, the $\infty$-category of $\I{U}$-small presheaves on a small $\infty$-category is precisely its ind-completion by $\kappa$-filtered colimits in the sense of~\cite[\S~5.3.5]{htt}. In general, however, the $\BB$-category $\Sml_{\BB}^{\I{U}}(\I{C})$ need not be a free cocompletion, see \S~\ref{sec:freeCocompletion} below.
\end{remark}
\begin{example}
	\label{ex:representablePresheafSmall}
	For any internal class $\I{U}$ of $\BB$-categories and for any $\BB$-category $\I{C}$, the presheaf represented by an object $c$ in $\I{C}$ in context $A\in\BB$ is $\I{U}$-small: the canonical section $\id_c\colon A\to \Over{\I{C}}{c}$ provides a final map from an object contained in $\I{U}^{\colim}(A)$, which implies that $\Over{\I{C}}{c}$ defines an object of $\I{U}^{\colim}$ as well.
	By making use of~\cite[Proposition~3.9.4]{martini2021}, the Yoneda embedding $h\colon \I{C}\into\IPSh(\I{C})$ thus factors through the inclusion $\Sml_{\BB}^{\I{U}}(\I{C})\into \IPSh(\I{C})$.
\end{example}

\begin{proposition}
	\label{prop:colimitAccessiblePresheaf}
	For any $\BB$-category $\I{C}$ and any internal class $\I{U}$ of $\BB$-categories, the $\BB$-category $\Sml_{\BB}^{\I{U}}(\I{C})$ is closed under $\I{U}$-colimits of representables in $\IPSh(\I{C})$. More precisely, for any object $A\to \I{U}$ in context $A\in\BB$ that corresponds to a $\Over{\BB}{A}$-category $\I{I}$, the colimit functor $\colim \colon \iFun[\Over{\BB}{A}](\I{I},\pi_A^\ast\IPSh(\I{C}))\to\pi_A^\ast\IPSh(\I{C})$ restricts to a functor
	\begin{equation*}
	\colim\colon \iFun[\Over{\BB}{A}](\I{I},\pi_A^\ast\I{C})\to \pi_A^\ast\Sml_{\BB}^{\I{U}}(\I{C}).
	\end{equation*}
\end{proposition}
\begin{proof}
	By using Example~\ref{ex:representabilityLocalCondition} and Remark~\ref{rem:localitySmall}, we may replace $\BB$ by $\Over{\BB}{A}$, so that it will be enough to show that for any diagram $d\colon B\to \iFun(\I{I},\I{C})$ in context $B\in\BB$ the colimit $\colim hd\colon B\to \IPSh[\BBB](\I{C})$ is a $\I{U}$-small presheaf on $\I{C}$. By the same argument and Remark~\ref{rem:etaleInvarianceLimits}, we may again replace $\BB$ with $\Over{\BB}{B}$, so that we can also reduce to $B\simeq 1$. Let $pi\colon \I{I}\to\I{P}\to\I{C}$ be the factorisation of $d$ into a final functor and a right fibration. By Proposition~\ref{prop:characterisationFinalLimit} we find $\colim hd\simeq\colim hp$, hence Proposition~\ref{prop:coYoneda} implies $\I{P}\simeq \Over{\I{C}}{\colim hd}$. Since $i$ is a final functor into $\I{P}$ from the $\BB$-category $\I{I}\in\I{U}(1)$, this shows that $\colim hd$ is $\I{U}$-small.
\end{proof}

We finish this section by showing that for any $\BB$-category $\I{C}$, the functor $h\colon \I{C}\into \Sml_{\BB}^{\I{U}}(\I{C})$ that is induced by the Yoneda embedding has a left adjoint whenever $\I{C}$ is $\I{U}$-cocomplete.

\begin{proposition}
	\label{prop:leftAdjointYonedaEmbedding}
	Let $\I{U}$ be an internal class of $\BB$-categories. If $\I{C}$ is a $\I{U}$-cocomplete $\BB$-category, the functor $h\colon\I{C}\into\Sml_{\BB}^{\I{U}}(\I{C})$ that is induced by the Yoneda embedding admits a left adjoint $L\colon \Sml_{\BB}^{\I{U}}(\I{C})\to\I{C}$. 
\end{proposition}
\begin{proof}
	As $\I{C}$ being $\I{U}$-cocomplete is equivalent to $\I{C}$ being $\I{U}^{\colim}$-cocomplete, we may assume without loss of generality that $\I{U}$ is already a colimit class.
	Let $F$ be an object in $\Sml_{\BB}^{\I{U}}(\I{C})$ in context $A\in\BB$. On account of Proposition~\ref{cor:representabilityCriterionAdjoint}, it suffices to show that the copresheaf $\map{\Sml_{\BB}^{\I{U}}(\I{C})}(F, h(-))$ is corepresentable by an object in $\I{C}$.
	Using Example~\ref{ex:representabilityLocalCondition} together with Remark~\ref{rem:localitySmall}, we may replace $\BB$ with $\Over{\BB}{A}$ and can therefore assume without loss of generality that $F$ is a $\I{U}$-small presheaf in context $1\in\BB$ (see Remark~\ref{rem:etaleInvarianceReductionGlobalContext}). In this case, we have $\Over{\I{C}}{F}\in\I{U}(1)$, where $p\colon \Over{\I{C}}{F}\to\I{C}$ is the right fibration that is classified by $F$. Now Proposition~\ref{prop:coYoneda} and Proposition~\ref{prop:characterisationFinalLimit} give rise to an equivalence $F\simeq\colim hp$. Thus, one obtains a chain of equivalences
	\begin{align*}
		\map{\Sml_{\BB}^{\I{U}}(\I{C})}(F, h(-)) &\simeq \map{\IPSh(\I{C})}(\colim hp, h(-))\\
		&\simeq \map{\iFun(\Over{\I{C}}{F},\IPSh(\I{C}))}( hp, \diag h(-))\\
		&\simeq \map{\iFun(\Over{\I{C}}{F},\I{C})}(p, \diag(-))\\
		&\simeq \map{\I{C}}(\colim p, -),
	\end{align*}
	which shows that the presheaf $\map{\Sml_{\BB}^{\I{U}}(\I{C})}(F, h(-))$ is represented by $L(F)=\colim p$.
\end{proof}

\subsection{The functor of left Kan extension}
\label{sec:KanExtensions}
Throughout this section, let $\I{C}$, $\I{D}$ and $\I{E}$ be $\BB$-categories and let $f\colon \I{C}\to\I{D}$ be a functor.
\begin{definition}
\label{def:KanExtension}
A \emph{left Kan extension} of a functor $F\colon \I{C}\to\I{E}$ along $f$ is a functor $f_!F\colon \I{D}\to\I{E}$ together with an equivalence 
\begin{equation*}
\map{\iFun(\I{D},\I{E})}(f_! F,-)\simeq \map{\iFun(\I{C},\I{E})}(F, f^\ast(-)).
\end{equation*}
Dually, a \emph{right Kan extension} of $F$ along $f$ is a functor $f_\ast F\colon \I{D}\to\I{E}$ together with an equivalence
\begin{equation*}
\map{\iFun(\I{D},\I{E})}(-,f_\ast F)\simeq \map{\iFun(\I{C},\I{E})}(f^\ast(-), F).
\end{equation*}
\end{definition} 

\begin{remark}[locality of Kan extensions]
	\label{rem:BCKanExtensions}
	In the situation of Definition~\ref{def:KanExtension}, if $A\in\BB$ is an arbitrary object, one easily deduces from Remark~\ref{rem:localityPrinciplePreservationStructure} and~\cite[Lemma~4.7.13]{martini2021} that the functor $\pi_A^\ast(f_! F)$ is a left Kan extension of $\pi_A^\ast F$ along $\pi_A^\ast f$.
\end{remark}

\begin{remark}
    As usual, the theory of right Kan extensions can be formally obtained from the theory of left Kan extensions by taking opposite $\BB$-categories. We will therefore only discuss the case of left Kan extensions.
\end{remark}
\begin{remark}
    The theory of Kan extensions for the special case $\BB\simeq\SS$ is discussed in \cite[\S 22]{joyal2008notes},~\cite[\S~4.3]{htt},  or \cite[\S 6.4]{cisinski2019a}.
\end{remark}
The main goal of this section is to prove the following theorem about the existence of left Kan extensions:
\begin{theorem}
	\label{thm:existenceKanExtension}
	Let $\I{U}$ be an internal class of $\BB$-categories such that for every object $d\colon A\to\I{D}$ in context $A\in\BB$ the $\Over{\BB}{A}$-category $\Over{\I{C}}{d}$ is contained in $\I{U}^{\colim}(A)$.
	Then, whenever $\I{E}$ is $\I{U}$-cocomplete, the functor $f^\ast\colon \iFun(\I{D},\I{E})\to\iFun(\I{C},\I{E})$ has a left adjoint $f_!$ which is fully faithful whenever $f$ is fully faithful.
\end{theorem}
\begin{proof}
	To begin with, by replacing $\I{U}$ with $\I{U}^{\colim}$, we may assume without loss of generality that $\I{U}$ is a colimit class and therefore that $\Over{\I{C}}{d}$ is contained in $\I{U}$ for every object $d$ in $\I{D}$.
	
	By Corollary~\ref{cor:existenceYonedaExtension}, the functor $(f\times\id)^\ast\colon \iFun(\I{D}\times \I{E}^{\op},\Univ)\to\iFun(\I{C}\times\I{E}^{\op},\Univ)$ admits a left adjoint $(f\times\id)_!$. 
	We now claim that the composition
	\begin{equation*}
	\iFun(\I{C},\I{E})\xrightarrow{h_\ast} \iFun(\I{C},\IPSh(\I{E}))\simeq\iFun(\I{C}\times\I{E}^{\op},\Univ)\xrightarrow{(f\times\id)_!}\iFun(\I{D}\times\I{E}^{\op},\Univ)\simeq\iFun(\I{D},\IPSh(\I{E}))
	\end{equation*}
	takes values in $\iFun(\I{D},\Sml_{\BB}^{\I{U}}(\I{E}))$. To see this, let $F\colon A\to \iFun(\I{C},\I{E})$ be an object in context $A\in\BB$. Using Example~\ref{ex:representabilityLocalCondition} together with Remark~\ref{rem:localitySmall} and the fact that as $\pi_A^\ast$ preserves adjunctions (Corollary~\ref{cor:geometricMorphismAdjunction}) we may identify $\pi_A^\ast(f\times\id)_!$ with $(\pi_A^\ast(f)\times\id)_!$, which allows us to replace $\BB$ with $\Over{\BB}{A}$ and therefore to reduce to the case where $A\simeq 1$ (see Remark~\ref{rem:etaleInvarianceReductionGlobalContext}). Let $p\colon \I{P}\to\I{C}\times\I{E}^{\op}$ be the left fibration that is classified by the transpose of $hF$, and let $qi\colon \I{P}\to\I{Q}\to \I{D}\times\I{E}^{\op}$ be the factorisation of $(f\times\id)p$ into an initial functor and a left fibration. Then $q\colon \I{Q}\to\I{D}\times\I{E}^{\op}$ classifies $(f\times\id)_!(hF)$, hence we need to show that for any object $d\colon A\to \I{D}$ in context $A\in\BB$ the fibre $\I{Q}\vert_d\to A\times\I{E}^{\op}$ is classified by a $\I{U}$-small presheaf on $\I{E}$. By the same argument as above, we may again assume that $A\simeq 1$. Consider the commutative diagram
	\begin{equation*}
	\begin{tikzcd}[column sep=small, row sep=small]
	&& && \I{Q}\vert_d\arrow[dd]\arrow[dl]\arrow[dr, "s"]  &\\
	& \Over{\I{P}}{d} \arrow[dd]\arrow[rr, "\Over{i}{d}"]\arrow[dl]&& \Over{\I{Q}}{d}\arrow[dd] \arrow[dl]\arrow[rr, crossing over, "j", near end]&& \I{R} \arrow[dd, "r"]\\
	\I{P}\arrow[dd, "p", crossing over]\arrow[rr, "i", crossing over, near end] && \I{Q} && \I{E}^{\op}\arrow[dl]\arrow[dr, "\id"]&\\
	& \Over{\I{C}}{d}\times\I{E}^{\op} \arrow[dl]\arrow[rr]&& \Over{\I{D}}{d} \times\I{E}^{\op}\arrow[dl] \arrow[rr]&& \I{E}^{\op} \\
	\I{C}\times\I{E}^{\op}\arrow[rr, "f\times\id"] && \I{D}\times\I{E}^{\op}\arrow[from=uu, "q", crossing over, near start] &&&
	\end{tikzcd}
	\end{equation*}
	in which $\I{R}$ is uniquely determined by the condition that $j$ be initial and $r$ be a left fibration. Since $\Over{i}{d}$ is the pullback of $i$ along a right fibration and since right fibrations are proper~\cite[Proposition~4.4.7]{martini2021}, this map is initial. As a consequence, the composition $j \Over{i}{d}$ is initial as well, which implies that the left fibration $r$ is classified by the colimit of the composition $\Over{\I{C}}{d}\to\I{C}\to \I{E}\into\IPSh(\I{E})$. By Proposition~\ref{prop:colimitAccessiblePresheaf} and the condition on $\Over{\I{C}}{d}$ to be contained in $\I{U}(1)$, the left fibration $r$ is classified by a  $\I{U}$-small presheaf. To prove our claim, we therefore need only show that the map $s\colon \I{Q}\vert_d\to\I{R}$ is an equivalence. As this is a map of right fibrations over $\I{E}^{\op}$, we may work fibrewise~\cite[Proposition~4.1.18]{martini2021}. If $e\colon A\to\I{E}^{\op}$ is an object in context $A\in\BB$, we obtain an induced commutative triangle
	\begin{equation*}
	\begin{tikzcd}[column sep=small]
	& (\I{Q}\vert_d)\vert_e\arrow[dl]\arrow[dr, "s\vert_e"] & \\
	(\Over{\I{Q}}{d})\vert_e\arrow[rr, "j\vert_e"] & & \I{R}\vert_e
	\end{tikzcd}
	\end{equation*}
	over $A$. Since the projections $\Over{\I{Q}}{d}\to\I{E}^{\op}$ and  $\I{R}\to\I{E}^{\op}$ are left fibrations and therefore smooth~\cite[Proposition~4.4.7]{martini2021} and since initial functors are a fortiori covariant equivalences (see~\cite[\S~4.4]{martini2021}), we deduce from~\cite[Proposition~4.4.10]{martini2021} that $j\vert_e$ exhibits $\I{R}\vert_e$ as the groupoidification of $(\Over{\I{Q}}{d})\vert_e$. Moreover, the map $(\I{Q}\vert_d)\vert_e\to(\Over{\I{Q}}{d})\vert_e$ is a pullback of the final map $A\to \Over{\I{D}}{d}\times A$ along a smooth map and therefore final as well. Since final functors induce equivalences on groupoidifications, we thus conclude that $s\vert_e$ must be an equivalence, as desired.
	
	By making use of the discussion thus far, we may now define $f_!$ as the composition of the two horizontal arrows in the top row of the commutative diagram
	\begin{equation*}
	\begin{tikzcd}
	\iFun(\I{C},\I{E})\arrow[d, hookrightarrow, "h"]\arrow[r] & \iFun(\I{D},\Sml_{\BB}^{\I{U}}(\I{E}))\arrow[d, hookrightarrow]\arrow[r, "L_\ast"] & \iFun(\I{D},\I{E})\\
	\iFun(\I{C}\times\I{E}^{\op},\Univ)\arrow[r, "(f\times\id)_!"] & \iFun(\I{D}\times\I{E}^{\op},\Univ)
	\end{tikzcd}
	\end{equation*}
	in which $L$ denotes the left adjoint to the Yoneda embedding that is supplied by Proposition~\ref{prop:leftAdjointYonedaEmbedding}. It is now clear from the construction of $f_!$ that this functor defines a left adjoint of $f^\ast$.
	
	Lastly, suppose that $f$ is fully faithful.
	We show that in this case the adjunction counit $\id_{\iFun(\I{C},\I{E})}\to f^\ast f_!$ is an equivalence. Since equivalences are computed objectwise (see~\cite[Corollary~4.7.17]{martini2021}), we only have to show that for every object $F$ in $\iFun(\I{C},\I{E})$ the induced map $F\to f^\ast f_! F$ is an equivalence. Since $\pi_A^\ast$ preserves adjunctions and the internal hom (Corollary~\ref{cor:geometricMorphismAdjunction} and Remark~\ref{rem:localityPrinciplePreservationStructure}), we may replace $\BB$ with $\Over{\BB}{A}$ and can therefore assume that $F$ is in context $1\in \BB$ (see Remark~\ref{rem:etaleInvarianceReductionGlobalContext}). By construction of the adjunction $f_!\dashv f^\ast$, the unit $F\to f^\ast f_! F$ is determined by the composition
	\begin{equation*}
	h_\ast(F)\xrightarrow{\eta_1 h_\ast(F)} (f\times\id)^\ast (f\times\id)_! h_\ast(F)\xrightarrow{(f\times\id)^\ast \eta_2 (f\times\id)_! h_\ast(F)}  (f\times\id)^\ast h_\ast L_\ast (f\times\id)_! h_\ast(F)
	\end{equation*}
	in which $\eta_1$ is the unit of the adjunction $(f\times\id)_!\dashv (f\times\id)^\ast$ and $\eta_2$ is the unit of the adjunction $L_\ast\dashv h_\ast$. By Corollary~\ref{cor:existenceYonedaExtension}, the first map is an equivalence, hence it suffices to show that the second one is an equivalence as well. Again, it suffices to show this objectwise. Let therefore $c$ be an object of $\I{C}$, as above without loss of generality in context $1\in\BB$. By the above argument, the object $(f\times \id)_! h_\ast(F)(c)$ is given by the colimit of the diagram $h F(\pi_c)_!\colon\Over{\I{C}}{c}\to \I{C}\to \I{E}\into\IPSh(\I{E})$. By making use of the final section $\id_c\colon 1\to\Over{\I{C}}{c}$, this presheaf is therefore representable by $F(c)$, which implies the claim.
\end{proof}

\begin{remark}
	\label{rem:counitKanExtension}
	In the situation of Theorem~\ref{thm:existenceKanExtension}, the construction of $f_!$ shows that if $F\colon \I{D}\to\I{E}$ is a functor, the counit $f_!f^\ast F\to F$ is given by the composition
	\begin{equation*}
	L_\ast (f\times \id)_! (f\times\id)^\ast h_\ast(F)\xrightarrow{L_\ast \epsilon_1 h_\ast F} L_\ast h_\ast(F)\xrightarrow{\epsilon_2} F
	\end{equation*}
	where $\epsilon_1$ is the counit of the adjunction $(f\times\id)_!\dashv (f\times\id)^\ast$ and $\epsilon_2$ is the counit of the adjunction $L_\ast\dashv h_\ast$. Since the latter is an equivalence, the functor $F$ arises as the left Kan extension of $f^\ast F$ precisely if the map $L_\ast \epsilon_1 h_\ast(F)$ is an equivalence. Let $q\colon \I{Q}\to \I{D}\times\I{E}^{\op}$ be the left fibration that is classified by $h_\ast(F)$ and let $p\colon \I{P}\to\I{C}\times\I{E}^{\op}$ be the pullback of $q$ along $f\times\id$. Let furthermore $q^\prime\colon \I{Q}^\prime\to\I{C}\times\I{E}^{\op}$ be the functor that arises from factoring $(f\times\id)p$ into an initial map and a left fibration. On the level of left fibrations over $\I{D}\times\I{E}^{\op}$, the map $\epsilon_1 h_\ast(F)$ is then given by the map $g$ that arises as the unique lift in the commutative diagram
	\begin{equation*}
	\begin{tikzcd}
	\I{P}\arrow[r, "i"]\arrow[d, "i^\ast"] & \I{Q}\arrow[d, "q"]\\
	\I{Q}^\prime\arrow[r, "q"]\arrow[ur, dotted, "g"] & \I{D}\times\I{E}^{\op}.
	\end{tikzcd}
	\end{equation*}
	 Then the condition that $L_\ast \epsilon_1 h_\ast F$ is an equivalence corresponds to the condition that for any object $d\colon A\to \I{D}$ in context $A\in\BB$ the map $g\vert_d\colon\I{Q}^\prime\vert_d\to \I{Q}\vert_d$, viewed as a map over $\pi_A^\ast\I{E}^{\op}$, induces an equivalence $\colim (q^\prime\vert_d^{\op})\simeq \colim (q\vert_d^{\op})$ in $\pi_A^\ast\I{E}$. Note that by a similar argument as in the proof of Theorem~\ref{thm:existenceKanExtension}, the map $g\vert_d$ fits into a commutative square
	 \begin{equation*}
	 \begin{tikzcd}
	 	\Over{\I{Q}^\prime}{d}\arrow[r, "\Over{g}{d}"]\arrow[d, "j^\prime"] & \Over{\I{Q}}{d}\arrow[d, "j"] \\
	 	\I{Q}^\prime\vert_d\arrow[r, "g\vert_d"] & \I{Q}\vert_d
	 \end{tikzcd}
	 \end{equation*}
	 in which $j^\prime$ and $j$ are initial. As a consequence, the map $g\vert_d$ is determined by the factorisation of the map $j\Over{i}{d}$ in the commutative diagram
	 \begin{equation*}
	 \begin{tikzcd}
	 \Over{\I{P}}{d}\arrow[d]\arrow[r, "\Over{i}{d}"] \arrow[d]& \Over{\I{Q}}{d} \arrow[r, "j"] \arrow[d]& \I{Q}\vert_d\arrow[d] \\
	 \Over{\I{C}}{d}\times\I{E}^{\op}\times A\arrow[r, "\Over{f}{d}\times\id"] & \Over{\I{D}}{d}\times\I{E}^{\op}\times A\arrow[r] & \I{E}^{\op}\times A
	 \end{tikzcd}
	 \end{equation*}
	 into an initial map and a right fibration. This argument shows that the map $g\vert_d$ classifies the canonical map 
	 \begin{equation*}
	 	 \colim hF(\pi_d)_!\Over{f}{d} \to \colim hF (\pi_d)_!
	 \end{equation*}
	 of presheaves on $\pi_A^\ast\I{E}$ that is induced by the functor $\Over{f}{d}\colon \Over{\I{C}}{d}\to\Over{\I{D}}{d}$. Since $L$ is a left inverse of $h$ that preserves colimits, we thus conclude that $F$ is a left Kan extension of its restriction $f^\ast F$ precisely if the map
	 $\Over{f}{d}\colon\Over{\I{C}}{d}\to\Over{\I{D}}{d}$ induces an equivalence
	 \begin{equation*}
	 \colim F(\pi_d)_!\Over{f}{d} \simeq \colim F (\pi_d)_!\simeq F(d)
	 \end{equation*}
	 in $\pi_A^\ast\I{E}$ for every object $d\colon A\to \I{D}$.
\end{remark}

Recall from~\cite[\S~4.7]{martini2021} that a large $\BB$-category $\I{D}$ is \emph{locally} small if the left fibration $\Tw(\I{D})\to\I{D}^{\op}\times\I{D}$ is small (in the sense of~\cite[\S~4.5]{martini2021}). Theorem~\ref{thm:existenceKanExtension} now implies:
\begin{corollary}
	\label{cor:leftKanExtensionsSmall}
	If $f\colon \I{C}\to\I{D}$ is a functor of $\BB$-categories such that $\I{C}$ is small and $\I{D}$ is locally small (but not necessarily small). 
        If $\I{E}$ is a cocomplete large $\BB$-category, the functor of left Kan extension $f_!$ always exists.
\end{corollary}
\begin{proof}
	By Theorem~\ref{thm:existenceKanExtension}, it suffices to show that for any object $d\colon A\to \I{D}$ in context $A\in\BB$ the $\Over{\BB}{A}$-category $\Over{\I{C}}{d}$ is small, which follows immediately from the observation that the right fibration $\Over{\I{C}}{d}\to \I{C}\times A$ a pullback of the small fibration $\Tw(\I{D})\to\I{D}^{\op}\times\I{D}$ and therefore small itself.
\end{proof}

We conclude this section with an application of the theory of Kan extensions to a characterisation of \emph{colimit cocones}. If $\I{I}$ is a $\BB$-category, recall from Remark~\ref{rem:cones} that the associated right cone $\I{I}^\triangleright$ comes equipped with two functors $\iota\colon\I{I}\to \I{I}^\triangleright$ and $\infty\colon 1\to \I{I}^\triangleright$. Our goal is to prove:

\begin{proposition}
\label{prop:colimitCocones}
	Let $\I{I}$ and $\I{C}$ be $\BB$-categories and suppose that $\I{C}$ admits $\I{I}$-indexed colimits. Then the functor of left Kan extension
	\begin{equation*}
	\iota_!\colon \iFun(\I{I},\I{C})\to \iFun(\I{I}^\triangleright,\I{C})
	\end{equation*}
	along $\iota\colon \I{I}\to \I{I}^\triangleright$ exists and is fully faithful, and its essential image coincides with the full subcategory of $\iFun(\I{I}^\triangleright,\I{C})$ that is spanned by the colimit cocones.
\end{proposition}
The proof of Proposition~\ref{prop:colimitCocones} relies on the following two general facts:

\begin{lemma}
    \label{lem:pullbackCoreflectiveSubcategory}
    Suppose that
	\begin{equation*}
	\begin{tikzcd}
		\I{P}\arrow[d, "p"]\arrow[r, "g"] & \I{Q}\arrow[d, "q"]\\
		\I{C}\arrow[r, "f"] & \I{D}
	\end{tikzcd}
	\end{equation*}
	is a cartesian square in $\Cat(\BB)$ such that $q$ admits a fully faithful left adjoint. Then $p$ admits a fully faithful left adjoint as well.
\end{lemma}
\begin{proof}
    By assumption $q$ has a section $l_1\colon \I{D}\to\I{Q}$ which pulls back along $f$ to form a section $l_0\colon \I{C}\into\I{P}$ of $p$. Moreover, the adjunction counit $\epsilon_1\colon \Delta^1\otimes \I{Q}\to\I{Q}$ fits into a commutative diagram
    \begin{equation*}
        \begin{tikzcd}
        \Delta^1\otimes \I{D}\arrow[r, "\id\otimes l_1"] \arrow[d, "s^0"] & \Delta^1\otimes \I{Q}\arrow[d, "\epsilon_1"]\arrow[r, "s^0"] & \I{Q}\arrow[d, "q"]\\
        \I{D}\arrow[r, "l_1"] &\I{Q}\arrow[r, "q"] & \I{D},
        \end{tikzcd}
    \end{equation*}
    hence pullback along $f$ defines a map $\epsilon_0\colon\Delta^1\otimes \I{P}\to\I{P}$ that fits into a commutative square
    \begin{equation*}
        \begin{tikzcd}
        \Delta^1\otimes \I{C}\arrow[r, "\id\otimes l_0"] \arrow[d, "s^0"] & \Delta^1\otimes \I{P}\arrow[d, "\epsilon_0"]\arrow[r, "s^0"] & \I{P}\arrow[d, "p"]\\
        \I{C}\arrow[r, "l_0"] &\I{P}\arrow[r, "p"] & \I{C},
        \end{tikzcd}
    \end{equation*}
    By construction, the map $\epsilon_0 d^0$ is equivalent to the identity on $\I{P}$, and the map $\epsilon_1 d^1$ recovers the functor $l_1 p$. The previous commutative diagram now precisely expresses that both $p\epsilon_0$ and $\epsilon_0 l_0$ are equivalence, hence the desired result follows from Corollary~\ref{cor:criterionSubcategoryReflective}.
\end{proof}
\begin{lemma}
	\label{lem:pushoutFF}
	Fully faithful functors in $\Cat(\BB)$ are stable under pushout.
\end{lemma}
\begin{proof}
	If
	\begin{equation*}
	\begin{tikzcd}
	\I{C}\arrow[d, hookrightarrow, "f"] \arrow[r, "h"] & \I{E}\arrow[d, "g"]\\
	\I{D}\arrow[r, "k"] & \I{F}
	\end{tikzcd}
	\end{equation*}
	is a pushout square in $\Cat(\BB)$ in which $f$ is fully faithful, applying the functor $\IPSh(-)$ results in a pullback square
	\begin{equation*}
	\begin{tikzcd}
	\IPSh(\I{F})\arrow[d, "k^\ast"] \arrow[r, "g^\ast"] & 	\IPSh(\I{E}) \arrow[d, "h^\ast"]\\
	\IPSh(\I{D})\arrow[r, "f^\ast"] & 	\IPSh(\I{C})
	\end{tikzcd}
	\end{equation*}
	in which $f^\ast$ admits a fully faithful left adjoint $f_!$. By Lemma~\ref{lem:pullbackCoreflectiveSubcategory}, this implies that $g^\ast$ admits a fully faithful left adjoint as well, hence that the functor of left Kan extension $g_!$ is fully faithful. This in turn implies that $g$ must be fully faithful too, see Corollary~\ref{cor:existenceYonedaExtension}.
\end{proof}

\begin{proof}[{Proof of Proposition~\ref{prop:colimitCocones}}]
	Let $\I{U}$ be the smallest colimit class in $\BB$ that contains $\I{I}$. Then $\I{C}$ is $\I{U}$-cocomplete (by Remark~\ref{rem:UcolimitsGenerators}). Hence the existence of $\iota_!$ follows from Theorem~\ref{thm:existenceKanExtension} once we show that for every object $j\colon A\to \I{I}^\triangleright$ the $\Over{\BB}{A}$-category $\Over{\I{I}}{j}$ is contained in $\I{U}(A)$. By definition of the right cone, we have a cover $\I{I}_0\sqcup 1\onto (\I{I}^\triangleright)_0$ which induces a cover $A_0\sqcup A_1\onto A$ by taking the pullback along $j\colon A\to (\I{I}^{\triangleright})_0$. Let $j_0\colon A_0\to \I{I}^\triangleright$ and $j_1\colon A_1\to \I{I}^{\triangleright}$ be the induced objects. Since $j_0$ factors through the inclusion $\iota_0\colon\I{I}_0\into (\I{I}^\triangleright)_0$ and since $\iota$ is fully faithful by Lemma~\ref{lem:pushoutFF}, we obtain an equivalence $\Over{\I{I}}{j_0}\simeq \Over{\I{I}}{j_0^\prime}$ over $A_0$, where $j_0^\prime$ is the unique object in $\I{I}$ such that $\iota(j_0^\prime)\simeq j_0$. Since $j_1$ factors through the inclusion of the cone point $\infty\colon 1\to \I{I}^\triangleright$ which defines a final object in $\I{I}^\triangleright$, we furthermore obtain an equivalence $\Over{\I{I}}{j_1}\simeq \pi_{A_i}^\ast\I{I}$. Therefore the $\Over{\BB}{A}$-category $\Over{\I{I}}{j}$ is \emph{locally} contained in $\I{U}$ and therefore contained in $\I{U}$ itself, for $\I{U}$ defines a sheaf on $\BB$. We therefore deduce that the functor of left Kan extension $\iota_!$ exists. Since Lemma~\ref{lem:pushoutFF} implies that $\iota$ is fully faithful, Corollary~\ref{cor:existenceYonedaExtension} furthermore shows that $\iota_!$ is fully faithful as well.
	
	We finish the proof by identifying the essential image of $\iota_!$. By combining Remark~\ref{rem:cones} with Lemma~\ref{lem:pullbackCoreflectiveSubcategory}, if $d\colon A\to\iFun(I,\I{C})$ is a diagram, the object $\iota_!(d)$ defines a fully faithful left adjoint $A\to \Under{\I{C}}{d}$ to the projection $\Under{\I{C}}{d}\to A$. By Example~\ref{ex:initialObjectColimit}, this precisely means that $\iota_!(d)$ is an initial section over $A$ and is therefore a colimit cocone. 
	Conversely, if $\bar d\colon A\to \iFun(\I{I}^\triangleright,\I{C})$ is a cocone under $d=\iota^\ast\bar d$, the map $\epsilon \bar d\colon \iota_! d\to \bar d$ defines a map in $\Under{\I{C}}{d}$. By the above argument, the domain of this map is a colimit cocone, hence if $\bar d$ defines a colimit cocone in $\Under{\I{C}}{d}$ as well, the map $\epsilon \bar d$ must necessarily be an equivalence since \emph{any} map between two initial objects in a $\Over{\BB}{A}$-category is an equivalence (see Corollary~\ref{cor:universalPropertyInitialObject}).
\end{proof}

\section{Cocompletion}
\label{chap:Cocompletion}
The main goal of this section is to construct and study the \emph{free cocompletion} by $\I{U}$-colimits of an arbitrary $\BB$-category, for any internal class $\I{U}$ of $\BB$-categories. In \S~\ref{sec:freeCocompletion} we give the construction of this $\BB$-category and prove its universal property. \S~\ref{sec:detectingCocompletion} contains a criterion to detect free cocompletions, and we finish this chapter by studying the $\I{U}$-cocompletion of the point in \S~\ref{sec:freeCocompletionPoint}.

\subsection{The free $\I{U}$-cocompletion}
\label{sec:freeCocompletion}

Let $\I{C}$ be a $\BB$-category and let $\I{U}$ be an internal class of $\BB$-categories. The goal of this section is to construct the free $\I{U}$-cocompletion of $\I{C}$, i.e.\ the initial $\I{U}$-cocomplete $\BB$-category that is equipped with a functor from $\I{C}$.

We begin our discussion of free cocompletions with the maximal case $\I{U}=\ICat_{\BB}$:
\begin{theorem}
	\label{thm:universalPropertyPSh}
	For any $\BB$-category $\I{C}$ and any cocomplete large $\BB$-category $\I{E}$, the functor of left Kan extension $(h_{\I{C}})_!$ along the Yoneda embedding $h_{\I{C}}\colon \I{C}\into\IPSh(\I{C})$ induces an equivalence
	\begin{equation*}
	(h_{\I{C}})_!\colon\iFun(\I{C},\I{E})\simeq \iFun^\cc(\IPSh(\I{C}),\I{E}).
	\end{equation*}
	In other words, the Yoneda embedding $h_{\I{C}}\colon \I{C}\into\IPSh(\I{C})$ exhibits the $\BB$-category of presheaves on $\I{C}$ as the free cocompletion of $\I{C}$.
\end{theorem}
\begin{remark}
    The analogue of Theorem~\ref{thm:universalPropertyPSh} for $\infty$-categories is the content of~\cite[Theorem~5.1.5.6]{htt} or \cite[Theorem 6.3.13]{cisinski2019a}.
\end{remark}
The proof of Theorem~\ref{thm:universalPropertyPSh} relies on the following lemma:
\begin{lemma}
	\label{lem:YonedaExtensionIsKanExtension}
	Let $f\colon \I{C}\to \I{D}$ be a functor of $\BB$-categories and assume that $\I{C}$ is small. Then the left Kan extension $(h_{\I{C}})_!(h_{\I{D}}f)\colon \IPSh(\I{C})\to\IPSh[\BBB](\I{D})=\iFun(\I{D}^\op,\Univ[\BBB])$ of $h_{\I{D}}F$ along $h_{\I{C}}$ is equivalent to the composition
	\begin{equation*}
	\IPSh(\I{C})\xhookrightarrow{i_\ast} \IPSh[\BBB](\I{C})\xrightarrow{f_!} \IPSh[\BBB](\I{D}),
	\end{equation*}
	where $i\colon\Univ[\BB]\into\Univ[\BBB]$ is the inclusion from \S~\ref{sec:universe}.
\end{lemma}
\begin{proof}
	Since $(h_{\I{C}})_!$ is fully faithful and since the restriction of $f_!i_\ast$ along $h_{\I{C}}$ recovers the functor $h_{\I{D}}f$, it suffices to show that $f_!i_\ast$ is a left Kan extension along its restriction. By Remark~\ref{rem:counitKanExtension}, this is the case precisely if for any presheaf $F$ on $\I{C}$ the inclusion $\Over{h}{F}\colon \Over{\I{C}}{F}\into \Over{\IPSh(\I{C})}{F}$ induces an equivalence
	\begin{equation*}
		\colim f_!(\pi_{F})_! \Over{h}{F}\simeq \colim f_!(\pi_F)_!\simeq f_!(F).
	\end{equation*}
	Since $f_!i_\ast$ commutes with small colimits (Proposition~\ref{prop:universeEnlargementColimits}) and since $\PSh_{\Univ}(\I{C})$ admits small colimits (Proposition~\ref{prop:FunctorCategoryCocomplete}), it suffices to show that the map
	\begin{equation*}
	\colim (\pi_F)_! \Over{h}{F}\to F
	\end{equation*}
	is an equivalence in $\IPSh(\I{C})$, which follows immediately from Proposition~\ref{prop:coYoneda}.
\end{proof}
\begin{proof}[{Proof of Theorem~\ref{thm:universalPropertyPSh}}]
	Let us first show that for any object $f\colon A\to \iFun(\I{C},\I{E})$ in context $A\in\BB$ the object $(h_{\I{C}})_!(f)$ is contained in $\iFun^\cc(\IPSh(\I{C}),\I{E})$. By making use of Remarks~\ref{rem:UColimitPreservingFunctorCategoryLocal},~\ref{rem:localityPrinciplePreservationStructure} and~\ref{rem:BCKanExtensions} as well as Example~\ref{ex:representabilityLocalCondition}, we may replace $\BB$ with $\Over{\BB}{A}$ and can therefore assume that $A\simeq 1$ (see Remark~\ref{rem:etaleInvarianceReductionGlobalContext}). Hence, we only need to show that $h_!(f)$ is cocontinuous. By again making use of Remark~\ref{rem:BCKanExtensions} and Example~\ref{ex:representabilityLocalCondition}, it is enough to show that $h_!(f)$ preserves $\I{I}$-indexed colimits for every small $\BB$-category $\I{I}$. By Lemma~\ref{lem:YonedaExtensionIsKanExtension} and the explicit construction of $h_!$ in Theorem~\ref{thm:existenceKanExtension}, the functor $h_!(f)$ is equivalent to the composition
	\begin{equation*}
			\IPSh(\I{C})\xhookrightarrow{i_\ast}\IPSh[\BBB](\I{C})\xrightarrow{f_!} \Sml_{\BBB}^{\ICat_{\BB}}(\I{E})\xrightarrow{L} \I{E}
	\end{equation*}
	in which $L$ is left adjoint to the Yoneda embedding $h_{\I{E}}$. Since all three functors preserve small colimits, the claim follows.
	
	By what we have just shown, the embedding $h_!$ takes values in $\iFun^\cc(\IPSh(\I{C}),\I{E})$ and therefore determines an inclusion $h_!\colon \iFun(\I{C},\I{E})\into\iFun^\cc(\IPSh(\I{C}),\I{E})$. To show that this functor is essentially surjective as well, we need only show that any object $g\colon A\to \iFun(\IPSh(\I{C}),\I{E})$ in context $A\in\BB$ whose associated functor in $\Cat(\Over{\BBB}{A})$ is cocontinuous is a left Kan extension of its restriction along $h$. By the same reduction argument as above, we may again assume $A\simeq 1$. By using Remark~\ref{rem:counitKanExtension}, the functor $g$ is a Kan extension of $gh$ precisely if for any presheaf $F\colon A\to \IPSh(\I{C})$ the functor $\Over{h}{F}\colon \Over{\I{C}}{F}\to \Over{\IPSh(\I{C})}{F}$ induces an equivalence
	\begin{equation*}
	\colim g(\pi_F)_! \Over{h}{F}\simeq g(F)
	\end{equation*}
	in $\I{E}$. Since Proposition~\ref{prop:coYoneda} implies that the canonical map $\colim (\pi_F)_! \Over{h}{F}\to F$ is an equivalence in $\IPSh(\I{C})$ and since $g$ is cocontinuous, this is immediate.
\end{proof}

\begin{remark}
	\label{rem:universalPropertyPShLocallySmall}
	In the situation of Theorem~\ref{thm:universalPropertyPSh}, suppose that $\I{E}$ is in addition locally small. If $f\colon \I{C}\to\I{E}$ is an arbitrary functor, its left Kan extension $h_!(f)$ is not only cocontinuous, but even admits a right adjoint. In fact, by the explicit construction of $h_!(f)$ in the proof of Theorem~\ref{thm:universalPropertyPSh}, we may compute
	\begin{align*}
		\map{\I{E}}(h_!(f)(-),-) &\simeq \map{\I{E}}(L f_! i_\ast(-), -)\\
		&\simeq \map{\IPSh[\BBB](\I{C})}(i_\ast(-), f^\ast h_{\I{E}}(-))
	\end{align*}
	and since $\I{E}$ is locally small, the functor $f^\ast h_{\I{E}}$ takes values in $\IPSh(\I{C})$, hence the claim follows. By replacing $\BB$ with $\Over{\BB}{A}$ and using Remark~\ref{rem:BCKanExtensions} and Example~\ref{ex:representabilityLocalCondition}, the same argument works for arbitrary objects $A\to\iFun(\I{C},\I{E})$, hence we conclude that the functor $h_!$ takes values in $\iFun(\IPSh(\I{C}),\I{E})^L$ and therefore gives rise to an equivalence
	\begin{equation*}
	\iFun(\IPSh(\I{C}),\I{E})^L\simeq \iFun^\cc(\IPSh(\I{C}),\I{E}).
	\end{equation*}
	This is a special (and in a certain sense universal) case of the adjoint functor theorem for presentable $\BB$-categories. We will treat the general case in future work.
\end{remark}
Our next goal is to generalise Theorem~\ref{thm:universalPropertyPSh} to an arbitrary internal class $\I{U}$ of $\BB$-categories. For this, we need to make the following general observation:
\begin{lemma}
	\label{lem:existenceSmallestSubcategory}
	Let $\I{E}$ be a $\BB$-category, let $\I{C}\into\I{E}$ be a full subcategory and let $\I{U}\subset \I{V}$ be two internal classes of $\BB$-categories. Suppose that $\I{E}$ is $\I{V}$-cocomplete. Then there exists a full subcategory $\I{D}\into \I{E}$ that is closed under $\I{U}$-colimits (i.e.\ that is $\I{U}$-cocomplete and the inclusion into $\I{E}$ is $\I{U}$-cocontinuous), contains $\I{C}$ and is the smallest full subcategory of $\I{E}$ with these properties, in that whenever $\I{D}^\prime\into \I{E}$ has the same properties there is an inclusion $\I{D}\into\I{D}^\prime$ over $\I{E}$.
\end{lemma}
\begin{proof}
	Recall that the full subposet $\Sub^{\mathrm{full}}(\I{E})\into \Sub(\I{E})$ that is spanned by the fully faithful functors is a reflective subcategory (cf.\ the discussion in~\cite[\S~3.9]{martini2021}), which implies that this subposet is closed under limits in $\Sub(\I{E})$, i.e.\ meets. To complete the proof, we therefore only need to show that the collection of full subcategories of $\I{E}$ that contain $\I{C}$ and that are closed under $\I{U}$-colimits in $\I{E}$ is closed under limits in $\Sub(\I{E})$. Clearly, if $(\I{D}_i)_{i\in I}$ is a collection of full subcategories in $\I{E}$ that each contain $\I{C}$, then so does their meet $\I{D}=\bigwedge_{i}\I{D}_i$. Similarly, suppose that each $\BB$-category $\I{D}_i$ is closed under $\I{U}$-colimits in $\I{E}$, and let $A\in\BB$ be an arbitrary context. Since $\pi_A^\ast$ commutes with limits and carries fully faithful functors to fully faithful functors, we may assume without loss of generality that $A\simeq 1$. We thus only need to show that the meet of the $\I{D}_i$ is closed under $\I{I}$-indexed colimits in $\I{E}$ for any $\I{I}\in\I{U}(1)$. Let $d\colon B\to \iFun(\I{I},\I{D})$ be a diagram in context $B\in\BB$. Since by assumption the object $\colim d$ is contained in $\I{D}_i$ for every $i\in I$ and thus defines an object in $\I{D}$, the result follows.
\end{proof}
In light of Lemma~\ref{lem:existenceSmallestSubcategory}, we may now define:
\begin{definition}
	\label{def:freeCocompletion}
	For any $\BB$-category $\I{C}$ and any internal class $\I{U}$ of $\BB$-categories, we define the large $\BB$-category $\IPSh^{\I{U}}(\I{C})$ as the smallest full subcategory of $\IPSh(\I{C})$ that contains $\I{C}$ and is closed under $\I{U}$-colimits.
\end{definition}
\begin{remark}
    \label{rem:freeCocompletionSmall}
    Suppose that $\I{U}$ is a \emph{small} internal class of $\BB$-categories and $\I{C}$ is a $\BB$-category. Then $\IPSh^{\I{U}}(\I{C})$ is small as well. To see this, let us first fix a small full subcategory of generators $\GG\subset \BB$ (i.e.\ a full subcategory such that every object in $\BB$ admits a small cover by objects in $\GG$). Since $\I{U}$ is small, there exists a small regular cardinal $\kappa$ such that for every $\BB$-category $\I{I}$ in $\I{U}$ in context $G\in\GG$ the object $\I{I}_0\in \Over{\BB}{G}$ is $\kappa$-compact. We construct a diagram $\I{E}^{\bullet}\colon\kappa\to \Sub^{\mathrm{full}}(\IPSh(\I{C}))$ by transfinite recursion as follows: set $\I{E}^0=\I{C}$ and $\I{E}^{\lambda}=\bigvee_{\tau < \lambda} \I{E}^{\tau}$ for any limit ordinal $\lambda < \kappa$, where the right-hand side denotes the join operation in the poset $\Sub^{\mathrm{full}}(\IPSh(\I{C}))$. For $\lambda < \kappa$, we furthermore define $\I{E}^{\lambda+1}$ to be the full subcategory of $\IPSh(\I{C})$ that is spanned by $\I{E}^{\lambda}$ together with those objects that arise as the colimit of a diagram of the form $d\colon \I{I}\to \pi_G^\ast\I{E}^{\lambda}$ for $G\in \GG$ and $\I{I}\in \I{U}(G)$. Let us set $\I{E}=\bigvee_{\tau < \kappa} \I{E}^{\tau}$. Since $\kappa$ is small and $\I{E}^\tau$ is a small large $\BB$-category for every $\tau<\kappa$, the large $\BB$-category $\I{E}$ is small as well. We claim that $\I{E}$ is $\I{U}$-cocomplete. In fact, it suffices to show that for every $G\in\GG$ and every diagram $d\colon \I{I}\to \pi_G^\ast \I{E}$ the object $\colim d$ is contained in $\pi_G^\ast\I{E}$ as well. Since $\I{I}_0$ is $\kappa$-compact in $\Over{\BB}{G}$ and since $\kappa$ is $\kappa$-filtered as it is regular, the map $d_0\colon \I{I}_0\to \I{E}_0=\bigvee_{\tau<\kappa}\I{E}^\tau_0$ factors through $\I{E}^\tau_0$ for some $\tau<\kappa$. As a consequence, the colimit $\colim d$ is contained in $\I{E}^{\tau+1}$ and therefore a fortiori in $\I{E}$, as claimed. Now since $\I{E}$ is $\I{U}$-cocomplete and contains $\I{C}$, it must also contain $\IPSh^{\I{U}}(\I{C})$, which is therefore small.
\end{remark}

In the situation of Definition~\ref{def:freeCocompletion},  Proposition~\ref{prop:coYoneda} implies that there are inclusions
\begin{equation*}
\I{C}\into\Sml_{\BB}^{\I{U}}(\I{C})\into \IPSh^{\I{U}}(\I{C})\into\IPSh(\I{C}).
\end{equation*}
In general, the middle inclusion is not an equivalence, as the following example shows.
\begin{example}
	Let $\BB=\SS$ be $\infty$-topos of spaces, let $\I{C}=(\Delta^1)^{\op}$ and let $\I{U}$ be the smallest colimit class that contains $\Lambda^2_0$. An $\infty$-category is thus $\I{U}$-cocomplete precisely if it admits pushouts. An object in $\Fun(\Delta^1,\SS)$ is representable when viewed as a presheaf on $(\Delta^1)^{\op}$ precisely if it is one of the two maps $0\to 1$ and $1\to 1$. Hence $\Sml_{\BB}^{\I{U}}(\I{C})$ is the full subcategory of $\Fun(\Delta^1,\SS)$ that is spanned by the maps $n\to 1$ for natural numbers $n\leq 2$. But this $\infty$-category is not closed under pushouts in $\Fun(\Delta^1,\SS)$: for example, the map $S^1\to 1$ is a pushout of objects in $\Sml_{\BB}^{\I{U}}(\I{C})$ which is not contained in $\Sml_{\BB}^{\I{U}}(\I{C})$ itself.
\end{example}

\begin{lemma}
	\label{lem:geometricMorphismCocontinuity}
	Let $A\in\BB$ be an arbitrary object, let $\I{U}$ be an internal class of $\BB$-categories and let $f\colon\I{C}\to\I{D}$ be a $\pi_A^\ast\I{U}$-cocontinuous functor of $\pi_A^\ast\I{U}$-cocomplete $\Over{\BB}{A}$-category. Then $(\pi_A)_\ast(f)$ is a $\I{U}$-cocontinuous functor of $\I{U}$-cocomplete $\BB$-categories.
\end{lemma}
\begin{proof}
	Let $B\in\BB$ be an arbitrary object. We need to show that for every $\I{I}\in\I{U}(B)$ the $\Over{\BB}{B}$-categories $\pi_B^\ast(\pi_A)_\ast\I{C}$ and $\pi_B^\ast(\pi_A)_\ast\I{D}$ admit $\I{I}$-indexed colimits and that $\pi_B^\ast(\pi_A)_\ast(f)$ preserves these. Note that if $\pr_0\colon A\times B\to A$ and $\pr_1\colon A\times B\to B$ are the two projections, the natural map $\pi_B^\ast(\pi_A)_\ast\to (\pr_1)_\ast \pr_0^\ast$
	is an equivalence, owing to the transpose map $(\pr_0)_!\pr_1^\ast\to \pi_A^\ast(\pi_B)_!$ being one. Thus, we may identify $\pi_B^\ast(\pi_A)_\ast(f)$ with $(\pr_1)_\ast\pr_0^\ast(f)$. Now since $f$ is a $\pi_A^\ast\I{U}$-cocontinuous functor between $\pi_A^\ast\I{U}$-cocomplete $\Over{\BB}{A}$-categories, it follows that $\pr_0^\ast(f)$ is a $\pi_{A\times B}^\ast\I{U}$-cocontinuous functor between $\pi_{A\times B}^\ast\I{U}$-cocomplete $\Over{\BB}{A\times B}$-categories (Remark~\ref{rem:CocompletenessLocalCondition}). Therefore, by passing to $\Over{\BB}{B}$, we can assume that $B\simeq 1$. In other words, we only need to show that for every $\I{I}\in\I{U}(1)$ the two horizontal maps in the commutative square
	\begin{equation*}
		\begin{tikzcd}
			(\pi_A)_\ast \I{C}\arrow[d, "(\pi_A)_\ast(f)"]\arrow[r, "\diag"] & \iFun(\I{I},(\pi_A)_\ast\I{C})\arrow[d, "(\pi_A)_\ast(f)_\ast"]\\
			(\pi_A)_\ast \I{C}\arrow[r, "\diag"] & \iFun(\I{I},(\pi_A)_\ast\I{C})
		\end{tikzcd}
	\end{equation*}
	have left adjoints and that the associated mate transformation is an equivalence. This is a consequence of the equivalence $\iFun[\Over{\BB}{A}](-,(\pi_A)_\ast (-))\simeq(\pi_A)_\ast\iFun(\pi_A^\ast(-),-)$ (which follows by adjunction from the evident equivalence $\pi_A^\ast(-\times -)\simeq\pi_A^\ast(-)\times_A\pi_A^\ast(-)$) and the fact that by Corollary~\ref{cor:geometricMorphismAdjunction} the geometric morphism $(\pi_A)_\ast$ preserves adjunctions.
\end{proof}

\begin{lemma}
	\label{lem:pullbackUCocomplete}
	Let $\I{U}$ be an internal class of $\BB$-categories and let
	\begin{equation*}
		\begin{tikzcd}
			\I{Q}\arrow[r, "j", hookrightarrow]\arrow[d, "q"] & \I{P}\arrow[d, "p"]\\
			\I{D}\arrow[r, "i", hookrightarrow] & \I{C}
		\end{tikzcd}
	\end{equation*}
	be a pullback square in $\Cat(\BB)$ in which $f$ and $g$ are fully faithful. Assume furthermore that $\I{D}$, $\I{C}$ and $\I{P}$ are $\I{U}$-cocomplete and $p$ and $i$ are $\I{U}$-cocontinuous. Then $\I{Q}$ is $\I{U}$-cocomplete and $j$ is $\I{U}$-cocontinuous.
\end{lemma}
\begin{proof}
	We need to show that for every $A\in\BB$ and every $\I{I}\in\I{U}(A)$, the $\Over{\BB}{A}$-category $\pi_A^\ast\I{Q}$ admits $\I{I}$-indexed colimits and the functor $\pi_A^\ast j$ preserves them.
	Since $\pi_A^\ast$ preserves pullbacks and fully faithful functors and on account of Remark~\ref{rem:CocompletenessLocalCondition}, we may replace $\BB$ with $\Over{\BB}{A}$ and can therefore assume that $A\simeq 1$. Now we obtain a commutative diagram
	\begin{equation*}
	\begin{tikzcd}[column sep={5em,between origins},row sep={3em,between origins}]
		& \iFun(\I{I},\I{Q})\arrow[rr, hookrightarrow]\arrow[dd, dashed]\arrow[dl] && \iFun(\I{I},\I{P})\arrow[dd, "\colim"]\arrow[dl]\\
		\iFun(\I{I},\I{D})\arrow[rr, hookrightarrow, crossing over] && \iFun(\I{I},\I{C})\arrow[dd] &\\
		& \I{Q}\arrow[rr, hookrightarrow] \arrow[dl]&& \I{P}\arrow[dl] \\
		\I{D}\arrow[rr, hookrightarrow] \arrow[from=uu, "\colim", near start]&& \I{C}\arrow[from=uu, "\colim", crossing over, near start]
	\end{tikzcd}
	\end{equation*}
	where the dashed arrow exists on account of the lower square being a pullback. Thus Proposition~\ref{prop:limitsFullyFaithfulFunctor} yields that $\I{Q}$ admits $\I{I}$-indexed colimits and that $j$ preserves these, as desired.
\end{proof}

\begin{proposition}[locality of $\IPSh^{{\I{U}}}(\I{C})$]
	\label{prop:BCCocompletion}
	For any $\BB$-category $\I{C}$, any internal class $\I{U}$ of $\BB$-categories and any object $A\in\BB$, there is a natural equivalence
	\begin{equation*}
	\pi_A^\ast\IPSh^{\I{U}}(\I{C})\simeq\IPSh[\Over{\BB}{A}]^{\pi_A^\ast\I{U}}(\pi_A^\ast\I{C}).
	\end{equation*}
\end{proposition}
\begin{proof}
	It follows from Example~\ref{ex:representabilityLocalCondition} that there is a commutative diagram
	\begin{equation*}
	\begin{tikzcd}
	\pi_A^\ast\I{C}\arrow[ddr, bend left, hookrightarrow, "h_{\pi_A^\ast\I{C}}"] \arrow[d, hookrightarrow, "\pi_A^\ast h"]& \\
	\pi_A^\ast\IPSh^{\I{U}}(\I{C})\arrow[dr, hookrightarrow] \arrow[d, hookrightarrow]& \\
	\pi_A^\ast\IPSh(\I{C})\arrow[r, "\simeq"] & \IPSh[\Over{\BB}{A}](\pi_A^\ast\I{C}),
	\end{tikzcd}
	\end{equation*}
	and it is clear that $\pi_A^\ast\IPSh^{\I{U}}(\I{C})$ is closed under $\pi_A^\ast\I{U}$-colimits in $\IPSh[\Over{\BB}{A}](\pi_A^\ast\I{C})$. It therefore suffices to show that if $\I{D}\into\IPSh[\Over{\BB}{A}](\pi_A^\ast\I{C})$ is a full subcategory that contains $\pi_A^\ast \I{C}$ and that is likewise closed under $\pi_A^\ast\I{U}$-colimits in $\IPSh[\Over{\BB}{A}](\pi_A^\ast\I{C})$, this subcategory must contain $\pi_A^\ast\IPSh^{\I{U}}(\I{C})$. Consider the commutative diagram
	\begin{equation*}
	\begin{tikzcd}
	\I{C}\arrow[d, "\eta_A"]\arrow[r, hookrightarrow] & \I{D}^\prime\arrow[d]\arrow[r, hookrightarrow] & \IPSh(\I{C})\arrow[d, "\eta_A"]\\
	(\pi_A)_\ast\pi_A^\ast\I{C}\arrow[r, hookrightarrow] & (\pi_A)_\ast\I{D}\arrow[r, hookrightarrow] & (\pi_A)_\ast\pi_A^\ast\IPSh(\I{C})
	\end{tikzcd}
	\end{equation*}
	in which $\eta_A$ denotes the adjunction unit of $\pi_A^\ast\dashv (\pi_A)_\ast$ and in which $\I{D}^\prime$ is defined by the condition that the right square is a pullback. Note that the triangle identities for the adjunction $\pi_A^\ast\dashv (\pi_A)_\ast$ imply that $\I{D}$ contains $\pi_A^\ast\I{D}^\prime$. The proof is therefore finished once we show that $\I{D}^\prime$ is closed under $\I{U}$-colimits in $\IPSh(\I{C})$. 
	To prove this claim, note that we may identify $(\pi_A)_\ast\pi_A^\ast\simeq\iFun(A,-)$. With respect to this identification, the unit $\eta_A$ corresponds to precomposition with the unique map $\pi_A\colon A\to 1$. Thus, Proposition~\ref{prop:FunctorCategoryCocomplete} implies that $\eta_A$ is a $\I{U}$-cocontinuous functor between $\I{U}$-cocomplete $\BB$-categories. Also, Lemma~\ref{lem:geometricMorphismCocontinuity} implies that the inclusion $(\pi_A)_\ast \I{D}\into(\pi_A)_\ast\pi_A^\ast\IPSh(\I{C})$ is closed under $\I{U}$-colimits. Therefore, the result follows from Lemma~\ref{lem:pullbackUCocomplete}.
\end{proof}

\begin{lemma}
	\label{lem:morphismColimitPreserving}
	Let $\I{U}$ be an internal class and let $\I{C}$ and $\I{D}$ be $\I{U}$-cocomplete $\BB$-categories. Let $\alpha\colon f\to g$ be a map in $\iFun^{\cocont{\I{U}}}(\I{C},\I{D})$ in context $1\in\BB$. Then $\alpha$ is $\I{U}$-cocontinuous when viewed as a functor $\I{C}\to\I{D}^{\Delta^1}$ (where $\I{D}^{\Delta^1}$ is indeed $\I{U}$-cocomplete by Proposition~\ref{prop:FunctorCategoryCocomplete}).
\end{lemma}
\begin{proof}
	We need to show that for every $A\in\BB$ and every $\I{I}\in\I{U}(A)$, the functor $\pi_A^\ast\alpha$ preserves $\I{I}$-indexed colimits. Since by Remark~\ref{rem:localityPrinciplePreservationStructure} the base change functor $\pi_A^\ast$ commutes with cotensoring, we may replace $\BB$ with $\Over{\BB}{A}$ and can therefore assume that $A\simeq 1$. Now consider the commutative diagram
	\begin{equation*}
	\begin{tikzcd}
	\I{C}\arrow[d, "\alpha"] \arrow[r, "\diag"]& \iFun(\I{I}, \I{C})\arrow[d, "\alpha_\ast"]\\
	{\I{D}^{\Delta^1}}\arrow[r, "\diag"] \arrow[d, "{(d_1,d_0)}"]& \iFun(\I{I},\I{D}^{\Delta^1})\arrow[d, "{(d_1,d_0)_\ast}"]\\
	\I{D}\times\I{D}\arrow[r, "\diag"] & \iFun(\I{I},\I{D}\times\I{D}).
	\end{tikzcd}
	\end{equation*}
	In order to show that $\alpha$ preserves $\I{I}$-indexed colimits, we need to verify that the mate transformation $\phi$ of the upper square is an equivalence. On account of Proposition~\ref{prop:limitsFunctorCategories}, the mate of the lower square is an equivalence. We claim that the mate of the composite square is an equivalence as well, i.e.\ that $(f,g)\colon\I{C}\to\I{D}\times\I{D}$ preserves $\I{I}$-indexed colimits. To see this, let $d\colon A\to \iFun(\I{I},\I{C})$ be a diagram in context $A\in\BB$. Using Remark~\ref{rem:etaleInvariancePreservationLimits}, we may once again replace $\BB$ by $\Over{\BB}{A}$ and can thus assume that $A\simeq 1$ (see Remark~\ref{rem:etaleInvarianceReductionGlobalContext}). Now as $\iFun(\I{I},-)$ commutes with limits, we obslitain an equivalence $\Under{(\I{D}\times\I{D})}{(f,g)_\ast d}\simeq\Under{\I{D}}{f_\ast d}\times\Under{\I{D}}{g_\ast d}$, so that the claim follows once we show that the image of the initial cocone $1\to\Under{\I{C}}{d}$ along the functor
	\begin{equation*}
		(f_\ast, g_\ast)\colon \Under{\I{C}}{d}\to \Under{\I{D}}{f_\ast d}\times\Under{\I{D}}{g_\ast d}
	\end{equation*}
	is initial as well. This in turn follows from the assumption that both $f$ and $g$ preserve $\I{I}$-indexed colimits, together with the fact that the product of two initial maps is again initial.
	
	As a consequence of what we've shown so far and the functoriality of mates, we conclude that postcomposing $\phi$ with the functor $(d_1,d_0)\colon \I{C}^{\Delta^1}\to\I{D}\times\I{D}$ yields an equivalence. Therefore, $\phi$ is itself an equivalence once we verify that $(d_1,d_0)$ is conservative, i.e.\ internally right orthogonal to the map $s^0\colon\Delta^1\to\Delta^0$ (see~\cite[Definition~4.1.10]{martini2021}). Unwinding the definitions, this amounts to showing that the functor $\I{D}^{(-)}$ carries the commutative square
	\begin{equation*}
		\begin{tikzcd}
		\Delta^1\sqcup\Delta^1\arrow[r, "({d^1,d^0)}"]\arrow[d, "s^0\sqcup s^0"] & \Delta^1\times\Delta^1\arrow[d, "\pr_1"]\\
		\Delta^0\sqcup\Delta^0\arrow[r, "{(d^1,d^0)}"] & \Delta^1
		\end{tikzcd}
	\end{equation*}
	to a pullback, which follows from the observation that this square is a pushout in $\CatS$. Thus, we conclude that $\alpha$ preserves $\I{I}$-indexed colimits.
\end{proof}

\begin{theorem}
	\label{thm:freeCocompletion}
	Let $\I{C}$ be a $\BB$-category, let $\I{U}$ be an internal class of $\BB$-categories and let $\I{E}$ be a $\I{U}$-cocomplete large $\BB$-category. Then the functor of left Kan extension along $h_{\I{C}}\colon \I{C}\into\IPSh^{\I{U}}(\I{C})$ exists and determines an equivalence
	\begin{equation*}
	(h_{\I{C}})_!\colon \iFun(\I{C},\I{E})\simeq \iFun^\cocont{\I{U}}(\IPSh^{\I{U}}(\I{C}),\I{E}) .
	\end{equation*}
	In other words, the $\BB$-category $\IPSh^{\I{U}}(\I{C})$ is the free $\I{U}$-cocompletion of $\I{C}$.
\end{theorem}
\begin{proof}
	Let us define $\I{E}^{\prime}=\iFun(\I{E},\Univ[\BBB])^{\op}$. By Proposition~\ref{prop:YonedaEmbeddingComplete}, the inclusion $h_{\I{E}}^{\op}\colon\I{E}\into \I{E}^{\prime}$ that is given by the Yoneda embedding is $\I{U}$-cocontinuous. Let $j\colon \IPSh^{\I{U}}(\I{C})\into\IPSh(\I{C})$ be the inclusion. By Theorem~\ref{thm:existenceKanExtension}, the functors of left Kan extension along $h_{\I{C}}$ and $j$ exist and define inclusions
	\begin{equation*}
		\iFun(\I{C},\I{E}^{\prime})\xhookrightarrow{(h_{\I{C}})_!} \iFun(\IPSh^{\I{U}}(\I{C}),\I{E}^{\prime})\xhookrightarrow{j_!}\iFun(\IPSh(\I{C}),\I{E}^{\prime}),
	\end{equation*}
	and by Theorem~\ref{thm:universalPropertyPSh} the essential image of the composition is spanned by those objects in $\iFun(\IPSh(\I{C}),\I{E}^\prime)$ which define cocontinuous functors. Since $j$ is by construction $\I{U}$-cocontinuous, the restriction functor $j^\ast\colon\iFun(\IPSh(\I{C}),\I{E}^{\prime})\to\iFun(\IPSh^{\I{U}}(\I{C}),\I{E}^{\prime})$ restricts to a functor
	\begin{equation*}
	j^\ast\colon\iFun^\cc(\IPSh(\I{C}),\I{E}^{\prime})\to\iFun^\cocont{\I{U}}(\IPSh^{\I{U}}(\I{C}),\I{E}^{\prime}).
	\end{equation*}
	Consequently, we deduce that the left Kan extension functor $(h_{\I{C}})_!\colon\iFun(\I{C},\I{E}^{\prime})\into \iFun(\IPSh^{\I{U}}(\I{C}),\I{E}^{\prime})$ factors through an inclusion
	\begin{equation*}
	(h_{\I{C}})_!\colon\iFun(\I{C},\I{E}^{\prime})\into \iFun^\cocont{\I{U}}(\IPSh^{\I{U}}(\I{C}),\I{E}^{\prime}).
	\end{equation*}
	We claim that this functor is essentially surjective and therefore an equivalence. On account of Remarks~\ref{rem:UColimitPreservingFunctorCategoryLocal} and~\ref{rem:BCKanExtensions} as well as Proposition~\ref{prop:BCCocompletion}, it suffices to show (by replacing $\BB$ with $\Over{\BB}{A}$, see Remark~\ref{rem:etaleInvarianceReductionGlobalContext}) that any $\I{U}$-cocontinuous functor $f\colon \IPSh^{\I{U}}(\I{C})\to\I{E}^{\prime}$ is a left Kan extension along its restriction to $\I{C}$. Let $\epsilon\colon (h_{\I{C}})_!h_{\I{C}}^\ast f\to f$ be the adjunction counit, and let $\I{D}$ be the full subcategory of $\IPSh^{\I{U}}(\I{C})$ that is spanned by those objects $F$ in $\IPSh^{\I{U}}(\I{C})$ (in arbitrary context) for which $\epsilon F$ is an equivalence. We need to show that $\I{D}=\IPSh^{\I{U}}(\I{C})$. By construction, we have $\I{C}\into\I{D}$, so that it suffices to show that $\I{D}$ is closed under $\I{U}$-colimits in $\IPSh^{\I{U}}(\I{C})$. 
	Note that the inclusion $\I{D}\into\IPSh^{\I{U}}(\I{C})$ is precisely the pullback of $s_0\colon\I{E}^\prime\into(\I{E}^\prime)^{\Delta^1}$ along $\epsilon\colon \IPSh^{\I{U}}(\I{C})\to(\I{E}^{\prime})^{\Delta^1}$. Since Proposition~\ref{prop:FunctorCategoryCocomplete} implies that $s_0$ is cocontinuous and Lemma~\ref{lem:morphismColimitPreserving} shows that $\epsilon$ is $\I{U}$-cocontinuous, we deduce from Lemma~\ref{lem:pullbackUCocomplete} that the inclusion $\I{D}\into\IPSh^{\I{U}}(\I{C})$ is indeed closed under $\I{U}$-colimits.
	
	To finish the proof, we still need to show that the equivalence $(h_{\I{C}})_!\colon \iFun(\I{C},\I{E}^\prime)\simeq\iFun^\cocont{\I{U}}(\IPSh^{\I{U}}(\I{C}),\I{E}^\prime)$ restricts to the desired equivalence
	\begin{equation*}
		(h_{\I{C}})_!\colon \iFun(\I{C},\I{E})\simeq \iFun^\cocont{\I{U}}(\IPSh^{\I{U}}(\I{C}),\I{E}) .
	\end{equation*}
	As clearly $h_{\I{C}}^\ast$ restricts in the desired way, it suffices to show that $(h_{\I{C}})_!$ restricts as well.
	By the same reduction steps as above, this follows once we show that for every functor $f\colon \I{C}\to\I{E}$, the left Kan extension $ (h_{\I{C}})_!f\colon \IPSh^{\I{U}}(\I{C})\to \I{E}^\prime$ factors through $\I{E}$. Consider the commutative diagram
	\begin{equation*}
	\begin{tikzcd}
	\I{C}\arrow[dr, hookrightarrow, "h"] \arrow[r, dotted, hookrightarrow] & \I{D}\arrow[r] \arrow[d, hookrightarrow]& \I{E}\arrow[d, hookrightarrow] \\
	& \IPSh^{\I{U}}(\I{C})\arrow[r, "(h_{\I{C}})_! f"] & \I{E}^\prime
	\end{tikzcd}
	\end{equation*}
	in which the square is a pullback. Since both $\hat f$ and $\I{E}\into \I{E}^\prime$ are $\I{U}$-cocontinuous, it follows from Lemma~\ref{lem:pullbackUCocomplete} that the inclusion $\I{D}\into \IPSh^{\I{U}}(\I{C})$ is closed under $\I{U}$-colimits and must therefore be an equivalence. As a consequence, the functor $(h_{\I{C}})_! f$ factors through $\I{E}$, as needed.
\end{proof}

\begin{corollary}
	\label{cor:CocompletionLeftAdjoint}
	Let $\I{C}$ be a $\BB$-category and let $\I{U}\subset\I{V}$ be internal classes such that $\IPSh^{\I{U}}(\I{C})$ is $\I{V}$-cocomplete. Then the inclusion $i\colon\IPSh^{\I{U}}(\I{C})\into\IPSh^{\I{V}}(\I{C})$ admits a left adjoint. In particular, if $\I{C}$ itself is $\I{V}$-cocomplete, the inclusion $h_{\I{C}}\colon\I{C}\into\IPSh^{\I{V}}(\I{C})$ admits a left adjoint.
\end{corollary}
\begin{proof}
	By choosing $\I{U}=\varnothing$ (i.e.\ the initial object in $\Cat(\BB)$), the second claim is an immediate consequence of the first. To prove the first statement, let $j\colon\I{C}\into\IPSh^{\I{U}}(\I{C})$ be the inclusion. Then Theorem~\ref{thm:freeCocompletion} allows us to construct a candidate for the left adjoint $L\colon \IPSh^{\I{V}}(\I{C})\to\IPSh^{\I{U}}(\I{C})$ of $i$ as the left Kan extension of $j$ along $ij$. By construction, $L$ is $\I{V}$-cocontinuous. As $i$ is $\I{U}$-cocontinuous and since we have equivalences $j^\ast (Li)\simeq (ij)^\ast(L)\simeq j$, Theorem~\ref{thm:freeCocompletion} moreover gives rise to an equivalence $Li\simeq \id_{\IPSh^{\I{U}}(\I{C})}$. Similarly, since $j^\ast(i)\simeq ij$, one obtains an equivalence $i\simeq j_!(ij)$. Therefore, transposing the identity on $ij$ across the adjunction $(ij)_!\dashv (ij)^\ast$ gives rise to a map $\eta\colon \id_{\IPSh^{\I{V}}(\I{C})}\to iL$ such that $\eta i$ is an equivalence, being a map between $\I{U}$-cocontinuous functors that restricts to an equivalence along $j$.
	By making use of Corollary~\ref{cor:criterionSubcategoryReflective}, we conclude that $L$ is a left adjoint once we verify that $L\eta$ is an equivalence as well. As both domain and codomain of this map are $\I{V}$-cocontinuous functors, this is the case already if its restriction along $ij$ is an equivalence, which follows from the construction of $\eta$.
\end{proof}

\begin{corollary}
    \label{cor:CocompletionIsLeftAdjoint}
    Let $\I{U}$ be a \emph{small} internal class of $\BB$-categories. Then the inclusion $\ICat_{\BB}^\cocont{\I{U}}\into\ICat_{\BB}$ admits a left adjoint that carries a $\BB$-category $\I{C}$ to its free $\I{U}$-cocompletion. Moreover, the adjunction unit is given by the Yoneda embedding $\I{C}\into\IPSh^{\I{U}}(\I{C})$.
\end{corollary}
\begin{proof}
    By Remark~\ref{rem:freeCocompletionSmall}, the free $\I{U}$-cocompletion $\IPSh^{\I{U}}(\I{C})$ is indeed a small $\BB$-category. Therefore, the Yoneda embedding $h_{\I{C}}\colon\I{C}\into \IPSh^{\I{U}}(\I{C})$ is a well-defined map in $\ICat_{\BB}$. By Corollary~\ref{cor:representabilityCriterionAdjoint}, it suffices to show that the composition
    \begin{equation*}
        \phi\colon\map{\ICat_{\BB}^\cocont{\I{U}}}(\IPSh^{\I{U}}(\I{C}),-)\into \map{\ICat_{\BB}}(\IPSh^{\I{U}}(\I{C}),-)\to \map{\ICat_{\BB}}(\I{C},-)
    \end{equation*}
    is an equivalence of functors $\ICat_{\BB}^\cocont{\I{U}}\to \Univ$. Using that equivalences of functors are detected object-wise~\cite[Corollary~4.7.17]{martini2021}, this follows once we show that the evaluation of this map at any object $A\to \ICat_{\BB}^{\cocont{\I{U}}}$ yields an equivalence of $\Over{\BB}{A}$-groupoids. By combining Remark~\ref{rem:objectsCatU} with Proposition~\ref{prop:BCCocompletion} and with Example~\ref{ex:representabilityLocalCondition}, we may pass to $\Over{\BB}{A}$ and can therefore assume that $A\simeq 1$ (see Remark~\ref{rem:etaleInvarianceReductionGlobalContext}). In this case, the result follows from Theorem~\ref{thm:freeCocompletion} in light of the observation that by~ Remark~\ref{rem:UColimitPreservingFunctorCategoryLocal}, the evaluation of $\phi$ at a $\I{U}$-cocomplete $\BB$-category $\I{E}$ is precisely the restriction of the equivalence from Theorem~\ref{thm:freeCocompletion} to core $\BB$-groupoids.
\end{proof}

\subsection{Detecting cocompletions}
\label{sec:detectingCocompletion}
In this section we give a characterisation when a functor $f\colon \I{C}\to\I{D}$ exhibits $\I{D}$ as the free $\I{U}$-cocompletion of $\I{C}$. To achieve this, we need the notion of \emph{$\I{U}$-cocontinuous objects}, which is in a certain way an internal analogue of the notion of a $\kappa$-compact object in an $\infty$-category:
\begin{definition}
	\label{def:smallObject}
	Let $\I{D}$ be a $\I{U}$-cocomplete $\BB$-category. We define the full subcategory $\I{D}^{\cocont{\I{U}}}\into\I{D}$ of \emph{$\I{U}$-cocontinuous objects} as the pullback
	\begin{equation*}
		\begin{tikzcd}
			\I{D}^{\cocont{\I{U}}}\arrow[r, hookrightarrow]\arrow[d, hookrightarrow] & \iFun^{\cocont{\I{U}}}(\I{D},\Univ)^\op\arrow[d, hookrightarrow]\\
			\I{D}\arrow[r, hookrightarrow, "h_{\I{D}^\op}"] & \iFun(\I{D},\Univ)^\op.
		\end{tikzcd}
	\end{equation*}
\end{definition}
\begin{remark}[locality of $\I{U}$-cocontinuous objects]
	\label{rem:USmallObjectsLocal}
	In the situation of Definition~\ref{def:smallObject}, we may combine Example~\ref{ex:representabilityLocalCondition} with Remark~\ref{rem:UColimitPreservingFunctorCategoryLocal} to deduce that there is a canonical equivalence $\pi_A^\ast(\I{D}^{\cocont{\I{U}}})\simeq(\pi_A^\ast\I{D})^{\cocont{\pi_A^\ast\I{U}}}$ of full subcategories of $\pi_A^\ast\I{D}$, for every $A\in\BB$.
\end{remark}
\begin{remark}[\'etale transposition invariance of $\I{U}$-cocontinuous objects]
	\label{rem:etaleInvarianceUCocontinuousObjects}
	By Remark~\ref{rem:USmallObjectsLocal}, an object $d\colon A\to\I{D}$ is contained in $\I{D}^{\cocont{\I{U}}}$ if and only if its transpose $\bar d\colon 1\to\pi_A^\ast\I{D}$ is $\pi_A^\ast\I{U}$-cocontinuous.
\end{remark}

The following proposition and its proof is an adaptation of~\cite[Proposition~5.1.6.10]{htt}.
\begin{proposition}
	\label{prop:criterionFreeCocompletion}
	Let $f\colon\I{C}\to\I{D}$ be a functor between $\BB$-categories such that $\I{D}$ is $\I{U}$-cocomplete, and let $\hat f\colon \IPSh^{\I{U}}(\I{C})\to\I{D}$ be its unique $\I{U}$-cocontinuous extension. Then the following are equivalent:
	\begin{enumerate}
		\item $\hat f$ is an equivalence;
		\item $f$ is fully faithful, takes values in $\I{D}^{\cocont{\I{U}}}$, and generates $\I{D}$ under $\I{U}$-colimits.
	\end{enumerate}
\end{proposition}
\begin{proof}
	We first note that $\IPSh^{\I{U}}(\I{C})^{\cocont{\I{U}}}$ contains $\I{C}$. Indeed, Yoneda's lemma implies that the composition
	\begin{equation*}
		\I{C}\xhookrightarrow{h_{\I{C}}}\IPSh^{\I{U}}(\I{C})\xhookrightarrow{h_{\IPSh^{\I{U}}(\I{C})^\op}}\iFun(\IPSh^{\I{U}}(\I{C}),\Univ)^\op
	\end{equation*}
	can be identified with the opposite of the transpose of the evaluation functor $\ev\colon\I{C}^\op\times\IPSh^{\I{U}}(\I{C})\to\Univ$. Together with Proposition~\ref{prop:BCCocompletion} and Remark~\ref{rem:localityPrinciplePreservationStructure}, this implies that the image of $c\colon A\to\I{C}$ along this composition transposes to the functor
	\begin{equation*}
		\IPSh[\Over{\BB}{A}]^{\pi_A^\ast\I{U}}(\pi_A^\ast\I{C})\into \IPSh[\Over{\BB}{A}](\pi_A^\ast\I{C})\xrightarrow{\ev_c} \Univ[\Over{\BB}{A}]
	\end{equation*}
	which is $\pi_A^\ast\I{U}$-cocontinuous by Proposition~\ref{prop:FunctorCategoryCocomplete}. Therefore,~(1) implies~(2).
	
	Conversely, suppose that condition~(2) is satisfied. We first prove that $\hat f$ is fully faithful. Tot that end, if $c\colon A\to \I{C}$ is an arbitrary object, we claim that the morphism
	\begin{equation*}
		\map{\IPSh^{\I{U}}(\I{C})}(c,-)\to\map{\I{D}}(\hat f(c), \hat f(-))
	\end{equation*}
	is an equivalence. By combining Remarks~\ref{rem:BCKanExtensions} and~\ref{rem:localitySmall} with Proposition~\ref{prop:BCCocompletion} and Example~\ref{ex:representabilityLocalCondition}, we may replace $\BB$ by $\Over{\BB}{A}$ and can thus assume that $A\simeq 1$ (see Remark~\ref{rem:etaleInvarianceReductionGlobalContext}). In this case, the fact that $\I{C}$ is contained in $\IPSh^{\I{U}}(\I{C})^{\cocont{\I{U}}}$ and condition~(2) imply that both domain and codomain of the morphism
	are $\I{U}$-cocontinuous functors. By using Lemma~\ref{lem:morphismColimitPreserving} and the fact that the above morphism restricts to an equivalence on $\I{C}$, the universal property of $\IPSh^{\I{U}}(\I{C})$ thus implies that this map is an equivalence of functors. By what we just have shown, if $F\colon A\to \IPSh^{\I{U}}(\I{C})$ is an arbitrary object, the natural transformation
	\begin{equation*}
		\map{\IPSh^{\I{U}}(\I{C})}(-,F)\to\map{\I{D}}(\hat f(-), \hat f(F))
	\end{equation*}
	restricts to an equivalence on $\I{C}$. As this map transposes to a morphism of $\pi_A^\ast\I{U}$-cocontinuous functors (using Proposition~\ref{prop:YonedaEmbeddingComplete} and the fact that $\hat f$ is $\I{U}$-cocontinuous), the same argument as above shows that the entire natural transformation is in fact an equivalence and therefore that $\hat f$ is fully faithful, as desired. As therefore $\hat f$ exhibits $\IPSh^{\I{U}}(\I{C})$ as a full subcategory of $\I{D}$ that is closed under $\I{U}$-colimits and that contains $\I{C}$, the assumption that $\I{D}$ is generated by $\I{C}$ under $\I{U}$-colimits implies that $\hat f$ is an equivalence.
\end{proof}

\subsection{Cocompletion of the point}
\label{sec:freeCocompletionPoint}
Let $\I{U}$ be an internal class of $\BB$-categories. Our goal in this section is to study the $\BB$-category $\IPSh^{\I{U}}(1)\into\Univ$. To that end, let us denote by $\gp(\I{U})\into\Univ$ the image of $\I{U}$ along the groupoidification functor $(-)^{\gp}\colon\ICat_{\BB}\to\Univ$ from Proposition~\ref{prop:internalCoreGroupoidification}.

\begin{definition}
    We call an internal class $\I{U}$ \emph{closed under groupoidification}, if for any $A \in \BB$ and $\I{I} \in \I{U}(A) $ the groupoidification  $\I{I}^\gp$ is also contained in $\I{U}$.
    For any internal class $\I{U}$ we can form its \emph{closure under groupoidification}, denoted $\gpdcl{\I{U}}$, that is defined as the internal class spanned by $ \I{U}$ and $\gp(\I{U})$.
\end{definition}

\begin{remark}
    Since for any $\BB$-category $\I{I}$, the morphism $\I{I} \to \I{I}^\gp$ is final, it follows that any colimit class (in the sense of Definition~\ref{def:colimitClass}) is closed under groupoidification.
    Furthermore, for any internal class $\I{U}$, we have inclusions $\I{U} \subseteq \gpdcl{\I{U}} \subseteq \I{U}^{\colim}$.
    In particular the discussion after Definition~\ref{def:colimitClass} shows that a $\BB$-category is $\I{U}$-cocomplete if and only if it is $\gpdcl{\I{U}}$-cocomplete.
    The same statement holds for $\I{U}$-cocontinuity.
\end{remark}

\begin{remark}
    \label{rem:colimitClassGroupoidification}
    If $\I{U}$ is closed under groupodification, the adjunction $(-)^{\gp}\dashv \iota\colon\ICat_{\BB}\leftrightarrows \Univ$ restricts to an adjunction
    \begin{equation*}
        ((-)^{\gp}\dashv i)\colon\I{U}\leftrightarrows\gp(\I{U}).
    \end{equation*}
\end{remark}

\begin{proposition}
    \label{prop:cocompletionPoint}
    For any internal class $\I{U}$ of $\BB$-categories, there is an inclusion $\gp(\I{U})\into \IPSh^{\I{U}}(1)$ which is an equivalence whenever $\gpdcl{\I{U}}$ is closed under $\I{U}$-colimits in $\ICat_\BB$.
\end{proposition}
\begin{proof}
    By construction, the canonical map $\I{U}\into \gpdcl{\I{U}}$ induces an equivalence $\gp(\I{U})\simeq\gp(\gpdcl{\I{U}})$. Therefore we may assume that $\I{U}$ is closed under groupoidification.
    For any $\BB$-category $\I{I}$ contained in $\I{U}(1)$, its groupoidification $\I{I}^{\gp}$ is the colimit of the functor $\I{I}\to 1 \into \Univ$ (see Proposition~\ref{prop:colimitsUniverse}) and therefore by definition contained in $\IPSh^{\I{U}}(1)$. Note that by using Remarks~\ref{rem:localityPrinciplePreservationStructure} and~\ref{rem:BCCatB} as well as Corollary~\ref{cor:geometricMorphismAdjunction}, for every $A\in\BB$ the functor $\pi_A^\ast$ carries the adjunction $(-)^\gp\dashv \iota\colon \ICat_{\BB}\leftrightarrows\Univ$ to the adjunction $(-)^\gp\dashv \iota\colon\ICat_{\Over{\BB}{A}}\leftrightarrows\Univ[\Over{\BB}{A}]$. Together with Proposition~\ref{prop:BCCocompletion}, this observation and the above argument also yields that for every $\I{I}\in\I{U}(A)$ the groupoidification $\I{I}^\gp$ defines an object $A\to\IPSh^{\I{U}}(1)$. Thus, the groupoidification functor $(-)^\gp\colon\ICat_{\BB}\to\Univ$ restricts to a functor $\I{U}\to\IPSh^{\I{U}}(1)$ and therefore gives rise to the desired inclusion $\gp(\I{U})\into\IPSh^{\I{U}}(1)$. Now by definition of $\IPSh^{\I{U}}(1)$, this inclusion is an equivalence if and only if $\gp(\I{U})$ is closed under $\I{U}$-colimits in $\Univ$.
    But if the subcategory $\I{U} \into \ICat_\BB$ is closed under $\I{U}$-colimits in $\ICat_\BB$ it follows by Remark~\ref{rem:colimitClassGroupoidification} that $\gp(\I{U}) = \I{U} \cap \Univ$, hence the claim follows from Lemma~\ref{lem:pullbackUCocomplete}.
\end{proof}

\begin{example}
    Let $S$ be a local class of maps in $\BB$ and let $\Univ[S]\into\Univ$ be the associated full subcategory of $\Univ$ (cf.~Proposition~\ref{prop:classificationSubuniverses}). Then $\Univ[S]$ is clearly closed under groupoidification.
    Recall that $\Univ[S]$ is closed under $\Univ[S]$-colimits in $\Univ$ precisely if the local class $S$ is stable under composition (see Example~\ref{ex:localClassCocomplete}). Therefore, if $S$ is stable under composition, Proposition~\ref{prop:cocompletionPoint} provides an equivalence $\Univ[S]\simeq \IPSh^{\Univ[S]}(1)$.
    
    If $S$ is not closed under composition, the free cocompletion $ \IPSh^{\Univ[S]}(1)$ still admits an explicit description.
    Namely, an object $c \colon A \to\Univ$ in context $A\in \BB$ defines an object of $\IPSh^{\Univ[S]}(1)$ if and only if it is locally a composition of two morphisms in $S$. 
    To be more precise, $c$ is in $\IPSh^{\Univ[S]}(1)$ if and only if there is a cover $(s_i)\colon\bigsqcup_i A_i\onto A$ in $\BB$
    such that every $s_i^* c \in \Univ(A_i) = \BB_{/A_i}$ can be written as a composition $g_i f_i$ of two morphisms $g_i \colon P_i \rightarrow Q_i$ and $f_i \colon Q_i \rightarrow A_i$ that are in $S$.
    This description holds since the full subcategory spanned by these objects is clearly closed under $\Univ[S]$-indexed colimits and it is easy to see that it is the smallest full subcategory of $\Univ$ with this property.
\end{example}

\begin{example}
	The following observation is due to Bastiaan Cnossen:
	Let $\BB = \PSh_{\SS}(\CC)$ for some small $\infty$-category $\CC$ with pullbacks and let $S$ be a class of morphisms in $\CC$ that is closed under pullbacks in $\CC$. It generates a local class in $\BB = \PSh_{\SS}(\CC)$ that we denote by $W$.
	As in Example~\ref{ex:U+LConstcocomplete}, we obtain an internal class $\I{U}_S=\left<W, \CatS \right>$, so that we may now consider the free $\I{U}_S$-cocompletion $\IPSh^{\I{U}_S}(1)$ of the point. It may be explicitly described as the presheaf on $\CC$ given by
	\begin{equation*}
		\IPSh^{\I{U}_S}(1) \colon \CC^\op \rightarrow \Cat_\infty, \quad c \mapsto \PSh_{\SS}(\Over{S}{c})
	\end{equation*}
	where $\Over{S}{c}$ denotes the full subcategory of $\CC_{/c}$ spanned by the morphisms in $S$.
	In particular it agrees with the $\PSh(\CC)$-category underlying the \emph{initial cocomplete pullback formalism} described in \cite[\S~4]{Drew2022}.
    One can use this observation to give an alternative proof of \cite[corollary 4.9]{Drew2022}.
	In fact one can prove something more general since the proof in~\cite{Drew2022} relies on $\CC$ being a $1$-category, which is not necessary in our framework.
\end{example}

We conclude this section by showing that any $\I{U}$-cocomplete large $\BB$-category $\I{E}$ is \emph{tensored} over $\IPSh^{\I{U}}(1)$ in the following sense:
\begin{definition}
	\label{def:tensoring}
	A large $\BB$-category $\I{E}$ is \emph{tensored} over $\IPSh^{\I{U}}(1)$ if there is a functor $-\otimes -\colon \IPSh^{\I{U}}(1)\times\I{E}\to\I{E}$ together with an equivalence
	\begin{equation*}
	\map{\I{E}}(-\otimes -, -)\simeq\map{\Univ[\BBB]}(-, \map{\I{E}}(-,-)).
	\end{equation*}
\end{definition}

\begin{proposition}
	\label{prop:tensoringCocomplete}
	If $\I{E}$ is a $\I{U}$-cocomplete large $\BB$-category, then $\I{E}$ is tensored over $\IPSh^{\I{U}}(1)$.
\end{proposition}
\begin{proof}
	Since $\I{E}$ is $\I{U}$-cocomplete, Proposition~\ref{prop:limitsFunctorCategories} implies that the functor $\BB$-category $\iFun(\I{E},\I{E})$ is $\I{U}$-cocomplete as well.
	As a consequence, we may apply Theorem~\ref{thm:freeCocompletion} to extend the identity $\id_{\I{E}}\colon 1\to \iFun(\I{E},\I{E})$ in a unique way to a $\I{U}$-cocontinuous functor $f\colon \IPSh^{\I{U}}(1)\to \iFun(\I{\I{E}},\I{E})$. We define the desired bifunctor $-\otimes -$ as the transpose of $f$. To see that it has the desired property, note that $\map{\I{E}}(-\otimes -, -)$ is the transpose of the composition
	\begin{equation*}
		\IPSh^{\I{U}}(1)^{\op}\xrightarrow{f^{\op}} \iFun(\I{E}^{\op},\I{E}^{\op})\xhookrightarrow{(h_{\I{E}^{\op}})_\ast} \iFun(\I{E}^{\op}\times\I{E},\Univ[\BBB]),
	\end{equation*}
	whereas the functor $\map{\Univ[\BBB]}(-, \map{\I{E}}(-,-))$ transposes to the functor
	\begin{equation*}
			\IPSh^{\I{U}}(1)^{\op}\xhookrightarrow{i} \Univ[\BBB]^{\op}\xhookrightarrow{h_{\Univ[\BBB]^{\op}}} \iFun(\Univ[\BBB],\Univ[\BBB])\xrightarrow{\map{\I{E}}^\ast} \iFun(\I{E}^{\op}\times\I{E},\Univ[\BBB]).
	\end{equation*}
	As the opposite of either of these functors is $\I{U}$-cocontinuous, Theorem~\ref{thm:freeCocompletion} implies that they are both uniquely determined by their value at the point $1\colon 1\to \Univ$. Since $\map{\Univ[\BBB]}(1,-)$ is equivalent to the identity functor, we find that both of these functors send $1\colon 1\to \Univ$ to $\map{\I{E}}$ and that they are therefore equivalent, as required.
\end{proof}

\begin{remark}
    \label{rem:cotensoring}
    By dualising Proposition~\ref{prop:tensoringCocomplete}, one obtains that a  $\I{U}$-complete large $\BB$-category $\I{E}$ is \emph{powered} over $\IPSh^{\op(\I{U})}(1)$: since $\IPSh^{\op(\I{U})}(1)^{\op}$ is the free $\I{U}$-completion of the final $\BB$-category $1\in\Cat(\BB)$, there is a functor $(-)^{(-)}\colon \IPSh^{\op(\I{U})}(1)^{\op}\times\I{E}\to\I{E}$
    that fits into an equivalence
    \begin{equation*}
        \map{\I{E}}(-,(-)^{(-)})\simeq \map{\Univ[\BBB]}(-, \map{\I{E}}(-,-)).
    \end{equation*}
\end{remark}

\appendix

\section{The large $\BB$-category of $\BB$-categories}
\label{sec:CatB}
\numberwithin{equation}{section}
The goal in this section is to define the large $\BB$-category of $\BB$-categories. What makes this possible is the following general construction:
\begin{construction}
	\label{constr:PresCatsAsInternalCats}
	Recall that Lurie's tensor product of presentable $\infty$-categories introduced in~\cite[\S~ 4.8.1]{lurie2017} defines a functor
	\begin{equation*}
		-\otimes-\colon\RPr\times\RPr\to\RPr, \quad (\CC,\DD) \mapsto \Shv_{\CC}(\DD)
	\end{equation*}
	that preserves limits in each variable.
	Since the functor $\BB_{/-} \colon \BB^\op \rightarrow \CatSS$ factors through the inclusion $\RPr \hookrightarrow \CatSS$ we may consider the composite
	\[
	\RPr \times \BB^\op \xrightarrow{\id \times \BB_{/-}} \RPr \times \RPr \xrightarrow{-\otimes-} \RPr \rightarrow \CatSS.
	\]
	Its transpose defines a functor $\RPr \rightarrow \Fun(\BB^\op,\CatSS)$.
	It follows from \cite[Theorem 5.5.3.18]{htt} that this map factors through the full subcategory spanned by the limit-preserving functors and thus defines a functor
	\[
	- \otimes \Univ \colon \RPr \rightarrow \Cat(\BBB).
	\]
	By the explicit description of the tensor product between presentable $\infty$-categories, this functor is equivalently given by $\Shv_{-}(\Over{\BB}{-})$. In other words, given any presentable $\infty$-category $\EE$, the associated large $\BB$-category $\EE\otimes\Univ$ is given by the composition
	\[
	\BB^{\op}\xrightarrow{\Over{\BB}{-}} (\LPr)^{\op}\xrightarrow{\Shv_{\EE}(-)} \CatSS.
	\]
\end{construction}
Let us now consider the above construction in the special case $\EE = \CatS$.
By definition, $\CatS \otimes \Univ$ is given by the composite
\[
\BB^\op \xrightarrow{\BB_{/-}} (\LPr)^\op \xrightarrow{\Shv_{\CatS}(-)} \CatSS
\]
and thus agrees with the presheaf of $\infty$-categories $\Cat(\BB_{/-})$ defined in \cite[\S~3.3]{martini2021}.
In particular it follows that the latter is indeed a sheaf.
Therefore we feel inclined to make the following definition:

\begin{definition}
	\label{def:categoryOfCategories}
	We define the \emph{large $\BB$-category $\ICat_{\BB}$ of (small) $\BB$-categories} to be $\ICat_\BB = \CatS \otimes \Univ$, i.e.\ as the large $\BB$-category that corresponds to the sheaf $\Cat(\BB_{/-})$.
\end{definition}

\begin{remark}[locality of the $\BB$-category of $\BB$-categories]
	\label{rem:BCCatB}
	By definition of $\ICat_{\BB}$, there is a canonical equivalence $\pi_A^\ast\ICat_{\BB}\simeq\ICat_{\Over{\BB}{A}}$ for every $A\in\BB$ (where $\pi_A^\ast\colon \Cat(\BBB)\to\Cat(\Over{\BBB}{A})$ denotes the base change functor induced by $\pi_A^\ast\colon \BB\to\Over{\BB}{A}$, cf.~Remark~\ref{rem:functorialityBCategories}). In fact, by Remark~\ref{rem:BCategoriesSheavesFunctoriality} we may compute $\pi_A^\ast\ICat_{\BB}\simeq \Cat(\Over{\BB}{(\pi_A)_!(-)})$, which is evidently equivalent to $\Cat(\Over{(\Over{\BB}{A})}{-})$.
\end{remark}

\begin{remark}
	\label{rem:oppositeCategoriesInternalFunctor}
	By applying $-\otimes \Univ$ to the equivalence $(-)^\op\colon\CatS\simeq\CatS$, one obtains an equivalence $(-)^\op\colon\ICat_{\BB}\simeq\ICat_{\BB}$. On global sections over $A\in\BB$, this equivalence recovers the equivalence $(-)^\op\colon\Cat(\Over{\BB}{A})\simeq\Cat(\Over{\BB}{A})$ that carries a $\Over{\BB}{A}$-category to its opposite (cf.~Remark~\ref{rem:sheafFromBCategoryExplicitly}).
\end{remark}

\begin{remark}
	\label{rem:CatBB}
	By working internal to $\BBB$, we may define the (very large) $\BB$-category $\ICat_{\BBB}$ of large $\BB$-categories. By regarding $\ICat_{\BB}$ as a very large $\BB$-category, we furthermore obtain a fully faithful functor $i\colon\ICat_{\BB}\into\ICat_{\BBB}$. In fact, by the discussion in ~\cite[\S~3.3]{martini2021}, the inclusion $\Cat(\Over{\BB}{A})\into\Cat(\Over{\BBB}{A})$ defines an embedding of presheaves $\Cat(\Over{\BB}{-})\into\Cat(\Over{\BBB}{-})$ on $\BB$. Since moreover restriction along the inclusion $\BB\into\BBB$ defines an equivalence
	\begin{equation*}
		\begin{tikzcd}
			\Shv_{\Cat(\SSSS)}(\BBB)\simeq \Shv_{\Cat(\SSSS)}(\BB)
		\end{tikzcd}
	\end{equation*}
	(see the argument in~\cite[Remark~2.4.1]{martini2021}), we obtain the desired fully faithful functor $\ICat_{\BB}\into\ICat_{\BBB}$ in $\Cat(\BBBB)$. Explicitly, an object $A\to \ICat_{\BBB}$ in context $A\in\BBB$ that corresponds to a $\Over{\BB}{A}$-category $\I{C}\to A$ is contained in $\ICat_{\BB}$ precisely if for any map $s\colon A^\prime\to A$ with $A^\prime\in\BB$ the pullback $s^\ast\I{C}$ is small.
\end{remark}

\section{Monomorphisms and subcategories of $\BB$-categories}
\numberwithin{equation}{subsection}
\subsection{Monomorphisms}
\label{sec:monomorphisms}
Recall that a \emph{monomorphism} in $\Cat(\BB)$ (i.e.\ a $(-1)$-truncated map) is a functor that is internally left orthogonal to the map $\Delta^0\sqcup\Delta^0\to\Delta^0$. In other words, a functor $f\colon \I{C}\to\I{D}$ between $\BB$-categories is a monomorphism if and only if the square
\begin{equation*}
	\begin{tikzcd}
		\I{C}\arrow[r, "f"]\arrow[d, "{(\id_{\I{C}}, \id_{\I{C}})}"] & \I{D}\arrow[d, "{(\id_{\I{D}}, \id_{\I{D}})}"]\\
		\I{C}\times\I{C}\arrow[r, "f\times f"] & \I{D}\times\I{D}
	\end{tikzcd}
\end{equation*}
is a pullback, or equivalently that the diagonal map $\I{C}\to \I{C}\times_{\I{D}}\I{C}$ is an equivalence. We say that a monomorphism $f\colon \I{C}\into\I{D}$ exhibits $\I{C}$ as a \emph{subcategory} of $\I{D}$. We will study subcategories more extensively in \S~\ref{sec:subcategories}.

\begin{proposition}
	\label{prop:conditionsMonomorphism}
	A functor $f\colon\I{C}\to\I{D}$ between $\BB$-categories is a monomorphism if and only if both $f_0$ and $f_1$ are monomorphisms in $\BB$. In particular, both the inclusion $\Grpd(\BB)\into\Cat(\BB)$ and the core $\BB$-groupoid functor $(-)^{\core}\colon\Cat(\BB)\to\Grpd(\BB)$ preserve monomorphisms.
\end{proposition}
\begin{proof}
	Since limits in $\Cat(\BB)$ are computed levelwise, the map $f$ is a monomorphism precisely if $f_n$ is a monomorphism in $\BB$ for all $n\geq 0$. Owing to the Segal conditions, this is automatically satisfied whenever only $f_0$ and $f_1$ are monomorphisms.
\end{proof}

\begin{proposition}
	\label{prop:characterizationMonomorphismMappingGroupoids}
	Let $f\colon\I{C}\to \I{D}$ be a functor between large $\BB$-categories. Then the following are equivalent:
	\begin{enumerate}
		\item $f$ is a monomorphism;
		\item $f^{\core}$ is a monomorphism in $\BBB$, and for any $A\in \BB$ and any two objects $c_0,c_1\colon A\to \I{C}$ in context $A\in\BB$, the morphism
		\begin{equation*}
			\map{\I{C}}(c_0,c_1)\to\map{\I{D}}(f(c_0), f(c_1))
		\end{equation*}
		that is induced by $f$ is a monomorphism in $\Over{\BBB}{A}$;
		\item for every $A\in \BB$ the functor $f(A)\colon \I{C}(A)\to\I{D}(A)$ is a monomorphism of $\infty$-categories;
		\item the map of cartesian fibrations over $\BB$ that is determined by $f$ is a monomorphism of $\infty$-categories.
	\end{enumerate}
\end{proposition}
\begin{proof}
	As monomorphisms are defined by a limit condition, one easily sees that conditions~(1),~(3) and~(4) are equivalent, by making use of the equivalence of $\infty$-categories $\PSh_{\CatSS}(\BB)\simeq\Cart(\BB)$ (here the latter denotes the $\infty$-category of carestian fibrations over $\BB$, see \S~\ref{sec:parametrisedCategories}) and the fact that the inclusion $\Cat(\BBB)\into\PSh_{\CatSS}(\BB)$ creates limits.
	Moreover, Proposition~\ref{prop:conditionsMonomorphism} implies that $f$ is a monomorphism if and only if both $f_0$ and $f_1$ are monomorphisms in $\BBB$. It therefore suffices to show that $f_1$ is a monomorphism if and only if for every $A\in \BB$ and any two objects   $c_0,c_1\colon A\to \I{C}$ in context $A$, the morphism
	\begin{equation*}
		\map{\I{C}}(c_0,c_1)\to\map{\I{D}}(f(c_0), f(c_1))
	\end{equation*}
	that is induced by $f$ is a monomorphism in $\Over{\BBB}{A}$, provided that $f_0$ is a monomorphism. By definition, the map that $f$ induces on mapping $\BB$-groupoids fits into the commutative diagram
	\begin{equation*}
		\begin{tikzcd}[column sep={5em,between origins},row sep={2em,between origins}]
			& \map{\I{C}}(c_0,c_1)\arrow[dl]\arrow[rr]\arrow[dd] & & \map{\I{D}}(f(c_0), f(c_1))\arrow[dl]\arrow[dd] \\
			\I{C}_1 \arrow[rr, crossing over, "f_1", near end]\arrow[dd]& & \I{D}_1 & \\
			& A \arrow[rr, "\id", near start]\arrow[dl] & & A\arrow[dl] \\
			\I{C}_0\times \I{C}_0 \arrow[rr, "f_0\times f_0"]& & \I{D}_0\times \I{D}_0. \arrow[from=uu,crossing over]
		\end{tikzcd}
	\end{equation*}
	in which the two squares on the left and on the right are pullbacks. As $f_0$ is a monomorphism, the bottom square is a pullback, which implies that the top square is a pullback as well. Hence if $f_1$ is a monomorphism, then the morphism on mapping $\BB$-groupoids must be a monomorphism as well. Conversely, suppose that $f$ induces a monomorphism on mapping $\BB$-groupoids. Let $P\simeq (\I{C}_0\times\I{C}_0)\times_{\I{D}_0\times\I{D}_0}\I{D}_1$ denote the pullback of the front square in the above diagram. Then $f_1$ factors as $\I{C}_1\to P\to \I{D}_1$ in which the second arrow is a monomorphism. It therefore suffices to show that the map $\I{C}_1\to P$ is a monomorphism as well. Note that the map $\map{\I{D}}(f(c_0), f(c_1))\to\I{D}_1$ factors through the inclusion $P\into\I{D}_1$ such that the induced map $\map{\I{D}}(f(c_0), f(c_1))\to P$ arises as the pullback of the map $P\to \I{C}_0\times\I{C}_0$ along  $(c_0,c_1)$.
	As the object $C_0\times C_0$ is obtained as the colimit of the diagram
	\begin{equation*}
		\Over{\BB}{C_0\times C_0}\to\BB\into\BBB,
	\end{equation*}
	we obtain a cover $\bigsqcup_{A\to \I{C}_0\times\I{C}_0} A\onto C_0\times C_0$ in $\BBB$ and therefore a cover
	\begin{equation*}
		\bigsqcup_{(c_0,c_1)}\map{\I{D}}(f(c_0),f(c_1))\onto P.
	\end{equation*}
	We conclude the proof by observing that there is a pullback diagram
	\begin{equation*}
		\begin{tikzcd}
			\bigsqcup_{(c_0,c_1)}\map{\I{C}}(c_0,c_1)\arrow[d, hookrightarrow]\arrow[r, twoheadrightarrow] & \I{C}_1\arrow[d]\\
			\bigsqcup_{(c_0,c_1)}\map{\I{D}}(f(c_0),f(c_1))\arrow[r, twoheadrightarrow] & P
		\end{tikzcd}
	\end{equation*}
	in which the left vertical map is a monomorphism. Thus $\I{C}_1 \to P$ is also a monomorphism by \cite[Proposition 6.2.3.17]{htt}.
\end{proof}
\begin{example}
	\label{ex:coreGroupoidMonomorphism}
	For any $\BB$-category $\I{C}$, the canonical map $\I{C}^{\core}\to \I{C}$ is a monomorphism. In fact, using Proposition~\ref{prop:characterizationMonomorphismMappingGroupoids} this follows from the observation that on the level of cartesian fibrations over $\BB$ this map is given by the inclusion of the wide subcategory of $\int \I{C}$ spanned by the cartesian arrows and that this defines a monomorphism of $\infty$-categories.
\end{example}

A \emph{strong epimorphism} in $\Cat(\BB)$ is a functor that is left orthogonal to the collection of monomorphisms. As a consequence of Proposition~\ref{prop:conditionsMonomorphism}, one finds:
\begin{proposition}
	\label{prop:strongEsoGroupoids}
	A functor between $\BB$-groupoids is a strong epimorphism if and only if it is essentially surjective. Furthermore, both the inclusion $\Grpd(\BB)\into\Cat(\BB)$ and the functor $(-)^{\gp}\colon \Cat(\BB)\to\Grpd(\BB)$ preserve strong epimorphisms.
\end{proposition}
\begin{proof}
	Let $f$ be a functor between $\BB$-categories. Then $f^{\gp}$ is left orthogonal to a map $g$ in $\Grpd(\BB)$ if and only if $f$ is left orthogonal to $g$ when viewing the latter as a map in $\Cat(\BB)$. Since by Proposition~\ref{prop:conditionsMonomorphism} $g$ is a monomorphism in $\Grpd(\BB)$ if and only if $g$ is a monomorphism in $\Cat(\BB)$, the map $f^{\gp}$ is a strong epimorphism whenever $f$ is one. Now if $f$ is an essentially surjective map between $\BB$-groupoids and if $g$ is a monomorphism in $\Cat(\BB)$, then $f$ is left orthogonal to $g$ if and only if $f$ is left orthogonal to $g^{\core}$, hence $f$ is a strong epimorphism in $\Cat(\BB)$ since the core $\BB$-groupoid functor preserves monomorphisms by Proposition~\ref{prop:conditionsMonomorphism} and since~\cite[Corollary~3.8.11]{martini2021} implies that a map between $\BB$-groupoids is a monomorphism if and only if it is fully faithful. As every strong epimorphism is in particular essentially surjective (since fully faithful functors are always monomorphisms and since essentially surjective maps are left orthogonal to fully faithful functors), this argument also shows that the inclusion $\Grpd(\BB)\into\Cat(\BB)$ preserves strong epimorphisms. 
\end{proof}

\begin{remark}
	\label{rem:strongEpiLevelwise}
	In light of Proposition~\ref{prop:conditionsMonomorphism} it might be tempting to expect that a map $f\colon\I{C}\to\I{D}$ in $\Cat(\BB)$ is a strong epimorphism if and only if $f_0$ and $f_1$ are covers. In fact, since the Segal conditions imply that $f_0$ and $f_1$ being a cover is equivalent to $f$ being a cover in the $\infty$-topos $\Simp\BB$ (where covers are given by levelwise covers in $\BB$), this is easily seen to be a sufficient condition. It is however not necessary. For example, the functor $(d_2,d_0)\colon\Delta^1\sqcup\Delta^1\to \Delta^2$ in $\CatS$ is a strong epimorphism since every subcategory of $\Delta^2$ that contains the image of this functor must necessarily be $\Delta^2$, but this map is not surjective on the level of morphisms.
\end{remark}

\subsection{Subcategories}
\label{sec:subcategories}
For any $\infty$-category $\CC$ with finite limits and any object $c\in\CC$, we write $\Sub_{\CC}(c)$ for the poset of \emph{subobjects} of $c$, i.e.\ the full subcategory of $\Over{\CC}{c}$ that is spanned by the $(-1)$-truncated objects. Since a functor $f\colon \I{C}\to\I{D}$ is a monomorphism in $\Cat(\BB)$ if and only if $f$ is a $(-1)$-truncated object in $\Over{\Cat(\BB)}{\I{D}}$, it makes sense to define:
\begin{definition}
	Let $\I{D}$ be a $\BB$-category. A \emph{subcategory} of $\I{D}$ is defined to be an object in $\Sub_{\Cat(\BB)}(\I{D})$.
\end{definition}

\begin{warning}
    \label{warn:simplicialSubobject}
    If $\I{C}$ is a $\BB$-category, not every subobject of $\I{C}$ in $\Simp\BB$ need to be a $\BB$-category. Therefore, the two posets $\Sub_{\Cat(\BB)}(\I{C})$ and $\Sub_{\Simp\BB}(\I{C})$ are in general different.
\end{warning}

Recall from the discussion in \S~\ref{sec:BCategories} (but see also \S~\ref{sec:objectsMorphisms}) that if $\I{C}$ is a $\BB$-category and $A$ is an object in $\BB$, the datum of a map $A\to \I{C}_1$ is equivalent to that of a map $A\to\I{C}^{\Delta^1}$, which is in turn equivalent to that of a map $\Delta^1\otimes A\to \I{C}$. Hence, the identity $\I{C}_1\to\I{C}_1$ transposes to a functor $\Delta^1\otimes \I{C}_1\to\I{C}$.
\begin{lemma}
	\label{lem:universalMapStrongEso}
	For any $\BB$-category $\I{C}$, the functor $\Delta^1\otimes \I{C}_1\to\I{C}$ is a strong epimorphism in $\Cat(\BB)$.
\end{lemma}
\begin{proof}
	In light of Remark~\ref{rem:strongEpiLevelwise}, it suffices to show that the functor $\Delta^1\otimes\I{C}_1\to \I{C}$ induces a cover on level $0$ and level $1$. On level $0$, the map is given by
	\begin{equation*}
		(d_1,d_0)\colon \I{C}_1\sqcup\I{C}_1\to \I{C}_0
	\end{equation*}
	which is clearly a cover since precomposition with $s_0\sqcup s_0\colon \I{C}_0\sqcup\I{C}_0\to \I{C}_1\sqcup\I{C}_1$ recovers the diagonal $\I{C}_0\sqcup\I{C}_0\to\I{C}_0$ which is always a cover in $\BB$. On level $1$, one obtains the map
	\begin{equation*}
		(s_0d_1,\id,s_0d_0)\colon \I{C}_1\sqcup\I{C}_1\sqcup\I{C}_1\to\I{C}_1
	\end{equation*}
	which is similarly a cover in $\BB$, as desired.
\end{proof}

\begin{proposition}
	\label{prop:mappingPropertySubcategory}
	Let $f\colon\I{C}\to\I{D}$ be a functor between large $\BB$-categories and let $\I{E}\into\I{D}$ be a subcategory. The following are equivalent:
	\begin{enumerate}
		\item $f$ factors through the inclusion $\I{E}\into\I{D}$;
		\item $f^{\core}$ factors through $\I{E}^{\core}\into\I{D}^{\core}$, and for each pair of objects $(c_0,c_1)\colon A\to \I{C}_0\times\I{C}_0$ in context $A\in \BB$, the map
		\begin{equation*}
			\map{\I{C}}(c_0,c_1)\to\map{\I{D}}(f(c_0),f(c_1))
		\end{equation*}
		that is induced by $f$ factors through the inclusion
		\begin{equation*}
			\map{\I{E}}(f(c_0),f(c_1))\into\map{\I{D}}(f(c_0),f(c_1));
		\end{equation*}
		\item for each map $\Delta^1\otimes A\to \I{C}$ in context $A\in\BB$ its image in $\I{D}$ is contained in $\I{E}$.
	\end{enumerate}
\end{proposition}
\begin{proof}
	It is immediate that~(1) implies~(2) and that~(2) implies~(3). Suppose therefore that condition~(3) holds. As in the proof of Proposition~\ref{prop:characterizationMonomorphismMappingGroupoids}, the collection of all maps $A\to\I{C}_1$ constitutes a cover
	\begin{equation*}
		\bigsqcup_{A\to\I{C}_1} A\onto \I{C}_1
	\end{equation*}
	in $\BBB$. By applying Proposition~\ref{prop:strongEsoGroupoids} and~\cite[Corollary~3.8.12]{martini2021}, we may view this map as a strong epimorphism between large $\BB$-groupoids. Since strong epimorphisms are \emph{internally} left orthogonal to monomorphisms and therefore closed under products in $\Cat(\BBB)$, we deduce that the induced map $ \bigsqcup_{A\to \I{C}_1}\Delta^1\otimes A\to \Delta^1\otimes\I{C}_1$ is a strong epimorphism. Together with Lemma~\ref{lem:universalMapStrongEso}, we therefore obtain a strong epimorphism $\bigsqcup_{A\onto \I{C}_1}\Delta^1\otimes A\to \I{C}$.
	Using the assumptions, we may now construct a lifting problem
	\begin{equation*}
		\begin{tikzcd}
			\bigsqcup_{A\to\I{C}_1} \Delta^1\otimes A\arrow[d, twoheadrightarrow]\arrow[r] & \I{E}\arrow[d, hookrightarrow]\\
			\I{C}\arrow[r, "f"]\arrow[ur, dashed] & \I{D}
		\end{tikzcd}
	\end{equation*}
	which admits a unique solution, hence condition~(1) follows.
\end{proof}

\begin{corollary}
	\label{cor:functorIntoCore}
	A functor $f\colon\I{C}\to\I{D}$ of $\BB$-categories factors through the inclusion $\I{D}^{\core}\into\I{D}$ if and only if $f$ sends all morphisms in $\I{C}$ to equivalences in $\I{D}$.\qed
\end{corollary}
\begin{definition}
	\label{def:Image}
	Let $f\colon \I{C}\to\I{D}$ be a map in $\Cat(\BB)$ and let $\I{C}\onto\I{E}\into\I{D}$ be the factorisation of $f$ into a strong epimorphism and a monomorphism. Then the subcategory $\I{E}\into\I{D}$ is referred to as the \emph{$1$-image} of $f$.
\end{definition}

In~\cite[\S~3.9]{martini2021} we have shown that full subcategories of a $\BB$-category $\I{C}$
can be parametrised by the subobjects of $\I{C}_0$ in $\BB$ (see also Proposition~\ref{prop:fullSubcategoriesParametrisation}). Our goal hereafter is to obtain a similar result for \emph{all} subcategories of $\I{C}$. To that end, note that the functor
\begin{equation*}
	(-)^{\Delta^1}\colon\Over{\Cat(\BB)}{\I{C}}\to \Over{\Cat(\BB)}{\I{C}^{\Delta^1}}
\end{equation*}
admits a left adjoint that is given by the composition
\begin{equation*}
	\Over{\Cat(\BB)}{\I{C}^{\Delta^1}}\xrightarrow{\Delta^1\otimes -} \Over{\Cat(\BB)}{\Delta^1\otimes\I{C}^{\Delta^1}}\xrightarrow{\ev_!} \Over{\Cat(\BB)}{\I{C}}
\end{equation*}
in which $\ev$ denotes the evaluation map. Similarly, the functor
\begin{equation*}
	(-)^{\simeq}\colon \Over{\Cat(\BB)}{\I{C}^{\Delta^1}}\to \Over{\BB}{\I{C}_1}
\end{equation*}
has a left adjoint that is given by the composition
\begin{equation*}
	\Over{\BB}{\I{C}_1}\into \Over{\Cat(\BB)}{\I{C}_1}\xrightarrow{i_!} \Over{\Cat(\BB)}{\I{C}^{\Delta^1}}
\end{equation*}
where $i\colon \I{C}_1\simeq (\I{C}^{\Delta^1})^{\simeq}\into \I{C}^{\Delta^1}$ denotes the canonical inclusion. By Proposition~\ref{prop:conditionsMonomorphism}, the functor $(-)_1=(-)^{\simeq}\circ (-)^{\Delta^1}$ sends a monomorphism $\I{D}\into\I{C}$ to the inclusion $\I{D}_1\into \I{C}_1$ and therefore restricts to a functor $\Sub_{\Cat(\BB)}(\I{C})\to \Sub_{\BB}(\I{C}_1)$.
Since the inclusion $\Sub_{\Cat(\BB)}(\I{C})\into \Over{\Cat(\BB)}{\I{C}}$ admits a left adjoint that sends a functor $f\colon \I{D}\to\I{C}$ to its $1$-image in $\I{C}$, we thus obtain an adjunction
\begin{equation*}
	(\Gen{-}\dashv (-)_1)\colon \Sub_{\BB}(\I{C}_1)\leftrightarrows \Sub_{\Cat(\BB)}(\I{C})
\end{equation*}
in which the left adjoint $\Gen{-}$ sends a monomorphism $S\into\I{C}_1$ to the $1$-image $\Gen{S}$ of the associated map $\Delta^1\otimes S\to \I{C}$. Note that for any subcategory $\I{D}\into\I{C}$, the counit $\Gen{\I{D}_1}\to \I{D}$ is given by the unique solution to the lifting problem
\begin{equation*}
	\begin{tikzcd}
		\Delta^1\otimes\I{D}_1\arrow[r] \arrow[d, twoheadrightarrow]& \I{D}\arrow[d, hookrightarrow]\\
		\Gen{\I{D}_1}\arrow[r, hookrightarrow]\arrow[ur, dashed] & \I{C}
	\end{tikzcd}
\end{equation*}
in which the upper horizontal map is the transpose of the identity $\I{D}_1\to\I{D}_1$. By Lemma~\ref{lem:universalMapStrongEso}, this is a strong epimorphism, hence we conclude that the map $\Gen{\I{D}_1}\to\I{D}$ must be an equivalence. We have thus shown:

\begin{proposition}
	\label{prop:parametrisationSubcategories}
	For any $\BB$-category $\I{C}$, the functor $(-)_1\colon \Sub_{\Cat(\BB)}(\I{C})\to \Sub_{\BB}(\I{C}_1)$ exhibits the poset $\Sub_{\Cat(\BB)}(\I{C})$ as a reflective subposet of $\Sub_{\BB}(\I{C}_1)$.\qed
\end{proposition}

\begin{remark}
	\label{rem:counterexampleMono}
	The inclusion $(-)_1\colon \Sub_{\Cat(\BB)}(\I{C})\into \Sub_{\BB}(\I{C}_1)$ is in general not an equivalence. For example, consider $\BB=\SS$ and $\I{C}=\Delta^2$: here the two maps $d^{\{0,1\}}\colon \Delta^1\to\Delta^2$ and $d^{\{1,2\}}\colon\Delta^1\to\Delta^2$ determine a proper subobject of $\Delta^2_1$, but the associated subcategory of $\Delta^2$ is nevertheless $\Delta^2$ itself.
\end{remark}

As Remark~\ref{rem:counterexampleMono} exemplifies, one obstruction to $(-)_1\colon \Sub(\I{C})\into\Sub(\I{C}_1)$ being an equivalence is that the collection of maps that determine a subobject $S\into \I{C}_1$ need not be stable under composition. In other words, to make sure that a subobject of $\I{C}_1$ arises as the object of morphisms of a subcategory of $\I{C}$, we need to impose a composability condition on this subobject. Altogether, we obtain the following characterisation of the essential image of $(-)_1$:
\begin{proposition}
	\label{prop:classificationSubcategories}
	For any $\BB$-category $\I{C}$, a subobject $S\into\I{C}_1$ lies in the essential image of the inclusion $\Sub_{\Cat(\BB)}\I{C})\into\Sub_{\BB}(\I{C}_1)$ if and only if
	\begin{enumerate}
		\item it is closed under equivalences, i.e.\ the map $(s_0d_1,s_0d_0)\colon S\sqcup S\to \I{C}_1$ factors through $S\into\I{C}_1$;
		\item it is closed under composition, i.e.\ the restriction of the composition map $d_1\colon \I{C}_1\times_{\I{C}_0}\I{C}_1\to \I{C}_1$ along the inclusion $S\times_{\I{C}_0}S\into\I{C}_1\times_{\I{C}_0}\I{C}_1$ factors through $S\into\I{C}_1$.
	\end{enumerate}
\end{proposition}
The remainder of this section is devoted to the proof of Proposition~\ref{prop:classificationSubcategories}. Our strategy is to make use of the intuition that the datum of a subcategory of $\I{C}$ should be equivalent to the datum of a collection of objects in $\I{C}$, together with a composable collection of maps between these objects. Our goal hereafter is turn this surmise into a formal statement.

For any integer $k\geq 0$, let $i_k\colon \Delta^{\leq k}\into\Delta$ denote the full subcategory spanned by $\ord{n}$ for $n\leq k$, and let $\Simp\BB^{\leq k}$ denote the $\infty$-category of $\BB$-valued presheaves on $\Delta^{\leq k}$. The truncation functor $i_k^\ast\colon\Simp\BB\to\Simp\BB^{\leq k}$ admits both a left adjoint $(i_k)_!$ and a right adjoint $(i_k)_\ast$ given by left and right Kan extension. Note that both $(i_k)_!$ and $(i_k)_\ast$ are fully faithful. We will generally identify $\Simp\BB^{\leq k}$ with its essential image in $\Simp\BB$ along the \emph{right} Kan extension $(i_k)_\ast$. We define the associated \emph{coskeleton} functor as $\cosk_k=(i_k)_\ast i_k^\ast$ and the \emph{skeleton} functor as $\sk_k=(i_k)_!i_k^\ast$. The unit of the adjunction $i_k^\ast\dashv (i_k)_\ast$ provides a map $\id_{\Simp\BB}\to \cosk_k$, and the counit of the adjunction $(i_k)_!\dashv i_k^\ast$  provides a map $\sk_k\to\id_{\Simp\BB}$. We say that $C\in \Simp\BB$ is \emph{$k$-coskeletal} if the map $C\to\cosk_k(C)$ is an equivalence, i.e.\ if $C$ is contained in $\Simp\BB^{\leq k}\subset\Simp\BB$. Dually, $C$ is \emph{$k$-skeletal} if the map $\sk_k(C)\to C$ is an equivalence.
Note that the adjunction  $\sk_k\dashv \cosk_k$ implies that a simplicial object is $k$-coskeletal if and only if it is local with respect to the maps $\sk_k(D)\to D$ for every $D\in\Simp\BB$.

\begin{definition}
	\label{def:coskeletalMaps}
	For any integer $k\geq 0$, a map $f\colon C\to D$ in $\Simp\BB$ is said to be \emph{$k$-coskeletal} if it is right orthogonal to $\sk_k(K)\to K$ for every $K\in\Simp\BB$.
\end{definition}
Note that by using the adjunction $\sk_k\dashv \cosk_k$ and Yoneda's lemma, one has the following criterion for a map between simplicial objects in $\BB$ to be $k$-coskeletal:
\begin{proposition}
	\label{prop:pullbackCriterionCoskeletalMap}
	For any integer $k\geq 0$, a map $f\colon C\to D$ in $\Simp\BB$ is $k$-coskeletal precisely if the canonical map $C\to D\times_{\cosk_k (D)}\cosk_k(C)$ is an equivalence.\qed
\end{proposition}
For any $n\geq 1$, denote by $\partial \Delta^n$ the simplicial $\infty$-groupoid $\sk_{n-1}\Delta^n$ and by $\partial\Delta^n\into\Delta^n$ the natural map induced by the adjunction counit.

For later use, we record the following obvious consqeuence of the skeletal filtration on simplical sets:

\begin{lemma}
	\label{lem:ObviousCombinatorialThing}
	Let $ j \colon K \hookrightarrow L $ be a monomorphism of finite simplicial sets and assume that $ \sk_k K = \sk_k L $ for some $ k \in \mathbb{N} $.
	Then $ j $ is contained in the smallest saturated class containing the maps $ \partial \Delta^l \to \Delta^l $ for $ k < l < \dim L$.\qed
\end{lemma}

\begin{lemma}
	\label{lem:generatorsSkeleta}
	Let $k\geq 0$ be an integer. Then the following sets generate the same saturated class of morphisms in $\Simp\BB$:
	\begin{enumerate}
		\item $\{\sk_k D\to D~\vert~ D\in\Simp\BB\}$;
		\item $\{\partial\Delta^n\otimes A\into \Delta^n\otimes A~\vert~ n>k,~A\in \BB\}$.
		\item $\{\partial\Delta^{k+1}\otimes D\into\Delta^{k+1}\otimes D~\vert~D\in\Simp\BB\}$.
	\end{enumerate}
\end{lemma}
\begin{proof}
	We start by showing that the saturations of (1) and (2) agree.	
	Given $A\in\BB$, note that since the truncation functor $i_k^\ast$ commutes with postcomposition by both the pullback functor $\pi_A^\ast\colon \BB\to\Over{\BB}{A}$ and its right adjoint $(\pi_A)_\ast$, the uniqueness of adjoints implies that the functor $\sk_k$ commutes with $-\times A\colon \Simp\BB\to\Simp\BB$. 
	By a similar argument, the functor $\sk_k$ commutes with $\const\colon\Simp\SS\to \Simp\BB$. We therefore obtain an equivalence $\sk_k(\Delta^m\otimes A)\simeq \sk_k(\Delta^m)\otimes A$ with respect to which the canonical map $\sk_k(\Delta^m\otimes A)\to\Delta^m\otimes A$ corresponds to the map obtained by applying the functor $-\otimes A$ to the map $\sk_k(\Delta^m)\to\Delta^m$. This already implies that the set in~(2) is contained in the set in~(1), so that the saturation of~(2) is contained in the saturation of~(1). Conversely, as any $ D \in \BB_{\Delta}$ can be written as a colimit of objects of the form $ \Delta^n \otimes A $ (see~\cite[Lemma~4.5.2]{martini2021}), the above argument also shows that every map in~(1) is a colimit of maps of the form $\sk_k(\Delta^n)\otimes A\to \Delta^n\otimes A$. Since moreover $\const_{\BB}$ and $-\otimes A$ are colimit-preserving functors, one finds that~(1) is contained in the saturation of~(2) as soon as we can show that any saturated class $S$ of maps in $\Simp\SS$ which contains $ \partial \Delta^n \to \Delta^n $ for all $ n >k $ must also contain the maps $ \sk_k \Delta^m \to \Delta^m $ for all $ m $. To prove this latter claim, we argue by induction over $ n > k $. 
	If $ n = k+1  $ this is clear by definition. For $ n > k+1 $ we consider the composite $ \sk_k \Delta^n \to \partial \Delta^n \to \Delta^n$.
	By our induction hypothesis and Lemma~\ref{lem:ObviousCombinatorialThing}, the first map is in $S$ and the composite is so by assumption.
	Since saturated classes have the left cancellation property (see \cite[Proposition 2.5.2 (2) and Proposition 2.5.6]{martini2021}), the claim now follows.

	Next, to show that the saturation of~(2) contains~(3), we may again assume $D\simeq\Delta^m\otimes A$. In this case, the inclusion $\partial\Delta^{k+1}\times \Delta^m\into\Delta^{k+1}\times\Delta^m$ can be obtained as an iterated pushout of maps of the form $\partial \Delta^n\into\Delta^n$ for $n>k$ (by Lemma~\ref{lem:ObviousCombinatorialThing}), hence the claim follows. 
	For the converse inclusion, we will use induction on $n$, the case $n={k+1}$ being satisfied by definition. Given that for a fixed $n>k$ the inclusion $\partial\Delta^n\otimes A\into\Delta^n\otimes A$ is contained in the saturation of~(3), Lemma~\ref{lem:ObviousCombinatorialThing} allows us to build the inclusion $\partial\Delta^n\times \Delta^1\into \sk_{n}(\Delta^n\times\Delta^1)$ as an iterated pushout along $\partial\Delta^n\into\Delta^n$. Therefore, the map $\sk_n(\Delta^n\times\Delta^1)\otimes A\into(\Delta^n\times\Delta^1)\otimes A$ is contained in the saturation of~(3) by the left cancellation property. Let $\alpha\colon \Delta^{n+1}\to \Delta^n\times\Delta^1$ be defined by $\alpha(i)=(i,0)$ for $i=0,\dots, n$ and $\alpha(n+1)=(n+1,1)$, and let $\beta\colon \Delta^n\times\Delta^1\to\Delta^{n+1}$ be given by $\beta(i,0)=i$ and $\beta(i,1)=n+1$. We then obtain a retract diagram
	\begin{equation*}
		\begin{tikzcd}
			\partial\Delta^{n+1}\arrow[d, hookrightarrow]\arrow[r, "\alpha^\prime"] & \sk_n(\Delta^n\times \Delta^1)\arrow[r, "\beta^\prime"] \arrow[d, hookrightarrow]& \partial\Delta^{n+1}\arrow[d, hookrightarrow]\\
			\Delta^{n+1}\arrow[r, "\alpha"] & \Delta^n\times\Delta^1\arrow[r, "\beta"] & \Delta^{n+1}
		\end{tikzcd}
	\end{equation*}
	in which $\alpha^\prime$ and $\beta^\prime$ are given by the restriction of $\alpha$ and $\beta$, respectively. We therefore conclude that the map $\partial\Delta^{n+1}\otimes A\into \Delta^{n+1}\otimes A$ is in the saturation of~(3), as desired.
\end{proof}

As a consequence of Lemma~\ref{lem:generatorsSkeleta}, one finds:
\begin{proposition}
	\label{prop:internalCharacterisationCoskeletal}
	For any integer $k\geq 0$, a map $f\colon C\to D$ in $\Simp\BB$ is $k$-coskeletal if and only if it is internally right orthogonal to the map $\partial\Delta^{k+1}\into\Delta^{k+1}$.\qed
\end{proposition}

We can use Proposition~\ref{prop:internalCharacterisationCoskeletal} to show that every monomorphism between $\BB$-categories is $1$-coskeletal. To that end, recall that we denote by $I^2\into\Delta^2$ the inclusion of the $2$-spine (see \S~\ref{sec:BCategories}). We now obtain:
\begin{lemma}
	\label{lem:monosCoskeletal}
	Let $S$ be the internal saturation of $\Delta^0\sqcup\Delta^0\to \Delta^0$ and $I^2\into\Delta^2$ in $\Simp\BB$. Then $S$ contains the map $\partial\Delta^2\into\Delta^2$.
\end{lemma}
\begin{proof}
	Let $f\colon K\to L$ be a map in $\Simp\BB$ that is internally right orthogonal to the maps $\Delta^0\sqcup\Delta^0\to\Delta^0$ and the inclusion of the 2-spine $I^2\into\Delta^2$. Then $f$ is a monomorphism. Now consider the commutative diagram
	\begin{equation*}
		\begin{tikzcd}[column sep={3em,between origins}, row sep={2.5em,between origins}]
			K^{\Delta^2}\arrow[dr]\arrow[drrr, bend left=20]\arrow[dddr, bend right=20] &&&&&\\
			& P\arrow[dd, hookrightarrow]\arrow[ddrr, phantom, "\ulcorner", very near start]\arrow[rr] \arrow[dl, crossing over]&& K^{\partial\Delta^2}\arrow[rr]\arrow[dd, hookrightarrow]\arrow[dl] && K^{I^2}\arrow[dd, hookrightarrow]\arrow[dl, "\id"]\\
			Q\arrow[dd, hookrightarrow]\arrow[rr, crossing over]\arrow[ddrr, phantom, "\ulcorner", very near start] && R\arrow[rr, crossing over] \arrow[ddrr, phantom, "\ulcorner", very near start]&& K^{I^2}&\\
			& L^{\Delta^2}\arrow[rr] \arrow[dl, "\id"]&& L^{\partial\Delta^2}\arrow[rr]\arrow[dl, "\id"] && L^{I^2}\arrow[dl, "\id"]\\
			L^{\Delta^2}\arrow[rr] && L^{\partial\Delta^2}\arrow[rr]\arrow[from=uu, hookrightarrow, crossing over] && L^{I^2}\arrow[from=uu, crossing over, hookrightarrow]&
		\end{tikzcd}
	\end{equation*}
	in which $P$, $Q$ and $R$ are defined by the condition that the respective square is a pullback diagram. We need to show that the map $K^{\Delta^2}\to P$ is an equivalence. As by assumption on $f$ the map $K^{\Delta^2}\to Q$ is an equivalence, it suffices to show that $P\to Q$ is an equivalence as well. But this map is already a monomorphism, hence the claim follows from the observation that $P\to Q$ must be a cover as the map $K^{\Delta^2}\to Q$ is one.
\end{proof}
\begin{proposition}
	\label{prop:monosFFCoskeletal}
	Every monomorphism between $\BB$-categories is $1$-coskeletal.
\end{proposition}
\begin{proof}
	Lemma~\ref{lem:monosCoskeletal} implies that every monomorphism between $\BB$-categories is internally right orthogonal to $\partial \Delta^2\into\Delta^2$  and therefore $1$-coskeletal.
\end{proof}
Let $\I{C}$ be a $\BB$-category and let $\Over{\Cat(\BB)^{\leq 1}}{\I{C}}$ be the full subcategory of $\Over{\Cat(\BB)}{\I{C}}$ that is spanned by the $1$-coskeletal maps into $\I{C}$. By restricting the inclusion $\Over{\Cat(\BB)^{\leq 1}}{\I{C}}\into\Over{\Cat(\BB)}{\I{C}}$ to $(-1)$-truncated objects (i.e.\ to monomorphisms into $\I{D}$), one obtains a full embedding
\begin{equation*}
	\Sub^{\leq 1}_{\Cat(\BB)}(\I{C})\into\Sub_{\Cat(\BB)}(\I{C})
\end{equation*}
of partially ordered sets.
Proposition~\ref{prop:monosFFCoskeletal} now implies:
\begin{corollary}
	\label{cor:monosCoskeletalSubobjects}
	For any $\BB$-category $\I{C}$, the inclusion $\Sub^{\leq 1}_{\Cat(\BB)}(\I{C})\into\Sub_{\Cat(\BB)}(\I{C})$ is an equivalence.\qed
\end{corollary}
For any $\BB$-category $\I{C}$, the functor $\Over{(\cosk_1)}{\I{C}}\colon\Over{(\Simp\BB)}{\I{C}}\to\Over{(\Simp\BB^{\leq 1})}{\cosk_1\I{C}}$ that is induced by the coskeleton functor on the slice $\infty$-categories admits a fully faithful right adjoint $\eta^\ast$ that is given by base change along the adjunction unit $\eta\colon\I{C}\to \cosk_1\I{C}$. Upon restricting to subobjects, we therefore obtain an adjunction
\begin{equation*}
	\begin{tikzcd}
		\Sub_{\Simp\BB}(\I{C})\arrow[from=r, shift left=1mm, hookrightarrow, "\eta^\ast"]\arrow[r, shift left=1mm, "\Over{(\cosk_1)}{\I{C}}"] & \Sub_{\Simp\BB^{\leq 1}}(\cosk_1\I{C}).
	\end{tikzcd}
\end{equation*}
In general, the functor $\eta^\ast$ does not take values in $\Sub_{\Cat(\BB)}(\I{C})$, but we may explicitly characterise those subobjects of $\cosk_1{\I{C}}$ that do give rise to a $\BB$-category. To that end, note that given a subobject $D\into \cosk_1\I{C}$ in $\Simp\BB^{\leq 1}$, the restriction of $d_1\colon\I{C}_1\times_{\I{C}_0}\I{C}_1\to \I{C}_1$ along the inclusion $D_1\times_{D_0}D_1\into \I{C}_1\times_{\I{C}_0}\I{C}_1$ determines a map $d_1\colon D_1\times_{D_0}D_1\to\I{C}_1$. 
\begin{definition}
	\label{def:subobjectTruncatedComposition}
	Let $\I{C}$ be a $\BB$-category. A subobject $D\into \cosk_1 \I{C}$ in $\Simp\BB^{\leq 1}$ is said to be \emph{closed under composition} if the map $d_1\colon D_1\times_{D_0} D_1\to \I{C}_1$ factors through $D_1\into \I{C}_1$. We denote by $\Sub^{\mathrm{comp}}_{\Simp\BB^{\leq 1}}(\cosk_1\I{C})$ the full subcategory of $\Sub_{\Simp\BB^{\leq 1}}(\cosk_1\I{C})$ that is spanned by these subobjects.
\end{definition}
\begin{lemma}
	\label{lem:simplifiedSegalCoskeletal}
	Let $A\in\BB$ be an arbitrary object and let $S$ be a saturated set of maps in $\Simp\BB$ that contains the internal saturation of $\partial\Delta^2\into\Delta^2$ as well as the map $I^2\otimes A\into\Delta^2\otimes A$. Then $S$ contains $I^n\otimes A\into \Delta^n\otimes A$ for all $n\geq 2$.
\end{lemma}
\begin{proof}
	We may assume $n>2$. By~\cite[Proposition~2.13]{joyal2008}, it suffices to show that for all $0<i<n$ the inclusion $\Lambda^n_{i}\otimes A\into\Delta^n\otimes A$ is contained in $S$. On account of the factorisation $\Lambda^n_{i}\into\partial\Delta^n\into\Delta^n$ in which the first map is obtained as a pushout along $\partial\Delta^{n-1}\into\Delta^n$, this is immediate.
\end{proof}

\begin{proposition}
	\label{prop:characterisationComposableSubgraph}
	Let $D\into \cosk_1\I{C}$ be a subobject in $\Simp\BB^{\leq 1}$. Then $\eta^\ast D$ is a $\BB$-category if and only if $D$ is closed under composition. In particular, $\eta^\ast$ defines an equivalence $\Sub^{\mathrm{comp}}_{\Simp\BB^{\leq 1}}(\cosk_1\I{C})\simeq \Sub_{\Cat(\BB)}(\I{C})$.
\end{proposition}
\begin{proof}
	If $\eta^\ast D$ is a $\BB$-category, the fact that applying $\cosk_1$ to the inclusion $\eta^\ast D\into\I{C}$ recovers the subobject $D\into\cosk_1\I{C}$ implies that $D$ is closed under composition. Conversely, suppose that $D$ is closed under composition. Since $E^1\to 1$ is a cover in $\Simp\BB$ (where $E^1$ is the walking equivalence, see \S~\ref{sec:BCategories}), every monomorphism of simplicial objects in $\BB$ is internally right orthogonal to $E^1\to 1$. Therefore $\eta^\ast D$ is univalent. We still need to show that $\eta^\ast D$ satisfies the Segal conditions. Since $\eta^\ast D\into \I{C}$ is $1$-coskeletal, Lemma~\ref{lem:simplifiedSegalCoskeletal} implies that we only need to show that $(\eta^\ast D)_2\to C_1\times_{D_0}D_1$ is an equivalence. As this map is a monomorphism, it furthermore suffices to show that it is a cover in $\BB$. Note that since the natural map $(-)^{\Delta^2}\to (-)^{\partial\Delta^2}$ induces an equivalence on $1$-coskeletal objects, the identification $\partial\Delta^2\simeq I^2\sqcup_{\Delta^0\sqcup\Delta^0}\Delta^1$ gives rise to a commutative square
	\begin{equation*}
		\begin{tikzcd}[column sep={3.5em,between origins}, row sep={2.5em,between origins}]
			& \I{C}_1\times_{\I{C}_0}\I{C}_1\arrow[rr]\arrow[ddrr, bend right, "d_1"'] && (\cosk_1\I{C})_2\arrow[rr]\arrow[dd] && \I{C}_1\times_{\I{C}_0}\I{C}_1\arrow[dd, near start]\\
			D_1\times_{D_0}D_1\arrow[rr, crossing over, dashed]\arrow[ur, hookrightarrow]\arrow[ddrr, bend right, "d_1"', dashed] && D_2\arrow[rr, crossing over]\arrow[dd, crossing over]\arrow[ur, hookrightarrow] && D_1\times_{D_0}D_1\arrow[ur, hookrightarrow] &\\
			& && \I{C}_1\arrow[rr] && \I{C}_0\times\I{C}_0 \\
			&& D_1\arrow[rr]\arrow[ur,hookrightarrow] && D_0\times D_0\arrow[from=uu, crossing over]\arrow[ur,hookrightarrow]
		\end{tikzcd}
	\end{equation*}
	in which the two squares in the front and in the back of the cube are pullbacks and where the dashed arrows exist as $D$ is closed under composition. By combining this diagram with the pullback square
	\begin{equation*}
		\begin{tikzcd}
			(\eta^\ast D)_2\arrow[r]\arrow[d, hookrightarrow] & D_2\arrow[d, hookrightarrow]\\
			\I{C}_2\arrow[r] & (\cosk_1\I{C})_2,
		\end{tikzcd}
	\end{equation*}
	one concludes that the map $(\eta^\ast D)_2\to D_1\times_{D_0}D_1$ admits a section and is therefore a cover, as desired. Lastly, the claim that that $\eta^\ast$ induces an equivalence $\SubObj^{\mathrm{comp}}(\cosk_1\I{C})\simeq \Sub(\I{C})$ now follows easily with Corollary~\ref{cor:monosCoskeletalSubobjects}.
\end{proof}

\begin{proof}[{Proof of Proposition~\ref{prop:classificationSubcategories}}]
	It is clear that any subobject $S\into\I{C}_1$ that arises as the object of morphisms of a subcategory of $\I{C}$ must necessarily satisfy the two conditions, so it suffices to prove the converse. Let $D_0\into \I{C}$ be the image of $(d_1,d_0)\colon S\sqcup S\to \I{C}_0$. As $S$ is closed under equivalences in $\I{C}$, the restriction of $s_0\colon \I{C}_0\to\I{C}_1$ to $D_0$ factors through $S\into\I{C}_1$. By setting $D_1=S$, we thus obtain a subobject $D\into \cosk_1\I{C}$ in $\Simp\BB^{\leq 1}$. By assumption, this subobject is closed under composition in the sense of Definition~\ref{def:subobjectTruncatedComposition}, hence Proposition~\ref{prop:characterisationComposableSubgraph} implies that $\eta^\ast D$ is a subcategory of $\I{C}$. Hence $S=D_1$ arises as the object of morphisms of $\eta^\ast D$ and is therefore contained in the essential image of $(-)_1\colon \Sub_{\Cat(\BB)}(\I{C})\into\Sub_{\BB}(\I{C}_1)$.
\end{proof}

\section{Localisations of $\BB$-categories}
\label{sec:localisation}
\numberwithin{equation}{section}
Recall that a functor between $\BB$-categories is said to be \emph{conservative} if it is internally right orthogonal to the map $\Delta^1\to\Delta^0$ (cf.~\cite[Definition~4.1.10]{martini2021}). Hereafter we discuss the left complement of the associated factorisation system, i.e.\ the saturated class that is internally generated by $\Delta^1\to\Delta^0$.

\begin{definition}
	\label{defn:conservativeIteratedlocalisation}
	A functor between $\BB$-categories is an \emph{iterated localisation} if it is left orthogonal to every conservative functor.
\end{definition}
The saturated class of iterated localisations in $\Cat(\BB)$ is internally generated by $\Delta^1\to\Delta^0$. Since this map is a strong epimorphism by Remark~\ref{rem:strongEpiLevelwise}, we deduce:
\begin{proposition}
	\label{prop:iteratedLocStrongEpi}
	Every iterated localisation in $\Cat(\BB)$ is a strong epimorphism and therefore in particular essentially surjective. Dually, every monomorphism is conservative.
	\qed
\end{proposition}

\begin{definition}
	\label{def:localizationAtSomeS}
	Let $\I{C}$ be a $\BB$-category and let $\I{S}\to \I{C}$ be a functor. The \emph{localisation} of $\I{C}$ at $\I{S}$ is the $\BB$-category $\I{S}^{-1}\I{C}$ that fits into the pushout square
	\begin{equation*}
		\begin{tikzcd}
			\I{S}\arrow[d]\arrow[r] \arrow[dr, phantom, "\lrcorner", very near end]& \I{S}^{\gp}\arrow[d]\\
			\I{C}\arrow[r] & \I{S}^{-1}\I{C}.
		\end{tikzcd}
	\end{equation*}
	We refer to the map $\I{C}\to \I{S}^{-1}\I{C}$ as the \emph{localisation functor} that is associated with the map $\I{S}\to\I{C}$. More generally, a functor $\I{C}\to \I{D}$ between $\BB$-categories is said to be a localisation if there is a functor $\I{S}\to\I{C}$ and an equivalence $\I{D}\simeq \I{S}^{-1}\I{C}$ in $\Under{\Cat(\BB)}{\I{C}}$.
\end{definition}

\begin{remark}
    The above definition is a direct analogue of the construction of localisations of $\infty$-categories, see \cite[Proposition 7.1.3]{cisinski2019a}.
\end{remark}

By definition, the groupoidification functor $\I{S}\to\I{S}^{\gp}$ in Definition~\ref{def:localizationAtSomeS} is an iterated localisation. One therefore finds:
\begin{proposition}
	For any $\BB$-category $\I{C}$ and any functor $\I{S}\to \I{C}$, the localisation functor $\I{C}\to \I{S}^{-1}\I{C}$ is an iterated localisation.\qed
\end{proposition}
\begin{lemma}
	\label{lem:strongEpimorphismDomainGroupoid}
	Let $\I{G}$ be a $\BB$-groupoid and let $\I{G}\to\I{C}$ be a strong epimorphism in $\Cat(\BB)$. Then $\I{C}$ is a $\BB$-groupoid as well.
\end{lemma}
\begin{proof}
	Since $\I{G}$ is a $\BB$-groupoid, Corollary~\ref{cor:functorIntoCore} implies that the functor $\I{G}\to\I{C}$ factors through the inclusion $\I{C}^{\core}\into\I{C}$. We may therefore construct a lifting problem
	\begin{equation*}
		\begin{tikzcd}
			\I{G}\arrow[r]\arrow[d] & \I{C}^\core\arrow[d, hookrightarrow]\\
			\I{C}\arrow[r, "\id"]\arrow[ur, dashed] & \I{C}
		\end{tikzcd}
	\end{equation*}
	which admits a unique solution. Hence the identity on $\I{C}$ factors through $\I{C}^\core\into\I{C}$, which evidently implies that $\I{C}^\core\into\I{C}$ is already an equivalence.
\end{proof}
\begin{lemma}
	\label{lem:pushoutStrongEpimorphism}
	For any strong epimorphism $f\colon\I{C}\to\I{D}$ in $\Cat(\BB)$, the commutative square
	\begin{equation*}
		\begin{tikzcd}
			\I{C}\arrow[r]\arrow[d, "f"] & \I{C}^\gp\arrow[d, "f^\gp"]\\
			\I{D}\arrow[r] & \I{D}^{\gp}
		\end{tikzcd}
	\end{equation*}
	is cocartesian.
\end{lemma}
\begin{proof}
	If $\I{P}=\I{D}\sqcup_{\I{C}} {\I{C}^{\gp}}$ denotes the pushout, we need to show that the induced functor $g\colon\I{P}\to\I{D}^{\gp}$ is an equivalence. Since iterated localisations are stable under pushout, the map $\I{D}\to\I{P}$ is an iterated localisation, which (by the left cancellation property) implies that $g$ must be an iterated localisation as well. We therefore only need to show that $g$ is conservative. Since $\I{D}^{\gp}$ is a $\BB$-groupoid, this is equivalent to $\I{P}$ being a $\BB$-groupoid as well~\cite[Corollary~4.1.17]{martini2021}. But since strong epimorphisms are also preserved by pushouts, the map $\I{C}^{\gp}\to\I{P}$ is a strong epimorphism, hence Lemma~\ref{lem:strongEpimorphismDomainGroupoid} implies the claim.
\end{proof}
\begin{proposition}
	\label{prop:invarianceLocalisationStrongEpimorphism}
	Let $f\colon\I{S}\to\I{T}$ and $g\colon \I{T}\to\I{C}$ be functors in $\Cat(\BB)$, and suppose that $f$ is a strong epimorphism. Then the induced functor $\I{S}^{-1}\I{C}\to\I{T}^{-1}\I{C}$ is an equivalence.
\end{proposition}
\begin{proof}
	This is an immediate consequence of the pasting lemma for pushout squares and Lemma~\ref{lem:pushoutStrongEpimorphism}.
\end{proof}
\begin{remark}
	\label{rem:localisationsSubcategories}
	Proposition~\ref{prop:invarianceLocalisationStrongEpimorphism} implies that when considering localisations of a $\BB$-category $\I{C}$, we may restrict our attention to \emph{subcategories} $\I{S}\into\I{C}$ instead of general functors, as we can always factor a functor $\I{S}\to\I{C}$ into a strong epimorphism followed by a monomorphism. Alternatively, by making use of the strong epimorphism $\Delta^1\otimes \I{S}_0\to \I{S}$ from Lemma~\ref{lem:universalMapStrongEso}, we can always assume that $\I{S}$ is of the form $\Delta^1\otimes A$ for some $A\in\BB$. 
\end{remark}

Let $f\colon\I{C}\to\I{D}$ be a functor between $\BB$-categories. Let $f^\ast\I{D}^{\core}\into\I{C}$ be the subcategory that is defined by the pullback square
\begin{equation*}
	\begin{tikzcd}
		f^\ast\I{D}^{\core}\arrow[d, hookrightarrow]\arrow[r] \arrow[dr, phantom, "\ulcorner", very near start] & \I{D}^{\core}\arrow[d, hookrightarrow]\\
		\I{C}\arrow[r, "f"] & \I{D}.
	\end{tikzcd}
\end{equation*}
Since $\I{D}^{\core}$ is a $\BB$-groupoid, the map $f^\ast\I{D}^{\core}\to \I{D}^{\core}$ factors through $f^\ast\I{D}^{\core}\to (f^\ast\I{D}^{\core})^{\gp}$. Consequently, one obtains a factorisation of $f$ into the composition
\begin{equation*}
	\I{C}\to (f^\ast\I{D}^{\core})^{-1}\I{C}\xrightarrow{f_1} \I{D}.
\end{equation*}
Let us set $\I{C}_1=(f^\ast\I{D}^{\core})^{-1}\I{C}$.
By replacing $\I{C}$ by $\I{C}_1$ and $f$ by $f_1$ and iterating this procedure, we obtain an $\mathbb N$-indexed diagram
in $\Over{\Cat(\BB)}{\I{D}}$. Let $f_{\infty}\colon\I{E}\to\I{D}$ denote the colimit of this diagram. By construction, the map $f$ factors into the composition $\I{C}\to\I{E}\to\I{D}$ in which the first map is a countable composition of localisations and therefore an iterated localisation in the sense of Definition~\ref{defn:conservativeIteratedlocalisation}. We claim that the map $f_\infty$ is conservative. To see this,  consider the cartesian square
\begin{equation*}
	\begin{tikzcd}
		f_{\infty}^\ast\I{D}^{\core}\arrow[r]\arrow[d]\arrow[dr, phantom, "\ulcorner", very near start] & \I{D}^{\core}\arrow[d]\\
		\I{E}\arrow[r, "f_\infty"]& \I{D}.
	\end{tikzcd}
\end{equation*}
On account of filtered colimits being universal in $\Cat(\BB)$ (see Proposition~\ref{prop:CatBPresentable}), we obtain an equivalence $f_\infty^\ast\I{D}^{\core}\simeq \colim_n f_n^\ast\I{D}^{\core}$. By construction, the categories $f_n^\ast\I{D}^{\core}$ sit inside the $\mathbb N$-indexed diagram
\begin{equation*}
	\cdots \to f_{n-1}^\ast\I{D}^{\core}\to (f_{n-1}^\ast\I{D}^{\core})^{\gp}\to f_{n}^\ast\I{D}^{\core}\to (f_{n}^\ast\I{D}^{\core})^{\gp}\to f_{n+1}^\ast\I{D}^{\core}\to (f_{n+1}^\ast\I{D}^{\core})^{\gp}\to\cdots
\end{equation*}
such that the functor $\cdot 2\colon \mathbb N\to\mathbb N$ that is given by the inclusion of all even natural numbers recovers the $\mathbb N$-indexed diagram $n\mapsto f_{n}^\ast\I{D}^{\core}$ that is defined by the cartesian square above. As both the inclusion of all even natural numbers and that of all odd natural numbers define final functors $\mathbb N\to\mathbb N$, we conclude that $f_{\infty}^\ast\I{D}^{\core}$ is obtained as the colimit of the diagram $n\mapsto (f_{n}^\ast\I{D}^{\core})^{\gp}$ and is therefore a \emph{groupoid} in $\BB$. Applying~\cite[Corollary~4.1.16]{martini2021}, this shows that $f_{\infty}$ is conservative. Therefore the factorisation of $f$ into the composite $\I{C}\to\I{E}\to\I{D}$ as constructed above is the unique factorisation of $f$ into an iterated localisation and a conservative functor. Applying this construction when $f$ is already an iterated localisation, one in particular obtains:
\begin{proposition}
	Every iterated localisation between $\BB$-categories is obtained as a countable composition of localisation functors.\qed
\end{proposition}

Our next goal is to prove the universal property of a localisation functor. To that end, given any two $\BB$-categories $\I{C}$ and $\I{D}$ and any functor $\I{S}\to \I{C}$, note that as the base change functor $\pi_A^\ast\colon \Cat(\BB)\to\Cat(\Over{\BB}{A})$ from Remark~\ref{rem:functorialityBCategories} preserves the internal hom $\iFun(-,-)$~\cite[Lemma~4.2.3]{martini2021}, an object of $\iFun(\I{C},\I{D})$ in context $A\in\BB$ is precisely given by a functor of $\Over{\BB}{A}$-categories $\pi_A^\ast\I{C}\to\pi_A^\ast\I{D}$. Therefore, the collection of functors $\pi_A^\ast\I{C}\to\pi_A^\ast\I{D}$ in arbitrary context $A\in\BB$ whose restriction along $\pi_A^\ast\I{S}\to\pi_A^\ast\I{C}$ factors through $\pi_A^\ast\I{D}^{\core}$ span a full subcategory of $\iFun(\I{C},\I{D})$ (see \S~\ref{sec:fullyFaithfulFunctors}) that we denote by $\iFun^{\I{S}}(\I{C},\I{D})$.

\begin{remark}[locality of $\iFun^{\I{S}}(\I{C},\I{D})$]
	\label{rem:invertingSubcategoryLocalCondition}
	Note that a functor $f\colon\pi_A^\ast \I{S}\to\pi_A^\ast\I{D}$ factors through $\pi_A^\ast\I{D}^\core$ if and only if the transposed map $A\times\I{S}\to\I{D}$ factors through $\I{D}^\core$. As the map $\I{D}^\core\into\I{D}$ is a monomorphism by Example~\ref{ex:coreGroupoidMonomorphism}, the same argument as in Example~\ref{ex:representabilityLocalCondition} shows that this condition is \emph{local}, in the sense that for every cover $(s_i)\colon\bigsqcup_i A_i\onto A$ in $\BB$, the functor $f$ factors through $\pi_A^\ast\I{D}^\core$ if and only if each of the functors $s_i^\ast(f)$ factors through $\pi_{A_i}^\ast\I{D}^\core$. As a consequence, \emph{every} object $A\to\iFun^{\I{S}}(\I{C},\I{D})$ encodes a functor $\pi_A^\ast\I{C}\to\pi_A^\ast\I{D}$ whose restriction along $\pi_A^\ast \I{S}\to\pi_A^\ast\I{C}$ factors through $\pi_A^\ast\I{C}^\core$. In conjunction with~\cite[Lemma~4.2.3]{martini2021}, this observation furthermore implies that there is a canonical equivalence $\pi_A^\ast\iFun^{\I{S}}(\I{C},\I{D})\simeq\iFun[\Over{\BB}{A}]^{\pi_A^\ast\I{S}}(\pi_A^\ast\I{C},\pi_A^\ast\I{D})$ for every $A\in\BB$, cf.~Remark~\ref{rem:localityPrincipleBaseChangeProposition}.
\end{remark}

\begin{remark}
	By Corollary~\ref{cor:functorIntoCore} and Remark~\ref{rem:invertingSubcategoryLocalCondition}, a functor $\pi_A^\ast\I{C}\to\pi_A^\ast\I{D}$ defines an object in $\iFun^{\I{S}}(\I{C},\I{D})$ precisely if its restriction along $\pi_A^\ast\I{S}\to\pi_A^\ast\I{C}$ sends every map in $\pi_A^\ast\I{S}$ to an equivalence in $\pi_A^\ast\I{C}$.
\end{remark}

\begin{proposition}
	\label{prop:universalPropertyLocalisation}
	Let $\I{C}$ be a $\BB$-category and let $\I{S}\to\I{C}$ be a functor. Then precomposition with the localisation functor $L\colon\I{C}\to \I{S}^{-1}\I{C}$ induces an equivalence
	\begin{equation*}
		L^\ast\colon\iFun(\I{S}^{-1}\I{C},\I{D})\simeq \iFun^{\I{S}}(\I{C},\I{D})
	\end{equation*}
	for any $\BB$-category $\I{D}$.
\end{proposition}
\begin{proof}
	By applying the functor $\iFun(-,\I{D})$ to the pushout square that defines the localisation of $\I{C}$ at $\I{S}$, one obtains a pullback square
	\begin{equation*}
		\begin{tikzcd}
			\iFun(\I{S}^{-1}\I{C},\I{D})\arrow[d]\arrow[r] & \iFun(\I{C},\I{D})\arrow[d]\\
			\iFun(\I{S}^{\gp},\I{D})\arrow[r] & \iFun(\I{S},\I{D}).
		\end{tikzcd}
	\end{equation*}
	We claim that the two horizontal functors are fully faithful. To see this, it suffices to consider the lower horizontal map. This is a fully faithful functor precisely if it is internally right orthogonal to the map $\Delta^0\sqcup\Delta^0\to \Delta^1$, and by making use of the adjunction between tensoring and powering in $\Cat(\BB)$, one sees that this is equivalent to the induced functor $\I{D}^{\Delta^1}\to \I{D}\times \I{D}$ being internally right orthogonal to the map $\I{S}\to\I{S}^{\gp}$. Hence it suffices to show that the functor $\I{D}^{\Delta^1}\to \I{D}\times \I{D}$ is conservative, i.e.\ internally right orthogonal to $\Delta^1\to\Delta^0$. Making use of the adjunction between tensoring and powering in $\Cat(\BB)$ once more, this is seen to be equivalent to $\I{D}$ being internally local with respect to the map $K\to \Delta^1$ that is defined by the commutative diagram
	\begin{equation*}
		\begin{tikzcd}[column sep=huge]
			\Delta^1\sqcup\Delta^1\arrow[r, "{(d_1\times\id,d_0\times\id)}"] \arrow[d, "s_0\sqcup s_0"']\arrow[dr, phantom, "\lrcorner", very near end]& \Delta^1\times\Delta^1\arrow[d]\arrow[ddr, bend left, "\pr_1"]& \\
			\Delta^0\sqcup\Delta^0\arrow[r]\arrow[drr, bend right, "{(d_1, d_0)}"] & K \arrow[dr] & \\
			& & \Delta^1 
		\end{tikzcd}
	\end{equation*}
	in which $\pr_1$ denotes the projection onto the second factor. By the same reasoning as in the proof of~\cite[Lemma~3.8.8]{martini2021}, the map $K\to \Delta^1$ is an equivalence in $\Cat(\BB)$, hence the claim follows.
	
	Since for any $A\in\BB$ a functor $\pi_A^\ast\I{S}\to\pi_A^\ast\I{D}$ factors through $\pi_A^\ast\I{D}^{\core}$ if and only if it factors through the map $\pi_A^\ast\I{S}\to \pi_A^\ast\I{S}^{\gp}$, one obtains a commutative square
	\begin{equation*}
		\begin{tikzcd}
			\iFun^{\I{S}}(\I{C},\I{D})\arrow[d]\arrow[r, hookrightarrow] & \iFun(\I{C},\I{D})\arrow[d]\\
			\iFun(\I{S}^{\gp},\I{D})\arrow[r, hookrightarrow] & \iFun(\I{S},\I{D}).
		\end{tikzcd}
	\end{equation*}
	and therefore a map $\iFun^{\I{S}}(\I{C},\I{D})\into \iFun(\I{S}^{-1}\I{C},\I{D})$. Since every object $A\to \iFun(\I{S}^{-1}\I{C},\I{D})$ by definition gives rise to an object in $\iFun^{\I{S}}(\I{C},\I{D})$, this map must also be essentially surjective and is thus an equivalence.
\end{proof}

\bibliographystyle{halpha}
\bibliography{references.bib}
\end{document}